\documentclass[final]{siamltex}
\usepackage{xcolor}
\usepackage{amsmath}
\usepackage{amssymb}
\usepackage{mathrsfs}
\usepackage{graphicx}
\usepackage{hyperref}
\usepackage{wrapfig}
\usepackage{color}
\usepackage{float}
\usepackage{epstopdf}

\usepackage{amsmath,amstext,amssymb,graphicx,hyperref}
\usepackage{verbatim}
\usepackage{color}
\usepackage{algorithm,algorithmic}

\usepackage{mathtools}

\usepackage{epstopdf}
\usepackage{color}
\usepackage{graphicx}
\usepackage{float}
\usepackage[section]{placeins}

\graphicspath{{figs/}}

\usepackage{pdflscape}

\usepackage{amsopn}
\usepackage{lipsum}
\usepackage{amsfonts}
\usepackage{graphicx}
\usepackage{algorithmic}
\usepackage{mathrsfs}
\usepackage{array} 
\usepackage{color}
\usepackage{url}
\usepackage[normalem]{ulem}
\usepackage{algorithm, algorithmic}
\usepackage{relsize}

\usepackage{enumitem}

\usepackage{subfigure, graphicx}
\usepackage{multicol}
\usepackage{multirow}
\usepackage{rotating}
\usepackage{arydshln}

\usepackage{hyperref}
\usepackage{microtype}
\usepackage{booktabs}
\usepackage{float}

\usepackage{cite}

\makeatletter
\def\hlinewd#1{%
  \noalign{\ifnum0=`}\fi\hrule \@height #1 \futurelet
   \reserved@a\@xhline}
\makeatother

\newcommand{\mcrot}[4]{\multicolumn{#1}{#2}{\rlap{\rotatebox{#3}{#4}~}}} 
\newcommand*{\twoelementtable}[3][l]%
{%
    \renewcommand{\arraystretch}{0.8}%
    \begin{tabular}[t]{@{}#1@{}}%
        #2\tabularnewline
        #3%
    \end{tabular}%
}

\definecolor{mygreen1}{rgb}{0, .5, 0}
\definecolor{myred1}{rgb}{.5, 0, 0}

\begin{document}

\title{Finite Differences in Forward and Inverse Imaging Problems--MaxPol Design}
\author{Mahdi S. Hosseini and Konstantinos N. Plataniotis}
\maketitle
\begin{abstract}
A systematic and comprehensive framework for finite impulse response (FIR) lowpass/fullband derivative kernels is introduced in this paper. Closed form solutions of a number of derivative filters are obtained using the maximally flat technique to regulate the Fourier response of undetermined coefficients. The framework includes arbitrary parameter control methods that afford solutions for numerous differential orders, variable polynomial accuracy, centralized/staggered schemes, and arbitrary side-shift nodes for boundary formulation. Using the proposed framework four different derivative matrix operators are introduced and their numerical stability is analyzed by studying their eigenvalues distribution in the complex plane. Their utility is studied by considering two important image processing problems, namely gradient surface reconstruction and image stitching. Experimentation indicates that the new derivative matrices not only outperform commonly used method but provide useful insights to the numerical issues in these two applications.
\end{abstract}

\section{Introduction}
Finite difference (FD) methods are used extensively in variety of image processing tasks. For example, they are used in digital approximation of order moments such as gradients, Hessian, and high-order tensor. They find numerous applications in either forward imaging problems to encode certain features such as edge and corner \cite{Canny1986, MokhtarianMackworth1986, MarrHildreth1980, PeronaMalik1990, lindeberg1998feature, lowe2004distinctive, JalbaWilkinsonRoerdink2006}, or discretizing the differential operators used in inverse problems for image restoration \cite{buades2005review, 7776871, EstellersSoatto2016, Papafitsoros2014, mecca2016single, 5887417, harker2015regularized, agrawal2006range}. Despite extensive applications, the problem of numerical differentiation suffers from several deficiencies such as the lack of estimation accuracy, sensitivity over perturbing artifacts, and improper formulation of boundary condition (BC). Such deficiencies can easily make the overall imaging applications highly complicated (ill-posed) and sometimes turn the whole exercise to be useless. Recent works highlight the need for a systematic and comprehensive framework of numerical differentiation to overcome these issues by reducing the complexities of error propagation in the system design and furthermore prevent computational burdens. Applications can be found in gradient surface reconstruction \cite{harker2015regularized, gambaruto2015processing}, edge detection \cite{Belyaev2013, aprovitola2014edge, bustacara2016comparison, aprovitola2016knee}, edge synthesis \cite{6319316, 7744595, elad2016style}, feature extraction \cite{DelibasisKechriniotisMaglogiannis2013, DelibasisKechriniotis2014, gambaruto2015processing}, and many more. This article concerns with the main question of what is the best numerical approach to approximate derivatives in an image processing related task? With this question in mind we present a comprehensive framework for differentiation applied in forward and inverse imaging problems. The \textit{``forward''} here is referred to find a solution of direct approximation in a measurement system $A(X)=Y$ where $A$ is an operator and $X$ is a given signal/image to find $Y$. While, the \textit{``inverse''} is referred to inferring $X$ from a similar system where $A$ and $Y$ are known.

\subsection{Properties expected in inverse imaging problems}\label{subsec_inverse_properties}
Often in many inverse problems the image of interest is restored from a regulated partial differential equation (PDE) problem such as in photometric stereo \cite{agrawal2006range, harker2015regularized}, surface rendering \cite{EstellersSoatto2016}, magnetic resonance imaging (MRI) \cite{5887417}, phase unwrapping \cite{bioucas2007phase}, and variational regularization \cite{bredies2010total, Papafitsoros2014}. The discrete approximation of the related PDE problems is usually accommodated by a fullband FIR differentiation which needs to satisfy the following criteria:

\begin{itemize}[leftmargin=*]
\item The filter response of the derivative coefficients should remain close to the ideal derivative response. The well known centralized and staggered formulations preserve this condition for even and odd order of differentiations, respectively \cite{fornberg1988, S0036144596322507}. Any deviation (miss-match) between ideal and discrete filter response can lead in to a perturbation in recovery. Example of such deviation is shown in \ref{fig_staggered_vs_centralized} for the first order FIR derivative coefficient.
\item The accuracy of discrete differentiation should be high in both interior and boundary domains of the signal/image to preserve the derivative continuity and hence to avoid recovery artifacts. High accuracy within the interior domain can be achieved by increasing the filter polynomial degree. However, the relative accuracy on the boundaries downgrades as a side effect. To elaborate the point consider the reflective (Neumman) and anti-reflective BCs \cite{trefethen2005spectra}, where they  preserve up to zero $\mathcal{C}^{0}$ and first $\mathcal{C}^{1}$ order continuities, respectively. Meaning that none of them suit higher degree polynomial formulation. This has gained recent attention in finite difference formulations \cite{Li2005, o2008algebraic, o2012framework, HassanMohamadAtteia2012} to design a derivative matrix that embed high degree polynomial accuracies.
\item The derivative coefficients should be obtained from a closed mathematical formulation to minimize the discrete round-off errors. Methods such as recursive formulations \cite{fornberg1988, sadiq2014finite} or non-simplified formulations \cite{Li2005, o2008algebraic, o2012framework, HassanMohamadAtteia2012} lead into numerical error/instabilities.
\item The numerical design should be capable of addressing arbitrary order of derivatives (OD). This stems from the fact that many PDE models are not limited to the first OD but high order estimations such as in variational regularizers \cite{bredies2010total, Papafitsoros2014}.
\end{itemize}

\begin{figure}[htp]
\centerline{
\subfigure[impulse response]{\includegraphics[height=0.24\textwidth]{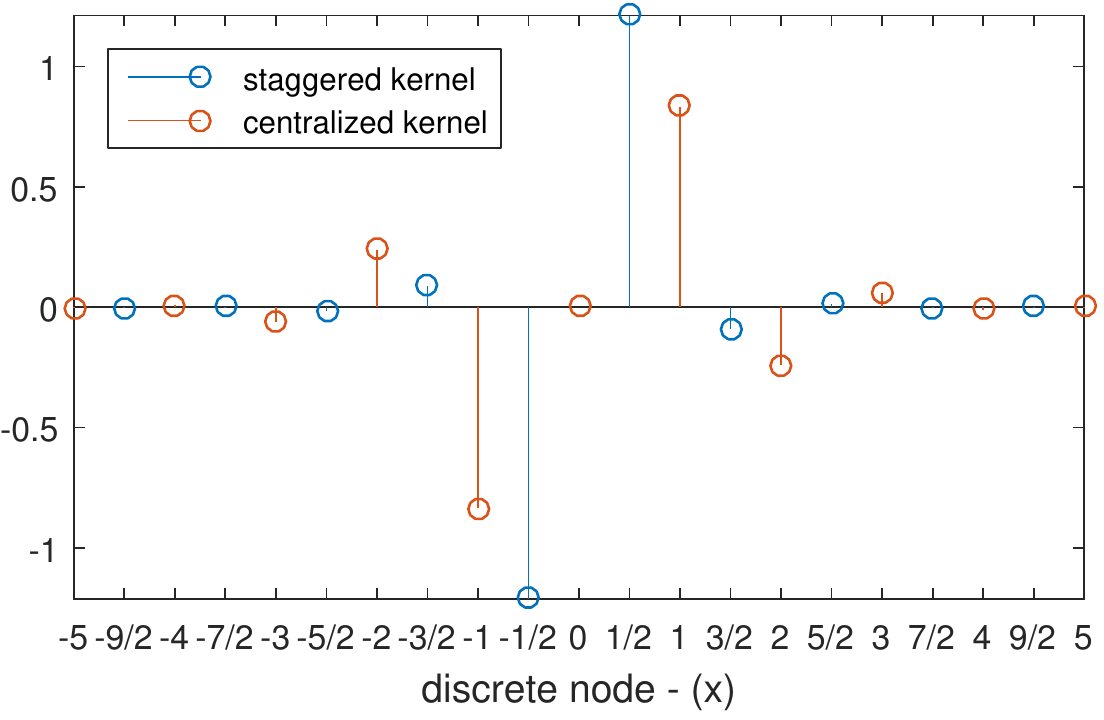}}
\subfigure[filter response]{\includegraphics[height=0.24\textwidth]{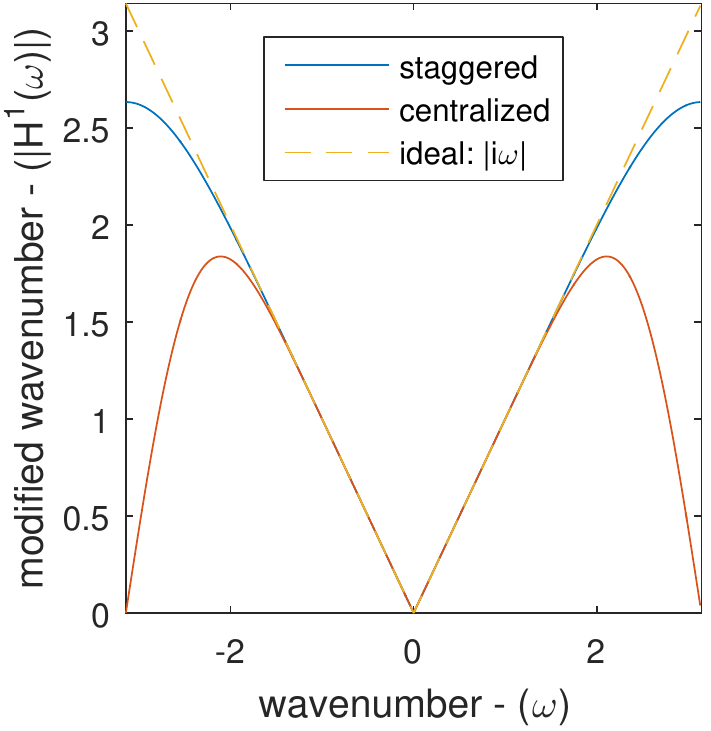}}
}
\caption{First order derivative fullband FIR coefficient: staggered and centralized schemes with high-order-accuracy polynomial of $10$. The courtesy of Fornberg derivative coefficients \cite{S0036144596322507}.}
\label{fig_staggered_vs_centralized}
\end{figure}

\subsection{Properties expected in forward imaging problems}\label{subsec_forward_properties}
The main purpose of image differentiation in forward problems is to approximate feature of interests from a given image. This is usually accommodated by convolution of an image with an FIR derivative kernel in a linear fashion. Examples include but not limited to gradient approximation for edge detection \cite{MarrHildreth1980, Canny1986, PeronaMalik1990} and local feature description \cite{lindeberg1998feature, lowe2004distinctive}, and Hessian approximation for curvature estimation \cite{MokhtarianMackworth1986, JalbaWilkinsonRoerdink2006}, optical flow \cite{UrasGirosiVerriTorre1988}, and image sharpening/enhancement \cite{PoleselRamponi2000,7563313}. A disadvantage common to all methods is the input image is contaminated with noise where a careful treatment is required to accurately estimate derivatives while suppressing the noise effect. In addition to the criteria listed in \ref{subsec_inverse_properties}, the following conditions are also necessary for proper lowpass differentiation:
\begin{itemize}[leftmargin=*]
\item A wide range of cutoff frequency delivers flexibility for costume design filter. Sampling rate, signal-to-noise-ratio (SNR), and frequency spectrum of the signal (assumed to be band-limited here) are the three main principles that define the reasonable cutoff range for differentiation. If the rate of sampling exceeds the maximum signal frequency, then proper cutoff should be set to avoid aliasing and noise amplification. Higher cutoffs are usually needed for image texture approximation that contains high frequency information.
\item Once a cutoff frequency is preset to design a lowpass differentiation, the filter's magnitude response should decay rapidly to mitigate the aliasing artifacts. 
\item The filter should not contain any perturbing residues after the cutoff level. However in most cases they are usually represented as ripples (side-lobes) such as in least square designs \cite{kaiser1977sharpening, McDevitt:2012, burch2005least, luo2005properties, SavitzkyGolay1964, gorry1990general, luo2005properties, schafer2011savitzky} or slow decay responses in smoothing kernels such as derivative-of-Gaussian (DoG) filters \cite{MarrHildreth1980, freeman1991design, deriche1993recursively}. Such residuals are shown in \ref{fig_lowpass_concept}.
\end{itemize}

\begin{figure}[htp]
\centerline{
\subfigure[Impulse Response]{\includegraphics[height=0.24\textwidth]{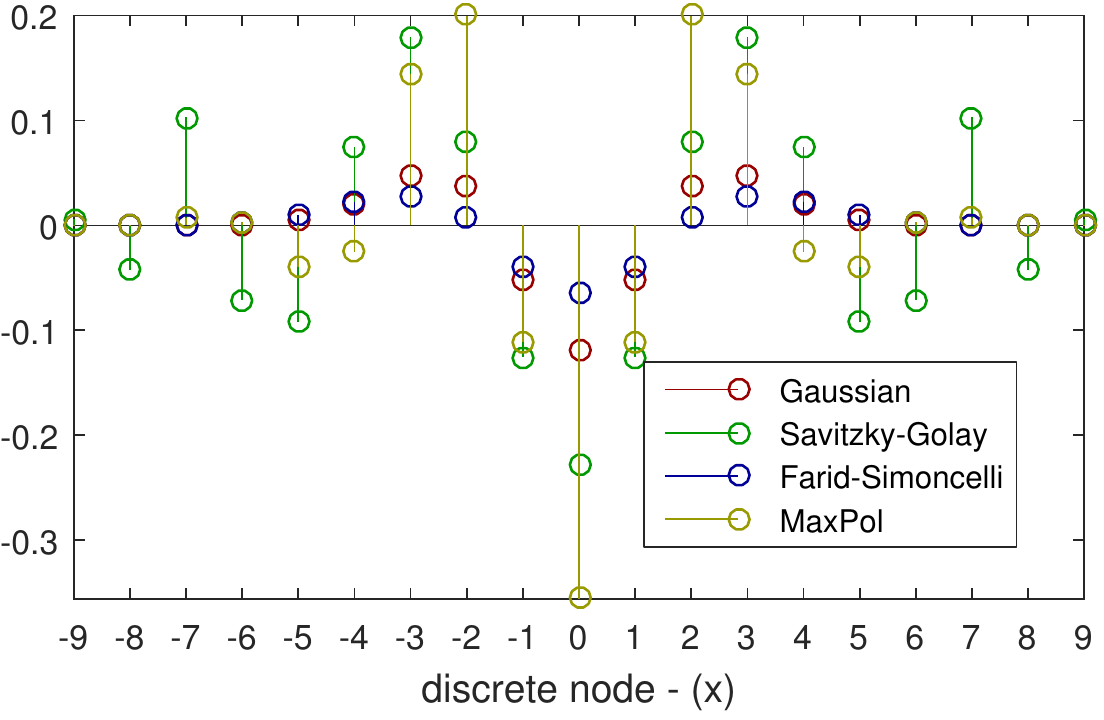}}
\subfigure[Filter response (linear)]{\includegraphics[height=0.24\textwidth]{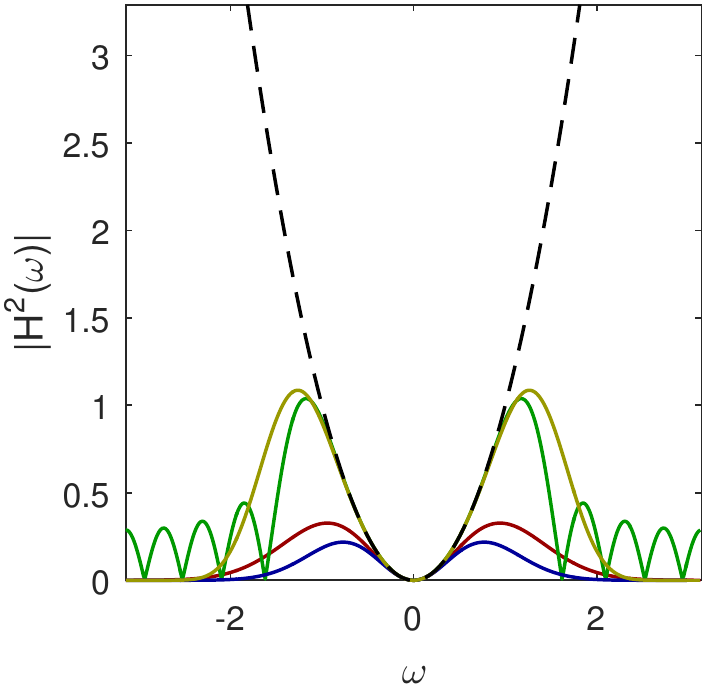}}
\subfigure[Filter response (log)]{\includegraphics[height=0.24\textwidth]{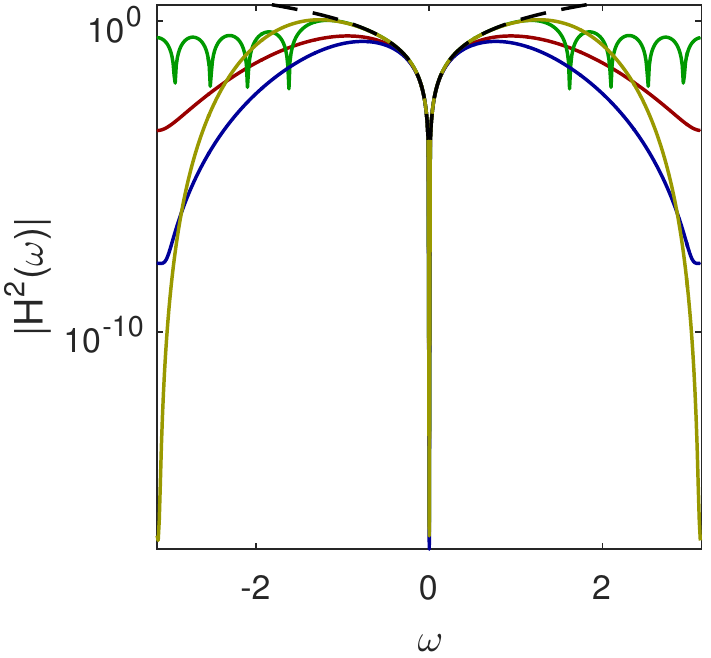}}
}
\caption{Second order lowpass FIR derivative of Gaussian, Savtizky-Golay \cite{luo2005properties}, Farid-Simoncelli \cite{FaridSimoncelli2004}, and MaxPol (a variant of proposed design). The tap-length of the filters is $19$. The Gaussian variance $\sigma=1.5$ and the cutoff frequency of SavitzkyGolay and MaxPol are set around half Nyquist band $\omega_c=\frac{\pi}{2}$.}
\label{fig_lowpass_concept}
\end{figure}



\subsection{Related works in explicit finite differentiation}\label{sec_related_works_FD}
Based on the two categories of forward and inverse problems discussed in previous sections, we divide the topic of numerical differentiation to fullband and lowpass filters as follows.

One of the early works of the finite difference method has been developed by Savitzky-Golay (SG) \cite{SavitzkyGolay1964} and it has been revised in several occasions \cite{gorry1990general, luo2005properties, schafer2011savitzky}. The design approach is based on the power polynomials regulated by the least square energy minimization where the solution fits a Vandermonde matrix system of equations. Associated by the highest accuracy design, one can extract the fullband response of the derivatives up to any order. However, the numerical solution of the SG filter is highly affected by the ill-conditioned property of the Vandermonde matrix that limits the degree of filter polynomial. A better resolution was made later by Fornberg \cite{fornberg1988, S0036144596322507} to find a recursive formulation using the Lagrangian interpolating polynomials. The latter methods are well  developed in literature that offer centralized and staggered schemes including the side-shift nodes to formulate the derivative at boundaries. Recent studies also address similar problem by introducing a closed form solutions such as the Lagrangian method in \cite{Li2005} and Taylor series expansion in \cite{HassanMohamadAtteia2012}. They propose more suitable computational approaches to overcome the problem of inverse Vandermonde calculation. Despite the effectiveness of the numerical solution, they are only limited to the centralized scheme. Another variant of such design is the orthogonal polynomial expansion in \cite{o2008algebraic, o2012framework} accommodated by QR factorization which is limited to only first order centralized differentiation. Khan et al \cite{RasoolKhan1999, RasoolKhan1999_2, RasoolKhan2000, RasoolKhan2003, RasoolKhan2003_2} proposed a comprehensive numerical solution for both centralized and staggered schemes based on the Taylor expansion. However, the method lacks of designing side-shift nodes for boundary derivative approximation and involves high computational complexity. Methods based on the MaxFlat technique also exist in \cite{Carlsson1991, KumarRoy1988, KumarRoy1989, KumarRoy1992, Selesnick2002} where they are mostly limited to the first order derivative design and do not formulate the boundary conditions.

The idea in lowpass differentiation is to allow certain range of frequency band to be transformed under the ideal derivative response and truncate the remaining frequencies by setting a cutoff threshold. The pioneering work done by Lanczos \cite{lanczosapplied, lanczos1988applied} designs lowpass responses by truncating the Fourier series coefficients and minimizing the energy of Gibbs oscillation in the least square sense. Different extensions of this work have been also reported in \cite{duchon1979lanczos, groetsch1998lanczo, burch2005least, McDevitt:2012} using Lagrange and Legendre polynomials. A more generalized design in least square sense would be the SG filters \cite{SavitzkyGolay1964, gorry1990general, meer1990smoothed, luo2005properties, schafer2011savitzky} that have been applied in many engineering problems. Though, variety of derivative orders with flexible cutoff ranges can be designed, they contain side-lobe artifacts on the stop-band. Less perturbing designs are also made by sinc interpolations \cite{kaiser1977sharpening, kaiser1977data} and Jacobi polynomials \cite{mboup2009numerical}. However they contain limited cutoff range and derivative orders. Wider range of derivative orders are also studied in \cite{UnserAldroubiEdenI1993, UnserAldroubiEdenII1993} using B-spline method, sinc interpolation methods in \cite{Simoncelli1994, FaridSimoncelli1997, FaridSimoncelli2004}, cubic interpolating polynomials in \cite{MollerMachiraju1997, moller1998design, moller1999spatial}, Gaussian based kernels in \cite{MarrHildreth1980, freeman1991design, deriche1993recursively}, and Legendre polynomial methods known as the chromatic derivatives \cite{ignjatovic2000signal, ignjatovic2001method, ignjatovic2009chromatic, ignjatovic2011multidimensional}. A common disadvantage to all of them is they provide limited cutoff range. Perhaps a more sophisticated approach can be addressed by Selesnick et al \cite{SelesnickBurrus1998, SelesnickBurrus1998TSP, SelesnickBurrus1999, Selesnick2002} which provides wide cutoff range with free-residual artifacts in the stop-band. The method is based on the maximally flat technique design and offers wide range of FIR filters up to only first order differentiation.

\subsection{Shortcomings and contributions}
The overview of our discussion on the finite difference methods is summarized in \ref{table_state_of_the_art_ND}. Despite various empirical methods developed for fullband differentiation, it appears that there is a lacking design of comprehensive derivative matrix that can handle both centralized and staggered schemes in arbitrary polynomial degree, different variations of derivative orders, and high accuracy BC formulation. The existing methods offer limited range of solutions such as centralized formulations \cite{Li2005, o2008algebraic, o2012framework, HassanMohamadAtteia2012} or pseudo-spectral approaches with limited accuracy designs \cite{welfert1997generation}. As we show later, the centralized scheme is numerically unstable for odd order differentiation. In contrast, odd order derivative matrices with staggered nodes are favorably applicable in problems of first order PDEs with much stable behaviors. Moreover, most lowpass differentiators are not fully guaranteed to satisfy all three main conditions together mentioned in \ref{subsec_forward_properties}. While many filters are incapable of offering flexible cutoff range such as in \cite{UnserAldroubiEdenI1993, UnserAldroubiEdenII1993, Simoncelli1994, FaridSimoncelli1997, FaridSimoncelli2004}, the existing ones with wide range solutions such as Savitzky-Golay and Lanczos contain residual artifacts that are closely tied with visual perception error in many imaging applications.

Motivated by the above discussions, our aim in this paper is to tackle the unresolved problems of finite difference method mainly from image processing perspective. Since both fullband and lowpass filters rise in different applications, it would be convenient to derive both filters from a generalized framework that can address diverse feature utilities similar to the SG approach. To this end, we propose a consensus model using the method of undetermined coefficients in conjunction with maximally flat technique to achieve a desired filter response. In a nutshell, the contributions made in this paper are as follows
\begin{itemize}[leftmargin=*]
\item A generalized framework called ``MaxPol'' is proposed to encode FIR coefficients of fullband/lowpass derivatives with comprehensive utilities including arbitrary degrees of polynomial, staggered/centralized stencil, wide cutoff frequency range, and arbitrary side-shift nodes. The derivation of fullband coefficients is accommodated by finding a closed form solution to the inverse of a Vandermonde matrix using the partial fraction technique. In the case of lowpass differentiation, the closed form is obtained by a symbolic calculation, providing a reasonable range of tap-length filters for computation.
\item Four different categories of derivative matrices are designed for tensor image decomposition that can be associated in many PDE problems for numerical calculations. The stability of the matrices is studied via the distribution of eigenvalues to provide a wide range of solutions with different parameter control design.
\item The utility of the MaxPol package is first studied on gradient surface recovery to validate the impact of different design matrices in inverse imaging problem. Image stitching is the second application studied in this paper to show the efficiency of the lowpass differentiation in conjunction with gradient surface problem for seamless image reconstruction. The numerical evaluations outrank the efficiency of the proposed method and highlight its versatility with respect to robustness and accuracy. 
\end{itemize}

Our early work on designing derivative kernels has been reported in \cite{hosseini2017derivative} to design lowpass filters based on the maximally flat technique with extended design to 2D kernels for image directional differentiation. This paper complements \cite{hosseini2017derivative} by further design of side-shift nodes and derivative matrices with high accuracy boundary formulation. \ref{table_state_of_the_art_ND} compares the additional numerical features added in this paper compared to \cite{hosseini2017derivative}. Also, from application point of view, the utility of the MaxPol design here mostly concentrates on the inverse imaging problems. Whereas, in  \cite{hosseini2017derivative} this was  mainly focused on the forward imaging problems. 

The remainder of this paper is organized as follows. The numerical foundations are introduced in \ref{sec_maxpol_numerics}. The construction of derivative matrices are explained in \ref{section_convolution_mtx} and the numerical validations are reported in \ref{sec_fullband_differentiation}.  Applications in the context of gradient surface reconstruction are given in \ref{sec_gradient_surface} . Finally, the paper is concluded in \ref{sec_conclusion}.

\begin{landscape}
\begin{table*}[htp]
\renewcommand{\arraystretch}{1.3}
\caption{List of numerical methods for explicit finite difference formulation with high degree polynomials}
\label{table_state_of_the_art_ND}
\centering
\scriptsize
\begin{tabular}{lccccccccccccccccccccccl}
\hlinewd{1pt}
Author & Year &
\mcrot{1}{c}{45}{Power polynomial} &
\mcrot{1}{c}{45}{Lagrange polynomial} &
\mcrot{1}{c}{45}{Legendre polynomial} &
\mcrot{1}{c}{45}{Jacobi polynomial} &
\mcrot{1}{c}{45}{Orthogonal polynomial} &
\mcrot{1}{c}{45}{Spline basis} &
\mcrot{1}{c}{45}{Sinc basis} &
\mcrot{1}{c}{45}{Undetermined coefficient}  &
\mcrot{1}{c}{45}{Taylor series expansion} &
\mcrot{1}{c}{45}{Least square minimization} &
\mcrot{1}{c}{45}{MaxFlat criteria} &
\mcrot{1}{c}{45}{Closed form solution} &
\mcrot{1}{c}{45}{Centralized node} &
\mcrot{1}{c}{45}{Staggered node} &
\mcrot{1}{c}{45}{Side-shift node} &
\mcrot{1}{c}{45}{Fullband design} &
\mcrot{1}{c}{45}{Lowpass design} &
\mcrot{1}{c}{45}{Lowpass: wide cutoff} &
\mcrot{1}{c}{45}{Lowpass: sharp rolloff} &
\mcrot{1}{c}{45}{Lowpass: residual--free} &
\mcrot{1}{c}{45}{Derivative matrix design} &
\mcrot{1}{c}{45}{Derivative order $(n)$} \\ \hlinewd{1pt}
Fornberg \cite{fornberg1988, S0036144596322507} & $1998$ &
&$\bullet$&&&&&&$\bullet$&&&&
$\textcolor{mygreen1}{\checkmark}$ &
$\textcolor{mygreen1}{\checkmark}$ &
$\textcolor{mygreen1}{\checkmark}$ &
$\textcolor{mygreen1}{\checkmark}$ &
$\textcolor{mygreen1}{\checkmark}$ &
$\textcolor{myred1}{\times}$ &
$\textcolor{myred1}{\times}$ &
$\textcolor{myred1}{\times}$ &
$\textcolor{myred1}{\times}$ &
$\textcolor{myred1}{\times}$ &
$n=\{0,1,\hdots\}$ \\ \hlinewd{.75pt}
Khan \cite{RasoolKhan1999, RasoolKhan1999_2, RasoolKhan2000, RasoolKhan2003, RasoolKhan2003_2} & $1999$ &
&&&&&&&$\bullet$&$\bullet$&&&
$\textcolor{mygreen1}{\checkmark}$ &
$\textcolor{mygreen1}{\checkmark}$ &
$\textcolor{mygreen1}{\checkmark}$ &
$\textcolor{myred1}{\times}$ &
$\textcolor{mygreen1}{\checkmark}$ &
$\textcolor{myred1}{\times}$ &
$\textcolor{myred1}{\times}$ &
$\textcolor{myred1}{\times}$ &
$\textcolor{myred1}{\times}$ &
$\textcolor{myred1}{\times}$ &
$n=\{0,1,\hdots\}$ \\ \hlinewd{.75pt}
Li \cite{Li2005} & $2005$ &
&$\bullet$&&&&&&$\bullet$&&&&
$\textcolor{mygreen1}{\checkmark}$ &
$\textcolor{mygreen1}{\checkmark}$ &
$\textcolor{myred1}{\times}$ &
$\textcolor{mygreen1}{\checkmark}$ &
$\textcolor{mygreen1}{\checkmark}$ &
$\textcolor{myred1}{\times}$ &
$\textcolor{myred1}{\times}$ &
$\textcolor{myred1}{\times}$ &
$\textcolor{myred1}{\times}$ &
$\textcolor{mygreen1}{\checkmark}$ &
$n=\{0,1,\hdots\}$ \\ \hlinewd{.75pt}
O'Leary \cite{o2008algebraic, o2012framework} & $2008$ &
&&&&$\bullet$&&&&&$\bullet$&&
$\textcolor{mygreen1}{\checkmark}$ &
$\textcolor{mygreen1}{\checkmark}$ &
$\textcolor{myred1}{\times}$ &
$\textcolor{mygreen1}{\checkmark}$ &
$\textcolor{mygreen1}{\checkmark}$ &
$\textcolor{myred1}{\times}$ &
$\textcolor{myred1}{\times}$ &
$\textcolor{myred1}{\times}$ &
$\textcolor{myred1}{\times}$ &
$\textcolor{mygreen1}{\checkmark}$ &
$n=1$ \\ \hlinewd{.75pt}
Hassan \cite{HassanMohamadAtteia2012} & $2012$ &
&&&&&&&$\bullet$&$\bullet$&&&
$\textcolor{mygreen1}{\checkmark}$ &
$\textcolor{mygreen1}{\checkmark}$ &
$\textcolor{myred1}{\times}$ &
$\textcolor{mygreen1}{\checkmark}$ &
$\textcolor{mygreen1}{\checkmark}$ &
$\textcolor{myred1}{\times}$ &
$\textcolor{myred1}{\times}$ &
$\textcolor{myred1}{\times}$ &
$\textcolor{myred1}{\times}$ &
$\textcolor{mygreen1}{\checkmark}$ &
$n=\{0,1,\hdots\}$ \\ \hlinewd{.75pt}
Carlsson \cite{Carlsson1991} & $1991$ &
&&&&&&&$\bullet$&&&$\bullet$&
$\textcolor{mygreen1}{\checkmark}$ &
$\textcolor{mygreen1}{\checkmark}$ &
$\textcolor{myred1}{\times}$ &
$\textcolor{myred1}{\times}$ &
$\textcolor{mygreen1}{\checkmark}$ &
$\textcolor{myred1}{\times}$ &
$\textcolor{myred1}{\times}$ &
$\textcolor{myred1}{\times}$ &
$\textcolor{myred1}{\times}$ &
$\textcolor{myred1}{\times}$ &
$n=1$ \\ \hlinewd{.75pt}
Kumar \cite{KumarRoy1988, KumarRoy1989, KumarRoy1992} & $1988$ &
&&&&&&&$\bullet$&&&$\bullet$&
$\textcolor{mygreen1}{\checkmark}$ &
$\textcolor{mygreen1}{\checkmark}$ &
$\textcolor{mygreen1}{\checkmark}$ &
$\textcolor{myred1}{\times}$ &
$\textcolor{mygreen1}{\checkmark}$ &
$\textcolor{myred1}{\times}$ &
$\textcolor{myred1}{\times}$ &
$\textcolor{myred1}{\times}$ &
$\textcolor{myred1}{\times}$ &
$\textcolor{myred1}{\times}$ &
$n=1$ \\ \hlinewd{.75pt}
Simoncelli \cite{Simoncelli1994, FaridSimoncelli1997, FaridSimoncelli2004} & $2004$ &
&&&&&&$\bullet$&&&$\bullet$&&
$\textcolor{myred1}{\times}$ &
$\textcolor{mygreen1}{\checkmark}$ &
$\textcolor{myred1}{\times}$ &
$\textcolor{myred1}{\times}$ &
$\textcolor{myred1}{\times}$ &
$\textcolor{mygreen1}{\checkmark}$ &
$\textcolor{myred1}{\times}$ &
$\textcolor{mygreen1}{\checkmark}$ &
$\textcolor{mygreen1}{\checkmark}$ &
$\textcolor{myred1}{\times}$ &
$n=\{0,1,\hdots\}$ \\ \hlinewd{.75pt}
Unser \cite{UnserAldroubiEdenI1993, UnserAldroubiEdenII1993} & $1993$ &
$\bullet$ &&&&&$\bullet$&&&&&&
$\textcolor{mygreen1}{\checkmark}$ &
$\textcolor{mygreen1}{\checkmark}$ &
$\textcolor{myred1}{\times}$ &
$\textcolor{myred1}{\times}$ &
$\textcolor{myred1}{\times}$ &
$\textcolor{mygreen1}{\checkmark}$ &
$\textcolor{myred1}{\times}$ &
$\textcolor{myred1}{\times}$ &
$\textcolor{myred1}{\times}$ &
$\textcolor{myred1}{\times}$ &
$n=\{0,1,\hdots\}$ \\ \hlinewd{.75pt}
Mboup \cite{mboup2009numerical} & $2009$ &
 &&&$\bullet$&&&&&&$\bullet$&&
$\textcolor{mygreen1}{\checkmark}$ &
$\textcolor{mygreen1}{\checkmark}$ &
$\textcolor{myred1}{\times}$ &
$\textcolor{myred1}{\times}$ &
$\textcolor{myred1}{\times}$ &
$\textcolor{mygreen1}{\checkmark}$ &
$\textcolor{myred1}{\times}$ &
$\textcolor{mygreen1}{\checkmark}$ &
$\textcolor{myred1}{\times}$ &
$\textcolor{myred1}{\times}$ &
$n=\{1,2,\hdots\}$ \\ \hlinewd{.75pt}
M{\"o}ller \cite{MollerMachiraju1997, moller1998design, moller1999spatial} & $1998$&
&&&&&$\bullet$&&&$\bullet$&$\bullet$&&
$\textcolor{mygreen1}{\checkmark}$ &
$\textcolor{myred1}{\times}$ &
$\textcolor{mygreen1}{\checkmark}$ &
$\textcolor{myred1}{\times}$ &
$\textcolor{myred1}{\times}$ &
$\textcolor{mygreen1}{\checkmark}$ &
$\textcolor{myred1}{\times}$ &
$\textcolor{mygreen1}{\checkmark}$ &
$\textcolor{myred1}{\times}$ &
$\textcolor{myred1}{\times}$ &
$n=\{0,1\}$ \\ \hlinewd{.75pt}
Gaussian \cite{MarrHildreth1980, freeman1991design, deriche1993recursively} & $1980$ &
 &&&&&&&&&&&
$\textcolor{mygreen1}{\checkmark}$ &
$\textcolor{mygreen1}{\checkmark}$ &
$\textcolor{mygreen1}{\checkmark}$ &
$\textcolor{myred1}{\times}$ &
$\textcolor{myred1}{\times}$ &
$\textcolor{mygreen1}{\checkmark}$ &
$\textcolor{mygreen1}{\checkmark}$ &
$\textcolor{myred1}{\times}$ &
$\textcolor{myred1}{\times}$ &
$\textcolor{myred1}{\times}$ &
$n=\{0,1,\hdots\}$ \\ \hlinewd{.75pt}
Kaiser \cite{kaiser1977sharpening, kaiser1977data} & $1977$ &
&&&&&&$\bullet$&&&$\bullet$&&
$\textcolor{mygreen1}{\checkmark}$ &
$\textcolor{myred1}{\times}$ &
$\textcolor{mygreen1}{\checkmark}$ &
$\textcolor{myred1}{\times}$ &
$\textcolor{myred1}{\times}$ &
$\textcolor{mygreen1}{\checkmark}$ &
$\textcolor{mygreen1}{\checkmark}$ &
$\textcolor{mygreen1}{\checkmark}$ &
$\textcolor{myred1}{\times}$ &
$\textcolor{myred1}{\times}$ &
$n=\{0,1\}$ \\ \hlinewd{.75pt}
Lanczos \cite{lanczosapplied, lanczos1988applied, duchon1979lanczos, groetsch1998lanczo, McDevitt:2012, burch2005least} & $1956$ &
$\bullet$ & $\bullet$ & $\bullet$ &&&&&$\bullet$&&$\bullet$&&
$\textcolor{mygreen1}{\checkmark}$ &
$\textcolor{myred1}{\times}$ &
$\textcolor{mygreen1}{\checkmark}$ &
$\textcolor{myred1}{\times}$ &
$\textcolor{myred1}{\times}$ &
$\textcolor{mygreen1}{\checkmark}$ &
$\textcolor{mygreen1}{\checkmark}$ &
$\textcolor{mygreen1}{\checkmark}$ &
$\textcolor{myred1}{\times}$ &
$\textcolor{myred1}{\times}$ &
$n=\{0,1,\hdots\}$ \\ \hlinewd{.75pt}
Savitzky-Golay \cite{SavitzkyGolay1964, gorry1990general, meer1990smoothed, luo2005properties, schafer2011savitzky} & $1964$ &
$\bullet$ &&&&&&&&&$\bullet$&&
$\textcolor{mygreen1}{\checkmark}$ &
$\textcolor{mygreen1}{\checkmark}$ &
$\textcolor{mygreen1}{\checkmark}$ &
$\textcolor{mygreen1}{\checkmark}$ &
$\textcolor{mygreen1}{\checkmark}$ &
$\textcolor{mygreen1}{\checkmark}$ &
$\textcolor{mygreen1}{\checkmark}$ &
$\textcolor{mygreen1}{\checkmark}$ &
$\textcolor{myred1}{\times}$ &
$\textcolor{myred1}{\times}$ &
$n=\{0,1,\hdots\}$ \\ \hlinewd{.75pt}
Ignjatovic \cite{ignjatovic2000signal, ignjatovic2001method, ignjatovic2009chromatic, ignjatovic2011multidimensional}& $2001$ & 
&&$\bullet$&&&&$\bullet$&&&&&
$\textcolor{mygreen1}{\checkmark}$ &
$\textcolor{mygreen1}{\checkmark}$ &
$\textcolor{myred1}{\times}$ &
$\textcolor{myred1}{\times}$ &
$\textcolor{myred1}{\times}$ &
$\textcolor{mygreen1}{\checkmark}$ &
$\textcolor{myred1}{\times}$ &
$\textcolor{mygreen1}{\checkmark}$ &
$\textcolor{mygreen1}{\checkmark}$ &
$\textcolor{myred1}{\times}$ &
$n=\{0,1,\hdots\}$ \\ \hlinewd{.75pt}
Selesnick \cite{SelesnickBurrus1998, SelesnickBurrus1998TSP, Selesnick2002}& $2002$ & 
$\bullet$ &&&&&&&$\bullet$&&&$\bullet$&
$\textcolor{mygreen1}{\checkmark}$ &
$\textcolor{mygreen1}{\checkmark}$ &
$\textcolor{mygreen1}{\checkmark}$ &
$\textcolor{myred1}{\times}$ &
$\textcolor{mygreen1}{\checkmark}$ &
$\textcolor{mygreen1}{\checkmark}$ &
$\textcolor{mygreen1}{\checkmark}$ &
$\textcolor{mygreen1}{\checkmark}$ &
$\textcolor{mygreen1}{\checkmark}$ &
$\textcolor{myred1}{\times}$ &
$n=\{0,1\}$ \\ \hlinewd{.75pt}
Hosseini-Plataniotis \cite{hosseini2017derivative} & $2017$ &
&&&&&&&$\bullet$&$\bullet$&&$\bullet$&
$\textcolor{mygreen1}{\checkmark}$ &
$\textcolor{mygreen1}{\checkmark}$ &
$\textcolor{mygreen1}{\checkmark}$ &
$\textcolor{myred1}{\times}$ &
$\textcolor{mygreen1}{\checkmark}$ &
$\textcolor{mygreen1}{\checkmark}$ &
$\textcolor{mygreen1}{\checkmark}$ &
$\textcolor{mygreen1}{\checkmark}$ &
$\textcolor{mygreen1}{\checkmark}$ &
$\textcolor{myred1}{\times}$ &
$n=\{0,1,\hdots\}$ \\ \hlinewd{.75pt}
MaxPol (Proposed) & $2017$ &
&&&&&&&$\bullet$&$\bullet$&&$\bullet$&
$\textcolor{mygreen1}{\checkmark}$ &
$\textcolor{mygreen1}{\checkmark}$ &
$\textcolor{mygreen1}{\checkmark}$ &
$\textcolor{mygreen1}{\checkmark}$ &
$\textcolor{mygreen1}{\checkmark}$ &
$\textcolor{mygreen1}{\checkmark}$ &
$\textcolor{mygreen1}{\checkmark}$ &
$\textcolor{mygreen1}{\checkmark}$ &
$\textcolor{mygreen1}{\checkmark}$ &
$\textcolor{mygreen1}{\checkmark}$ &
$n=\{0,1,\hdots\}$ \\ \hlinewd{.75pt}
\end{tabular}
\end{table*}
\end{landscape}

\section{Global finite difference model--numerical foundation}\label{sec_maxpol_numerics}
In this section we introduce the generalized numerical framework of finite difference calculation to render closed form solution to diverse parameter design of fullband/lowpass filters with arbitrary side-shift nodes.

\subsection{Maximally polynomial (MaxPol) design}\label{section_global_method}
Consider the global values of finite difference for both centralized and staggered nodes is defined by undetermined coefficients
\begin{align}
\text{centralized:~}f^{(n)}_{j+s} := \sum_{k=-l}^{l} c^n_{s}(k) f_{j+k},~~~ &
\text{staggered:~}f^{(n)}_{j+s-1/2}:= \sum_{k=-l+1}^{l} c^n_{s-\frac{1}{2}}(k) f_{j+k}
\label{eq1}
\end{align}
where, $c^n_{s}$ and $c^n_{s-\frac{1}{2}}$ stand for unknown coefficients of $n$th order derivative with side-shift $s$ either to left or right boundaries. From the definition \ref{eq1} the corresponding transform functions are obtained by the discrete Fourier transform (DFT)
\begin{subequations}
\begin{align}
H^{(n)}_{s}(i\omega) & = \frac{\mathfrak{F}\{f^{(n)}_{j+s}\}}{\mathfrak{F}\{f_{j+s}\}}=\sum_{k=-l}^{l} c^n_{s}(k) e^{i(k-s)\omega}  \label{eq2_1} \\
H^{(n)}_{s-\frac{1}{2}}(i\omega) & = \frac{\mathfrak{F}\{f^{(n)}_{j+s-\frac{1}{2}}\}}{\mathfrak{F}\{f_{j+s-\frac{1}{2}}\}}=\sum_{k=-l+1}^{l} c^n_{s-\frac{1}{2}}(k) e^{i(k-s+\frac{1}{2})\omega}.  \label{eq2_2}
\end{align}\label{eq2}
\end{subequations}
The exponential function at particular point $\omega_0$ can be expanded by Taylor series
\begin{align}
e^{\alpha\omega} =  e^{\alpha\omega_0}\sum\limits^{\infty}_{m=0}{\frac{\alpha^m\left(\omega-\omega_0\right)^m}{m!}}
\label{eq3_tyalor_expansion}
\end{align}
where $\alpha\in\mathbb{C}$ is an arbitrary complex number. By substituting \ref{eq3_tyalor_expansion} in \ref{eq2} yields
\begin{subequations}
\begin{align}
H^{(n)}_s(i\omega) & = \sum_{k=-l}^{l}{e^{i(k-s)\omega_0}c^n_{s}(k)}\sum_{m=0}^{\infty}\frac{\left(i(k-s)(\omega-\omega_0)\right)^m}{m!} \label{eq3_1} \\
H^{(n)}_{s-\frac{1}{2}}(i\omega) & = \sum_{k=-l+1}^{l}{e^{i(k-s+\frac{1}{2})\omega_0}c^n_{s-\frac{1}{2}}(k)}\sum_{m=0}^{\infty}\frac{\left(i(k-s+\frac{1}{2})(\omega-\omega_0)\right)^m}{m!} \label{eq3_2}
\end{align}\label{eq3}
\end{subequations}
The goal is to obtain lowpass FIR derivative coefficients where `by fixing $s$' different filter responses are modeled. Choosing positive and negative values for $s$ cause a shift towards right and left boundaries, respectively. Whereas $s=0$ implies zero-shift transformation. A plausible solution to this design is to regulate the transfer functions in \ref{eq3} using the maximally flat constraints \cite{Herrmann1970, Selesnick2002}. Specifically, we rename this design by \textit{``maximally polynomial (MaxPol)''} to resemble different polynomial orders ``$n$'' as of different derivative-order response $H^{(n)}_{c}(i\omega)=(i\omega)^n$ in frequency domain. The rationale of the MaxPol conditions is to regulate filter response to preserve the frequency information up to certain level. In other words, our desire is to truncate frequency residuals beyond a certain level which we refer to as `cut-off level'. This has huge advantage in forward imaging applications where texture details contain high-frequency information while they are perturbed by noise artifacts. For more information on the topic we refer the reader to \cite{hosseini2017derivative}. In particular, two main criteria should be satisfied
\begin{enumerate}
\item the discrete filter responses \ref{eq3} trace the continuous trajectory of the $n$th order derivative up to $P$th order polynomial at $\omega=0$ such that
\begin{align}\label{eq4_1}
\frac{\partial^p H^{(n)}_s(i\omega)}{\partial\omega^p}{\Big{\vert}}_{\omega=0} & = 
\frac{\partial^p H^{(n)}_c(i\omega)}{\partial\omega^p}{\Big{\vert}}_{\omega=0}, & p\in\{0,\cdots, P\}
\end{align}
where, $P$ controls the degree of differential accuracy.
\item An optional lowpass property is imposed on the discrete responses \ref{eq3} up to $Q$th order polynomial tangent to zero at $\omega=\pi$ to control the flatness degree of lowpass transition 
\begin{align}\label{eq4_2}
\frac{\partial^q H^{(n)}_s(i\omega)}{\partial\omega^q}{\Big{\vert}}_{\omega=\pi} & =0, & q\in\{0,\cdots, Q\}
\end{align}
\end{enumerate}
Combinations of $\{P,Q\}$ create different filter library for lowpass/fullband FIR design. Higher $P$ indicates fullband response by relaxing the cutoff frequency. Whereas higher $Q$ implies more emphasize on lowpass design by drifting the cutoff closer to zero. $P$ and $Q$ are linearly dependent, where $Q$ can be indirectly manipulated as function of $P$. To proceed with design, let first calculate the $p$th derivative of continuous response by
\begin{equation}\label{eq5}
\frac{\partial^p H^{(n)}_{c}(i\omega)}{{\partial\omega}^p}=\frac{i^n n!}{(n-p)!}\omega^{n-p}.
\end{equation}

Next, the $p$th derivative response of the transfer functions \ref{eq3} yields
\begin{subequations}
\begin{align}
\frac{\partial^p H^{(n)}_s(i\omega)}{\partial\omega^p} & =\sum_{k=-l}^{l}{e^{i(k-s)\omega_0}c^n_{s}(k)}\sum_{m=p}^{\infty}\frac{\left(i(k-s)\right)^m}{(m-p)!}(\omega-\omega_0)^{m-p} \label{eq6_1} \\
\frac{\partial^p H^{(n)}_{s-\frac{1}{2}}(i\omega)}{\partial\omega^p} & =\sum_{k=-l+1}^{l}{e^{i(k-s+\frac{1}{2})\omega_0}c^n_{s-\frac{1}{2}}(k)}\sum_{m=p}^{\infty}\frac{\left(i(k-s+\frac{1}{2})\right)^m}{(m-p)!}(\omega-\omega_0)^{m-p} \label{eq6_2}
\end{align}\label{eq6}
\end{subequations}
By substituting \ref{eq6_1} in MaxPol conditions \ref{eq4_1}-\ref{eq4_2} and evaluating the derivatives at $\omega=\{0,\pi\}$ the polynomials corresponding to $p,q\neq m$ will vanish. Hence, the remaining is given by
\begin{subequations}\label{eq7_centralized}
\begin{align}
\sum_{k=1 }^{2l+1} (a+k)^p c^n_{s}(k) =
\left\{\begin{array}{ll}
n!, & p=n \\
0, &\text{else}
\end{array}\right.,~~~p\in\mathcal{S}_P=\{0,\cdots,P\} \label{eq7_centralized_1}\\
\sum_{k=1 }^{2l+1} (-1)^{k}(a+k)^q c^n_{s}(k) =0,~~~q\in\mathcal{S}_Q=\{0,\cdots,Q\}.~~~~~~~~~ \label{eq7_centralized_2}
\end{align}
\end{subequations}
The centralized system of equations \ref{eq7_centralized} contains $2l+1$ unknowns and the uniqueness of the solution is guaranteed if $|\mathcal{S}_P|+|\mathcal{S}_Q|=2l+1$. Note the summation limits in \ref{eq7_centralized} are obtained by shift technique on sub-indexes. Similarly, by substituting the staggered transfer function \ref{eq6_2} in MaxPol conditions \ref{eq4_1}-\ref{eq4_2} and evaluating at points $\omega=\{0,\pi\}$
\begin{subequations}\label{eq7_staggered}
\begin{align}
\sum_{k=1}^{2l} (b+k)^p c^n_{s-\frac{1}{2}}(k) =
\left\{\begin{array}{ll}
n!, & p=n \\
0, &\text{else}
\end{array}\right.,~~~p\in\mathcal{S}_P=\{0,\cdots,P\} \label{eq8_staggered_1}\\
\sum_{k=1}^{2l} (-1)^{k}(b+k)^q c^n_{s-\frac{1}{2}}(k) =0,~~~q\in\mathcal{S}_Q=\{0,\cdots,Q\}.~~~~~~~~~\label{eq8_staggered_2}
\end{align}
\end{subequations}
The uniqueness of the solution to the staggered equations \ref{eq7_staggered} is guaranteed if $|\mathcal{S}_P|+|\mathcal{S}_Q|=2l$.

\begin{table}[htp]
\renewcommand{\arraystretch}{1.5}
\caption{Generalization of polynomial term for different node design}
\label{table_BC_model}
\centering
\scriptsize
\begin{tabular}{ccc} 
\hlinewd{1pt}
 & \text{Centralized Node:} \ref{eq7_centralized} & \text{Staggered Node:} \ref{eq7_staggered} \\
\hlinewd{1pt}
\text{zero-shift ($s=0$)} & $a=-l-1$ & $b=-l+1/2$ \\
\hlinewd{.75pt}  
\text{Left Boundary} & $a=-l+s-1$ & $b=-l+s-1/2$ \\
\hlinewd{.75pt} 
\text{Right Boundary} & $a=-l-s-1$ & $b=-l-s+1/2$ \\
\hlinewd{.75pt} 
\end{tabular}
\end{table}

The system of equations \ref{eq7_centralized} and \ref{eq7_staggered} are comprised of different polynomials $k+a$ and $k+b$ where the terms $a$ and $b$ can be generalized for positive or negative shifting index $s$ as discussed earlier. \ref{table_BC_model} elaborates on all possible selections for such generalization.


\subsection{Matrix-vector representation}\label{subsec:matrix_vector}
The matrix-vector formulation analogous to the centralized system of equations \ref{eq7_centralized} can be constructed by
\begin{align}
\begin{bmatrix}
V_{a,P}\\
V_{a,Q}\Lambda_s
\end{bmatrix}
c^n_{s} = 
b_n,~~~\text{where}~~~
V_{a,J} = 
\left[\hspace{-.075in}
\begin{array}{c@{\hspace{.5em}}c@{\hspace{.5em}}c@{\hspace{.5em}}c@{\hspace{.5em}}}
1 & 1 & \hdots & 1 \\
(a+1)^1 & (a+2)^1 & \hdots & (a+2l+1)^1 \\
\vdots  & \vdots  & \ddots & \vdots     \\
(a+1)^J & (a+2)^J & \hdots & (a+2l+1)^J
\end{array}
\hspace{-.075in}\right].
\label{eq9_centralized}
\end{align}
The multiplying matrix in \ref{eq9_centralized} is composed of partitioned Vandermonde blocks where $\Lambda_s = \diag{\left((-1)^1,(-1)^2,\ldots,(-1)^{2l+1}\right)}$ is a diagonal matrix, $c^n_s=\left[c^n_{s}(1),c^n_{s}(2),\ldots,c^n_{s}(2l+1)\right]^T$ is the undetermined vector coefficients, and $b_{n}=\left[0,\ldots,n!,\ldots,0\right]^T$ containing  one non-zero element at $n+1$th row. With similar notations, the analogous matrix-vector formulation to the staggered system of equations \ref{eq7_staggered} can be formulated by
\begin{align}
\begin{bmatrix}
V_{b,P}\\
V_{b,Q}\Lambda_{s-\frac{1}{2}}
\end{bmatrix}
c^n_{s-\frac{1}{2}} = 
b_n,~~~\text{where}~~~
V_{b,J} = 
\left[\hspace{-.075in}
\begin{array}{c@{\hspace{.5em}}c@{\hspace{.5em}}c@{\hspace{.5em}}c@{\hspace{.5em}}}
1 & 1 & \hdots & 1 \\
(b+1)^1 & (b+2)^1 & \hdots & (b+2l)^1 \\
\vdots  & \vdots  & \ddots & \vdots     \\
(b+1)^J & (b+2)^J & \hdots & (b+2l)^J
\end{array}
\hspace{-.075in}\right].
\label{eq9_staggered}
\end{align}
Here, $\Lambda_{s-\frac{1}{2}} = \diag{\left((-1)^1,(-1)^2,\ldots,(-1)^{2l}\right)}$ and $c^n_{s-\frac{1}{2}}=\left[c^n_{s-\frac{1}{2}}(1),c^n_{s-\frac{1}{2}}(2),\ldots,c^n_{s-\frac{1}{2}}(2l)\right]^T$.

The partitioned-block Vandermonde matrices in \ref{eq9_centralized} and \ref{eq9_staggered} are non-singular. This can be shown by decomposing the matrix in \ref{eq9_centralized} to a $2\times 2$  block matrix
\begin{align}
\begin{bmatrix}
V_{a,P}\\
V_{a,Q}\Lambda_{s}
\end{bmatrix} =
\begin{bmatrix}
V^L_{a,P} & V^R_{a,P}\\
V^L_{a,Q}\Lambda^U_{s} & V^R_{a,Q}\Lambda^B_{s}
\end{bmatrix},\text{where}
\begin{array}{l}
\left[V^L_{a,P}\right]_{P+1\times P+1},~
\left[V^R_{a,P}\right]_{P+1\times Q+1},~
\left[V^L_{a,Q}\right]_{Q+1\times P+1},\nonumber \\
\left[V^R_{a,Q}\right]_{Q+1\times Q+1},~
\left[\Lambda^U_{s}\right]_{P+1\times P+1},~
\left[\Lambda^B_{s}\right]_{Q+1\times Q+1},\nonumber
\end{array}
\end{align}
where, $\Lambda_{s}$ is separated to low $\Lambda^U_{s}$ and up $\Lambda^B_{s}$ block-diagonals. The block elements $V^L_{a,P}$ and $V^R_{a,Q}\Lambda^B_{s}$ are non-singular square Vandermonde matrices with following determinants \cite{fiedler2008special}
\begin{align}
\det V^L_{a,P} =\prod\limits_{i=1}^{P+1} (i-1)!
~\text{and}~
\det{V^R_{a,Q}\Lambda^B_{s}} = 
\det V^R_{a,Q}\det\Lambda^B_{s} =(-1)^P\prod\limits_{i=1}^{Q+1} (i-1)!
\nonumber
\end{align}
which are non-zero. Similarly the determinant of partitioned matrix \ref{eq9_staggered} yields non-zero.

The solutions to \ref{eq9_centralized} and \ref{eq9_staggered} require inverse calculation of partitioned-block Vandermonde matrix which is highly ill-conditioned. Accordingly, the numerical precision of inverse Vandermonde is highly perturbed in digital computation due to limited floating points for numerical computation. Nevertheless, both latter matrices are non-singular and hence there exist a closed form solution. In fact the remaining problem of this section reduces to find a proper closed form solutions to \ref{eq9_centralized} and \ref{eq9_staggered}.

\subsection{Exact solution via symbolic calculation}
One easy way to solve \ref{eq9_centralized} and \ref{eq9_staggered} is to perform symbolic computation in computer programming such as MATLAB/MATHEMATICA. With existing digital computing powers one can easily solve an acceptable range of matrix size up to $l<50$ in a reasonable time frame. For instance to compute the second derivative coefficients for parameters $l=15$, $P=8$, and $s=0$ it takes only $0.3$ seconds to run on a Windows 7 HP Z440 Workstation desktop computer with an Intel(R) Xeon(R) CPU. Often in many forward imaging applications the tap-length of the lowpass filter is fixed and therefore it can be pre-computed and restored in a digital memory for processing. \ref{fig_lowpass} displays the filter response of few centralized lowpass kernels generated by two methods of MaxPol (proposed) and Savitzky-Golay filters \cite{SavitzkyGolay1964, luo2005properties, schafer2011savitzky}. The derivative orders here are set to $n=\{0,1,2,3,4\}$ where $n=0$ corresponds to a lowpass (interpolation) filter. The residual free of the MaxPol response is highly evident compared to the side-lobes generated by SG filters.

\begin{figure}[htp]
\centerline{
\subfigure[MaxPol ($n=0$)]{\includegraphics[height=0.18\textwidth]{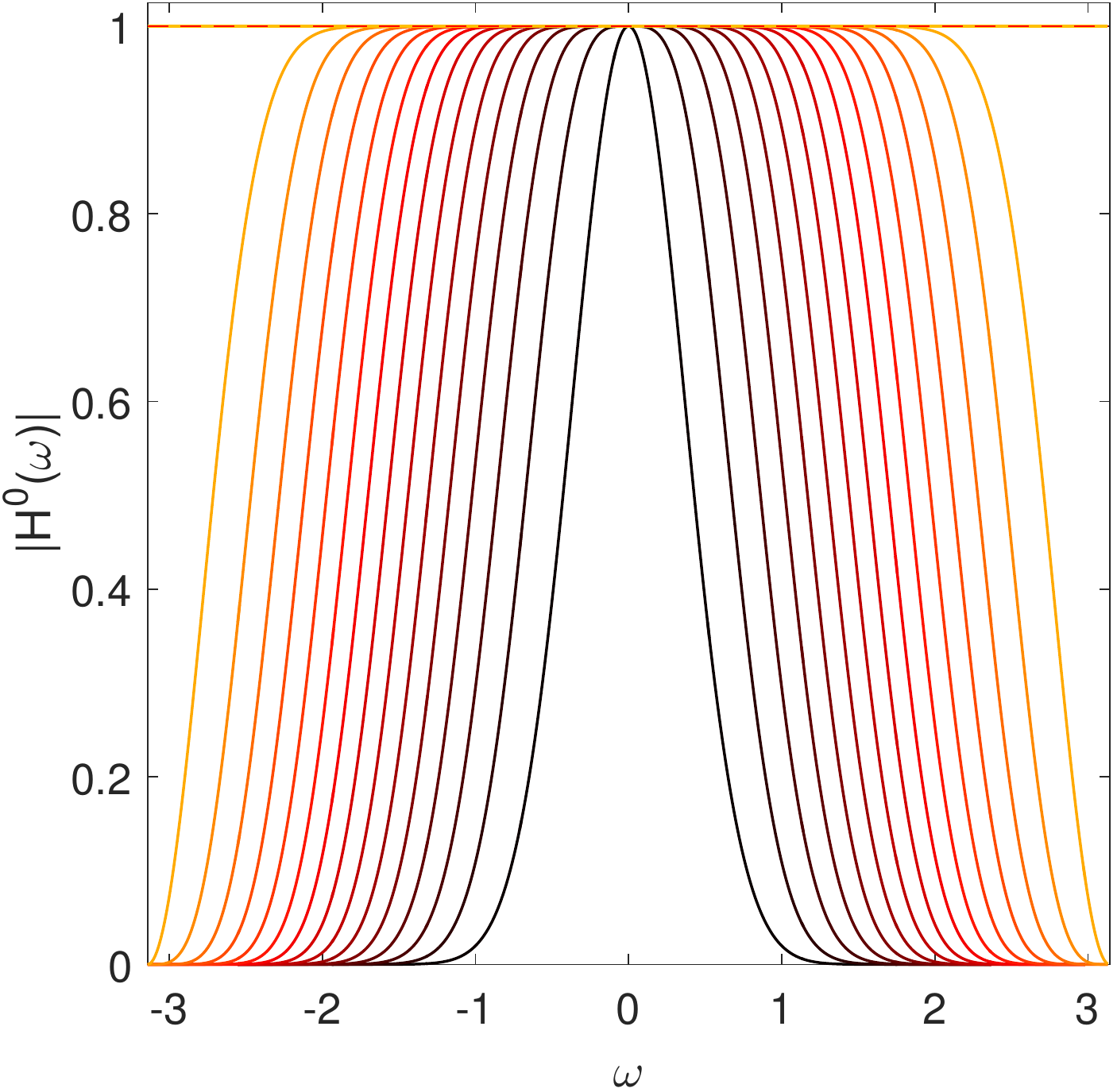}\label{fig_lowpass_centralized_maxpol_n_ord_0}}
\subfigure[MaxPol ($n=1$)]{\includegraphics[height=0.18\textwidth]{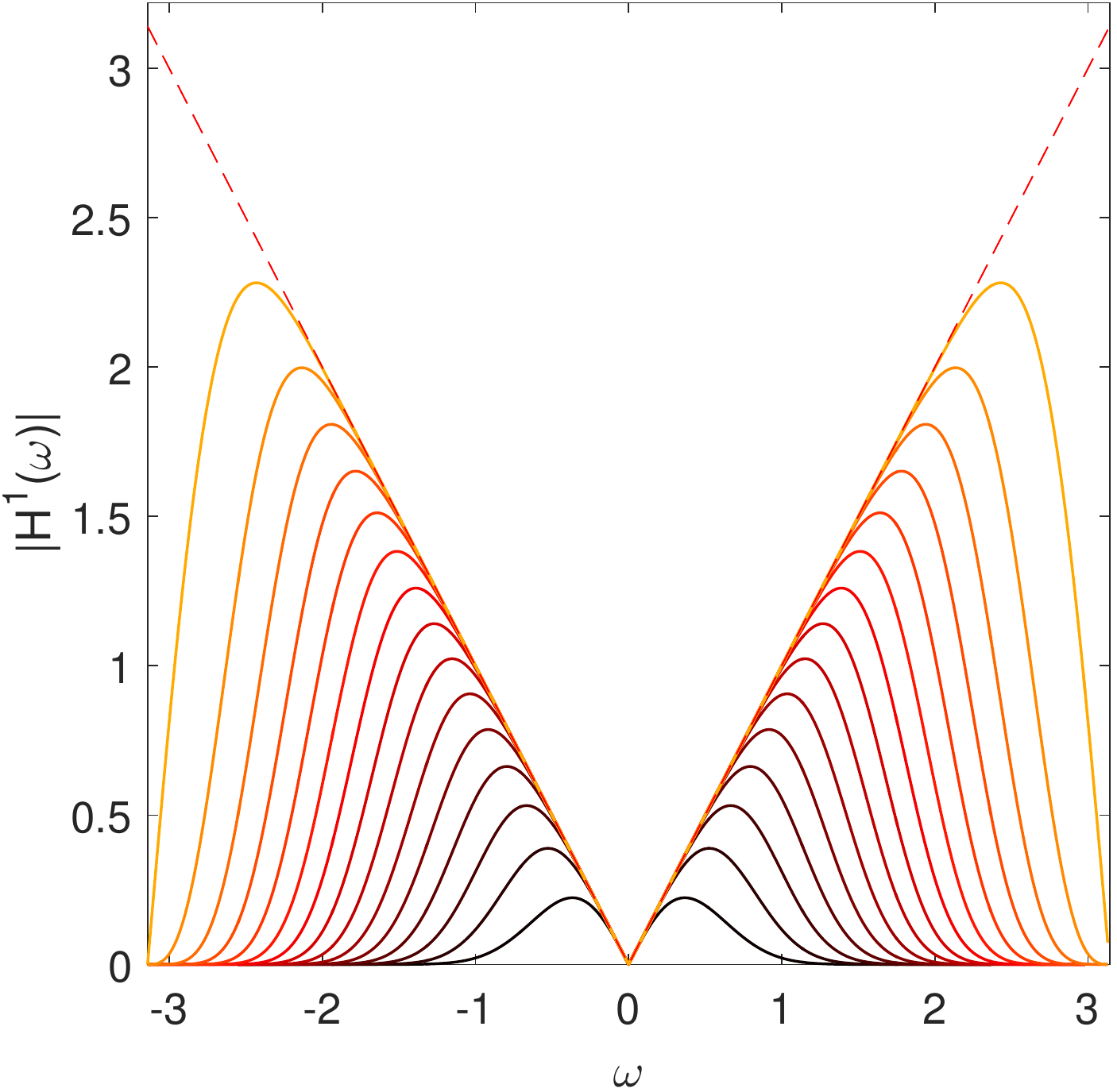}\label{fig_lowpass_centralized_maxpol_n_ord_1}}
\subfigure[MaxPol ($n=2$)]{\includegraphics[height=0.18\textwidth]{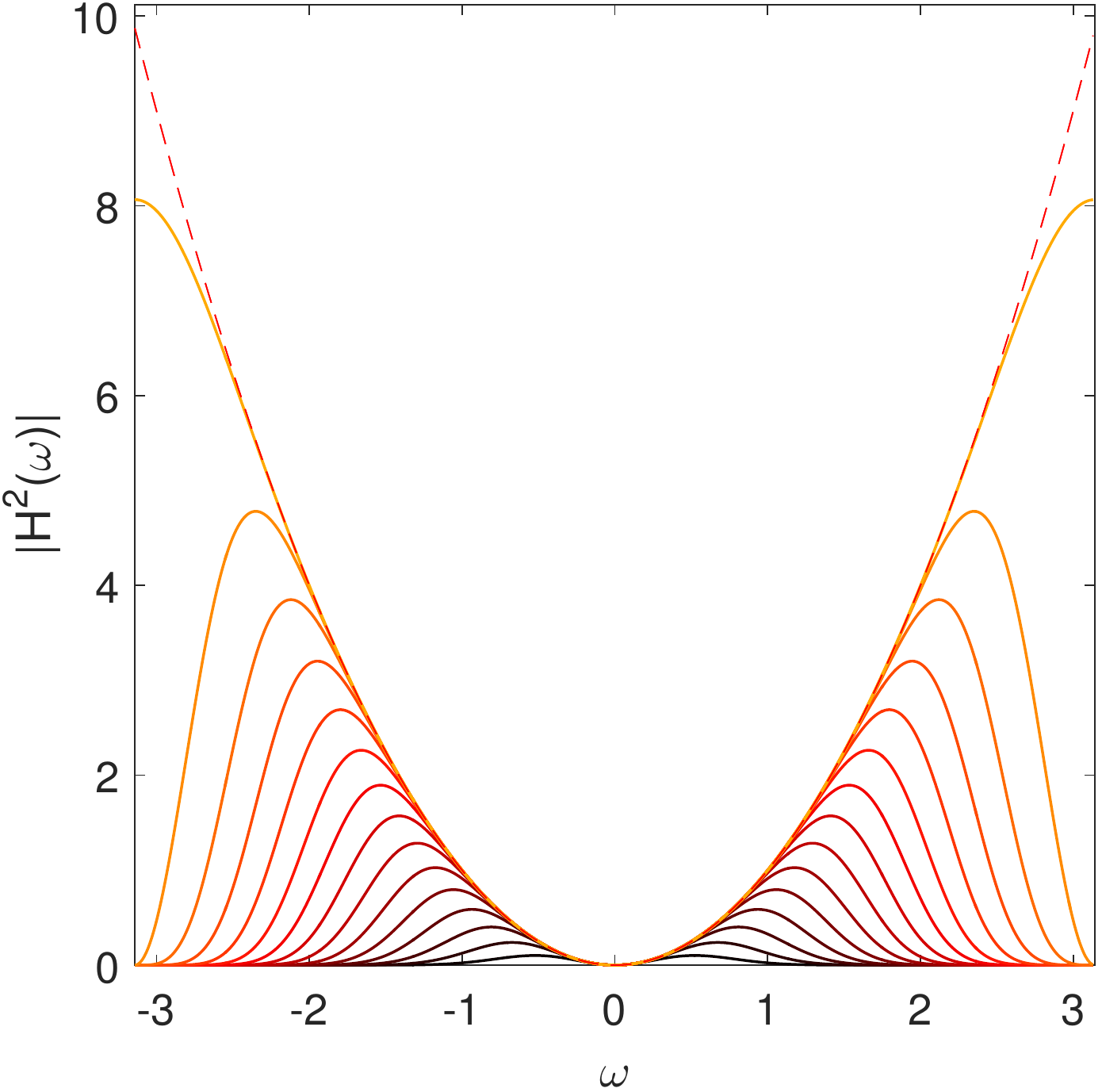}\label{fig_lowpass_centralized_maxpol_n_ord_2}}
\subfigure[MaxPol ($n=3$)]{\includegraphics[height=0.18\textwidth]{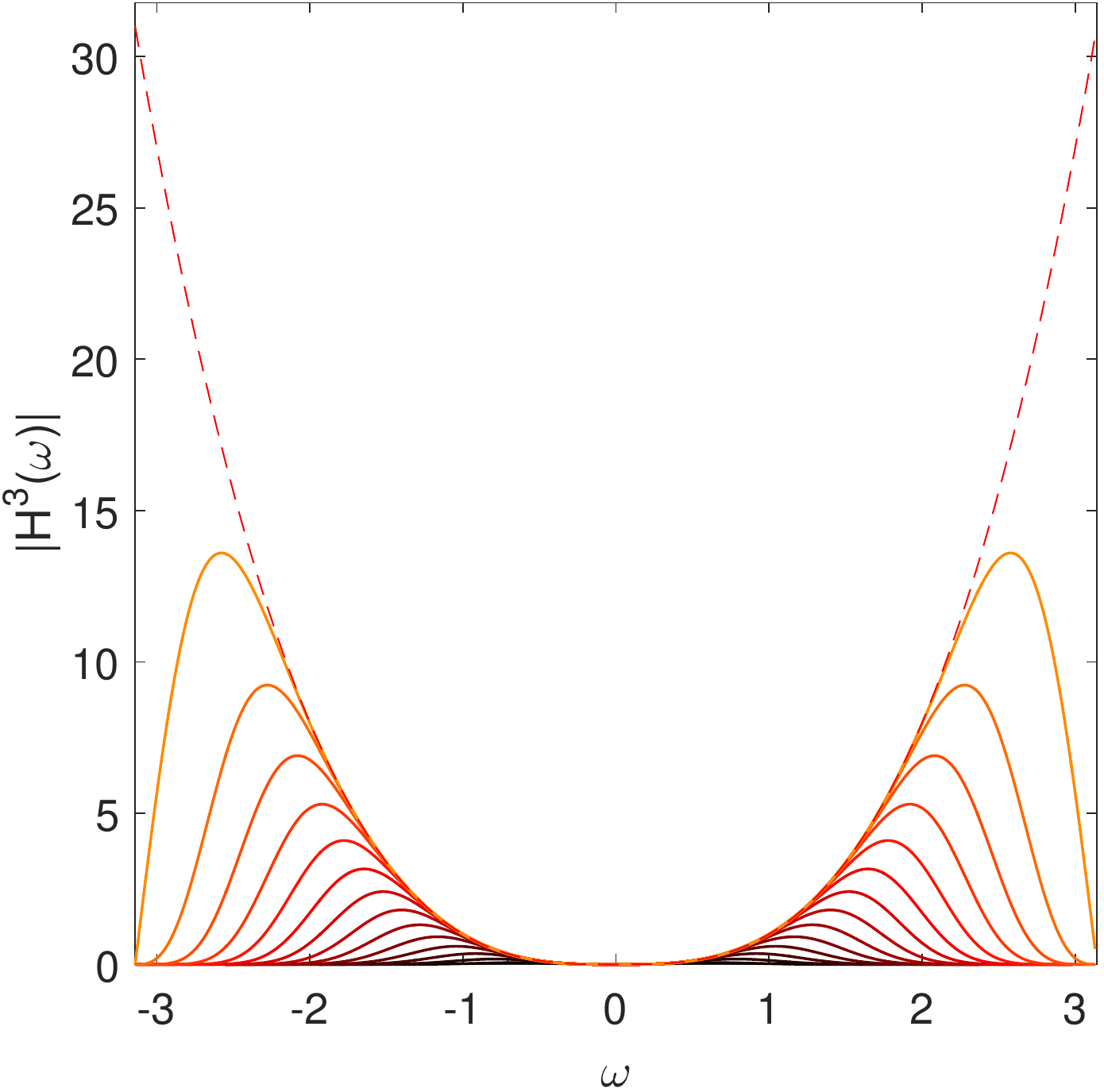}\label{fig_lowpass_centralized_maxpol_n_ord_3}}
\subfigure[MaxPol ($n=4$)]{\includegraphics[height=0.18\textwidth]{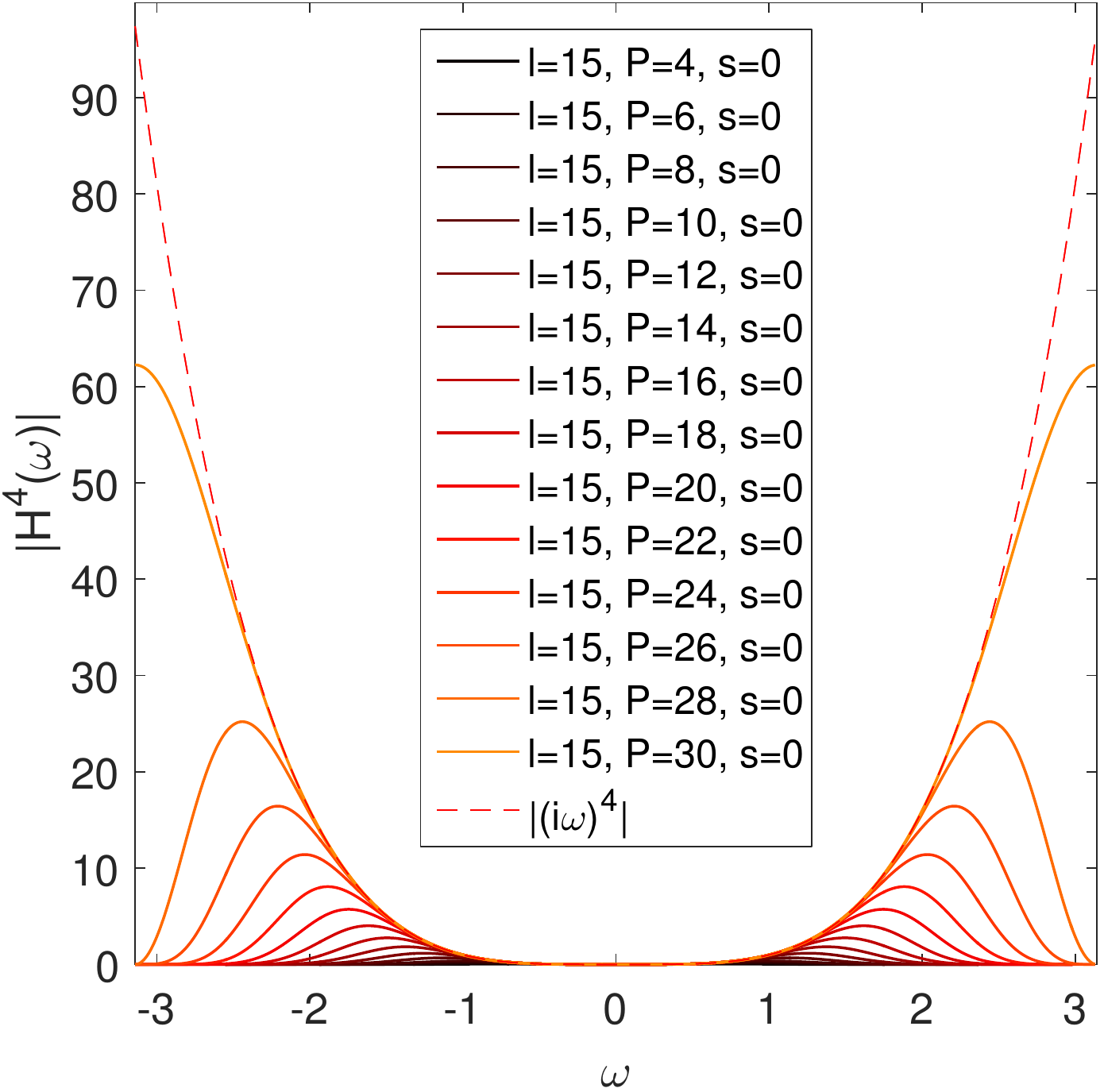}\label{fig_lowpass_centralized_maxpol_n_ord_4}}
}
\centerline{
\subfigure[SG ($n=0$)]{\includegraphics[height=0.18\textwidth]{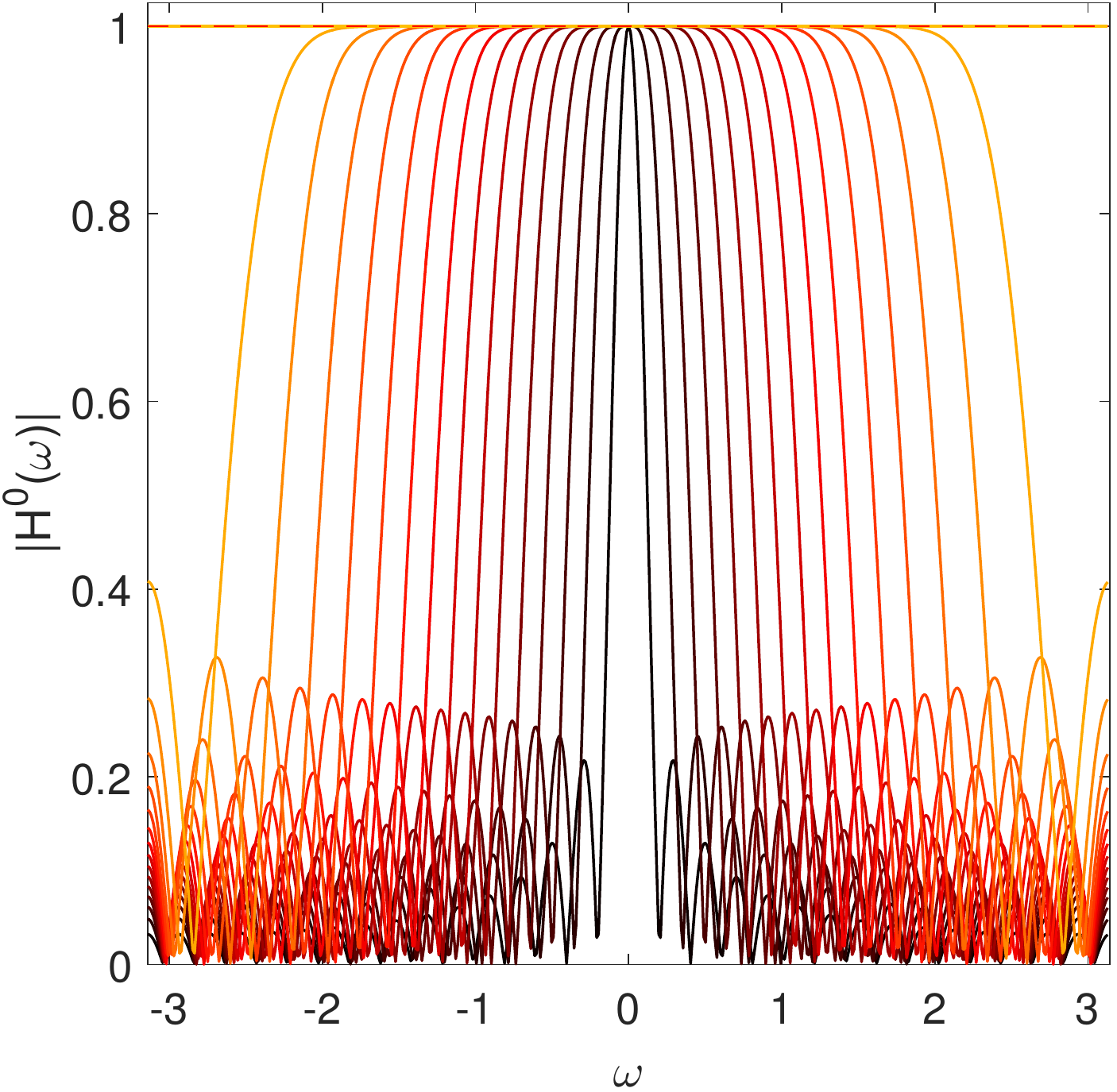}\label{fig_lowpass_centralized_savitzkygolay_n_ord_0}}
\subfigure[SG ($n=1$)]{\includegraphics[height=0.18\textwidth]{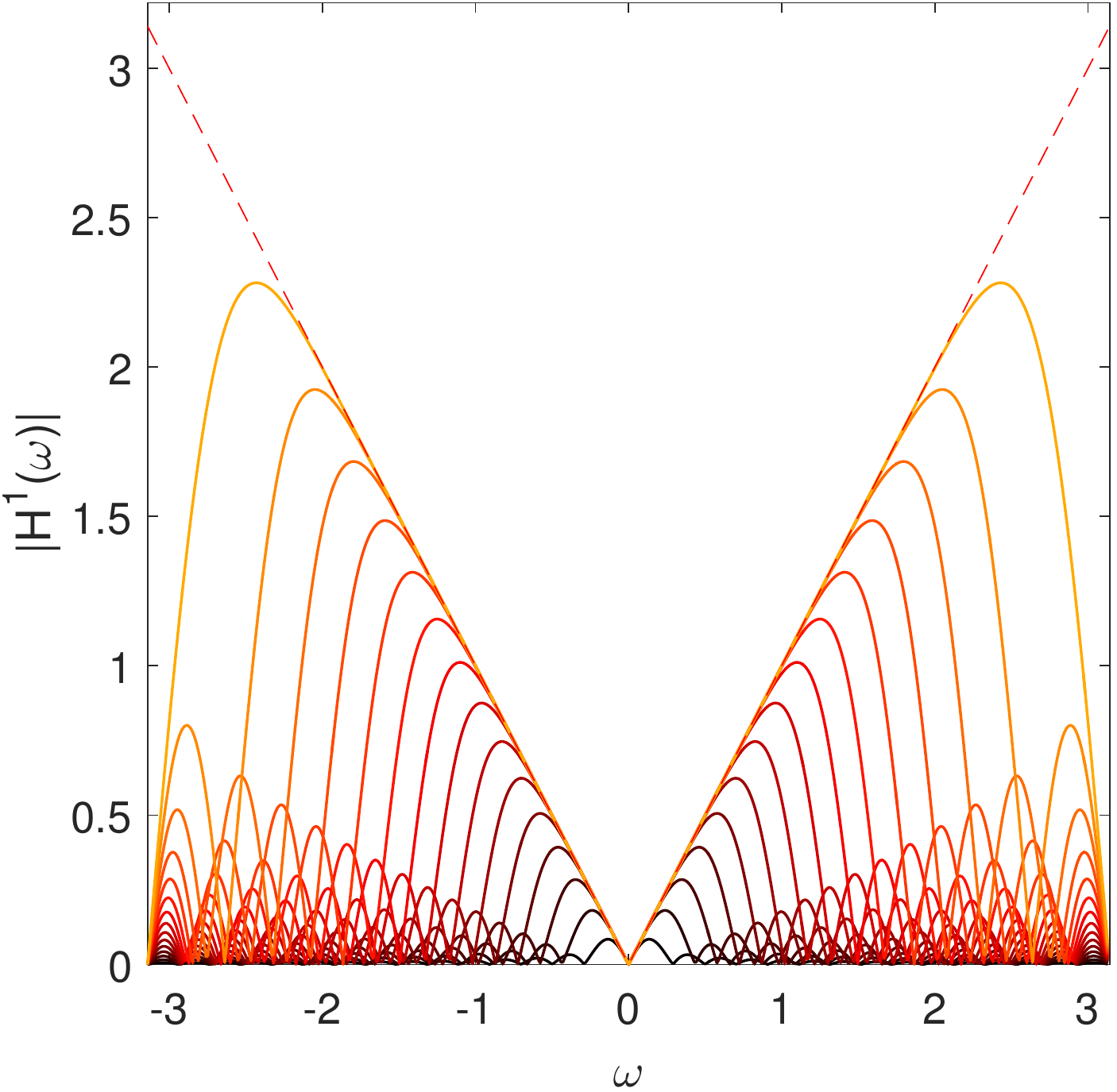}\label{fig_lowpass_centralized_savitzkygolay_n_ord_1}}
\subfigure[SG ($n=2$)]{\includegraphics[height=0.18\textwidth]{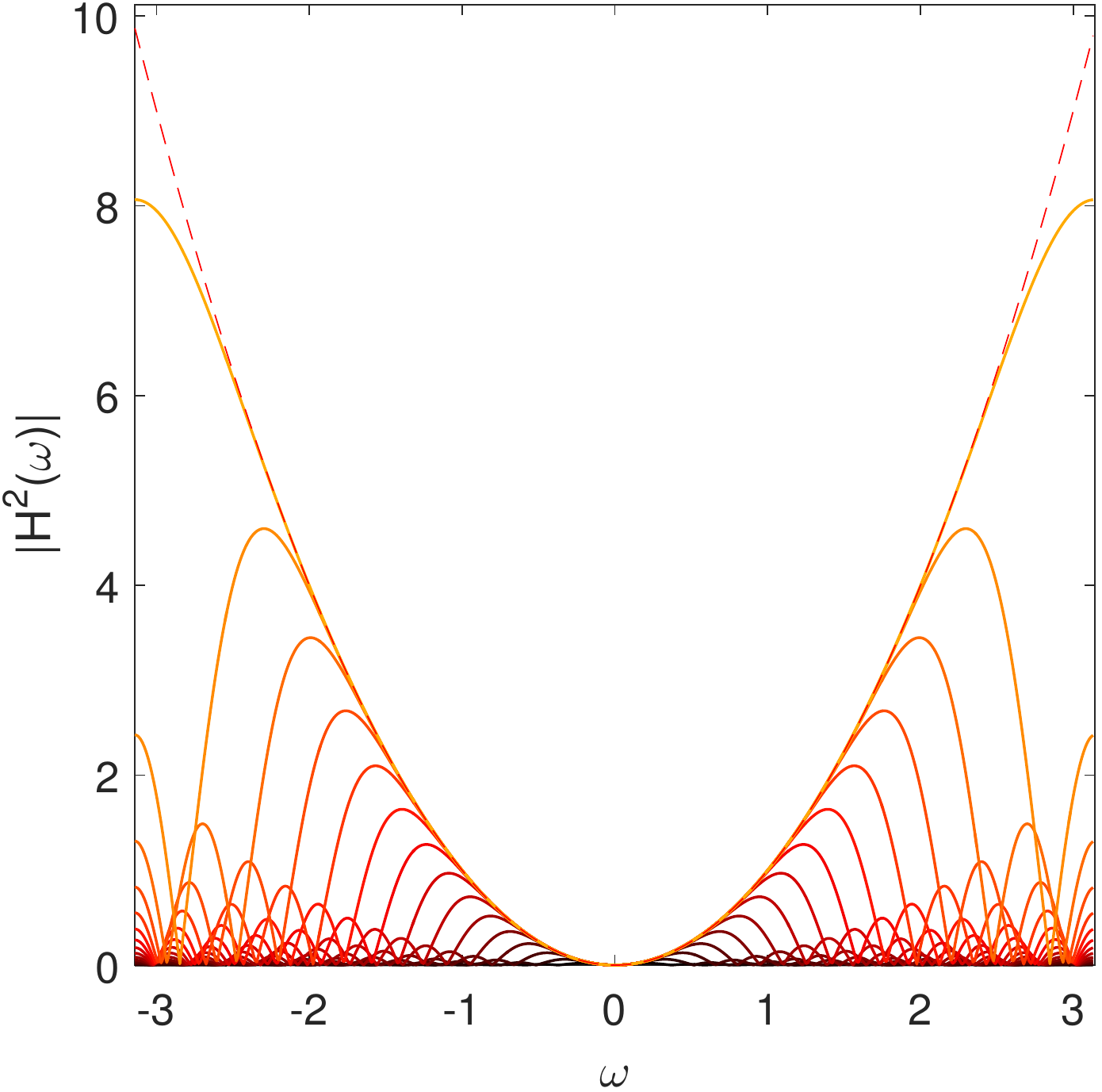}\label{fig_lowpass_centralized_savitzkygolay_n_ord_2}}
\subfigure[SG ($n=3$)]{\includegraphics[height=0.18\textwidth]{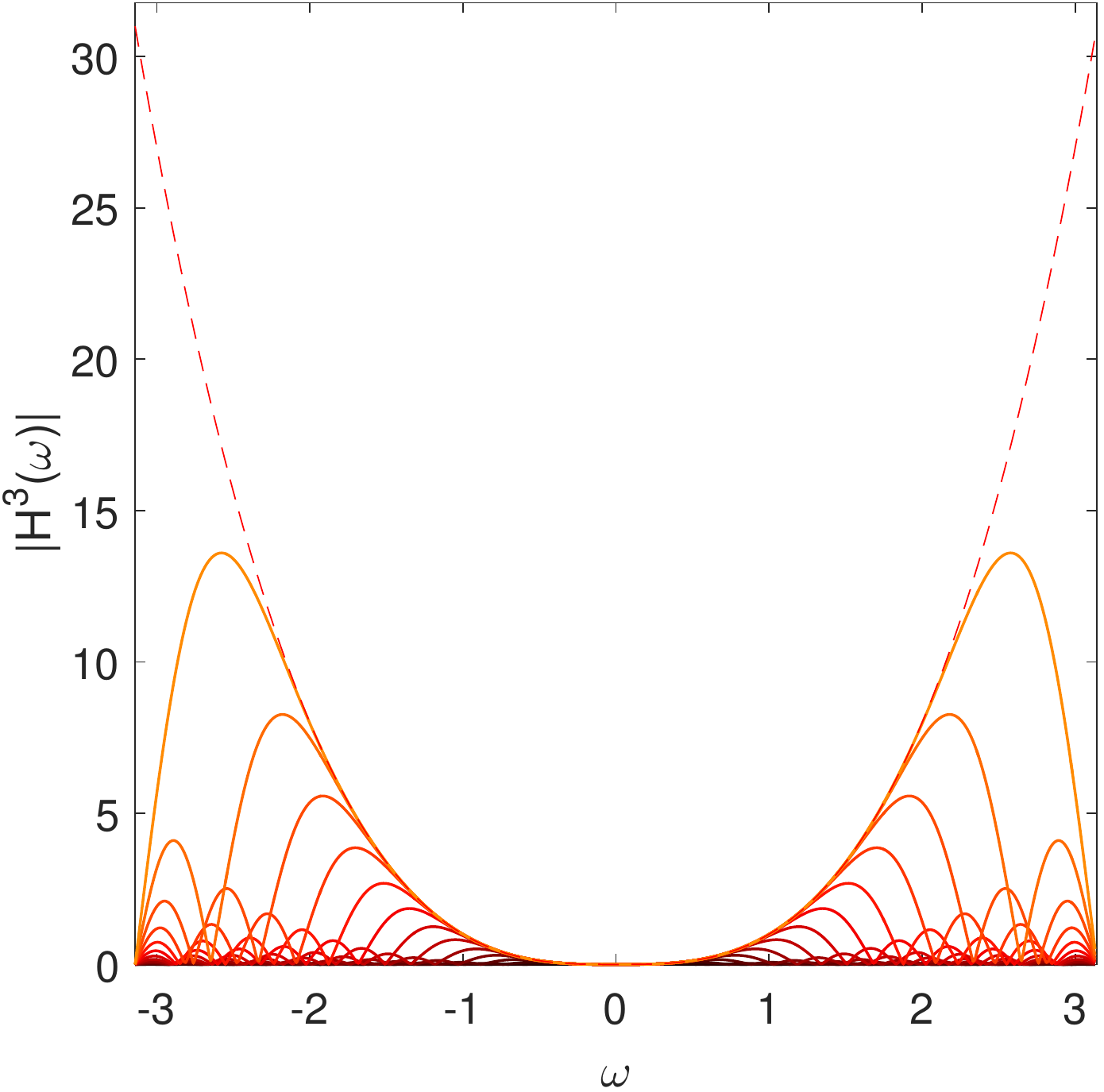}\label{fig_lowpass_centralized_savitzkygolay_n_ord_3}}
\subfigure[SG ($n=4$)]{\includegraphics[height=0.18\textwidth]{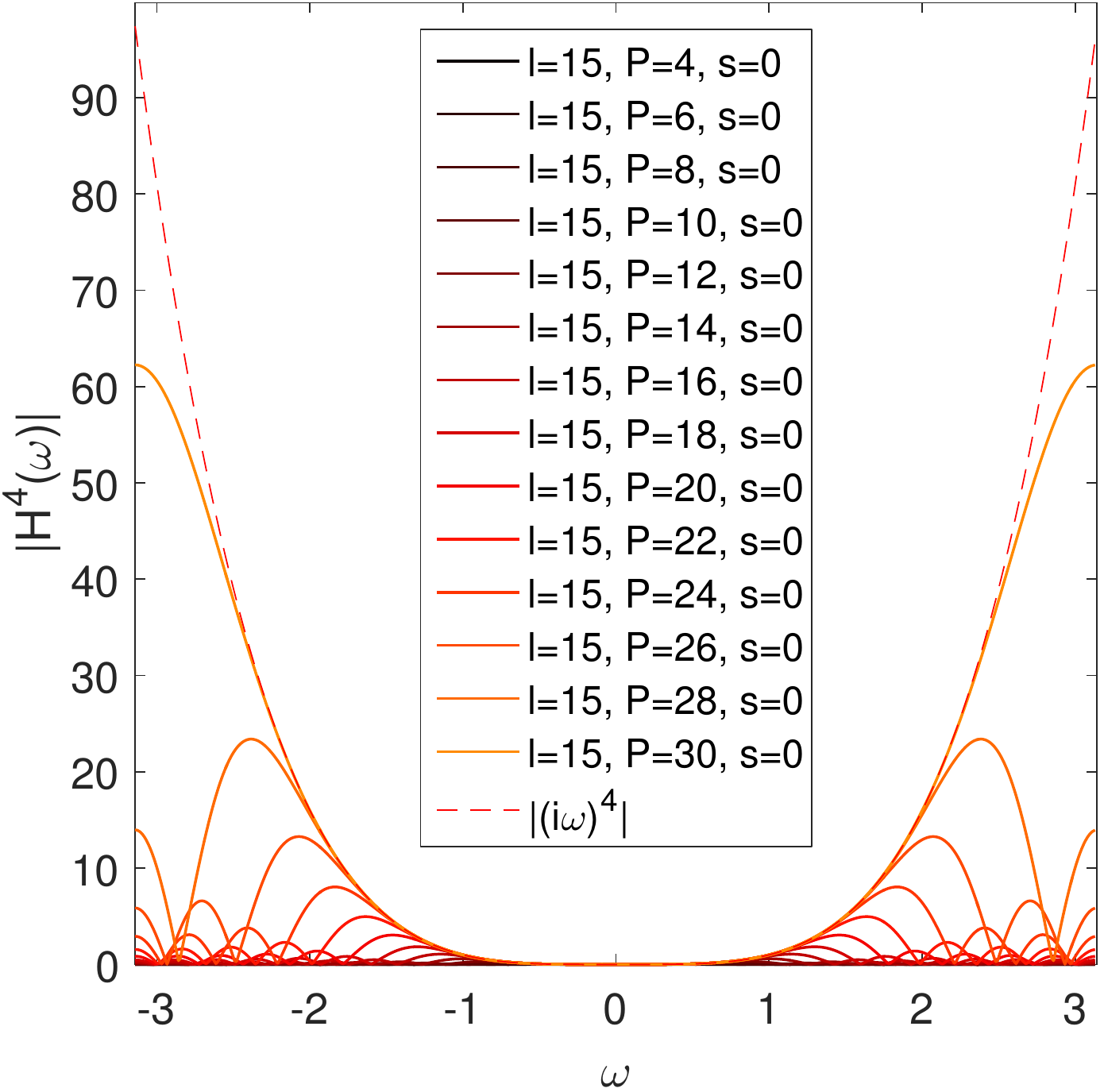}\label{fig_lowpass_centralized_savitzkygolay_n_ord_4}}
}
\caption{Transfer functions of FIR lowpass differentiation with centralized scheme using MaxPol (Proposed) and Savitzky-Golay (SG) design methods \cite{SavitzkyGolay1964, luo2005properties, schafer2011savitzky}.}
\label{fig_lowpass}
\end{figure}

It should be noted that the above filters are shown for zero shift design $s=0$. Without the loss of generality, one can associate non-zero side shifts to generate one-sided kernels that comply with causal filters. This can be favorably applied in dynamic applications such as system controls to estimate feature characteristics of a noisy signal in real-time processing. 

\subsection{Closed form solution to fullband design}\label{section_closed_form}
In the case of fullband design where the cutoff is relaxed, i.e. $P=2l-1$ for staggered and $P=2l$ for centralized schemes, the matrix system of equations \ref{eq9_centralized} and \ref{eq9_staggered} are simplified to square Vandermonde systems
\begin{align}
\text{staggered:}~V_{b,P} c^n_{s-\frac{1}{2}} = b_n,~~~ &
\text{centralzied:}~V_{a,P} c^n_{s} = b_n.
\label{eq_fullband_1}
\end{align}
We define the general notation $V=V_{v,N}$ where $v\in\mathbb{R}$ and $N$ is the number of equidistant sampling grids for discrete polynomial interpolation. While algorithmic solutions such as \cite{EisinbergFedele2006} exist for inverse Vandermonde calculation of \ref{eq_fullband_1} with suboptimal accuracies, recent analysis show the partial fraction expansion of a high degree polynomial of rational function can be associated to decompose the inverse Vandermonde matrix into multiplication of two non-singular matrices, namely $V^{-1}=WX$ \cite{man2014inversion, chen1981new}. Here, $W$ is a square matrix
\begin{equation}\label{eq12}
W=\left[\begin{array}{c@{\hspace{.5em}}c@{\hspace{.5em}}c@{\hspace{.5em}}c@{\hspace{.5em}}}
\lambda(1) & 0 & \cdots & 0 \\
0 & \lambda(2) & \cdots & 0 \\
\vdots & \vdots & \ddots & \vdots \\
0 & 0 & \cdots & \lambda(N)
\end{array}\right]^{-1}
\left[
\begin{array}{l@{\hspace{.5em}}l@{\hspace{.5em}}l@{\hspace{.5em}}l@{\hspace{.5em}}}
(v+1)^{N-1} & (v+1)^{N-2} & \cdots & 1 \\
(v+2)^{N-1} & (v+2)^{N-2} & \cdots & 1 \\
\vdots & \vdots & \ddots & \vdots \\
(v+N)^{N-1} & (v+N)^{N-2} & \cdots & 1
\end{array}
\right].
\end{equation}
The pivot elements of the diagonal matrix in LHS of \ref{eq12} are computed and simplified to
\begin{equation}\label{eq13}
\lambda(k)=\prod\limits_{\substack{j=1\\ j\neq k}}^{N}(k-j)=(-1)^{k}(k-1)! (N-k)!.
\end{equation}
Moreover, the lower triangular matrix $X$ known as the Stanely matrix \cite{stanley1964time} is defined by 
\begin{equation}\label{eq14}
X=\left[
\begin{array}{l@{\hspace{.45em}}l@{\hspace{.45em}}l@{\hspace{.45em}}l@{\hspace{.45em}}}
1 & 0  & \cdots & 0 \\
x_1 & 1 & \cdots & 0 \\
x_2 & x_1 & \cdots & 0 \\
\vdots & \vdots & \ddots & \vdots \\
x_{N-1} & x_{N-2} & \cdots & 1
\end{array}
\right],~\text{where}~
x_j=(-1)^j\sum\limits_{r_1\neq r_2\neq\cdots\neq r_j}{(v+r_1)(v+r_2)\cdots(v+r_j)}.
\end{equation}
The coefficients $x_j$ in $X$ are the coefficients of denominator of a rational function where $r_j\in\{1,\ldots,N\}$ is a positive integer. The computational complexity of the coefficient $x_j$ in \ref{eq14} requires choosing $j$ out of $N$ combinations $\binom {N} {j}$ with no repetitions. Here we find the closed form expression of such combinatorial problem.

\begin{proposition}\label{theorem_hypercube_summation}
Consider a combinatorial summation on an $n$-dimensional hypercube
\begin{equation}
S = \sum\limits_{\mathcal{S}}{C(r_1)C(r_2)\cdots C(r_n)}
\label{eq_hypercube_1}
\end{equation}
where, $C(\cdot)$ is a function of a variable from a set of real numbers $r_j\in\{R_1,R_2,\ldots,R_N\}$ and $\mathcal{S}=\{r_1\neq r_2\neq\cdots\neq r_{n}\}$ is a set of real numbers with no repetition. The $n$ recursive formulation with initialized parameters $i=0$ and $C^{(0)}=C$
\begin{equation}
\left\{
\begin{array}{l}
C^{(i)}(r) = C(r)\left[S^{(i-1)} - (n-i)C^{(i-1)}(r)\right] \\
S^{(i)} = \sum\limits_{r}{C^{(i)}(r)} \\
i\leftarrow i+1
\end{array}
\right.
\label{eq_hypercube_2}
\end{equation}
provides the closed form expression to the summation \ref{eq_hypercube_1} where $S = S^{(n-1)}/n!$.
\end{proposition}
\begin{proof}
The conditions of the summation set $\mathcal{S}$ in \ref{eq_hypercube_1} can be reformed by
\begin{equation}
S =
{\sum\limits_{\substack{r_1}}}
{\sum\limits_{\substack{r_2\\ r_2\neq r_1}}}
\ldots
{\sum\limits_{\substack{r_n\\ r_n\neq r_{n-1}\\ \vdots \\ r_n\neq r_{1}}}}
{C(r_1)C(r_2)\cdots C(r_n)}.
\label{eq_hypercube_proof_1}
\end{equation}
An equivalent interpretation of the sum of combinations \ref{eq_hypercube_proof_1} is to accumulate the multiplying functions of discrete coordinates $r_i$ of an $n$ dimensional hypercube except on its diagonals. Since the repetitions are not allowed, the summation is only valid on one corner of the hypercube out of $n!$ possible corners. The combinatorial summations \ref{eq_hypercube_proof_1} can be decomposed by relaxing the conditions of the last sum and written as
\begin{multline} \label{eq_hypercube_proof_3}
S =
{\sum\limits_{\substack{r_1}}}
{\sum\limits_{\substack{r_2\\ r_2\neq r_1}}}
\ldots
{\sum\limits_{\substack{r_{n-1}\\ r_{n-1}\neq r_{n-2}\\ \vdots \\ r_{n-1}\neq r_{1}}}}
{\sum\limits_{r_n}}
{C(r_1)C(r_2)\cdots C(r_n)} \\
-(n-1)
{\sum\limits_{\substack{r_1}}}
{\sum\limits_{\substack{r_2\\ r_2\neq r_1}}}
\ldots
{\sum\limits_{\substack{r_{n-1}\\ r_{n-1}\neq r_{n-2}\\ \vdots \\ r_{n-1}\neq r_{1}}}}
{C(r_1)C(r_2)\cdots {C(r_{n-1})}^2}~~~~~~~~~~~~~~~~
\end{multline}
The common terms in \ref{eq_hypercube_proof_3} can be factorized and written as
\begin{equation} 
S =
{\sum\limits_{\substack{r_1}}}
{\sum\limits_{\substack{r_2\\ r_2\neq r_1}}}
\ldots
{\sum\limits_{\substack{r_{n-1}\\ r_{n-1}\neq r_{n-2}\\ \vdots \\ r_{n-1}\neq r_{1}}}}
{C(r_1)C(r_2)\cdots C(r_{n-1})}
\left[
{\sum\limits_{r_n}C(r_n)-(n-1)C(r_{n-1})}
\right]
\label{eq_hypercube_proof_4}
\end{equation}
By renaming $C^{(1)}(r_{n-1})=C(r_{n-1})[\sum\limits_{r_n}C(r_n)-(n-1)C(r_{n-1})]$ the combinatorial sum \ref{eq_hypercube_proof_4} is reduced to $n-1$ summations. This reduction can be recursively repeated for $n$ times until no summation is remained and hence $S$ can be calculated by recursive formulations in \ref{eq_hypercube_2}.
\end{proof}

Next we find the closed form solutions to the matrix equations defined in \ref{eq_fullband_1}. The main result are provided in \ref{lemma_staggered} and \ref{lemma_centralized}.
\begin{theorem}\label{lemma_staggered}
Let integers $l\geq 1$, $-l\leq s \leq l$ and $1\leq k \leq 2l$ be given. The fullband FIR coefficients of $n$th order derivative at staggered nodes shifted by $s-\frac{1}{2}$ is given by
\begin{equation}
c^n_{s-\frac{1}{2}}(k) =\frac{(-1)^{n+1}n!}{\lambda(k)}
\left[\prod\limits_{\substack{j=1 \\ j\neq k}}^{2l}{(b+j)}\right]
\sum\limits_{\mathcal{S}}{\frac{1}{(b+r_1)(b+r_2)\cdots(b+r_{n})}}
\label{eq19}
\end{equation}
where, $\mathcal{S}=\{r_1\neq r_2\neq\cdots\neq r_{n}\neq k\}$ is a non-repetitive positive integer set. The combinatorial summations, $\lambda(k)$, and `$b$' in \ref{eq19} are obtained from \ref{theorem_hypercube_summation}, \ref{eq13}, and \ref{table_BC_model}, respectively.
\end{theorem}
\begin{proof}
By substituting the decomposing matrices $V^{-1}=WX$ defined in \ref{eq12} and \ref{eq13} to the staggered form in \ref{eq_fullband_1}, the solution to the $k$th coefficient is obtained by (Here $N=2l$)
\begin{align}
c^n_{s-\frac{1}{2}}(k)&=W[k\text{th row}] \times X[n+1\text{th column}] \times {n!}\nonumber\\
 &=\frac{{n!}}{\lambda(k)}\left[(b+k)^{2l-n-1}+(b+k)^{2l-n-2}x_1+(b+k)^{2l-n-3}x_2+\cdots+x_{2l-n-1}\right]. \label{eq16}
\end{align}
Substituting the coefficients $x_i$ defined in \ref{eq14} to \ref{eq16} and expanding the terms yields
\begin{align}
c^n_{s-\frac{1}{2}}(k) &=\left[(b+k)^{2l-n-1}-(b+k)^{2l-n-2}\sum\limits_{r_1}(b+r_1)+ (b+k)^{2l-n-3}\sum\limits_{r_1\neq r_2}(b+r_1)(b+r_2)
\right. \nonumber\\
& + \cdots +\left.(-1)^{2l-n-1}\sum\limits_{r_1\neq r_2\neq\cdots\neq r_{2l-n-1}}{(b+r_1)(b+r_2)\cdots(b+r_{2l-n-1})}\right]\times\frac{{n!}}{\lambda(k)} \label{eq17}
\end{align}
The summand in \ref{eq17} can be written as 
\begin{align}
c^n_{s-\frac{1}{2}}(k) &=\left[(b+k)^{2l-n-1}-(b+k)^{2l-n-2}\left((b+k)+\sum\limits_{r_1\neq k}(b+r_1)\right)+ \right. \nonumber \\
& (b+k)^{2l-n-3}\left((b+k)\sum\limits_{r_1\neq k}(b+r_1) + \sum\limits_{r_1\neq r_2\neq k}(b+r_1)(b+r_2)\right)+ \nonumber\\
& \cdots +\left.(-1)^{2l-n-1}\sum\limits_{r_1\neq r_2\neq\cdots\neq r_{2l-n-1}\neq k}{(b+r_1)(b+r_2)\cdots(b+r_{2l-n-1})}\right]\times \frac{{n!}}{\lambda(k)} \nonumber
\end{align}
where by induction the remaining terms are simplified to
\begin{equation}
c^n_{s-\frac{1}{2}}(k) =\frac{(-1)^{n+1}{n!}}{\lambda(k)}\sum\limits_{r_1\neq r_2\neq\cdots\neq r_{2l-n-1}\neq k}{(b+r_1)(b+r_2)\cdots(b+r_{2l-n-1})} \label{eq18}
\end{equation}
which gives \ref{eq19} after factoring out the terms inside the summation.
\end{proof}
\begin{theorem}\label{lemma_centralized}
Let integers $l\geq 1$, $-l\leq s \leq l$ and $k\in\{1,\ldots,2l+1\}$ be given. The fullband FIR coefficients of $n$th order derivative at centralized nodes shifted by $s$ is given by
\begin{equation}
c^n_{s}(k) =\frac{(-1)^{n} n!}{\lambda(k)}
\left[\prod\limits_{\substack{j=1 \\ j\neq -a \\ j\neq k}}^{2l+1}{(a+j)}\right]
\sum\limits_{\mathcal{S}}{\frac{1}{(a+r_1)(a+r_2)\cdots(a+r_{n-1})}}, \label{eq20}
\end{equation}
where, $\mathcal{S}=\{r_1\neq r_2\neq\cdots\neq r_{n-1}\neq k\}$ is a set of positive integers with no repetition. Here, the combinatorial summations, $\lambda(k)$ and `$a$' are obtained by \ref{theorem_hypercube_summation}, \ref{eq13} and \ref{table_BC_model}, respectively. Furthermore, the coefficient at node $-a$ is determined by $c^n_{s}(-a)=-\sum\limits_{k\neq -a} c^n_{s}(k)$.
\end{theorem}
\begin{proof}
By substituting the decomposing matrices $V^{-1}=WX$ defined in \ref{eq12} and \ref{eq13} to the centralized form in \ref{eq_fullband_1}, the solution to the $k$th coefficient is obtained by (Here $N=2l+1$)
\begin{align}
c^n_{s}(k)&= W[k\text{th row}] \times X[n+1\text{th column}]\times {n!} \nonumber\\
 &=\frac{{n!}}{\lambda(k)}\left[(a+k)^{2l-n}+(a+k)^{2l-n-1}x_1+(a+k)^{2l-n-2}x_2+\cdots+x_{2l-n}\right] \label{eq21}
\end{align}
By substituting the terms $x_i$ defined in \ref{eq14} into \ref{eq21} and canceling the mutual terms, similar to the derivations done for staggered formulation, the remaining terms are simplified to
\begin{equation}
c^n_{s}(k) =\frac{(-1)^{n}{n!}}{\lambda(k)}\sum\limits_{r_1\neq r_2\neq\cdots\neq r_{2l-n}\neq k}{(a+r_1)(a+r_2)\cdots(a+r_{2l-n})} \label{eq22}
\end{equation}
which gives \ref{eq20} after factoring out the terms inside the summation.  
\end{proof}

The complexities of \ref{eq19} and \ref{eq20} require $n$ and $n-1$ flops for computation and hence are of orders $\mathcal{O}(n)$ and $\mathcal{O}(n-1)$, respectively. The closed form expressions in \ref{lemma_staggered} and \ref{lemma_centralized} are favorably comparable to the Lagrangian approach in \cite{fornberg1988, S0036144596322507}. Specifically, the MaxPol method maintains much lower computational complexity and higher numerical precision as opposed to \cite{fornberg1988, S0036144596322507}. This will be further discussed in \ref{sec_fullband_differentiation}. 
\begin{figure}[htp]
\centerline{
\subfigure[centr. ($n=1$)]{\includegraphics[height=0.18\textwidth]{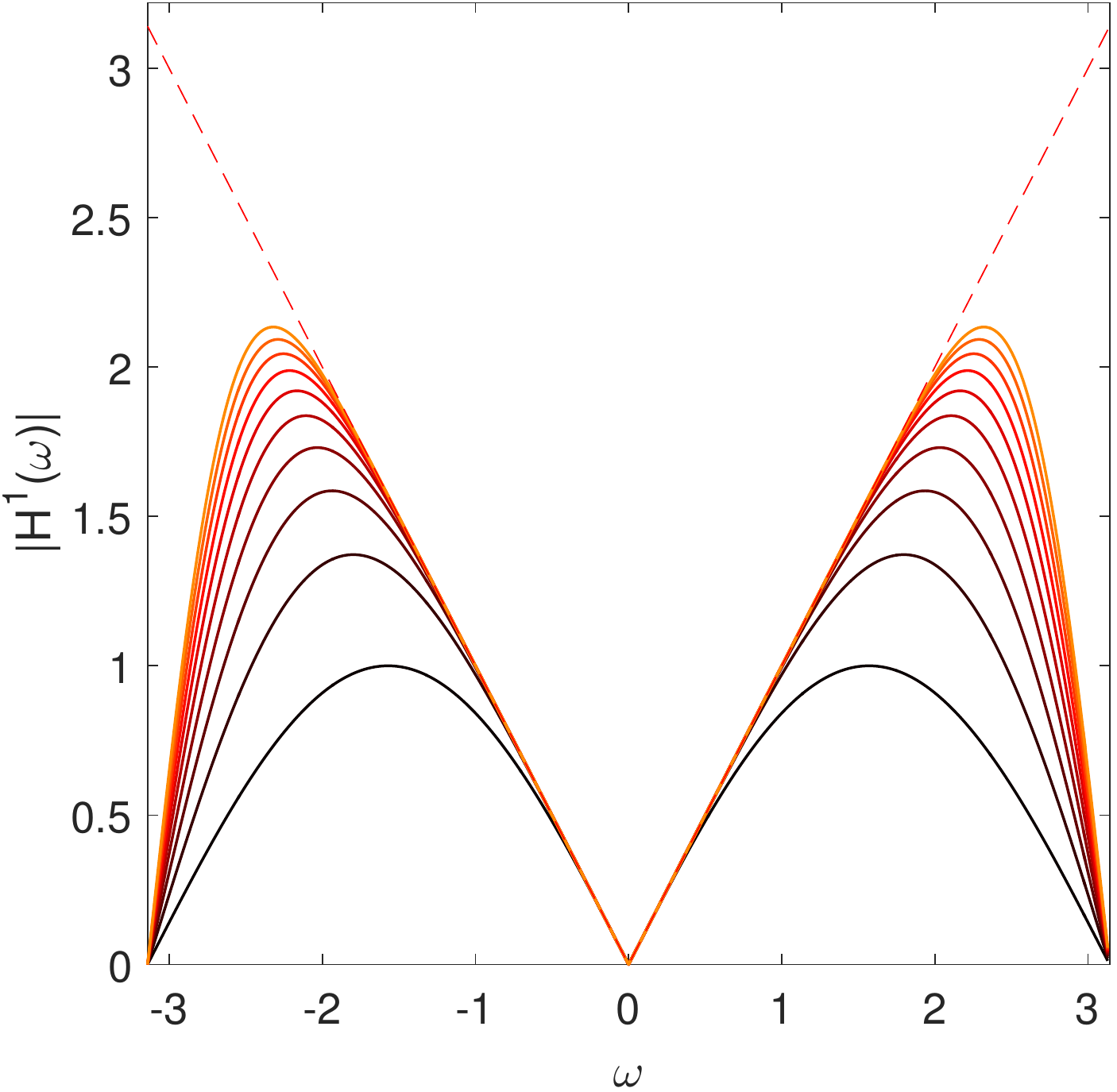}\label{fig_fullband_centralized_maxpol_n_ord_1}}
\subfigure[centr. ($n=2$)]{\includegraphics[height=0.18\textwidth]{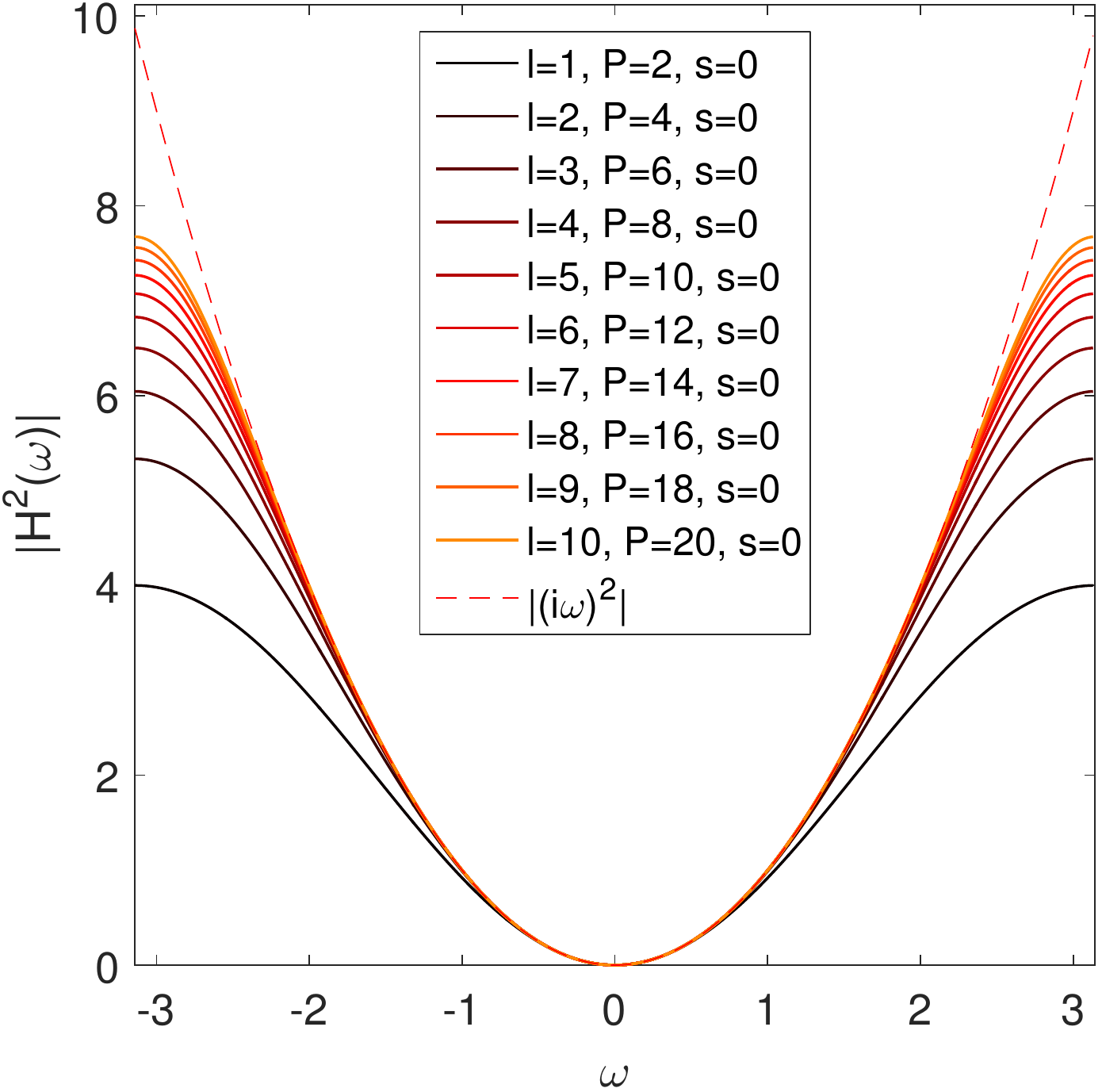}\label{fig_fullband_centralized_maxpol_n_ord_2}}
\subfigure[stagg. ($n=1$)]{\includegraphics[height=0.18\textwidth]{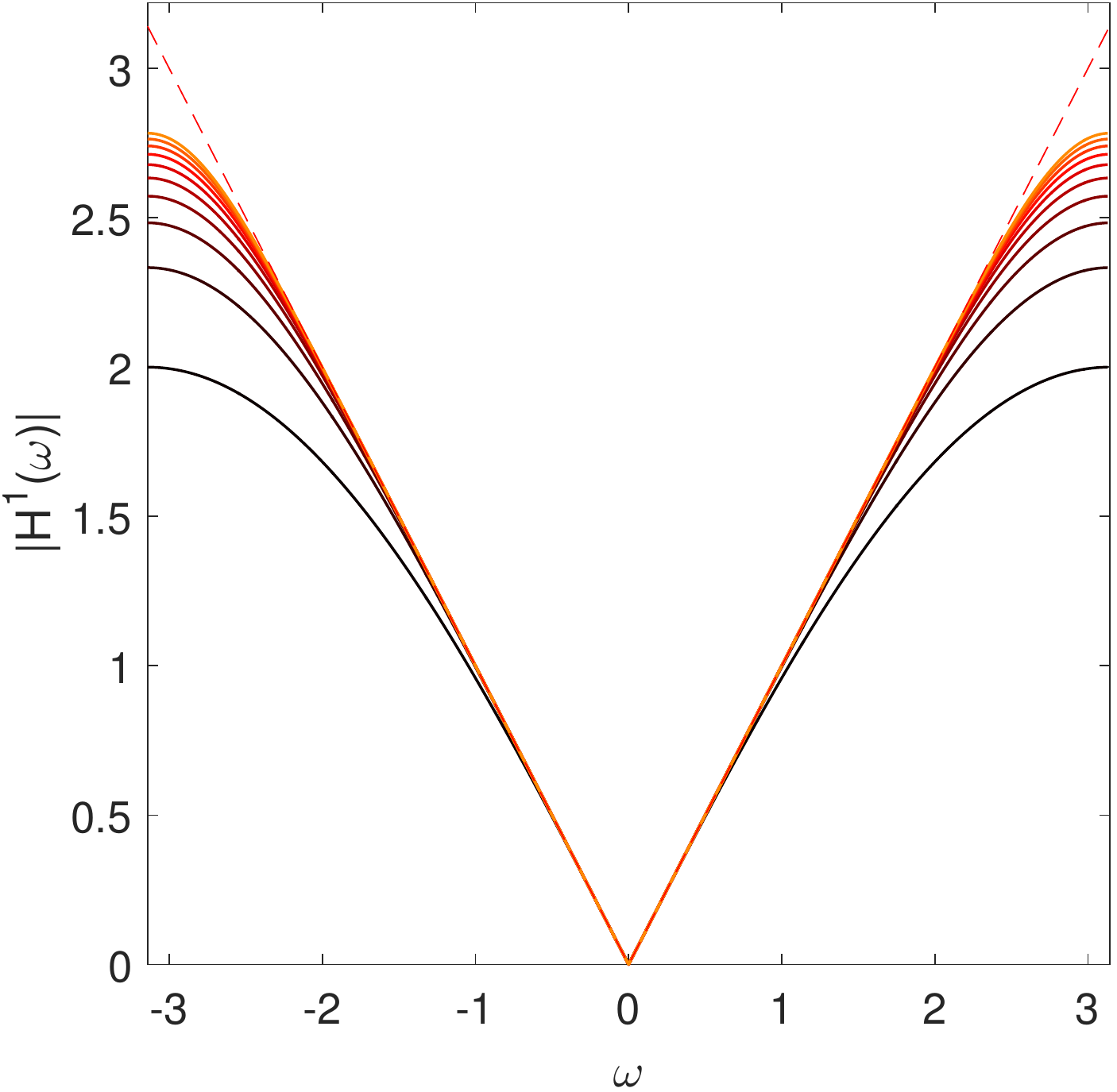}\label{fig_fullband_staggered_maxpol_n_ord_1}}
\subfigure[stagg. ($n=1$)]{\includegraphics[height=0.18\textwidth]{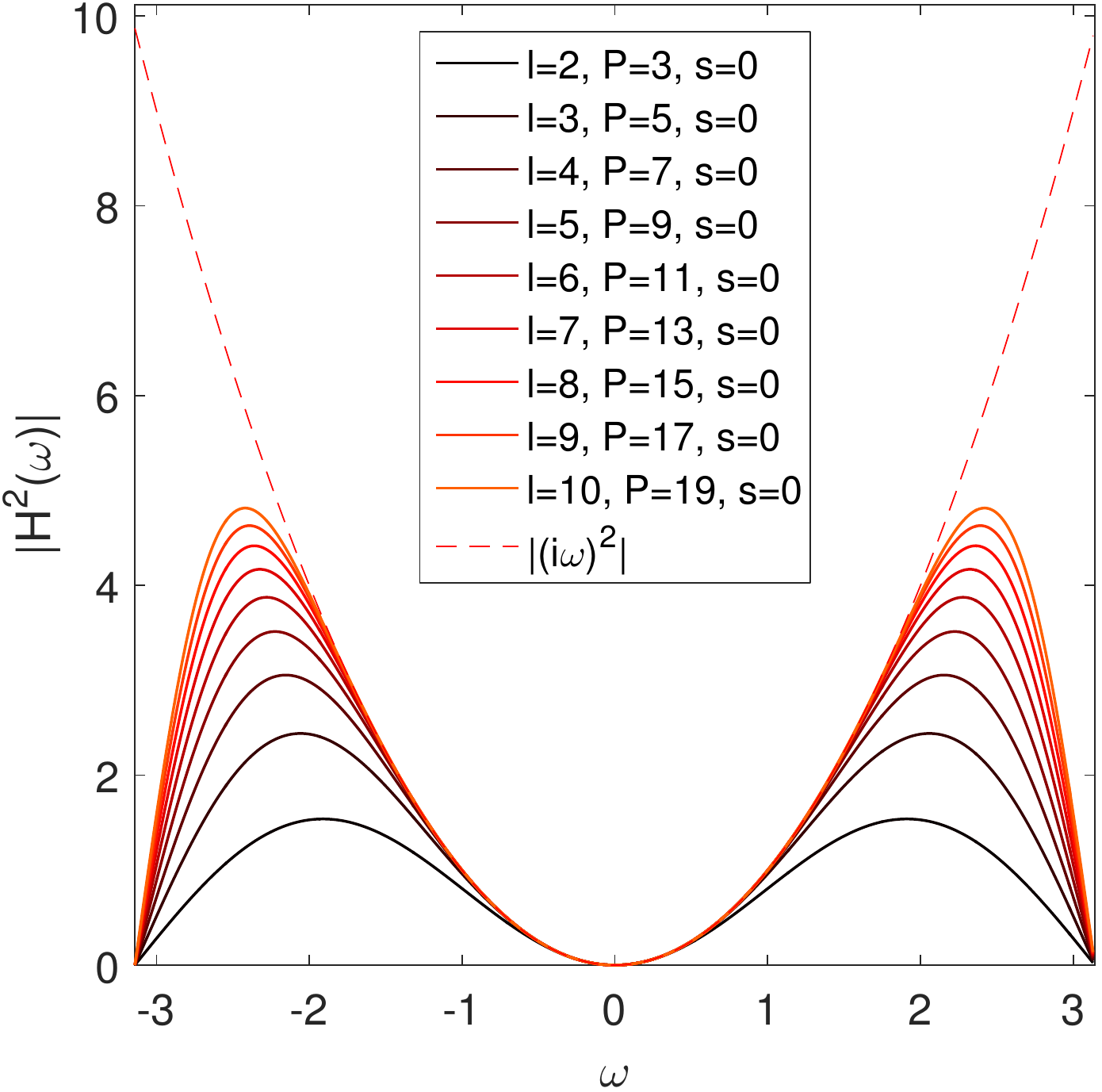}\label{fig_fullband_staggered_maxpol_n_ord_2}}
}
\centerline{
\subfigure[centr. ($n=1$)]{\includegraphics[height=0.18\textwidth]{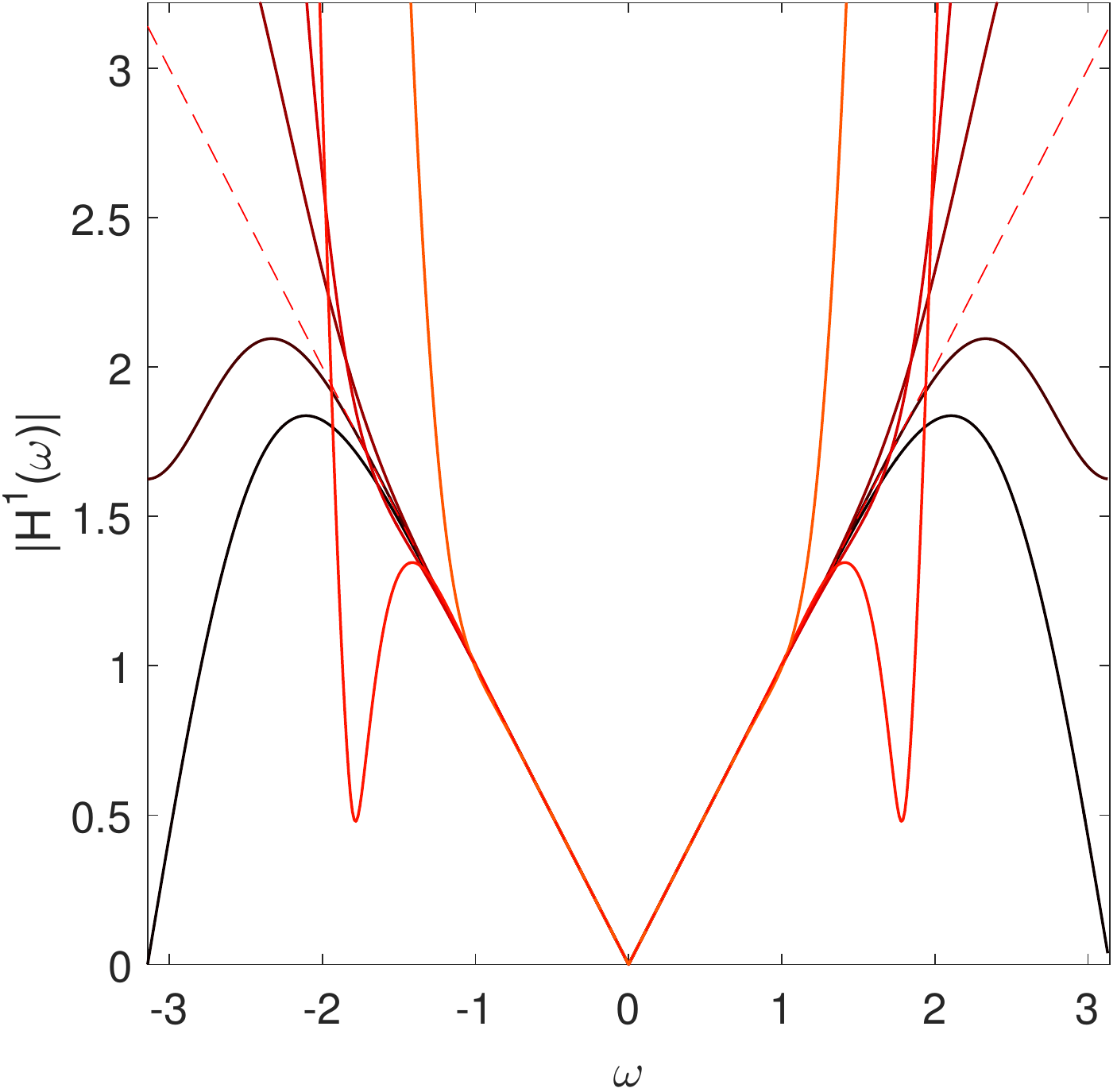}\label{fig_sideshift_centralized_maxpol_n_ord_1}}
\subfigure[centr. ($n=2$)]{\includegraphics[height=0.18\textwidth]{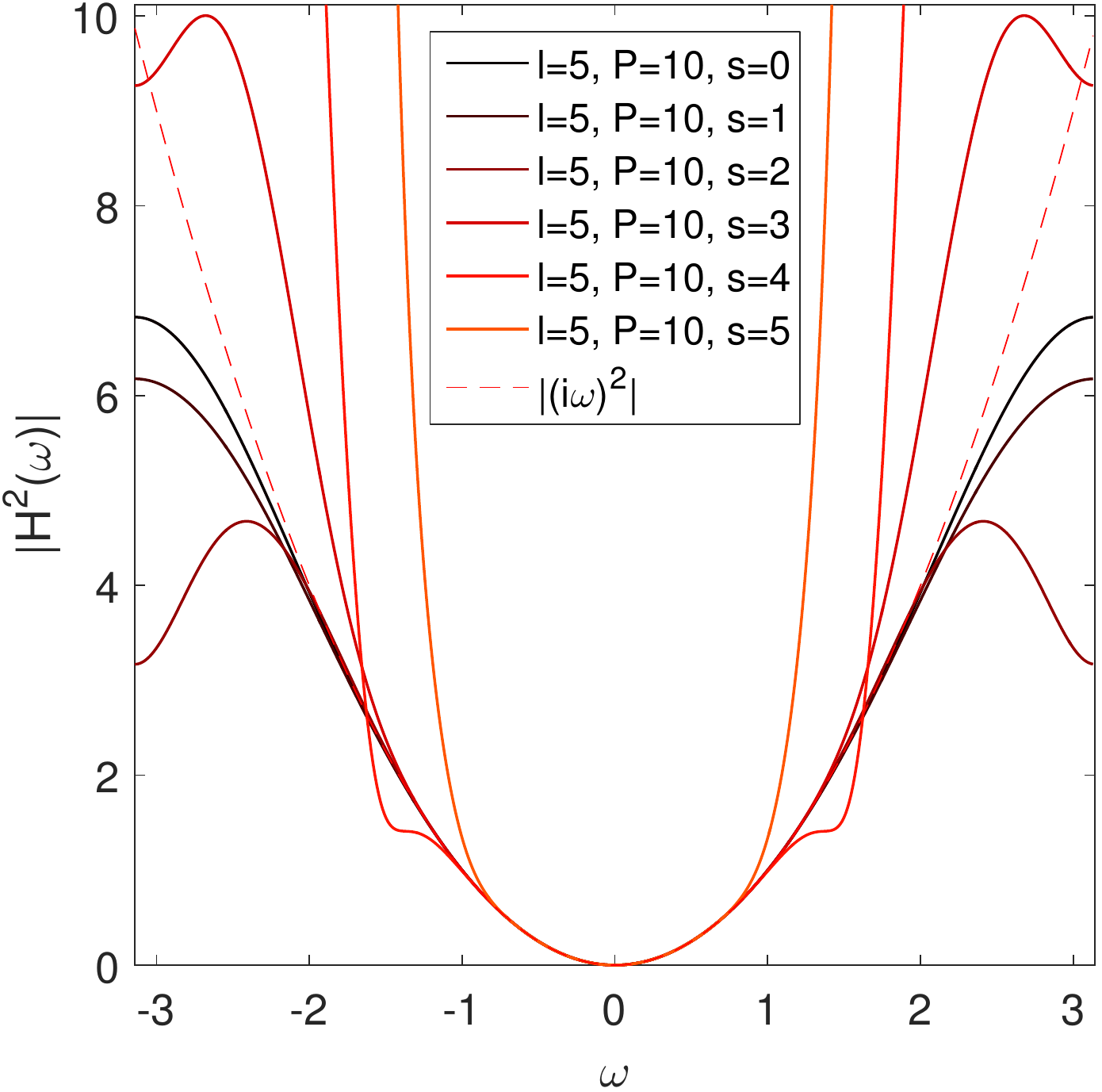}\label{fig_sideshift_centralized_maxpol_n_ord_2}}
\subfigure[stagg. ($n=1$)]{\includegraphics[height=0.18\textwidth]{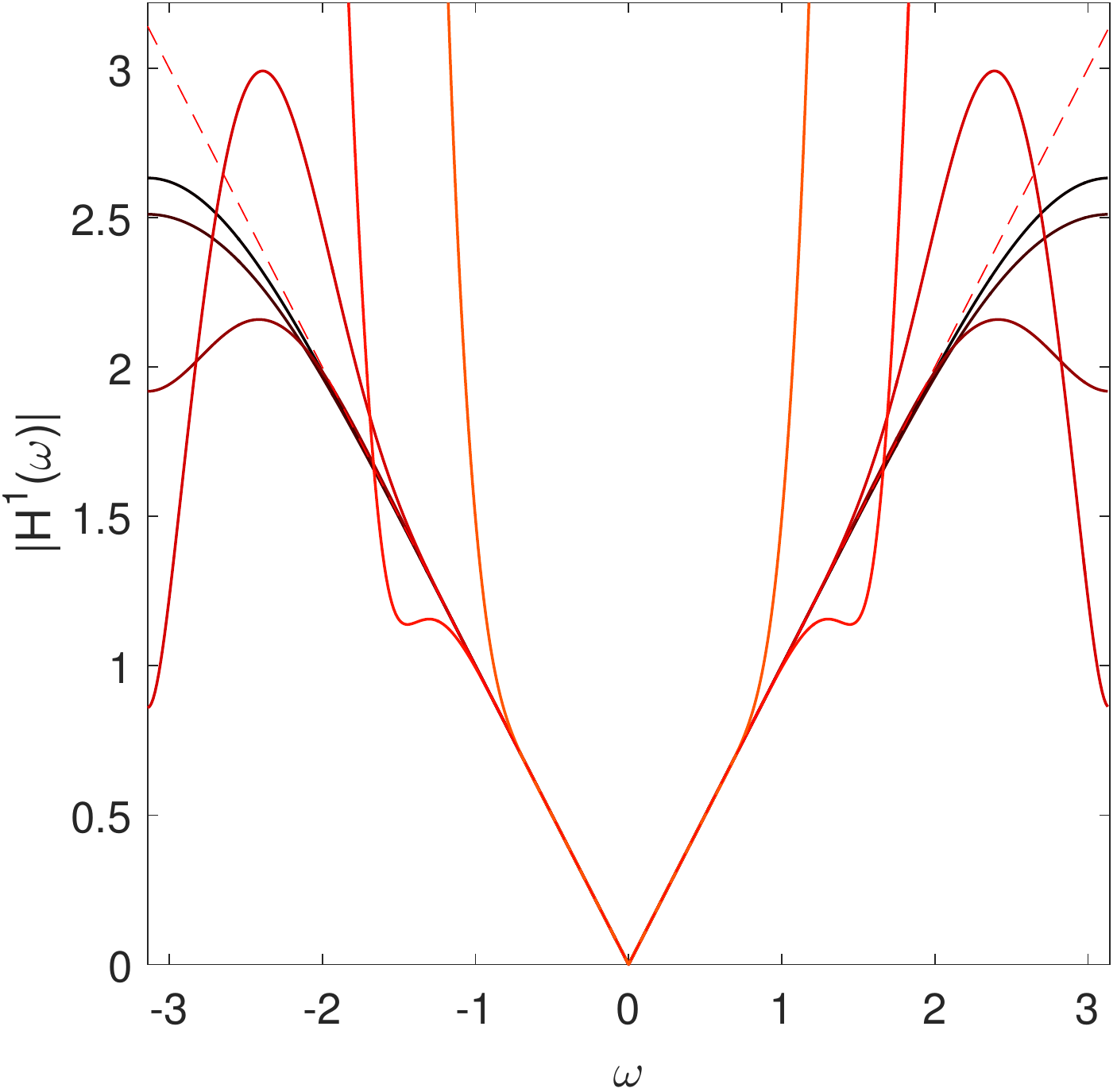}\label{fig_sideshift_staggered_maxpol_n_ord_1}}
\subfigure[stagg. ($n=2$)]{\includegraphics[height=0.18\textwidth]{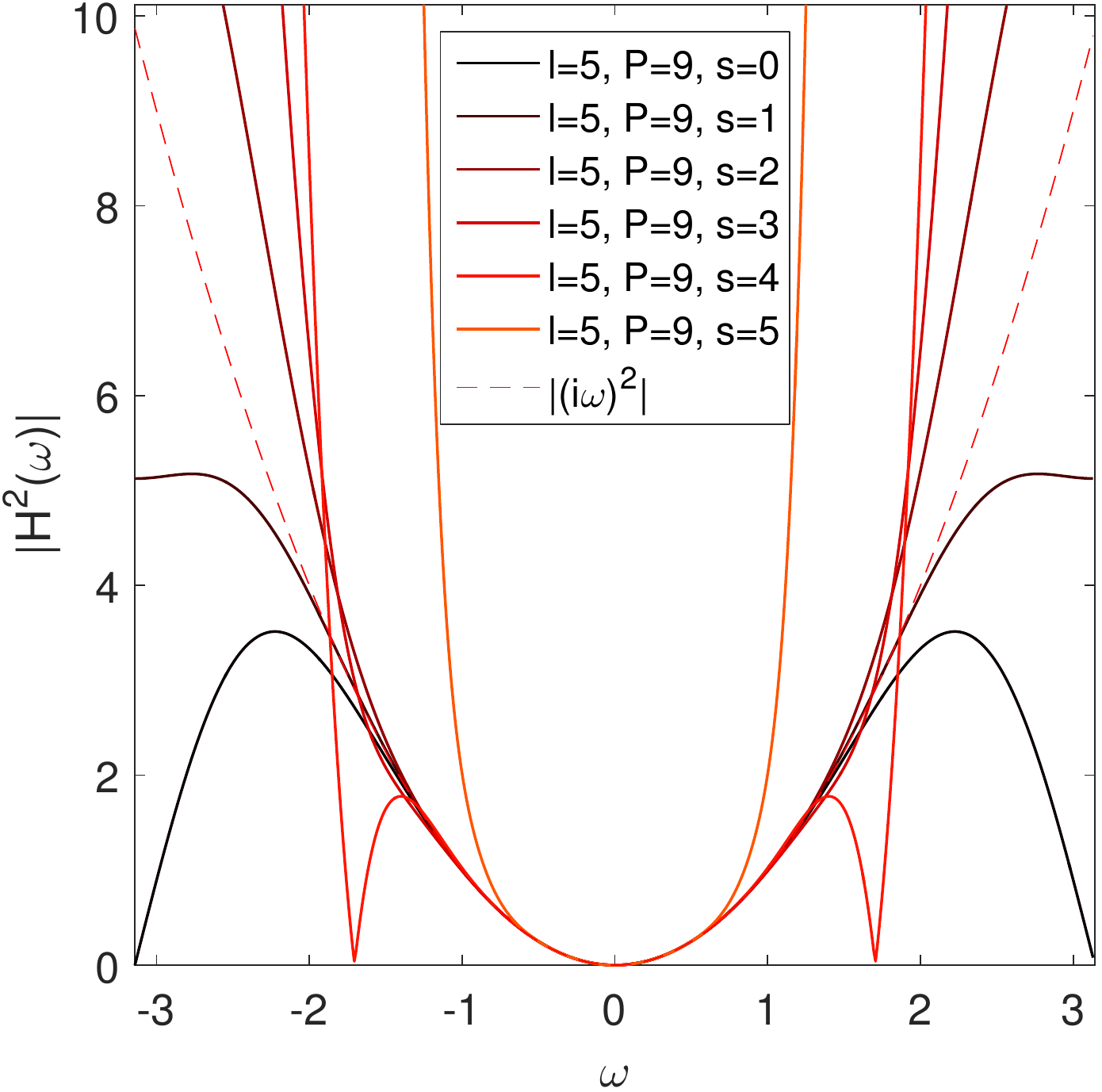}\label{fig_sideshift_staggered_maxpol_n_ord_2}}
}
\caption{Transfer functions (filter response) of fullband differentiators derived from MaxPol method. First and second rows correspond to zero-side shift $s=0$ and non-zero side-shift kernels, respectively.}
\label{fig_fullband_maxpol}
\end{figure}

\ref{fig_fullband_maxpol} displays few filter responses derived from the latter theorems for both staggered and centralized derivatives on both zero and non-zero side-shift nodes. As it shown, the staggered model accurately approximates odd order of derivatives (first order in this example) such that the deviation of the filter response from the ideal derivative is minimal using high order of accuracies i.e. high order polynomials. This is valid for both cases of zero/non-zero side shift nodes. Similarly, the even order of derivatives are well approximated by centralized schemes using high order of accuracy filters. As mentioned earlier, this is of paramount importance in inverse problems to establish complete and accurate frequency transforms between derivative and the recovery domains. For instance, the centralized scheme on first order of differentiation plagues the frequencies and cause perturbation in the recover stage of inverse problems. We will discuss this further in application studies.

In next section we employ the above kernels to design derivative matrices. It is worth noting that for non-zero shift the filter response is plagued by Runge phenomenon for either case of staggered and centralized formulation. We observe that $l<8$ provides a reasonable range of numerical accuracy for filter-tap length selection.

\subsection{Non-uniformly spaced sampling}
The definition of MaxPol differentiators are based on the undetermined coefficients defined in \ref{eq1} for both centralized and staggered schemes. The corresponding polynomial function value $f_{j+k}$ is defined for evenly-spaced samples i.e. $k\in\{-l,-l+1,\hdots, l\}$ in this paper. In the case of non-uniformly spaced data, the sample domain can be defined on a set such that $k\in\{\Omega=k_1,k_2,\hdots,k_N\}$ where $k_i\in\mathbb{Z}$ is an arbitrary integer value. This assumption extends the MaxPol design in \ref{eq9_centralized} and \ref{eq9_staggered} to non-uniform sampling case, which is a more generalized version of the current setup. Without loss of generality, our derivations remain valid for non-uniform case. This includes the closed form solutions we have found for Vandermonde matrix inversion utilized in \ref{lemma_staggered} and \ref{lemma_centralized}.

While the extension of evenly-spaced formulae to non-even setup seems to be straight forward by such minor revision, however, this requires more investigation on the stability analysis of the filter solutions. A direct relation of such non-uniformity case can be named under different notions such as under-sampling/up-sampling in digital signal processing, partial sampling in compressed sensing problems, 3D reconstruction in computer vision problems, etc.

\section{Derivative matrix design}\label{section_convolution_mtx}
Throughout this section we introduce the discrete differential operators designed in matrix form to approximate the $n$th-order derivative of a vector-valued function $f=[f_0,f_1,\ldots,f_{N-1}]\in\mathbb{R}^N$ in a matrix-vector multiplication form
\begin{equation}
\frac{d^nf}{dx^n}\approx D_{n} f,~~~\text{where}~\mathcal{N}\left(D_{n}\right)=\left\{x: D_{n} \left(\sum\limits_{p=0}^{n-1}\alpha_p x^p\right)=0, \alpha_p\in\mathbb{R}\right\}.
\label{eq23}
\end{equation}
Here, $D_{n}\in\mathbb{R}^{N\times N}$ denotes the $n$th-order derivative matrix. The operator is $n$-rank deficient due to non-empty nullspace of $n-1$ order polynomials. For instance, the constant, linear and quadratic vector-valued functions lie in the nullspaces of first, second, and third order derivatives, respectively. A non-empty nullspace implies that the matrix is not invertible. Therefore the reverse operation to recover vector values from approximated derivatives returns to be an inverse problem. The design of centralized nodes has been introduced in \cite{Li2005, o2008algebraic, o2012framework, HassanMohamadAtteia2012}. Here we take a further step to generalize this design for staggered/centralized formulations.  

\subsection{Building blocks}\label{sec_matrix_design}
The discrete numerical design of derivative matrix $D_{n}$ is demonstrated in \ref{fig_bc_model_1}. Every row of the matrix is associated with the $n$th-order derivative coefficients to approximate the numerical derivatives of a vector-valued function $f\in\mathbb{R}^N$ at particular node $0\leq j\leq N-1$. Two types of staggered and centralized matrices are constructed. We refer the reader to previous section for numerical computation of the derivative coefficients. The centralized type matrices approximate the derivatives at exact nodes, while the staggered type shift the approximation by half-a-grid sample. Specifically, the matrix $\left[D^n\right]_{N\times N}$ in \ref{fig_bc_model_1} is structured by three building blocks:
\begin{itemize}[leftmargin=*]
\item \textit{left-block:} formulates the left boundary derivatives using the positive side-shifts $s>0$
\item \textit{middle-block:} formulates the interior domain differentiation with zero side-shift $s=0$
\item \textit{right-block:} formulates the right boundary derivatives using the negative side-shifts $s<0$
\end{itemize}
The coefficients of left and right blocks are stacked vertically which are composed of incremental side-shifts towards the boundaries. While the coefficients in middle block are diagonally shifted in a Toeplitz structure and estimates the derivatives within the interior domain.

\begin{figure}[htp]
\centerline{
\subfigure[stagg. left BC]{\includegraphics[width=0.37\textwidth]{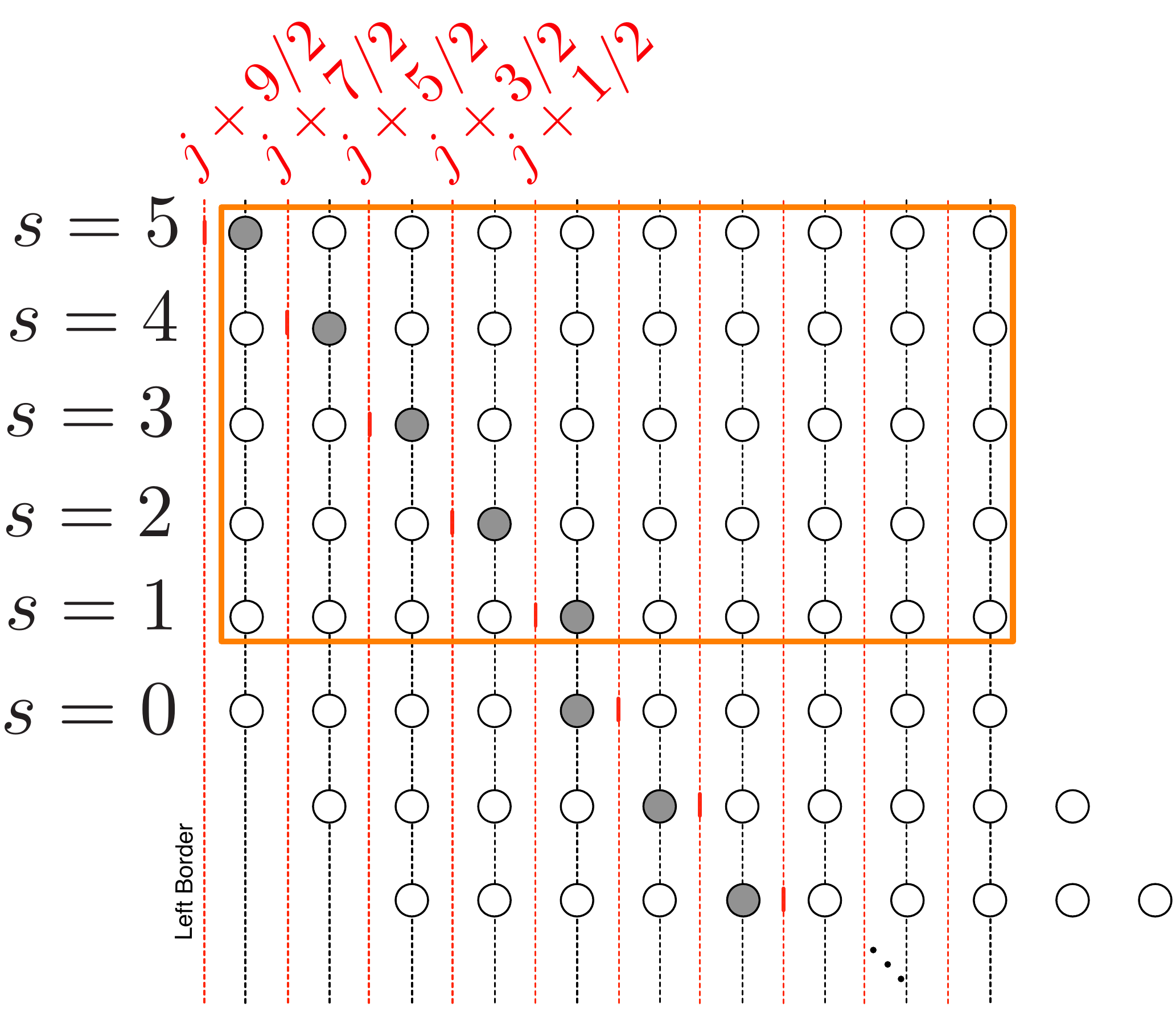}\label{fig_left_bc_model}}
\subfigure[stagg. right BC]{\includegraphics[width=0.37\textwidth]{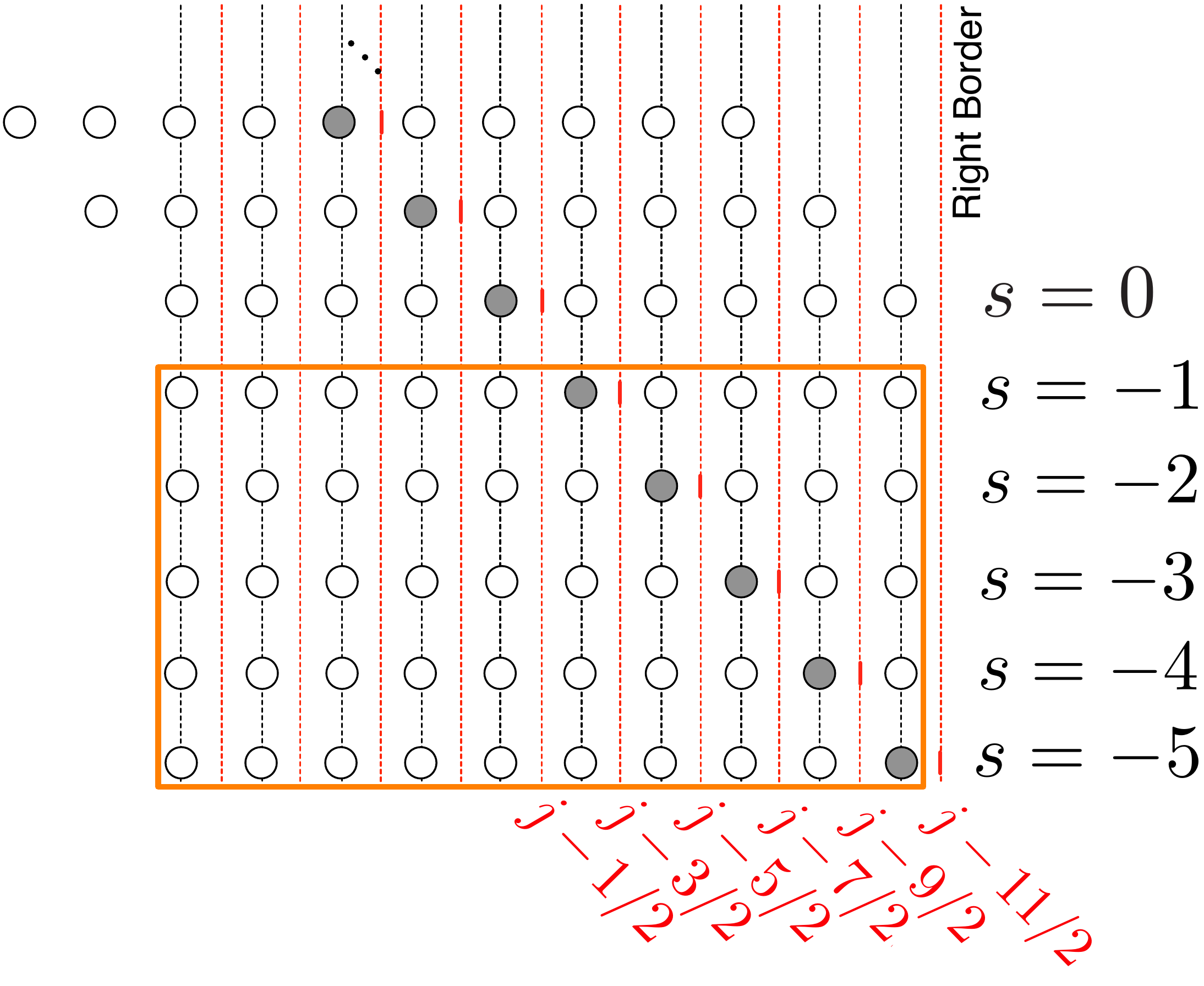}\label{fig_right_bc_model}}
}
\centerline{
\subfigure[centr. left BC]{\includegraphics[width=0.37\textwidth]{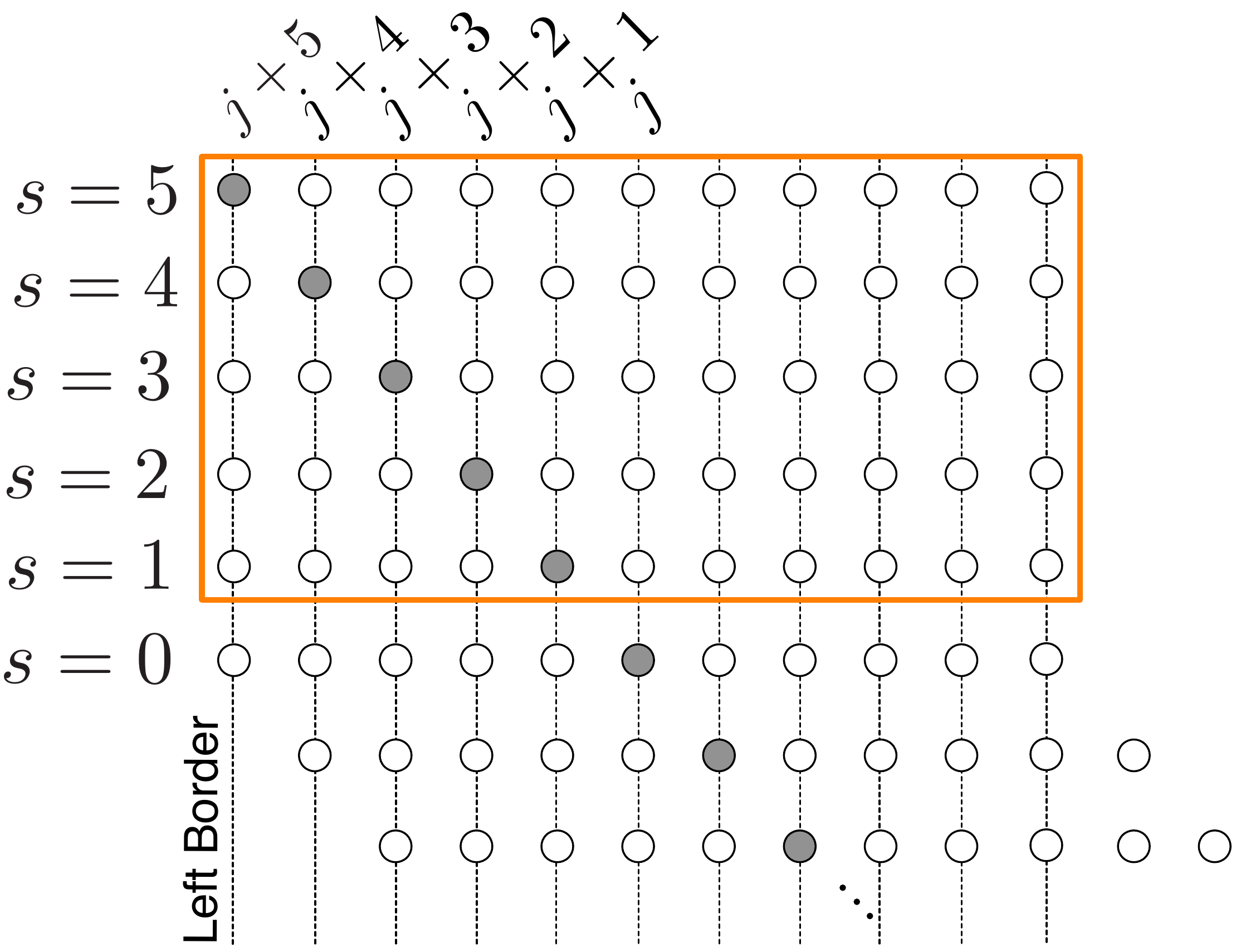}\label{fig_left__central_bc_model}}
\subfigure[centr. right BC]{\includegraphics[width=0.37\textwidth]{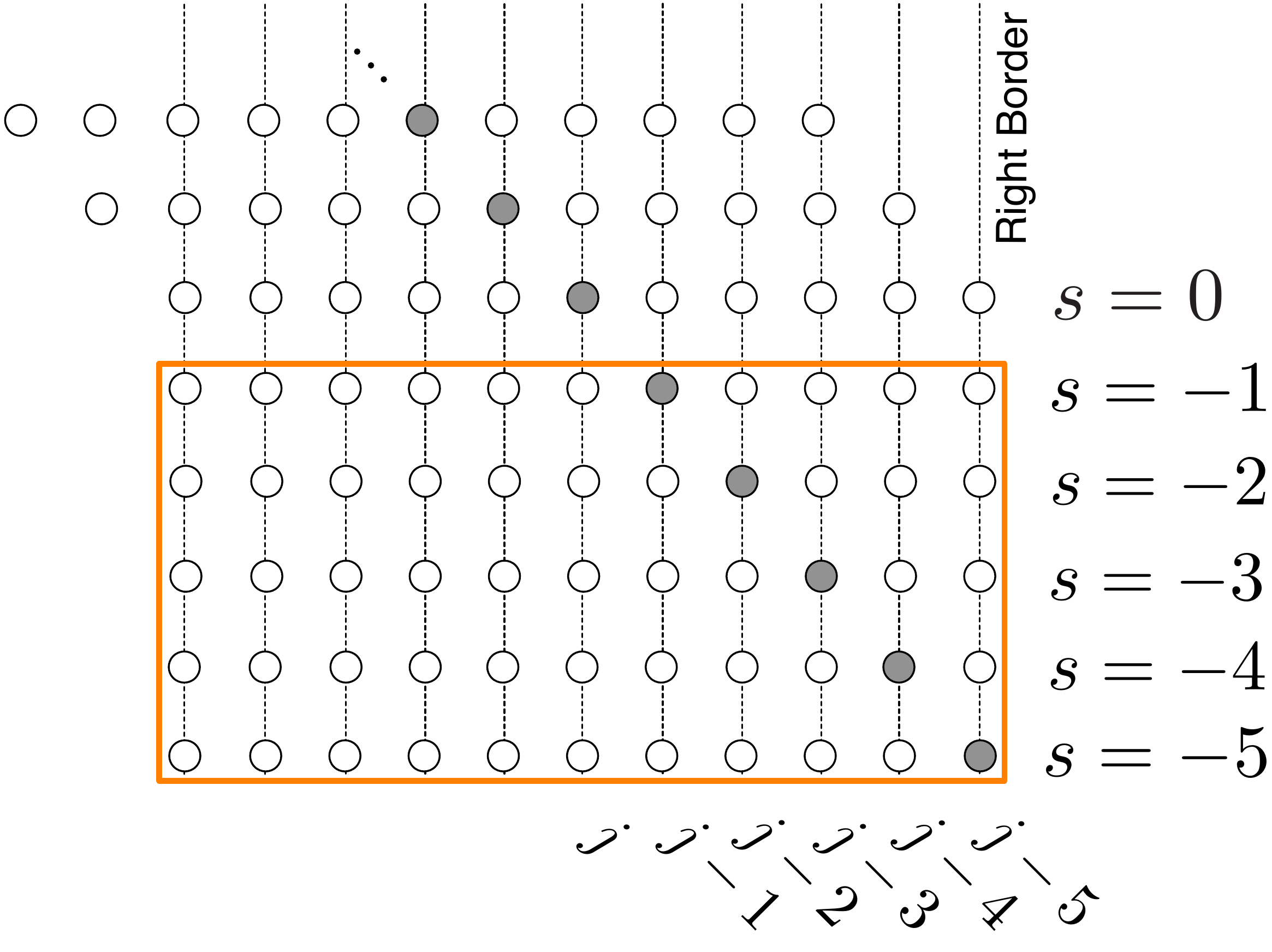}\label{fig_right_central_bc_model}}
\subfigure[$D_n$]{\includegraphics[width=0.25\textwidth]{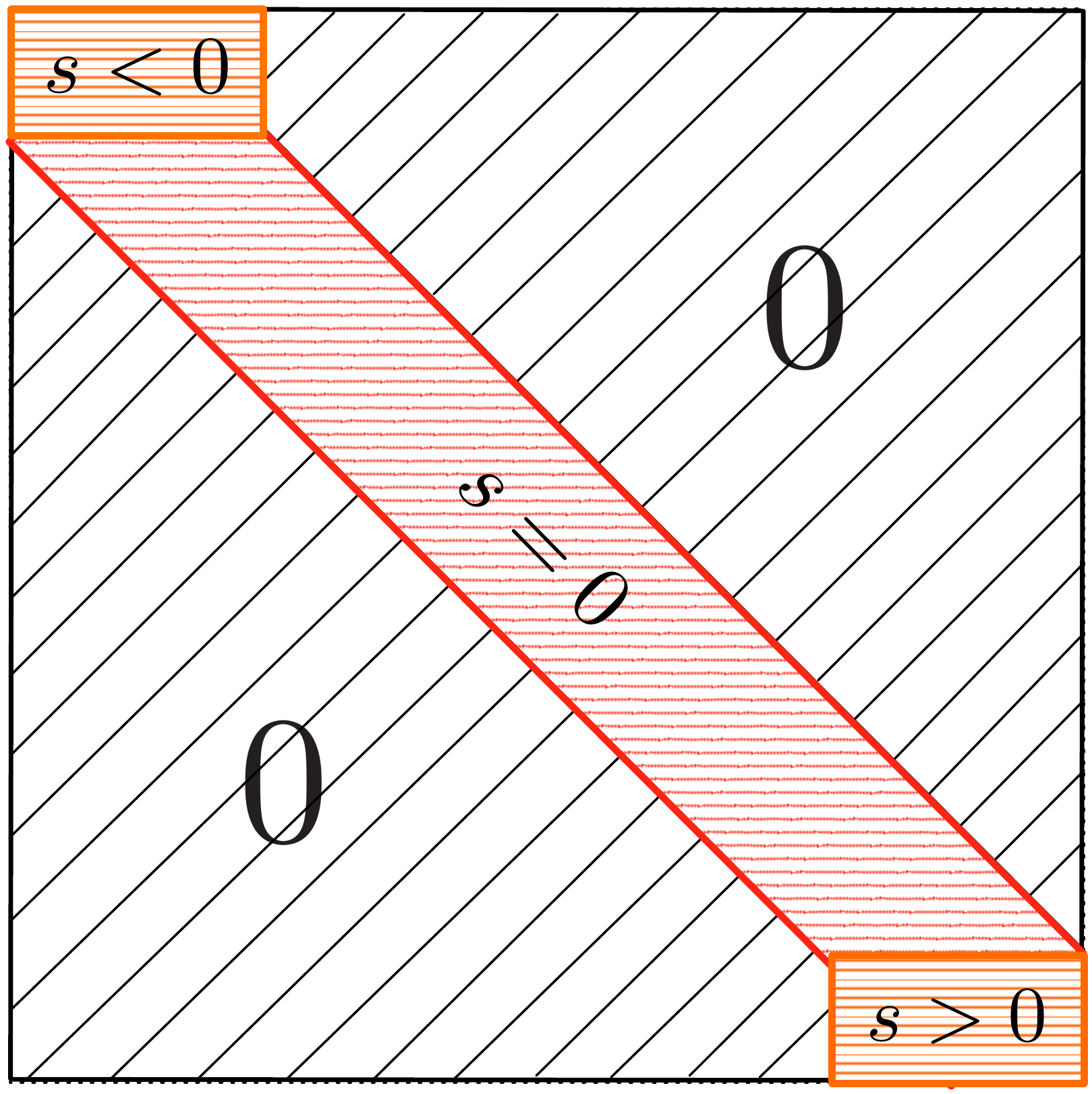}\label{fig_fig_diffmtx}}
}
\caption{Stencil model of staggered/centralized derivative matrix $D_{n}$ with tap-length polynomial $l=5$}
\label{fig_bc_model_1}
\end{figure}

\begin{figure}[htp]
\centerline{
\subfigure[$l=5$, $n=1$]{\includegraphics[height=0.35\textwidth]{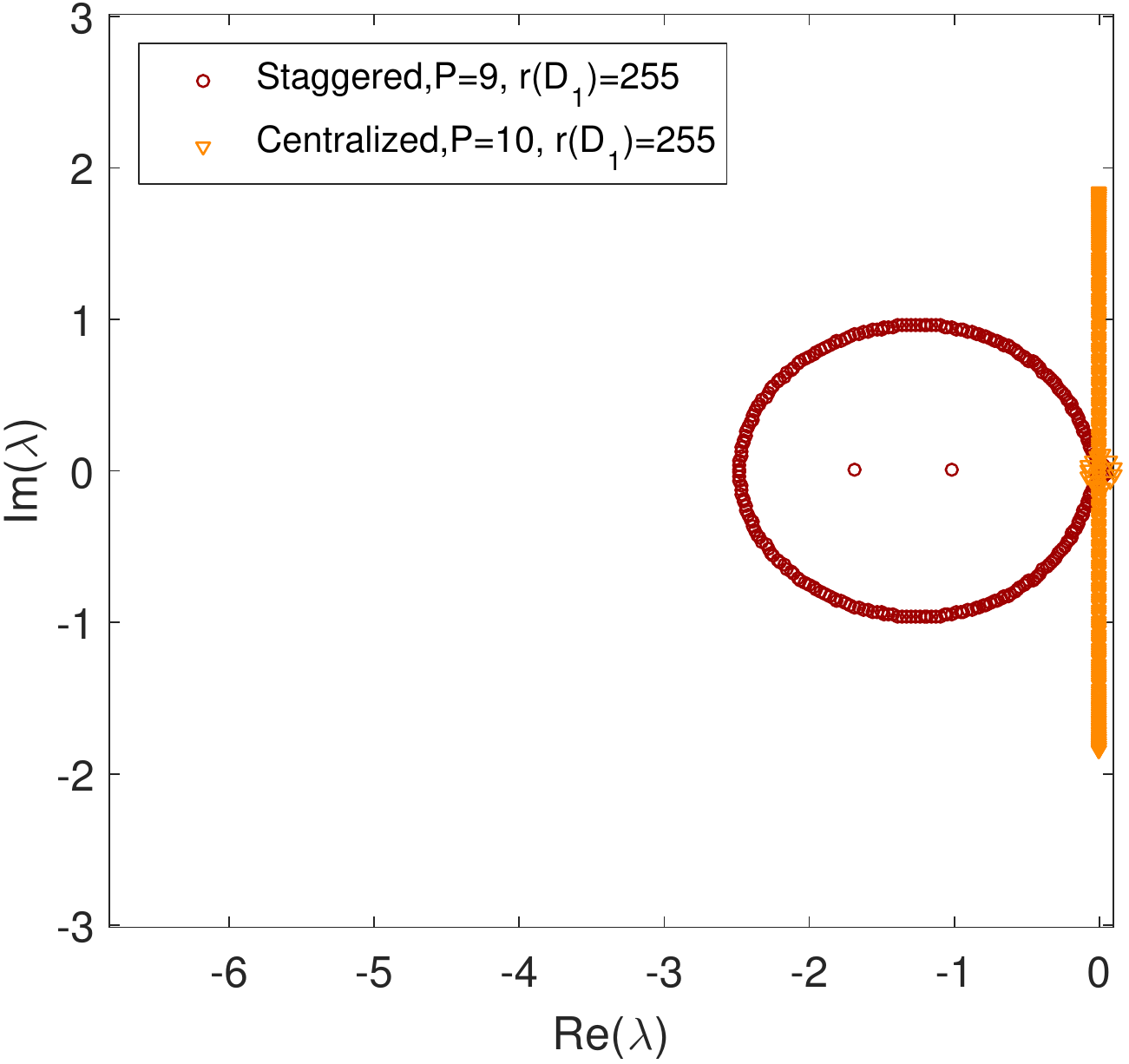}\label{fig_Eigenvalue_dstribution_staggered_vs_centralized_fullband_n1}}
\subfigure[$l=5$, $n=2$]{\includegraphics[height=0.35\textwidth]{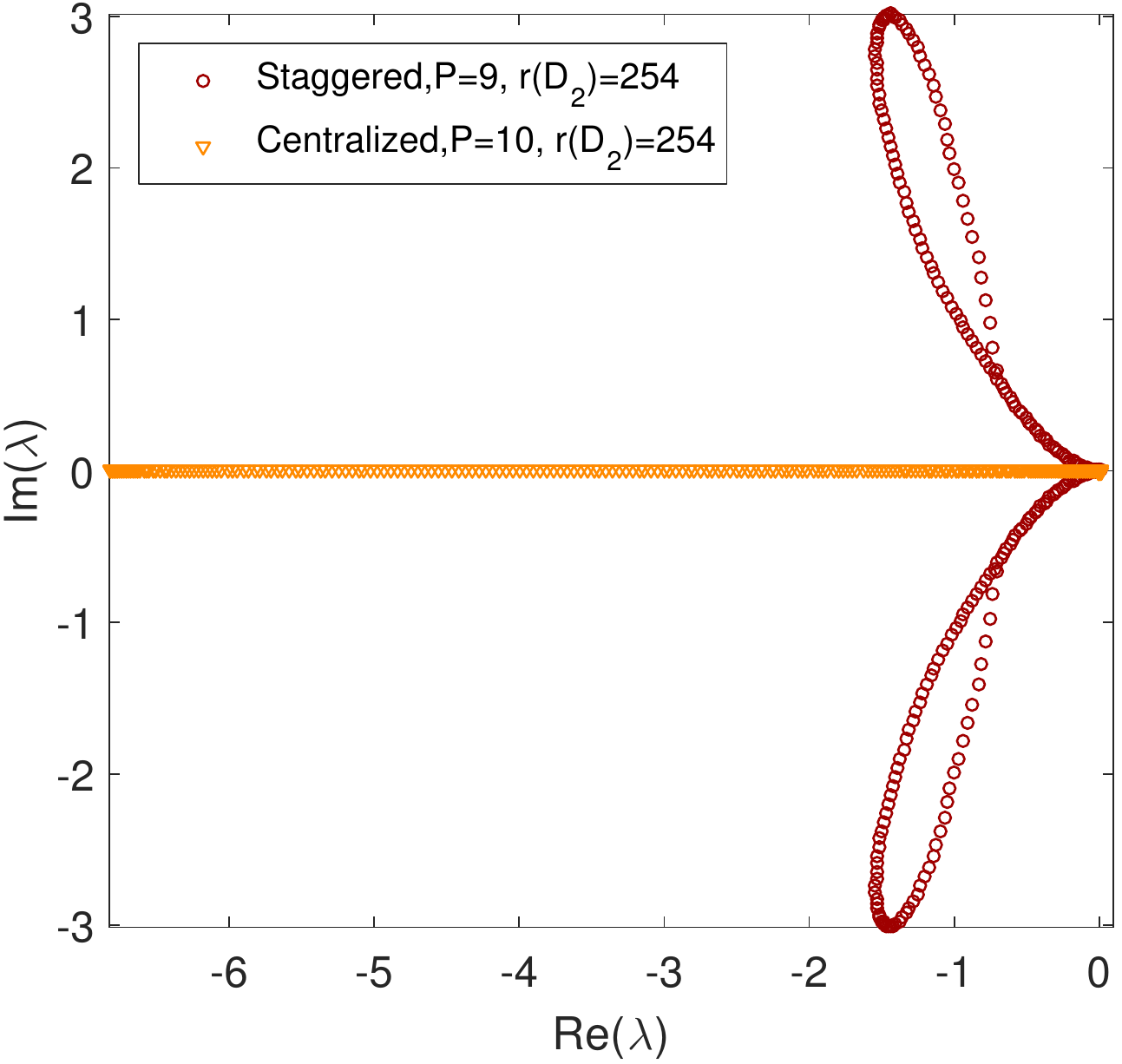}\label{fig_Eigenvalue_dstribution_staggered_vs_centralized_fullband_n2}}
}
\centerline{
\subfigure[Staggered $n=1$]{\includegraphics[height=0.4\textwidth]{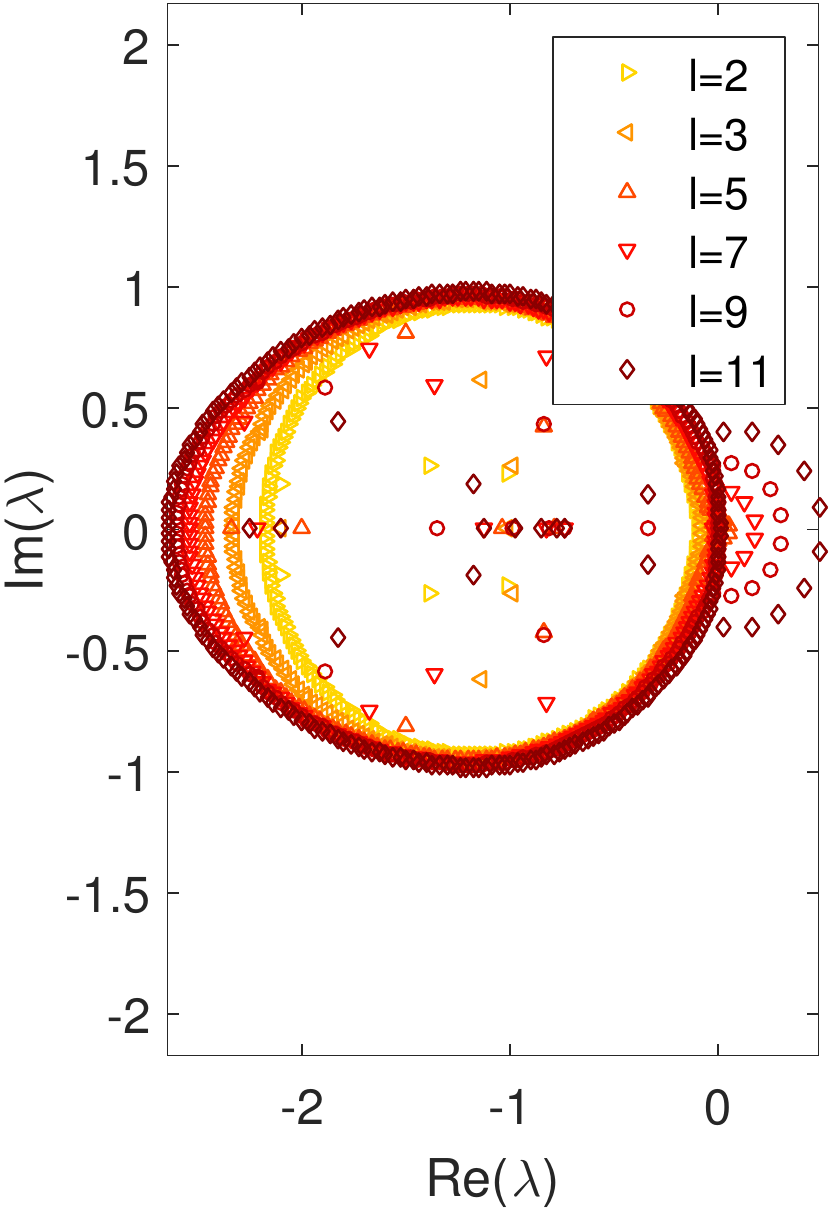}\label{fig_Eig_dist_on_different_polynomial_full_stag}}
\subfigure[Centralized $n=1$]{\includegraphics[height=0.4\textwidth]{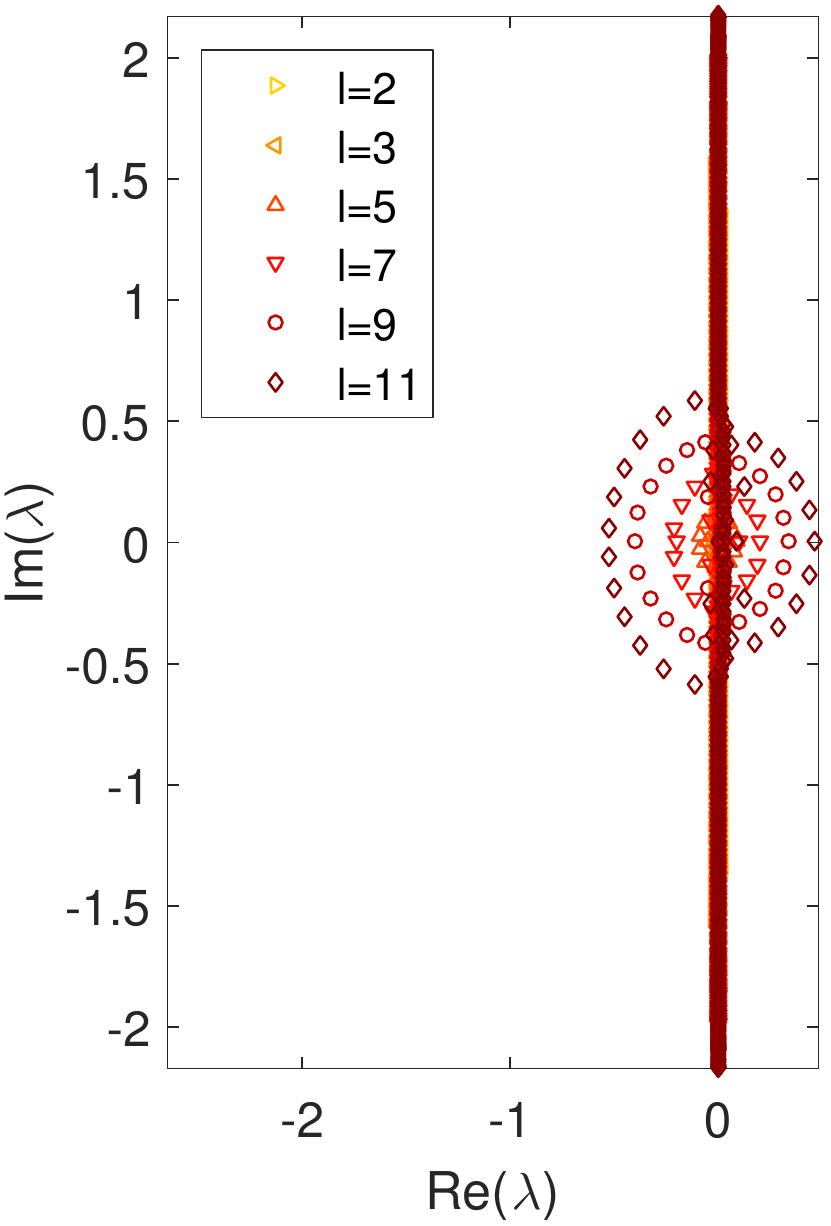}\label{fig_Eig_dist_on_different_polynomial_full_cent}}
\subfigure[Centralized $n=1$]{\includegraphics[height=0.4\textwidth]{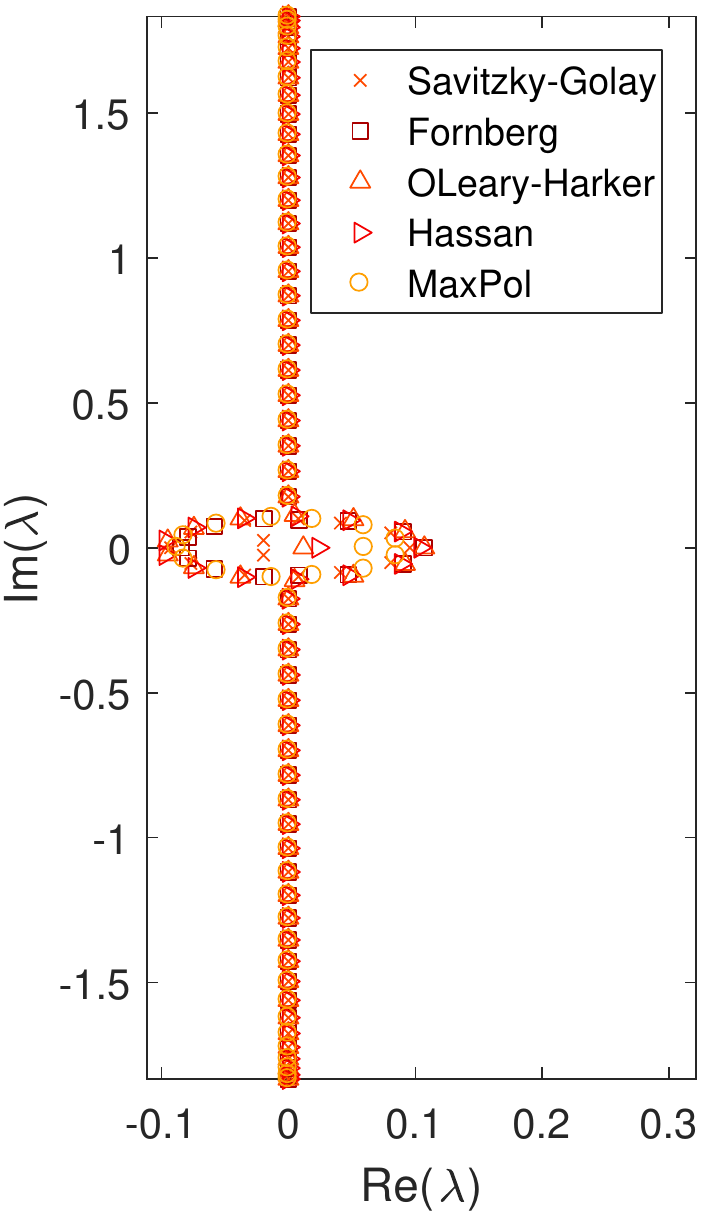}\label{fig_Eigenvalue_dstribution_centralized_fullband_n1_comparision_methods}}
\subfigure[Centralized $n=3$]{\includegraphics[height=0.4\textwidth]{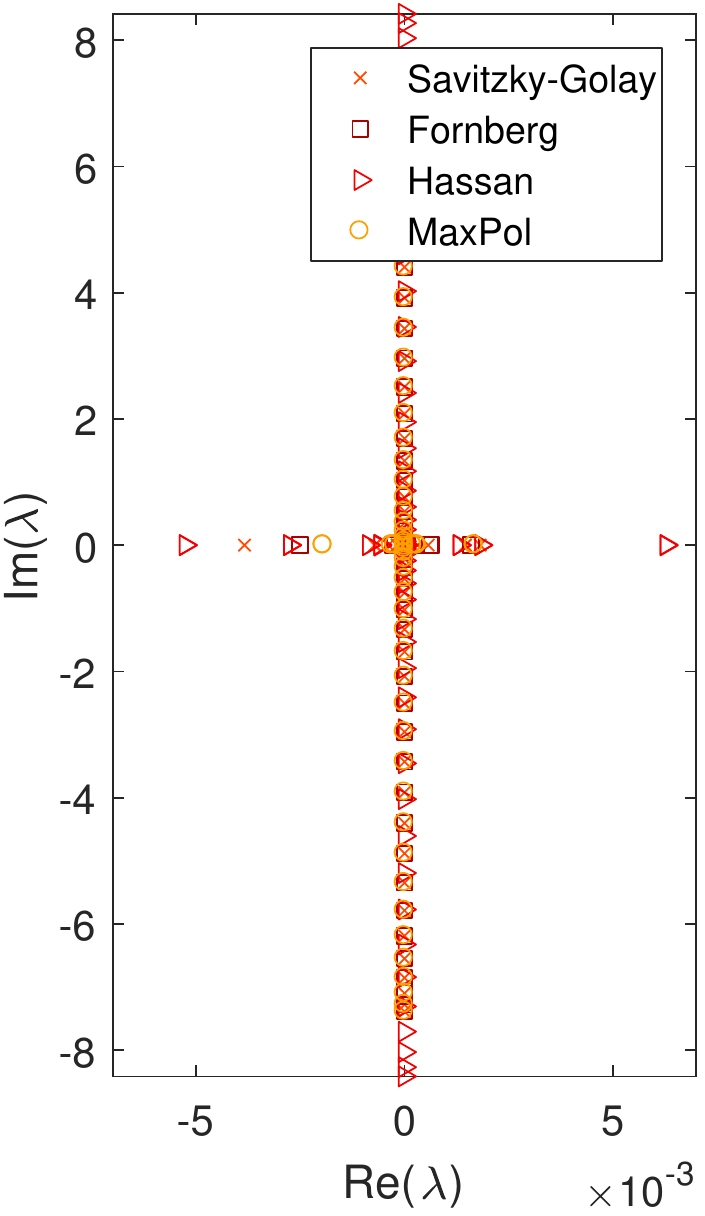}\label{fig_Eigenvalue_dstribution_centralized_fullband_n3_comparision_methods}}
}
\caption{(a)-(d) Distribution of eigenvalues of staggered/centralized derivative matrices ($N=256$) proposed by MaxPol method, (e)-(f) Comparison of the spectral decomposition of centralized derivative matrix ($N=64$ and $l=5$) constructed by different methods. In all setups, the fullband coefficients are used with no-cutoff. Eigenvalues of different parameter setup are overlaid on the same canvas.}
\label{fig_bc_model_2}
\end{figure}

\subsection{Stability via location of eigenvalues}\label{sec_stability_eigenvlaues}
The numerical stability of the derivative matrices are studied here by means of spectral analysis \cite{trefethen2005spectra}. In particular we consider the scalar linear ODE model problem $f^{(n)}=\lambda f$ where the solution is obtained by solving an eigenvalue problem $D_{n}f=\lambda f$. The numerical stability of the problem is validated through locating the eigenvalues in complex plane where they should reside on the left-half plane. A better understanding of the problem can be also made by Gersgorin theorem  \cite{horn2012matrix}. It indicates the eigenvalues of a matrix which is compactly localized around its main diagonal are in the union of Gresgorin discs prescribed around each diagonal entries. This provides a means to determine the closeness of the eigenvalues to the diagonal entries by measuring the norms of off-diagonals at every row of the matrix. Given the odd/even order of differentiation and centralized/staggered schemes, four possible cases are defined:
\begin{enumerate}[leftmargin=*]
\item \textbf{Case--I: staggered odd derivative matrix.} contains two dominant coefficients around the center of the even tap-length filter. The absolute values of these coefficients are close to the off-diagonals indicating that the eigenvalues contain real components. \ref{fig_Eigenvalue_dstribution_staggered_vs_centralized_fullband_n1} demonstrates the distribution of eigenvalues of first derivative matrix where they are concentrated in the left-half of the complex plane bounded by a closed loop.
\item \textbf{Case--II: centralized odd derivative matrix.} The tap-length filter is odd where the diagonal entries are mainly zeros (except left and right blocks). This implies the eigenvalues are far away from zero. \ref{fig_Eigenvalue_dstribution_staggered_vs_centralized_fullband_n1} demonstrates the distribution of the eigenvalues for centralized first derivative matrix, where the majority of the eigenvalues are purely imaginary. Similar designs are also introduced in \cite{Li2005, o2008algebraic, o2012framework, HassanMohamadAtteia2012}. The distribution of the eigenvalues of the comparison methods are also shown in \ref{fig_Eigenvalue_dstribution_centralized_fullband_n1_comparision_methods} and \ref{fig_Eigenvalue_dstribution_centralized_fullband_n3_comparision_methods} for numerical construction of the first and third order order derivative matrices, respectively. Here, we have also generated the Fornberg's and Savitzky-Golay's version of the derivative matrices by replacing the coefficients of our design case using \cite{fornberg1988, S0036144596322507, SavitzkyGolay1964, gorry1990general, meer1990smoothed, luo2005properties, schafer2011savitzky}. As it shown, the imaginary eigenvalues of the proposed MaxPol derivative matrix are less perturbed compared to all comparison methods. This stems from the fact that the numerical calculation of the coefficients in MaxPol are accommodated by the closed form solutions introduced in \ref{lemma_centralized} with less computational complexity. Such improvement is more noticeable on higher order derivatives.
\item \textbf{Case--III: staggered even derivative matrix.} similar to Case--I, the eigenvalues have non-zero real parts. The distribution of eigenvalues are shown in \ref{fig_Eigenvalue_dstribution_staggered_vs_centralized_fullband_n2} with complex structure, where they reside in the left-plane circling around a butterfly curve.
\item \textbf{Case--IV: centralized even derivative matrix.} With odd tap-length filter, the diagonal entries are mainly dominated by non-zero coefficients. So, the eigenvalues are close to this real value based on the Gersgorin theorem. The distribution of eigenvalues is shown in \ref{fig_Eigenvalue_dstribution_staggered_vs_centralized_fullband_n2} for second derivative matrix where all the eigenvalues are purely real and reside on the left-plane. Similar design case is introduced in \cite{Li2005, HassanMohamadAtteia2012}.
\end{enumerate}
The pertinent eigenvalues of left/right blocks in all type of matrices are mainly concentrated around zero plane looping around a small circle. As per the order of tap-length polynomial $l$ increases, the radius of this circle increases towards positive plane according to the Runge phenomenon. \ref{fig_Eig_dist_on_different_polynomial_full_cent} and \ref{fig_Eig_dist_on_different_polynomial_full_stag} elaborate on this issue where a careful selection of the polynomial is necessary to avoid numerical divergence. Case--I and Case--IV are highly stable since their eigenvalues contain non-zero real parts and mainly reside on the left-half plane. However, Case--II contain mostly pure imaginary eigenvalues and cause oscillations in recovery. This design case has been vastly used in inverse imaging problems in \cite{o2008algebraic, o2012framework, xie2014surface, harker2015regularized, queaunormal2016, EstellersSoatto2016}. We show in Section \ref{sec_gradient_surface}, the oscillatory characteristics of such design case is the main leading artifacts in recovery that can be greatly improved by an alternative design Case--I. 

We conduct our next stability analysis by observing the expansion of eigenvalues in complex plane in terms of matrix size $N$. \ref{fig_Eigenvalue_range_matrix_size_N_staggered_fullband_n1} and \ref{fig_Eigenvalue_range_matrix_size_N_centralized_fullband_n1} show the upper and lower bounds of the eigenvalues of first derivative matrix. As per the matrix size increases, both imaginary and real components converge. However, the relative comparison of imaginary and real eigenvalues in centralized scheme yields an unbalanced ratio with dominant imaginary components. While this ratio is much balanced in staggered scheme.
\begin{figure}[htp]
\centerline{
\subfigure[Staggered, $P=9$]{\includegraphics[height=0.23\textwidth]{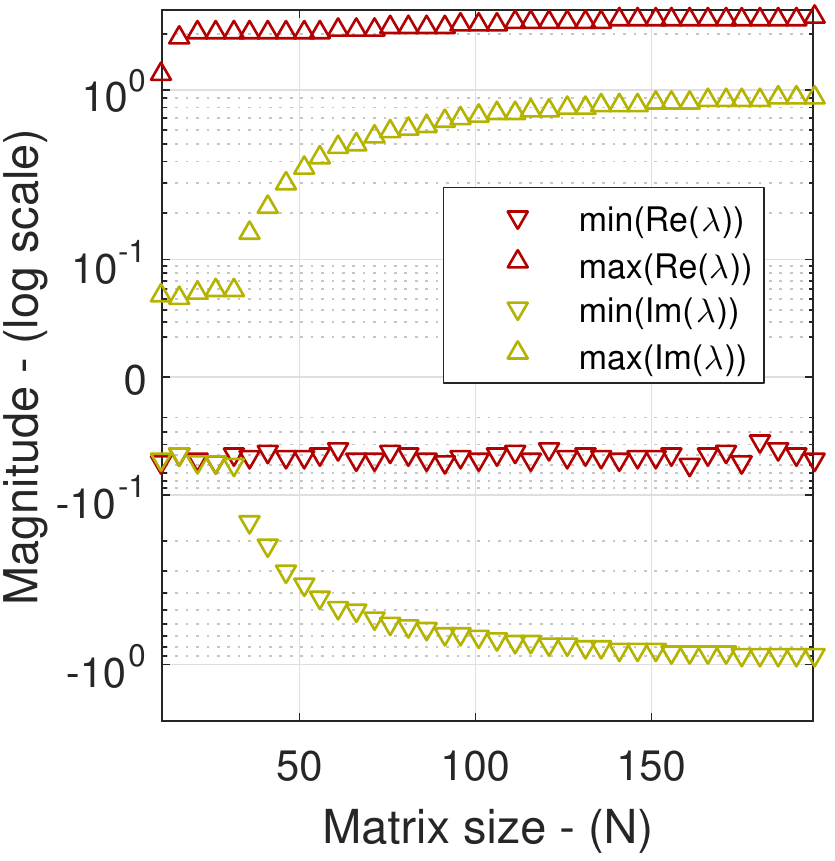}\label{fig_Eigenvalue_range_matrix_size_N_staggered_fullband_n1}}
\subfigure[Centralized, $P=10$]{\includegraphics[height=0.23\textwidth]{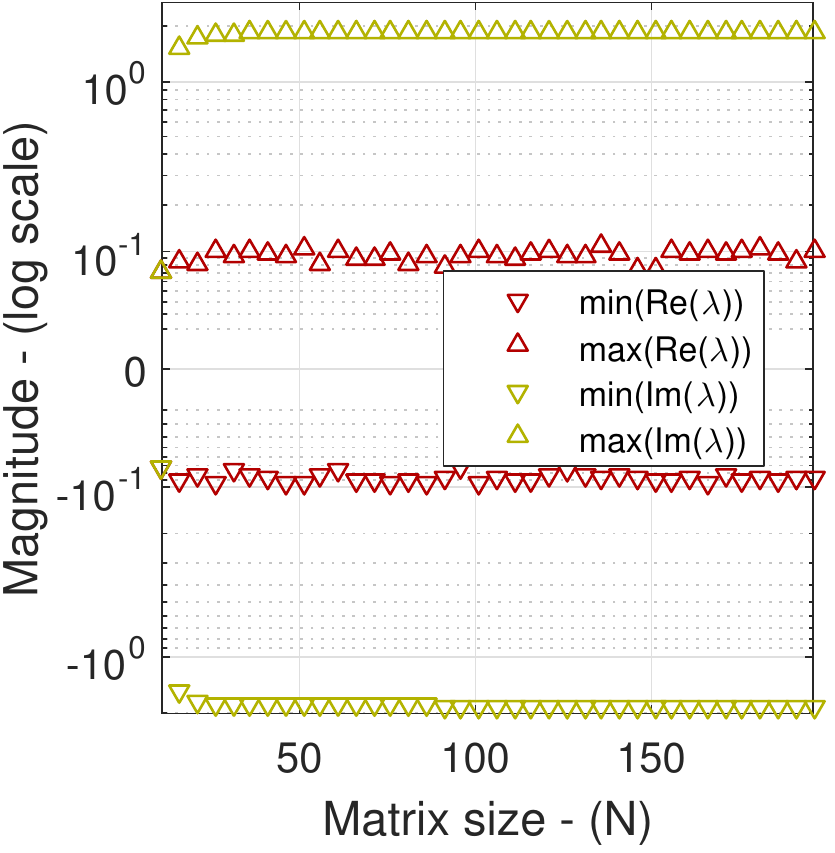}\label{fig_Eigenvalue_range_matrix_size_N_centralized_fullband_n1}}
\subfigure[Savitzky-Golay \cite{luo2005properties, schafer2011savitzky}]{\includegraphics[height=0.23\textwidth]{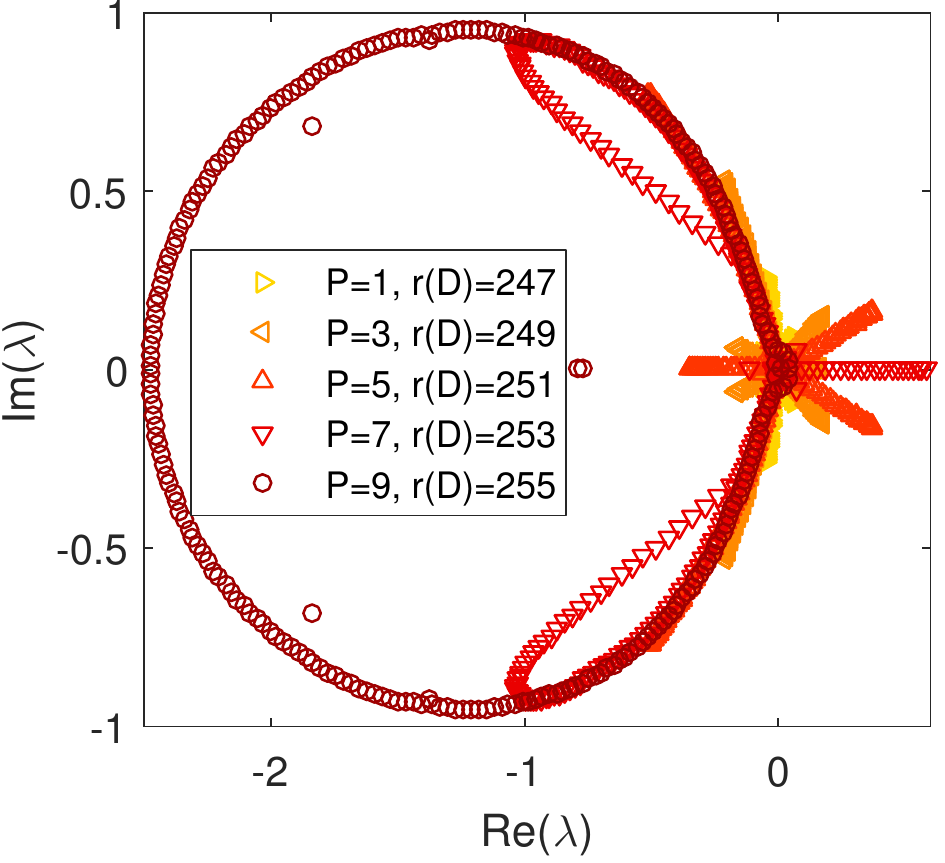}\label{fig_savitzkygolay_eigenvalue_dist_lowpass_to_fullband}}
\subfigure[MaxPol (Proposed)]{\includegraphics[height=0.23\textwidth]{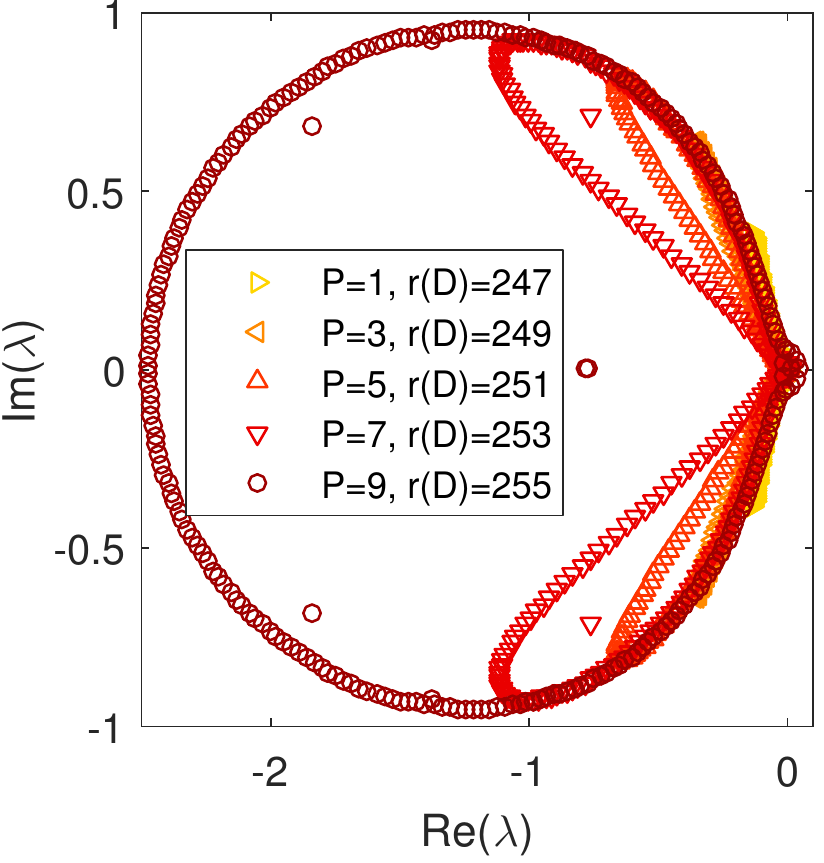}\label{fig_maxpol_eigenvalue_dist_lowpass_to_fullband}}
}
\caption{(a)-(b) Evolution of upper and lower bounds of eigenvalues for fullband first order derivative matrix as a function of matrix size ($N$), (c)-(d) Distribution of eigenvalues of first derivative matrix with staggered nodes ($N=256$, $l=5)$. The cutoff ranges from highest $P=1$ to no-cutoff (fullband) $P=9$.}
\label{fig_Eigenvalue_range_matrix_size_N_fullband_n1}
\end{figure}

Finally we check the stability of the design Case--I by varying the polynomial degree $P$ to set different cutoff parameters. Here we embed the first derivative coefficients from (a) Savitzky-Golay filters \cite{luo2005properties, schafer2011savitzky}, and (b) MaxPol. The distribution of the eigenvalues are shown in \ref{fig_savitzkygolay_eigenvalue_dist_lowpass_to_fullband} and \ref{fig_maxpol_eigenvalue_dist_lowpass_to_fullband}. The positive real eigenvalues for Savitzky-Golay is highly evident due to the rippling artifacts in stopband. While this is resolved in MaxPol design. 

\subsection{Tensor decomposition}\label{sec_tensor_decomposition}
In this section we generalize the separate mode decomposition of discrete data in higher dimensions such as 2D image and 3D volumetric spaces using the matrix framework proposed in \ref{sec_matrix_design}. Recall the matrix-vector multiplication $f^{(n)}=D_n f$ where $f\in\mathbf{R}^N$ is the vector-valued function and $D_n\in M_N$ is a square size derivative matrix associated by centralized/staggered scheme. We add the size of matrix to the notation as another subscript denoted by ${D}_{n,N}$ for centralized, $\overrightarrow{{D}}_{n,N}$ for staggered-forward, and $\overleftarrow{{D}}_{n,N}$ for staggered-backward. Here, the staggered coefficients are embedded either in backward or forward schemes according to the half-sample shift to left and right, respectively. The backward scheme can be obtained by forward scheme $\overleftarrow{{D}}_{n,N}=-J\overrightarrow{{D}}_{n,N} J$ (and vice versa) where $J$ is the permutation matrix. Accordingly, the decomposition framework for 2D images using the above notations is carried as follows. Assuming the orientation of the coordinates $x$ and $y$ in 2D domain corresponds to columns and rows of an image ${X}\in\mathbb{R}^{N_1\times N_2}$, respectively, the image partial derivatives are estimated using the algebraic scheme in \ref{table_OD_scheme_gradien_hessian}.
\begin{table}[htp]
\renewcommand{\arraystretch}{1.65}
\caption{2D algebraic scheme for partial differentiation}
\label{table_OD_scheme_gradien_hessian}
\centering
\scriptsize
\begin{tabular}{cccc} 
\hlinewd{1pt}
\multicolumn{1}{c}{} & \text{Backward} & \text{Forward} & \text{Central} \\ \hlinewd{1pt}
$\frac{\partial^n{X}}{\partial x^n}$ & ${X}{\overleftarrow{D}^T_{n,N_2}}$ & ${X}{\overrightarrow{D}^T_{n,N_2}}$ & ${X}{D}^T_{n,N_2}$ \\ \cline{1-4}
$\frac{\partial^n{X}}{\partial y^n}$ & ${\overleftarrow{D}_{n,N_1}}{X}$ & ${\overrightarrow{D}_{n,N_1}}{X}$ & ${D}_{n,N_1}{X}$ \\ \cline{1-4}
$\frac{\partial^{n_1+n_2}{X}}{{\partial x^{n_1}}{\partial y^{n_2}}}$ & ${\overleftarrow{D}_{n_1,N_1}}{X}{\overleftarrow{D}^T_{n_2,N_2}}$ & ${\overrightarrow{D}_{n_1,N_1}}{X}{\overrightarrow{D}^T_{n_2,N_2}}$ & ${D}_{n_1,N_1}{X}{D}^T_{n_2,N_2}$\\ \hlinewd{1pt} 
\end{tabular}
\end{table}

A practical guideline for such implementation in computer programming is the derivative matrices are sparse and hence the computations can be facilitated by \texttt{sparse} mode operation such as in MATLAB. Moreover, to avoid memory allocation for high pixel image resolutions, one can break down the differential operation to three steps according to the blocks defined in \ref{sec_matrix_design}. To generalize the image differentiation, we  extend this process to $m$-dimensional data, e.g. volumetric 3D images such as ultrasound and MRI, using the tensor decomposition approach \cite{7038247}. According to the Kronecker decomposition of multidimensional signals, the tensor ${X}\in\mathbb{R}^{N_1\times N_2\times\cdots\times N_m}$ can be decomposed in each dimension by means of separate factor bases using vectorized representation
\begin{align}
\text{vec}\left(\frac{\partial^{n_m}\cdots\partial^{n_2}\partial^{n_1}X}{\partial x_m^{n_m}\cdots\partial x_2^{n_2}\partial x_1^{n_1}}\right)= \left({D}_{n_m}\otimes\cdots\otimes {D}_{n_2} \otimes {D}_{n_1}\right)\text{vec}\left({X}\right),
\label{eq_25}
\end{align}
where, $\otimes$ stands for Kronecker product and $\text{vec}(\cdot)$ stands for vectorized representation of the input tensor which unfolds the tensor along the first dimension. It is also known as mode-$1$ vector of the tensor which is obtained by fixing all modes in the tensor except mode-$1$ \cite{7038247}. Here, ${D}_{n_j}\in M_{N_j}$ is the derivative matrix defined for mode-$j$ decomposition basis. The model presented in \ref{eq_25} decomposes all possible dimensions, where every corresponding matrix can be replaced by an identity to relax the decomposition at a particular mode. For instance, to approximate the first order derivative along the second dimension of a 3D volume with staggered-forward coefficients, the operator for decomposition is defined by $I_{N_3}\otimes {\overrightarrow{D}_{1,N_2}} \otimes I_{N_1}$ where $I_{N_1}$ is an identity matrix of size $N_1$. For more information on Kronecker decomposition we refer the readers to \cite{7038247, 6784037} and references therein.

\begin{table}[htp]
\renewcommand{\arraystretch}{1.5}
\caption{List of numerical methods as variants of MaxPol method}
\label{table_maxpol_generalizes}
\centering
\scriptsize
\begin{tabular}{ll}
\hlinewd{1pt}
& ~~~~~~~~~~~~~~~~~~~Selected parameters in MaxPol \\ \hlinewd{.75pt}
\multirow{2}{*}{Fornberg \cite{fornberg1988, S0036144596322507}}
& {node={`centralized'}}, {$l\geq 0$}, {$-l\leq s\leq l$}, {$P=2l$}, {$n=\{0,1,\hdots\}$} \\
& {node={`staggered'}}~~,~{$l\geq 0$}, {$-l\leq s\leq l$}, {$P=2l-1$}, {$n=\{0,1,\hdots\}$} \\ \hlinewd{.75pt}
\multirow{2}{*}{Savitzky-Golay (fullband) \cite{SavitzkyGolay1964, gorry1990general, meer1990smoothed, luo2005properties, schafer2011savitzky}}
& {node=\text{`centralized'}}, {$l\geq 0$}, {$-l\leq s\leq l$}, {$P=2l$}, {$n=\{0,1,\hdots\}$} \\
& {node=\text{`staggered'}}~~, {$l\geq 0$}, {$-l\leq s\leq l$}, {$P=2l-1$}, {$n=\{0,1,\hdots\}$} \\ \hlinewd{.75pt}
\multirow{2}{*}{Khan \cite{RasoolKhan1999, RasoolKhan1999_2, RasoolKhan2000, RasoolKhan2003, RasoolKhan2003_2}} & {node=\text{`centralized'}}, {$l\geq 0$}, {$s=0$}, {$P=2l$}, {$n=\{0,1,\hdots\}$} \\
& {node=\text{`staggered'}}~~, {$l\geq 0$}, {$s=0$}, {$P=2l-1$}, {$n=\{0,1,\hdots\}$} \\ \hlinewd{.75pt}
Li \cite{Li2005} & {node=\text{`centralized'}}, {$l\geq 0$}, {$-l\leq s\leq l$}, {$P=2l$}, {$n=\{0,1,\hdots\}$}, $D_n$ \\ \hlinewd{.75pt}
O'Leary \cite{o2008algebraic, o2012framework} & {node=\text{`centralized'}}, {$l\geq 0$}, {$-l\leq s\leq l$}, {$P=2l$}, {$n=1$}, $D_n$ \\ \hlinewd{.75pt}
Hassan \cite{HassanMohamadAtteia2012} & {node=\text{`centralized'}}, {$l\geq 0$}, {$-l\leq s\leq l$}, {$P=2l$}, {$n=\{0,1,\hdots\}$}, $D_n$ \\ \hlinewd{.75pt}
Carlsson \cite{Carlsson1991} & {node=\text{`centralized'}}, {$l\geq 0$}, {$s=0$}, {$P=2l$}, {$n=1$} \\ \hlinewd{.75pt}
\multirow{2}{*}{Kumar \cite{KumarRoy1988, KumarRoy1989, KumarRoy1992}}& {node=\text{`centralized'}}, {$l\geq 0$}, {$s=0$}, {$P=2l$}, {$n=1$} \\
& {node=\text{`staggered'}}~~, {$s=0$}, {$n=1$}, {$l\geq 0$}, {$P=2l-1$} \\ \hlinewd{.75pt}
\multirow{2}{*}{Selesnick \cite{SelesnickBurrus1998, SelesnickBurrus1998TSP, Selesnick2002}} & {node=\text{`centralized'}}, {$l\geq 0$}, {$s=0$}, {$n\leq P\leq 2l$}, {$n=\{0,1\}$} \\
& {node=\text{`staggered'}}~~, {$s=0$}, {$0\leq n\leq 1$}, {$l\geq 0$}, {$n\leq P\leq 2l-1$} \\ \hlinewd{.75pt}
\multirow{2}{*}{Hosseini-Plataniotis \cite{hosseini2017derivative}} & node=\text{`centralized'}, $l\geq 0$, $s=0$, $n\leq P\leq 2l$, {$n=\{0,1,\hdots\}$} \\
& node=\text{`staggered'}~~,~$l\geq 0$, $s=0$, $n\leq P\leq 2l-1$, {$n=\{0,1,\hdots\}$} \\ \hlinewd{.75pt}
\end{tabular}
\end{table}

\subsection{Generalization of MaxPol library package}
In this section we provide a summary on how \texttt{MaxPol} generalizes many methods in the literature. With retrospect to the finite difference methods listed in \ref{table_state_of_the_art_ND}, the cross parameter selected for \texttt{MaxPol} are demonstrated for equivalent representations in \ref{table_maxpol_generalizes}. For more information on different methods and how they compare to each other in terms of feature utilities please refer to \ref{table_state_of_the_art_ND}.

A comprehensive library code written in MATLAB called \texttt{MaxPol} is delivered along the publication of this paper which includes derivation of lowpass/fullband derivative coefficients, derivative matrices of all four kinds, and tensor decomposition utilities. For more information on this library please refer to the user's guidelines of \texttt{MaxPol} and provided demo examples.

\section{Numerical validation}\label{sec_fullband_differentiation}
Throughout this section we numerically validate the newly designed derivative matrices in terms of approximation accuracy. The application of these matrices and their usefulness will be fully discussed later in \ref{sec_gradient_surface}. An analytical 2D sinusoidal zone plate $f(x,y)=\sin\left((\omega_x x)^2+(\omega_y y)^2\right)$ with a single focus point is considered similar to the one in \cite{freeman1991design}. The function consists of rotation invariant sinusoidal grating patterns of multiple frequencies within $(x,y)\in [-1,1]\times [-1,1]$. The harmonic parameter for both axes are chosen  $\omega_x=\omega_y=1.6\pi$ and the function is uniformly discretized on a $128\times 128$ meshgrid. \ref{fig_2D_sinusoidal_grating} displays the discrete representation of this function and its partial derivatives, where by moving towards the image boundaries higher frequencies are observed. This is the main rationale of choosing this plate for our experiment to determine the accuracy of differentiation on both interior and exterior boundaries with respect to different harmonies.

\begin{figure}[htp]
\centerline{
\subfigure[$f(x,y)$]{\includegraphics[height=0.15\textwidth]{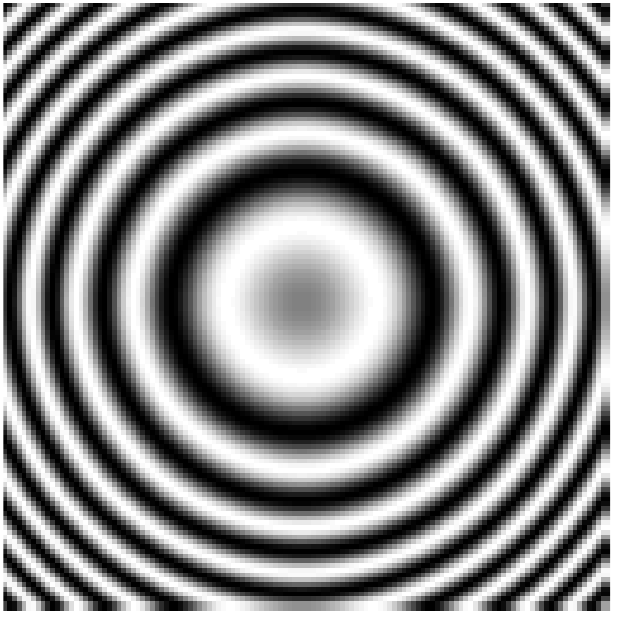}\label{fig_2D_sinusoidal_grating}}
\subfigure[$\frac{\partial f}{\partial x}$]{\includegraphics[height=0.15\textwidth]{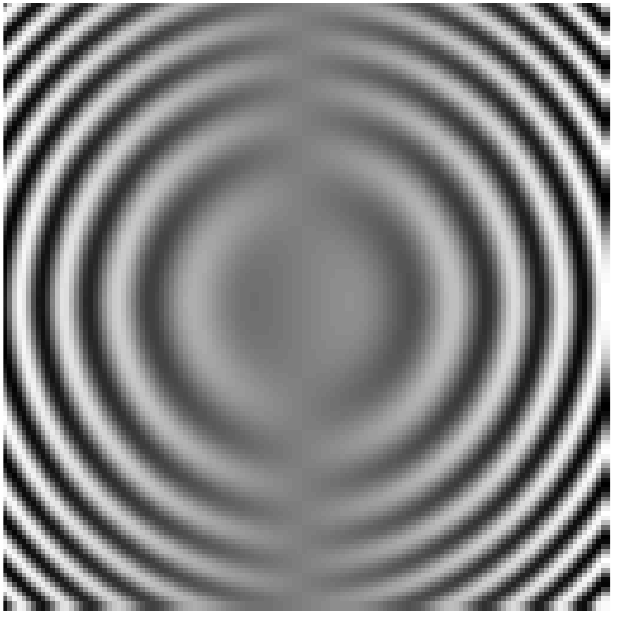}\label{fig_2D_sinusoidal_grating_x}}
\subfigure[$\frac{\partial^2 f}{\partial x \partial y}$]{\includegraphics[height=0.15\textwidth]{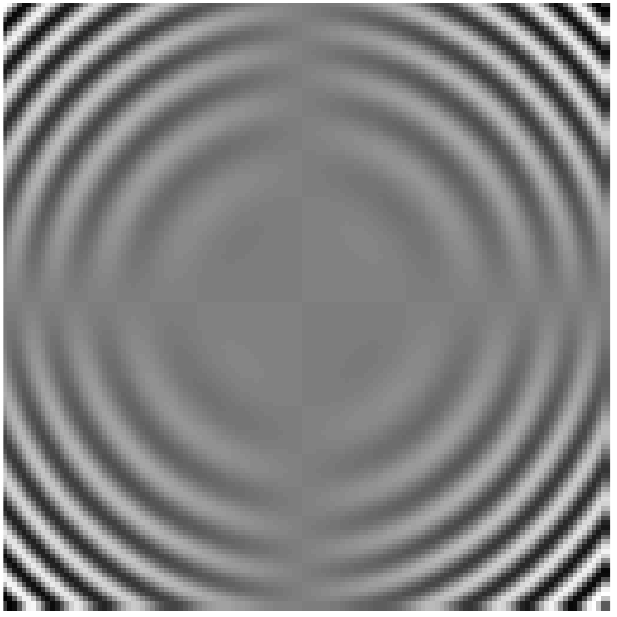}\label{fig_2D_sinusoidal_grating_xy}}
\subfigure[$\frac{\partial^2 f}{\partial x^2}$]{\includegraphics[height=0.15\textwidth]{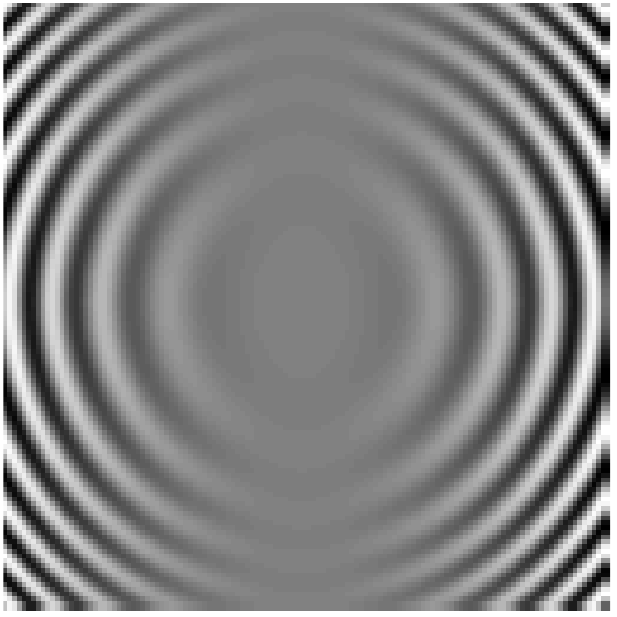}\label{fig_2D_sinusoidal_grating_xx}}
}
\caption{Analytical representation of 2D-sinusoidal grating and its first and second partial deriavtives}
\label{fig_analytical_2D_sinusoidal_grating}
\end{figure}

The empirical analysis is performed here to approximate the partial derivatives of the discrete image $f$ by means of staggered-forward third order derivative matrix $\partial^3 f/\partial x^3\approx {f\overrightarrow{D}^T_{3,128}}$ and centralized fourth order derivative matrix $\partial^4 f/\partial x^4 \approx {f}{D}^T_{4,128}$. The tap-length polynomial to construct the fullband differential matrices is $l=11$. The associated coefficients are numerically calculated by three different methods of Fornberg \cite{S0036144596322507}, Savitzky-Golay \cite{luo2005properties, schafer2011savitzky}, and our MaxPol formulations proposed in \ref{lemma_staggered} and \ref{lemma_centralized}. Also, different harmonic frequencies are considered to generate the discrete zone plate images varying from small to higher frequencies. The means of evaluation is performed by normalized-means-square-error (NMSE) between approximated and analytical derivative images. The evolution of the error approximation are demonstrated in \ref{fig:contour_error}. The domain of recovery is separate to interior and boundary domains for approximation according to zero and non-zero side shift nodes, respectively. The rank observation of MaxPol differentiation is consistent through different harmonic levels. Due to the lack of closed form solution to calculate the inverse Vandermonde in Savitzky-Golay, the numerical accuracy of FIR coefficients are plagued for high order polynomial, which we have selected $l=11$ in this experiment. Accordingly, the approximation error is deviated for higher order of harmonics due to the filter deviation of numerical differentiator from its ideal response, explained in previous section. With respect to the Fornberg's performance, the error levels maintain their robustness towards approximating high frequency information similar to the MaxPol. However, for high order of differentiation, the method loses its robustness particularly on approximating low frequency information using high order polynomials. This is due to the error propagation caused by rapid increase of the numerical iterations in generating high order derivative coefficients within the Lagrangian algorithm introduced in \cite{S0036144596322507}. Whereas, the closed form solution via MaxPol is scalable towards high order derivatives with more robust approximation.

\begin{figure}[htp]
\centerline{
\subfigure[interior ($n=3$)]{\includegraphics[width=0.19\textwidth]{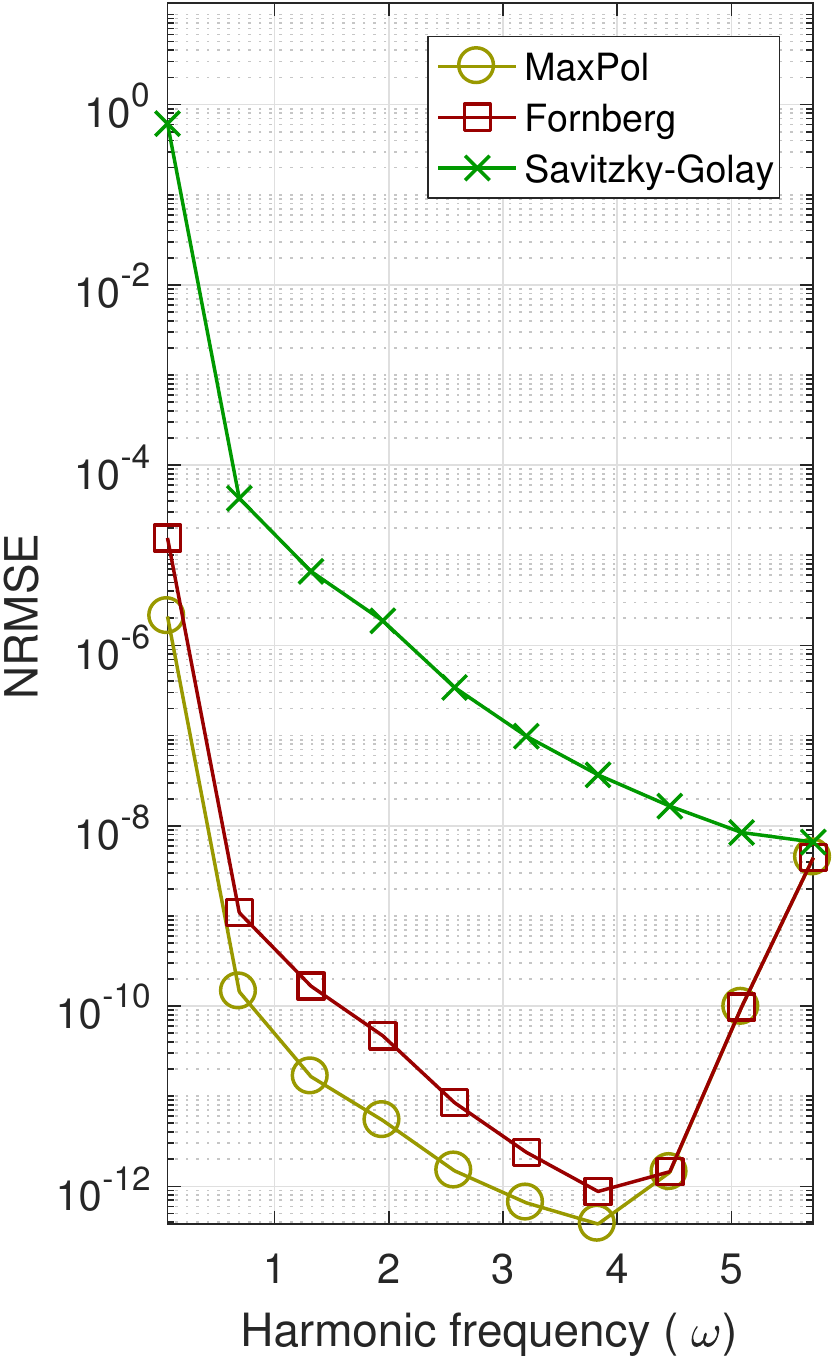}\label{fig:Numerical_error_2D_Zone_Plate_l_11_n_ord_3_staggered_interior}}
\subfigure[interior ($n=4$)]{\includegraphics[width=0.19\textwidth]{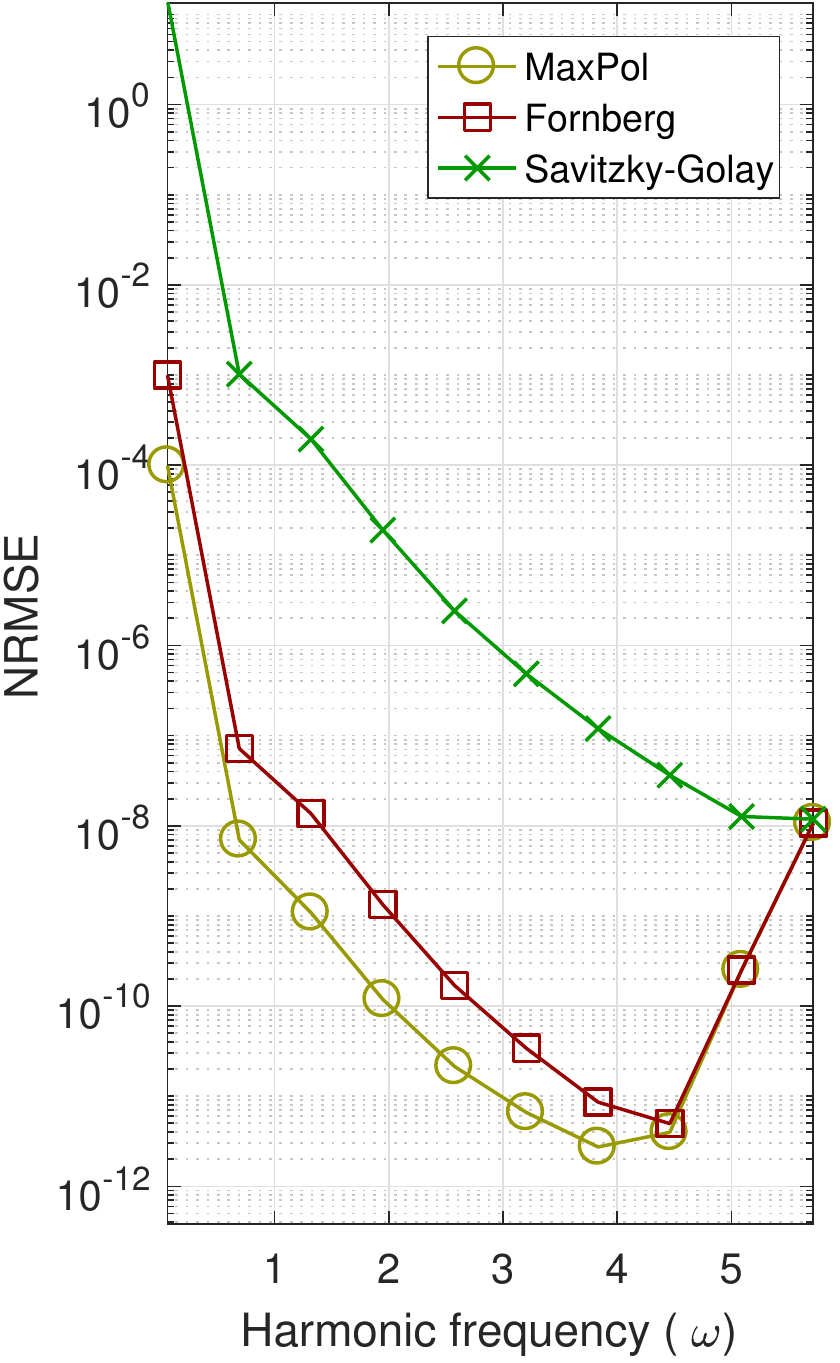}\label{fig:Numerical_error_2D_Zone_Plate_l_11_n_ord_4_centralized_interior}}
\subfigure[boundary ($n=3$)]{\includegraphics[width=0.19\textwidth]{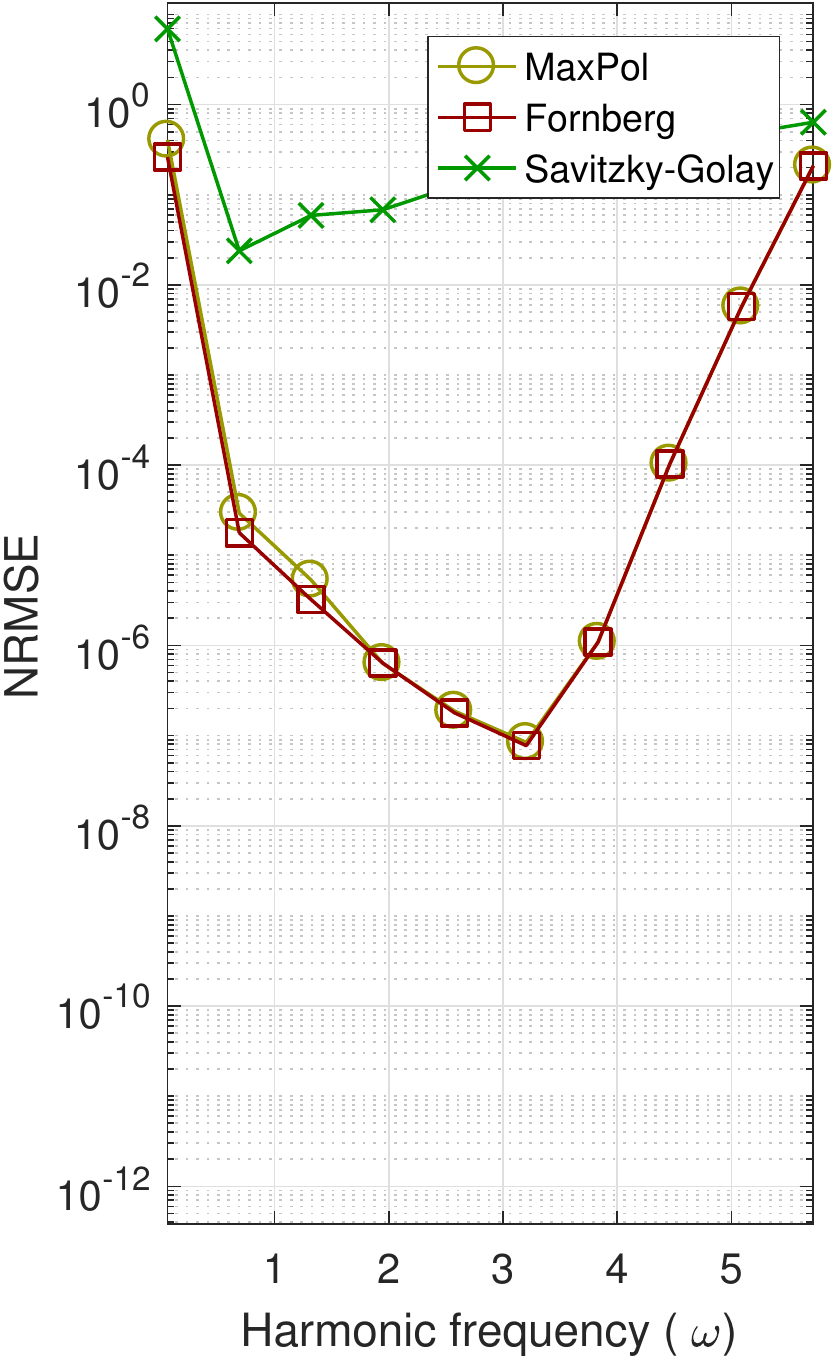}\label{fig:Numerical_error_2D_Zone_Plate_l_11_n_ord_3_staggered_boundary}}
\subfigure[boundary ($n=4$)]{\includegraphics[width=0.19\textwidth]{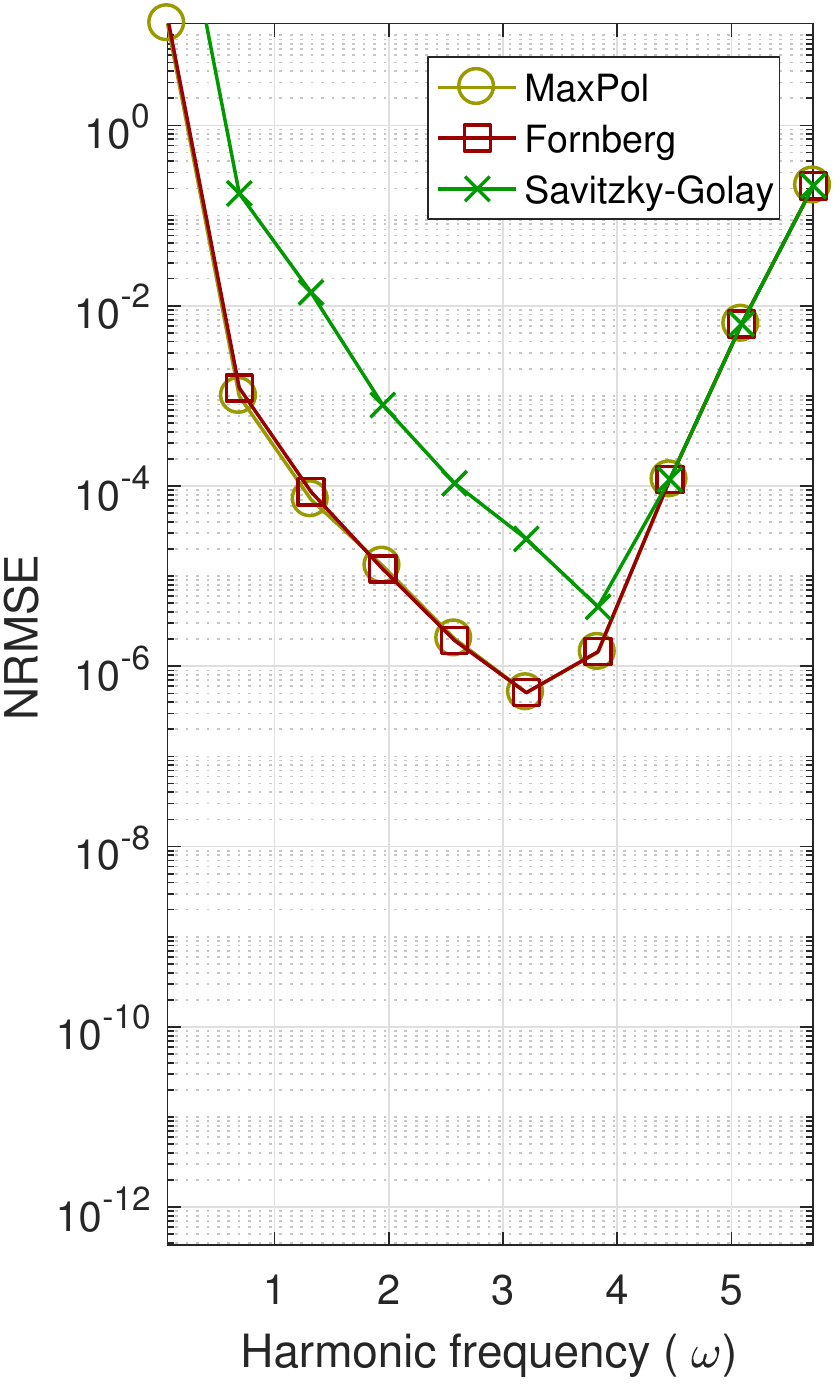}\label{fig:Numerical_error_2D_Zone_Plate_l_11_n_ord_4_centralized_boundary}}
\subfigure[Elapsed time]{\includegraphics[width=0.193\textwidth]{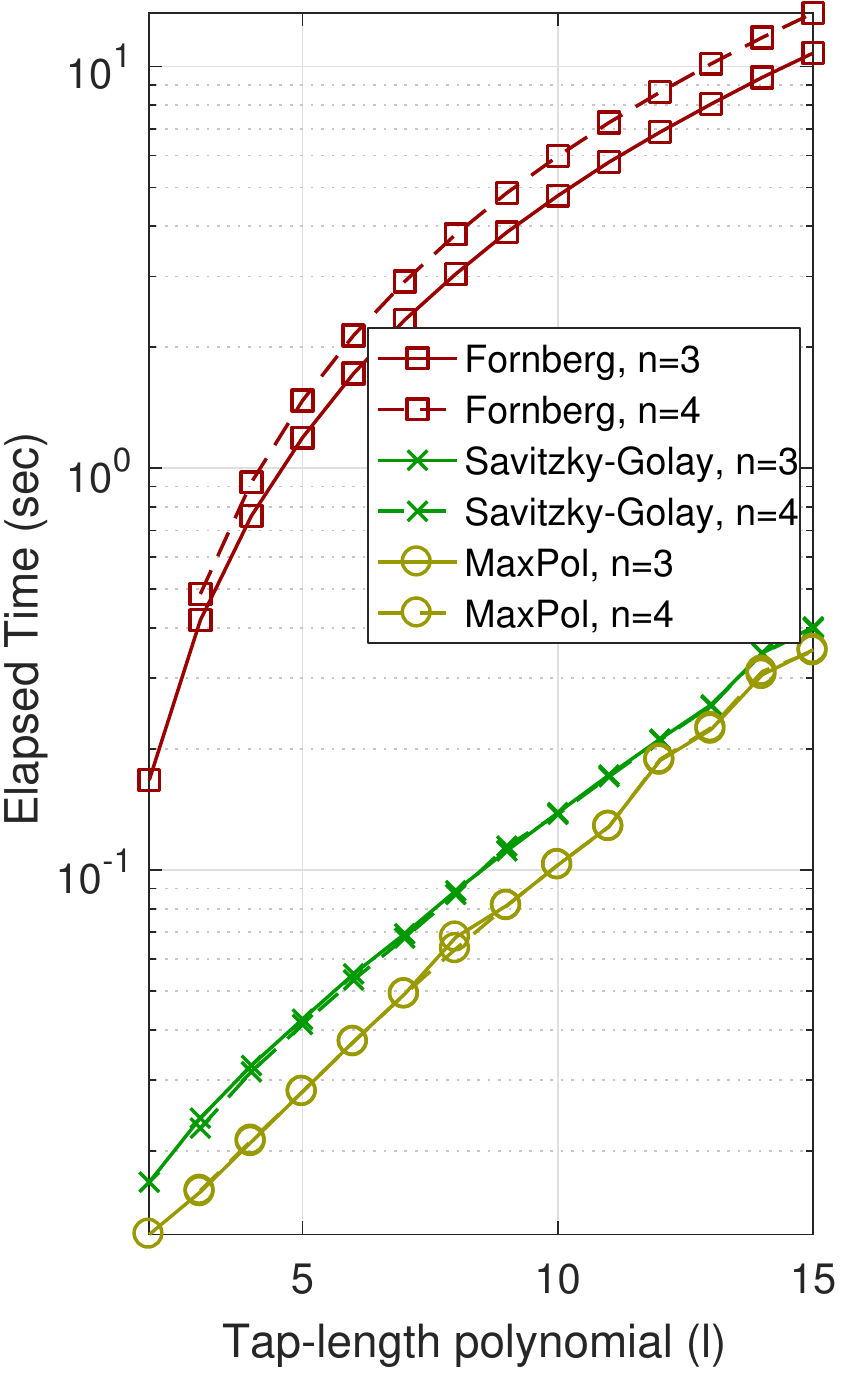}\label{fig:run_time_comparision_symbolic}}
}
\caption{(a)-(d) Approximation error of third and fourth order partial derivatives of discrete 2D zone plate on interior and boundary domains. The associated derivative matrices are constructed by means of numerical calculation of staggered and centralized schemes for third and fourth order FIR coefficients. (e) run time for symbolic calculation of coefficients by three different methods.}
\label{fig:contour_error}
\end{figure}

One alternative solution to avoid such inconsistencies is to deploy symbolic calculation of FIR coefficients over numerical computations. Though, the calculated coefficients will be the same under all three aforementioned methods, the computational complexities as per individual method will be different. \ref{fig:run_time_comparision_symbolic} elaborates on the time elapsed for symbolic calculation of the fullband FIR coefficients. As it is shown, the rank observation of the MaxPol is still consistent towards different tap-length polynomials and derivative orders.

\section{Applications in gradient surface reconstruction}\label{sec_gradient_surface}
The problem of recovering image surfaces from gradient measurements is of broad interest in computer vision tasks such as photometric stereo \cite{harker2015regularized}, surface rendering \cite{EstellersSoatto2016, mecca2016single}, and recovery in incomplete Fourier measurements \cite{5887417}. In general one needs to infer an underlying solution from the given gradient measurements. In a nutshell, the PDE problem is determined by $\nabla\phi=\overline{\partial\phi}$ where $\overline{\partial\phi}$ is the gradient measurement and $\phi\in\mathbb{R}^{N_1\times N_2}$ is the image surface to be recovered. One easy approach to solve this problem is to render the global solution by means of a quadratic minimization
\begin{align}
\min\limits_{\phi}{\|\frac{\partial\phi}{\partial x} - \overline{\partial_x\phi}\|^2_2 + \|\frac{\partial\phi}{\partial y} - \overline{\partial_y\phi}\|^2_2}
\label{Poisson_eq1}
\end{align}
which is known as the Poisson equation \cite{horn1986variational, simchony1990direct}. By discretizing the operators ${\partial\phi}/{\partial x}:=\phi D^T_{1,N_2}$ and ${\partial\phi}/{\partial y}:=D_{1,N_1}\phi$ using the matrix notations introduced in \ref{table_OD_scheme_gradien_hessian}, the global minimum of Poisson equation \ref{Poisson_eq1} in $\ell_2$-norm space fits a matrix Lyapunov (Sylvester) equation \cite{o2008algebraic, harker2015regularized}
\begin{align}
\left(D^T_{1,N_2}D_{1,N_2}\otimes I_{N_1} + I_{N_2}\otimes D^T_{1,N_1}D_{1,N_1}\right)\phi = 
\overline{\partial_x\phi}D_{1,N_2} + D^T_{1,N_1}\overline{\partial_y\phi}.
\label{Poisson_eq2}
\end{align}
The solution to the Sylvester equation can be accommodated by the well known Schur decomposition algorithms \cite{bartels1972solution, stewart2001matrix}. Since the differentiation matrix $D_1$ is rank-one deficient, the solution to \ref{Poisson_eq2} is not unique with arbitrary constant level. Therefore, a prior information is needed from the underlying image to tune the recovery constant level. The differential operator here is embedded as a correspondence between the measurements and regularizing variables $\nabla:\phi\mapsto\partial\phi$, where a fullband differentiation is required to establish a complete transformation of the frequency range. 

\subsection{Impact of derivative matrix accuracy: noise-free case}\label{sec_Sylvester_noise_free}
Our contribution to the Sylvester equation \ref{Poisson_eq2} is to deploy high accuracy derivative matrices with staggered schemes (case--I) to increase the stability of the recovery compared to its counterpart solutions in the literature. We shall elaborate on this utilization with more experimental detail as follows.
\begin{figure}[htp]
\centerline{
\subfigure[Interior recovery]{\includegraphics[height=0.45\textwidth]{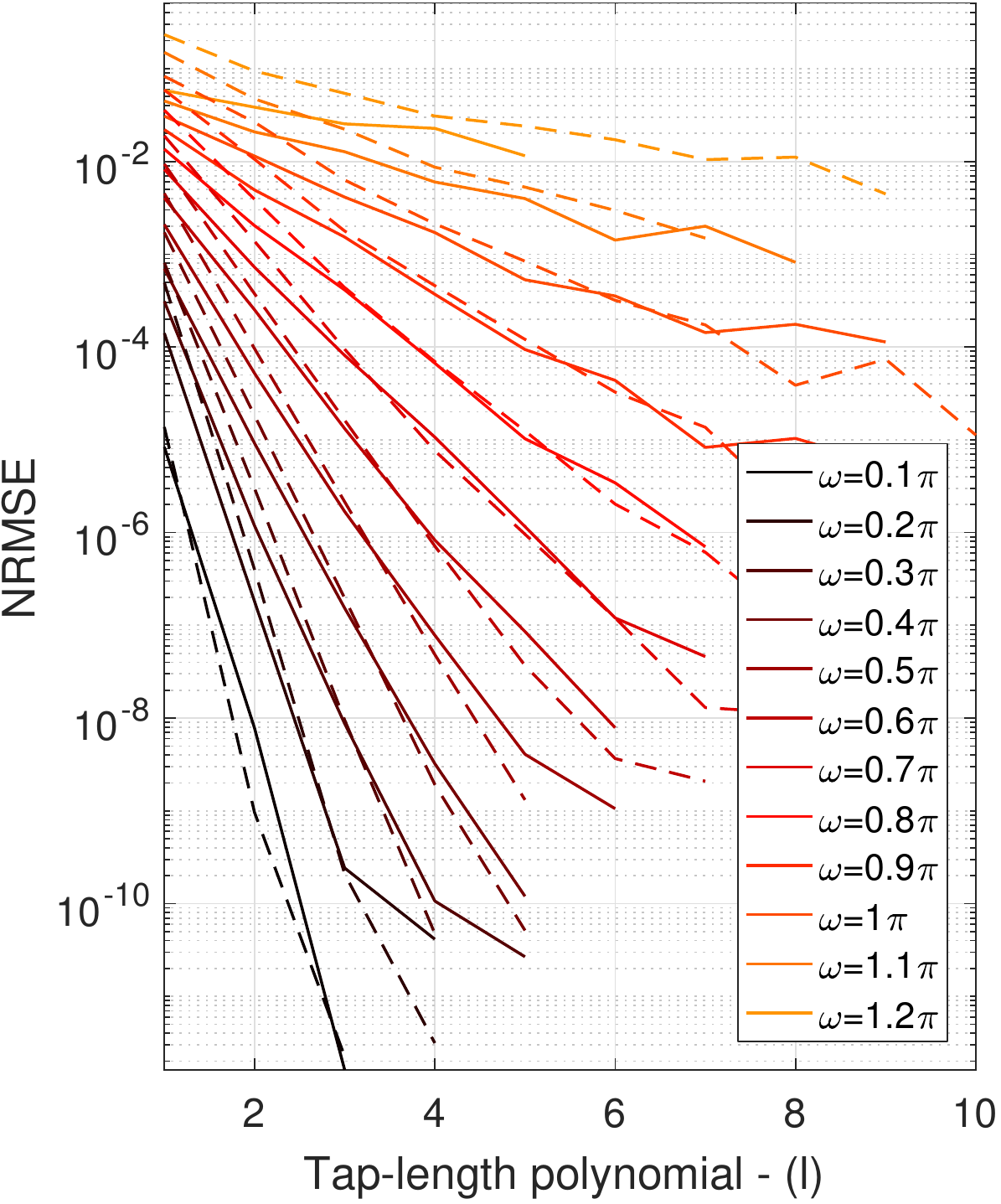}\label{fig_error_evolution_gradient_recovery_noise_free_interior}}
\subfigure[Boundary recovery]{\includegraphics[height=0.45\textwidth]{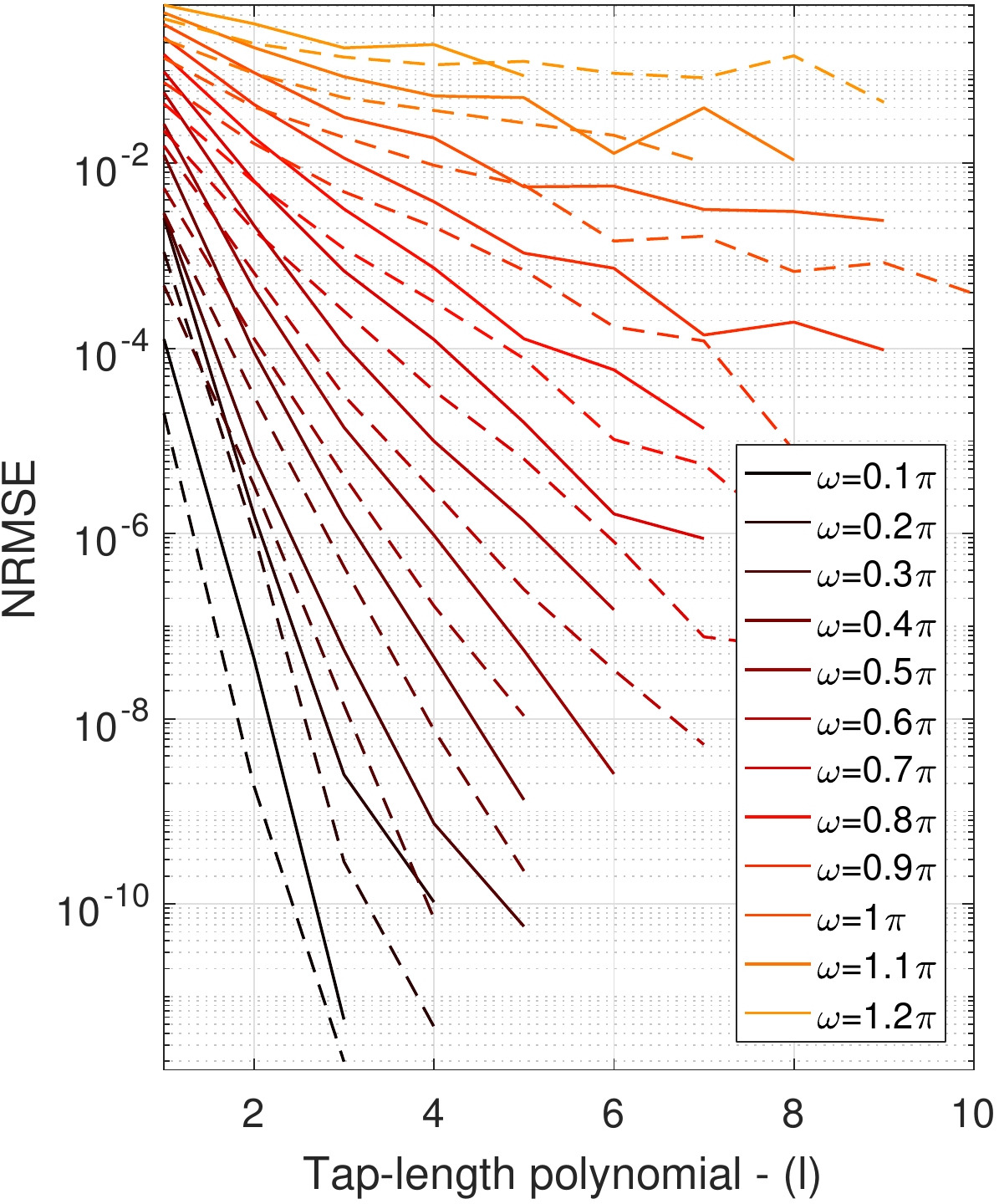}\label{fig_error_evolution_gradient_recovery_noise_free_boundary}}
}
\caption{Reconstruction error obtained from solving Sylvester equation associated by case--I (solid lines) and case--II (dashed lines) derivative matrices. Different tap-length polynomials are used in corporation of derivative matrices.}
\label{fig_error_evolution_gradient_recovery_noise_free}
\end{figure}
The applied image here is considered to be the same zone plate 2D sinusoid studied in \ref{sec_fullband_differentiation}. The gradient measurements $\overline{\partial_x\phi}$ and $\overline{\partial_y\phi}$ are obtained by sampling the gradient function with different harmonic parameters $\omega_x=\omega_y\in\{0.2\pi,0.4\pi,\ldots,2\pi\}$. The sampled gradients are then fed to Sylvester equation \ref{Poisson_eq2} to recover the underlying surface image $\phi$. Both case--I (staggered) and case--II (centralized) of fullband first order derivative matrices, defined in \ref{sec_stability_eigenvlaues}, are considered to be the candid operators of the Sylvester equation. The tap-length polynomial of the matrices are selected from a lowest to a highest accuracy $l\in\{1,\ldots,5\}$. It worth noting that the analogous design of $l=1$ of case--I has been extensively studied in \cite{agrawal2006range, durou2007integration}. Moreover, similar design of case--II of polynomial order $l\geq 1$ is also studied in \cite{o2008algebraic, harker2015regularized}. The terms of evaluation is performed here by measuring the normalized root mean square error $\text{NRMSE}=\|\phi - \phi_A\|_2/\|\phi_A\|_2$ between the reconstructed surface $\phi$ and its analytical representation (ground truth) $\phi_A$. The results of reconstructions are demonstrated in \ref{fig_error_evolution_gradient_recovery_noise_free}. To better interpret the recovery results, we have separated the recovery error of interior and boundary approximation. Note that we have omitted the grid points for deviated performances i.e. the NRMSE after certain polynomial size start to raise for both cases, therefore we have ignored them to avoid visual miss interpretation. The overall performance indicates that reconstruction error monotonically reduces by increasing the tap-length polynomial $l$. In particular, Case-I mainly out performs Case-II inside interior domain for higher frequency harmonics using wide range of tap length polynomials. However, this range is limited for low frequency zone plate recovery where Case-II outperforms Case-I using higher tap lengths. The root cause of this issue backs to the formation of the staggered derivative matrices as opposed to the centralized case. The last row of the derivative matrix in staggered form, in fact, extrapolates the derivative value outside the range of tap nodes about half-grid shifted to right. See \ref{fig_right_bc_model} on the right BC construction for clarification. This particular row in derivative matrix contains obliterated coefficients and, as a response, it perturbs the solution of the Sylvester Equation in \ref{Poisson_eq2}. This is further noticeable in recovering the boundary domain shown in \ref{fig_error_evolution_gradient_recovery_noise_free_boundary}. Nevertheless, the energy of such perturbation is much lower than the effect of non-conservative gradient sampling due to noise contamination. This is confirmed in the next experiment conducted in \ref{sec_noisy_gradient_surface_recovery}.
\begin{figure}[htp]
\centerline{
\subfigure[Interior ($l=1$)]{\includegraphics[height=0.425\textwidth]{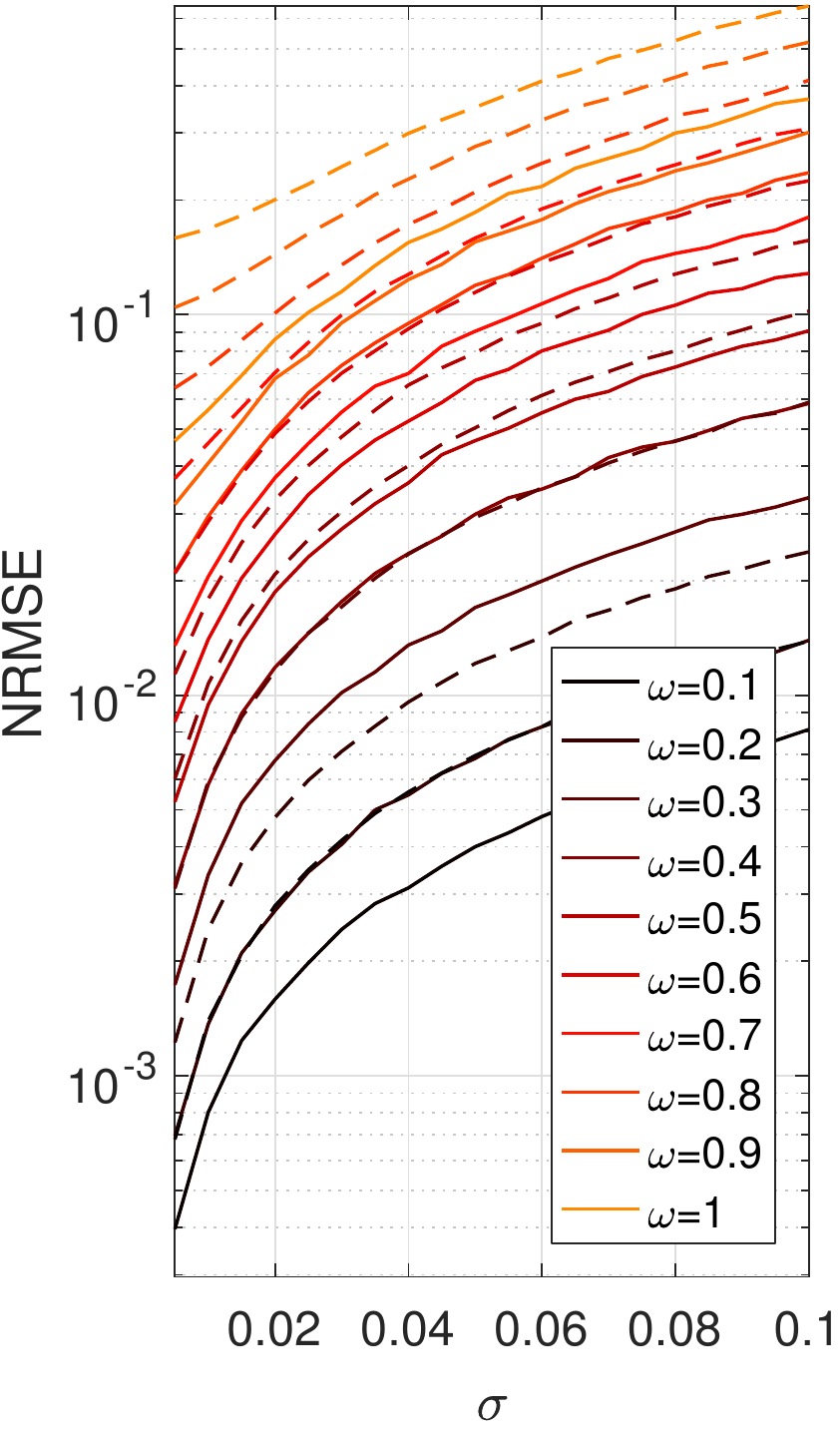}\label{fig_error_evolution_gradient_recovery_noise_sensitivity_interior_l1}}\hspace{-.05in}
\subfigure[Interior ($l=5$)]{\includegraphics[height=0.425\textwidth]{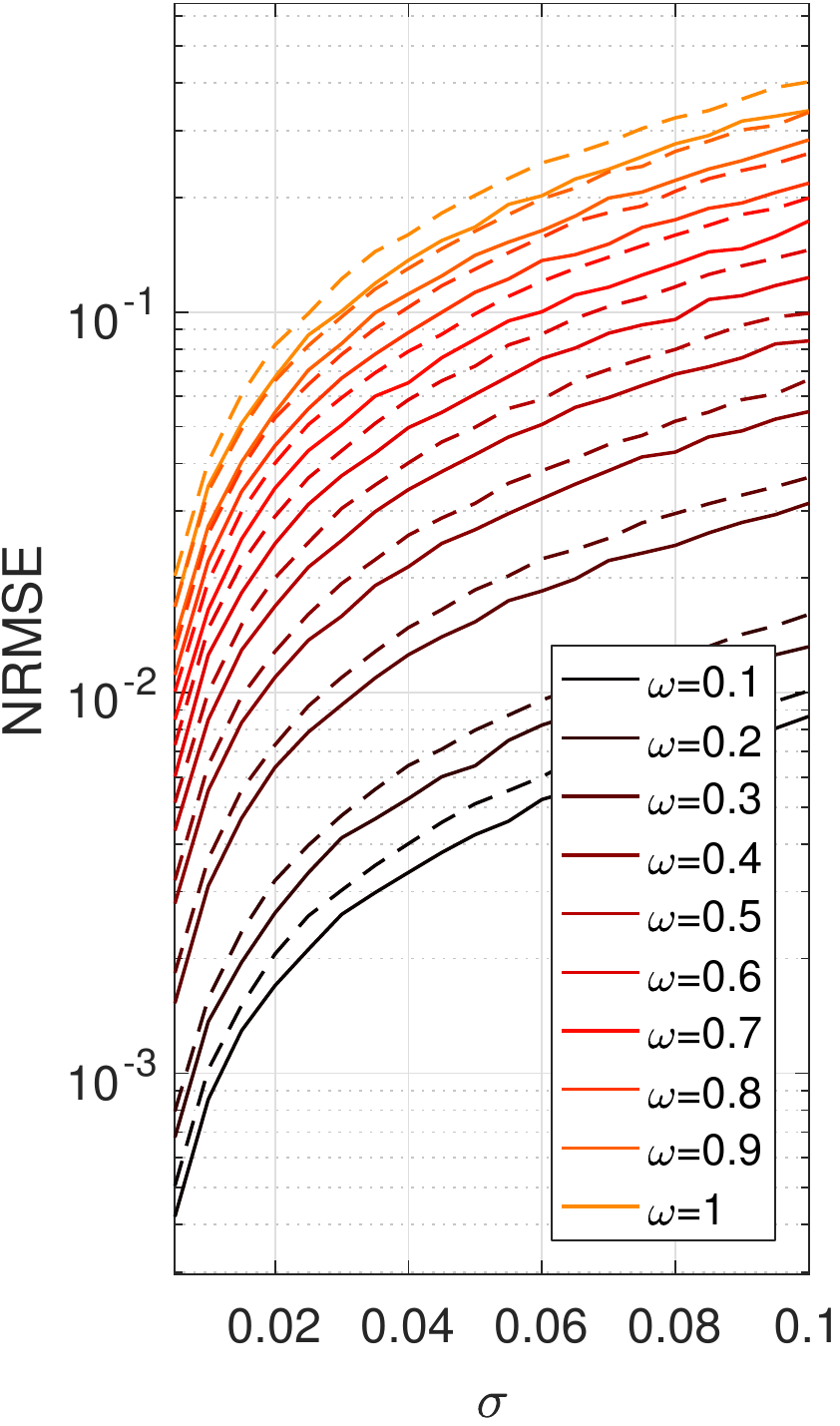}\label{fig_error_evolution_gradient_recovery_noise_sensitivity_interior_l5}}\hspace{-.05in}
\subfigure[Boundary ($l=1$)]{\includegraphics[height=0.425\textwidth]{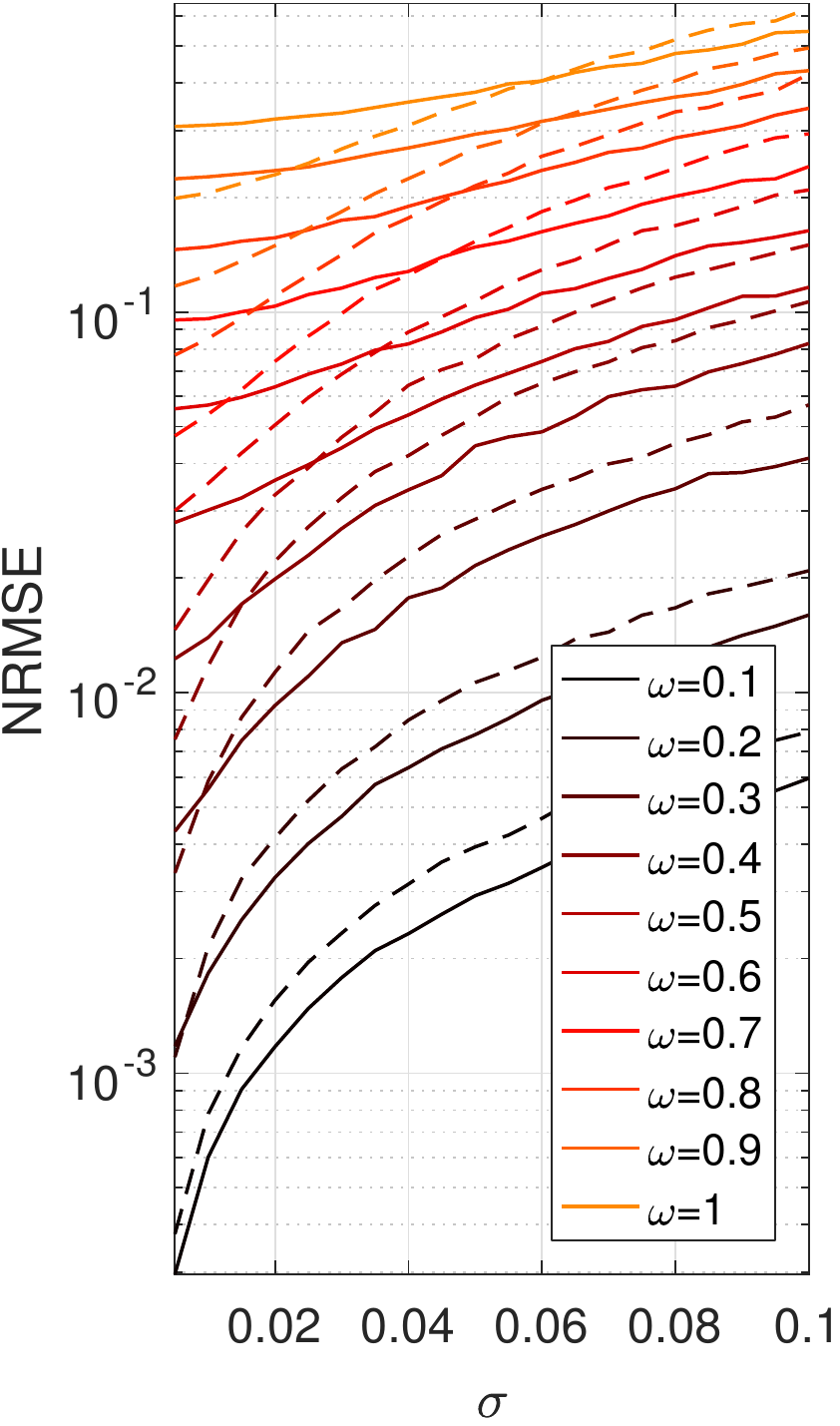}\label{fig_error_evolution_gradient_recovery_noise_sensitivity_boundary_l1}}\hspace{-.05in}
\subfigure[Boundary ($l=5$)]{\includegraphics[height=0.425\textwidth]{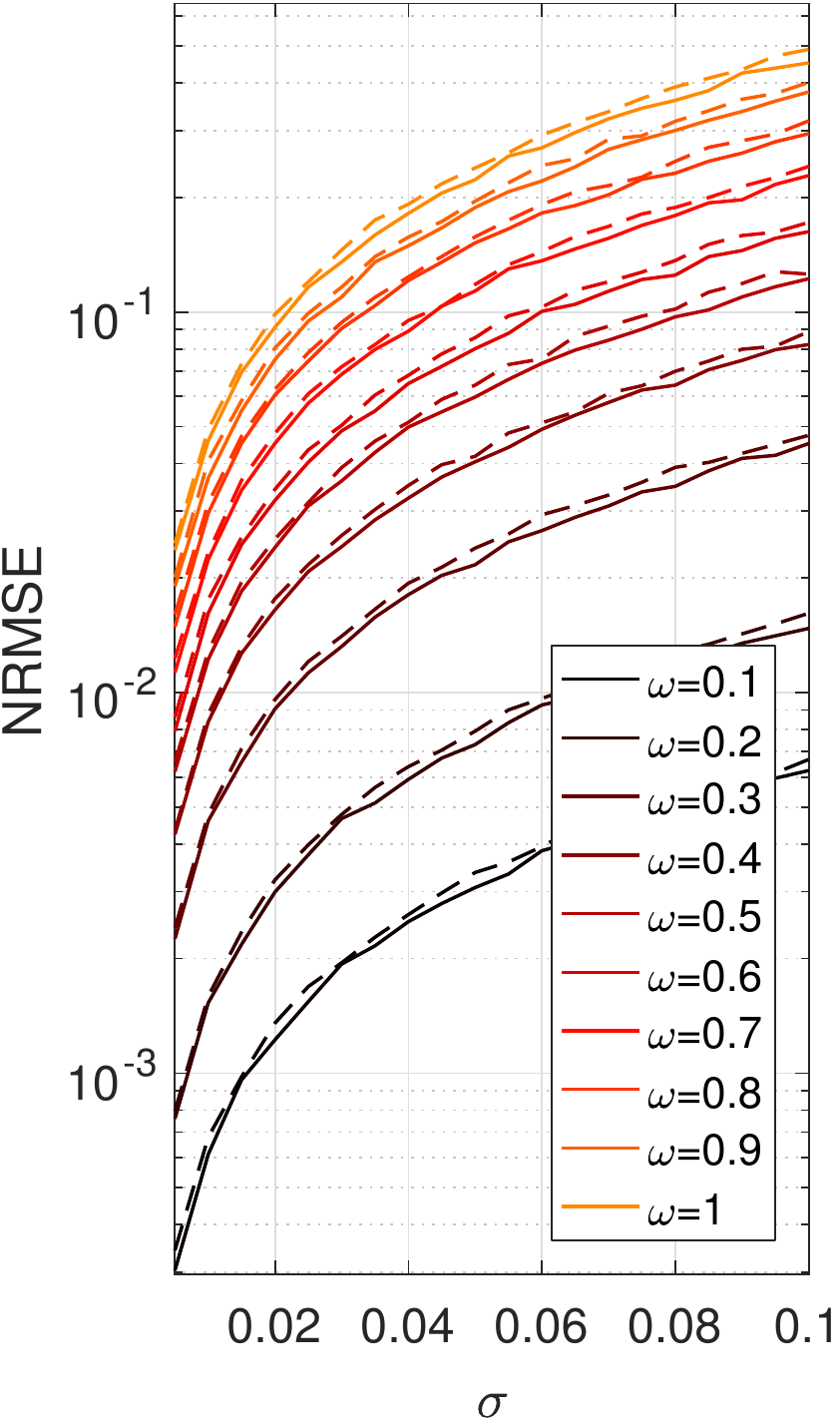}\label{fig_error_evolution_gradient_recovery_noise_sensitivity_boundary_l5}}
}
\caption{Reconstruction error obtained from solving noisy Sylvester equation associated by case--I (solid lines) and case--II (dashed lines) derivative matrices.}
\label{fig_error_evolution_gradient_recovery_noise_sensitivity}
\end{figure}
\subsection{Impact of noisy gradient-measurement}\label{sec_noisy_gradient_surface_recovery}
The results obtained in \ref{sec_Sylvester_noise_free} simulates the Sylvester reconstruction under the noise-free case. However, in real application samples are provided under noisy measurements. In this section we study the impact of noise contamination in recovering the same analytical 2D image used in previous section with the same selected parameters. We consider independent and identically distributed (i.i.d) Gaussian noise for the gradient field contamination since the maximum likelihood reconstruction fits the solution to the Sylvester equation \ref{Poisson_eq2}. The real noise application in many natural imaging problems is also modeled by i.i.d Gaussian random variables \cite{harker2015regularized}. The assumption of noisy gradients implies a non-conservative field meaning the curl of gradient is not zero. This creates residual perturbation in the recovery stage according to the noise energy. The Monte-Carlo simulation is done here for $100$ average simulations on generating random perturbations with different standard deviation magnitudes $\sigma\in\{0.01,0.02,\ldots,0.1\}$. These perturbations are added to the ground-truth gradients obtained from the analytical function $\overline{\partial_y\phi}_{\eta} \leftarrow \overline{\partial_y\phi}+\eta$ where $\eta\in\mathcal{N}(0,\sigma)$. \ref{fig_error_evolution_gradient_recovery_noise_sensitivity} presents the monte-carlo simulation on 2D zone plate recovery in terms of different noise magnitudes. Similar to previous simulations the error recovery is separated to interior and boundary recovery to identify the strength of the associated derivative schemes on different canvas area. The interior is the main canvas after the reconstruction which we expect the recovery should be minimized for presentation. The boundary is also important to keep the full-size image recovery if it is needed. Though, one may exclude these area after reconstruction if not required. The rank observation of staggered scheme (solid-lines) within the interior domains is highly evident across varying all three parameters of noise magnitudes, sinusoid frequencies, and tap-length polynomials.

\begin{table}[htp]
\renewcommand{\arraystretch}{1.3}
\caption{Recovery of zone plate from noisy gradients ($\sigma=0.01$) using different derivative matrices}
\label{fig_zone_plate_recovery_sigma_2}
\centering
\scriptsize
\begin{tabular}{|l|c|c|c|c|}
\cline{2-5}
\multicolumn{1}{c}{} & \multicolumn{2}{|c}{case--II (centralized scheme)} & \multicolumn{2}{|c|}{case--I (staggered scheme)} \\ \cline{2-5}
\multicolumn{1}{c|}{} & $l=1$ & $l=5$ & $l=1$ & $l=5$ \\ \cline{1-5}
{\hspace{-.05in}\begin{sideways} \hspace{.45in} $\tilde{f}$ \end{sideways}\hspace{-.05in}}   & 
{\includegraphics[height=0.18\textwidth]{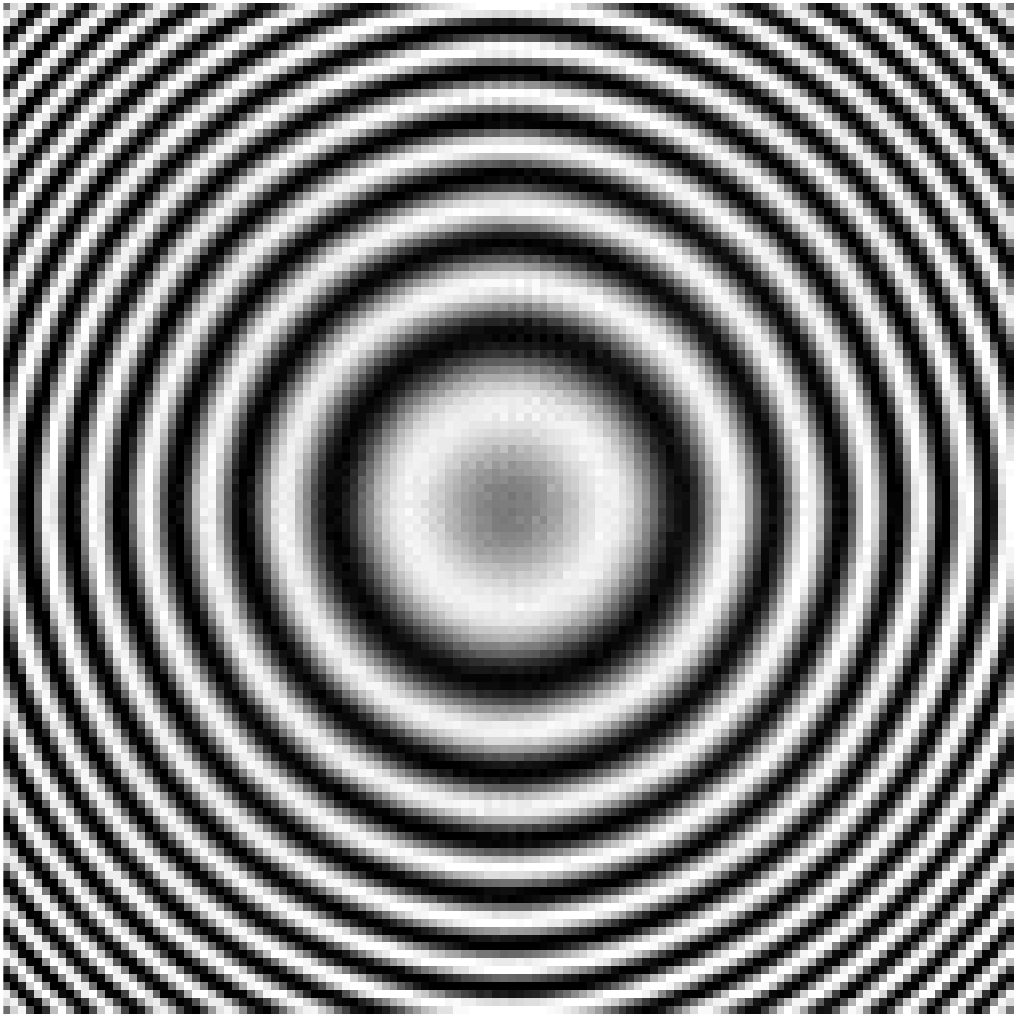}} &
{\includegraphics[height=0.18\textwidth]{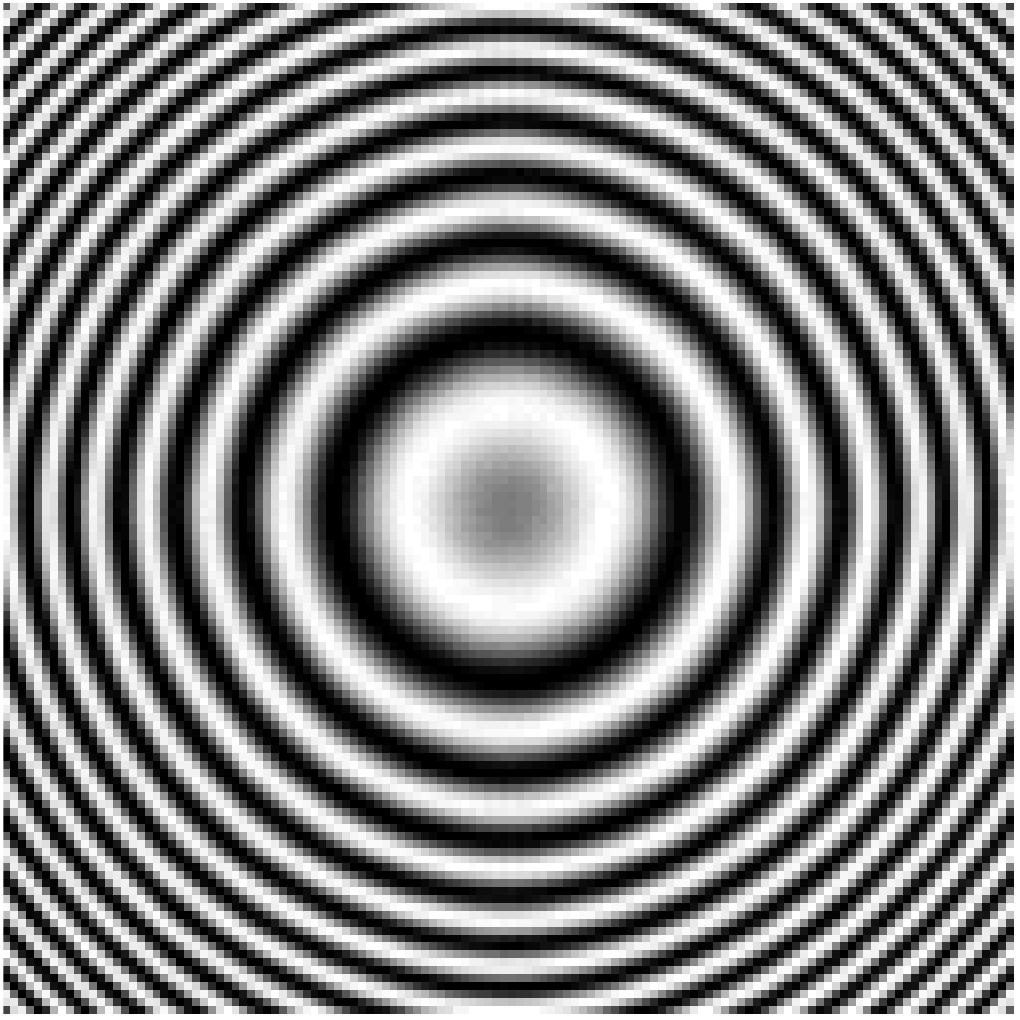}} &
{\includegraphics[height=0.18\textwidth]{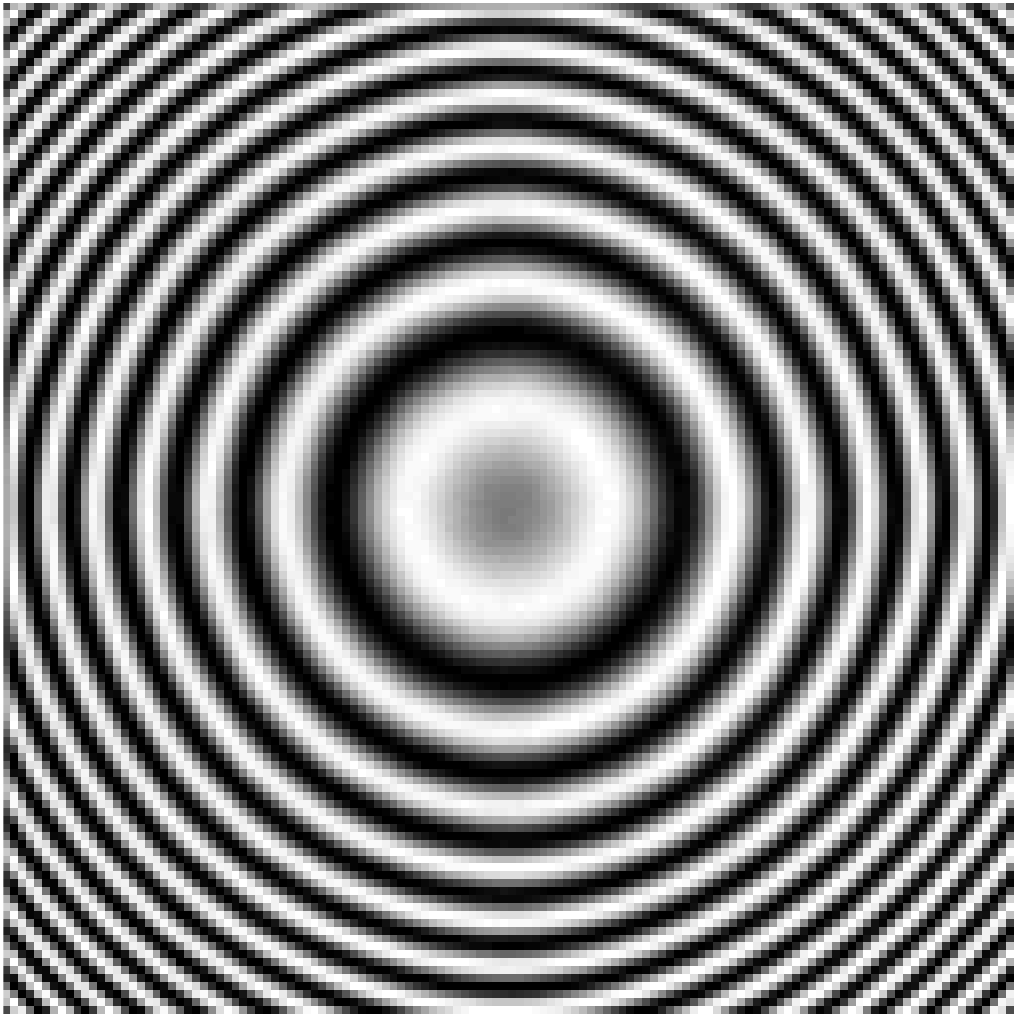}} &
{\includegraphics[height=0.18\textwidth]{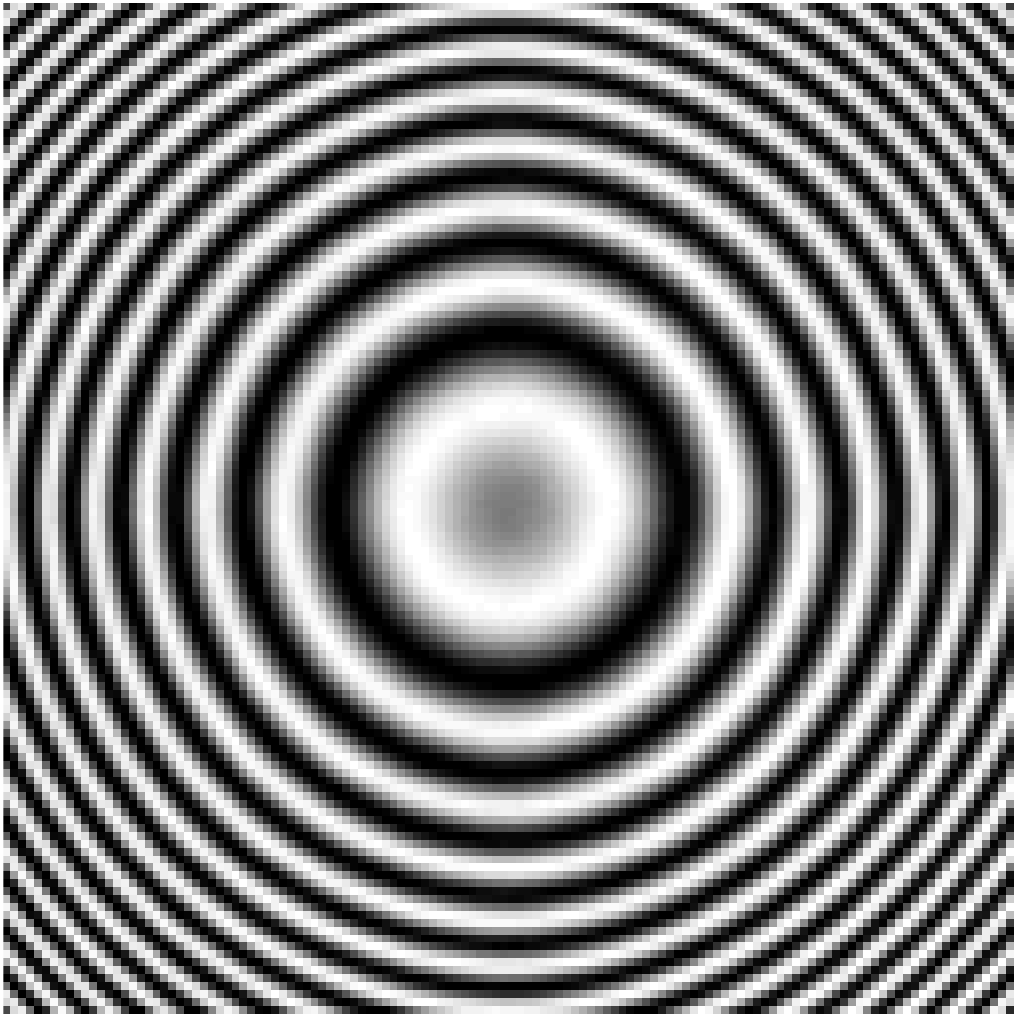}} \\ \cline{1-5}
{\hspace{-.05in}\begin{sideways} \hspace{.35in} $|\tilde{f}-f|$ \end{sideways}\hspace{-.05in}}   & 
{\includegraphics[height=0.18\textwidth]{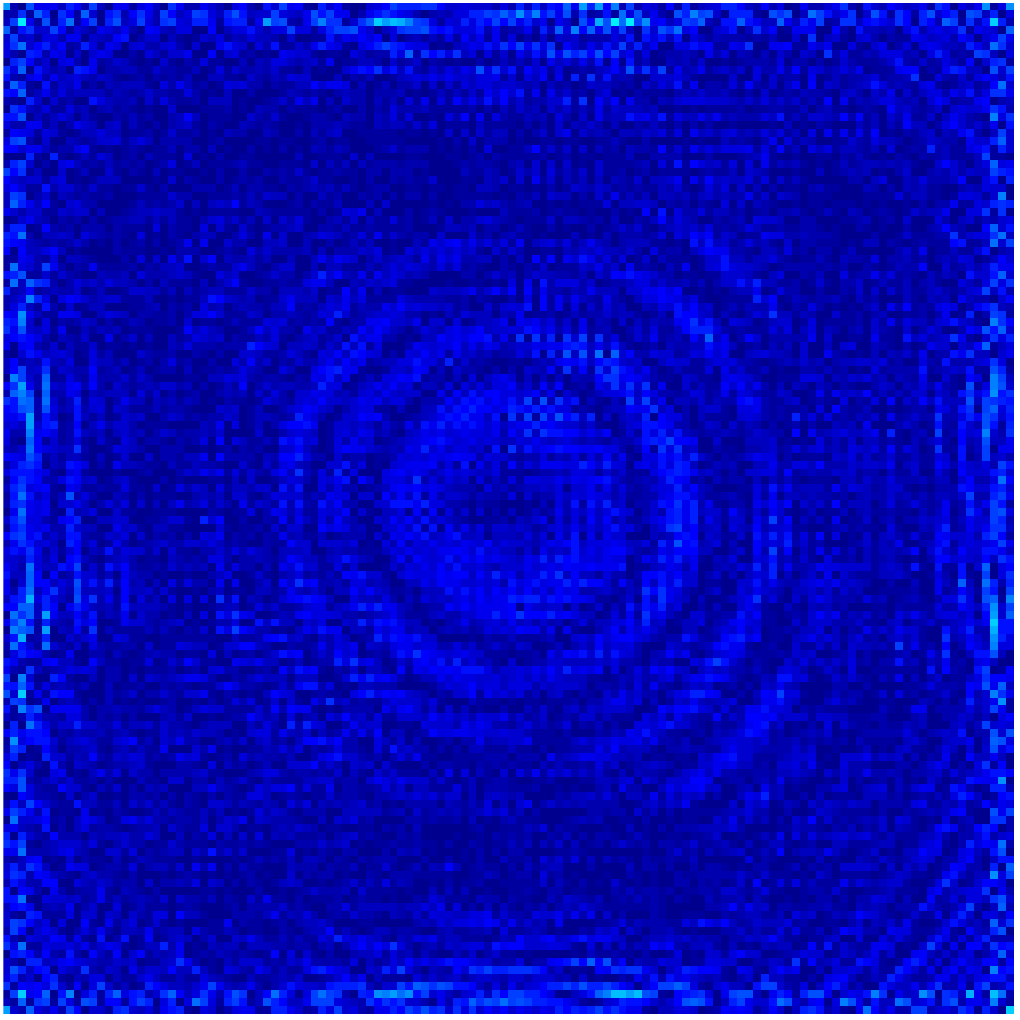}} &
{\includegraphics[height=0.18\textwidth]{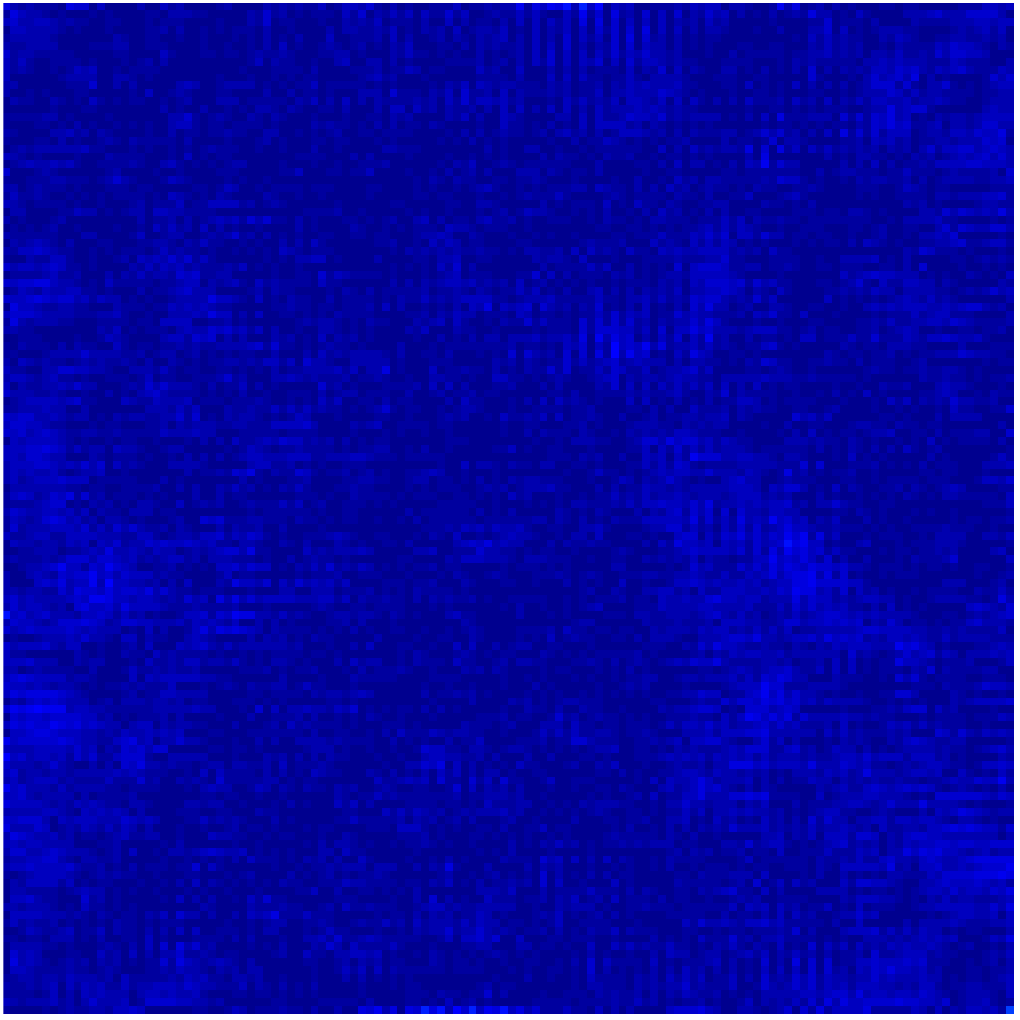}} &
{\includegraphics[height=0.18\textwidth]{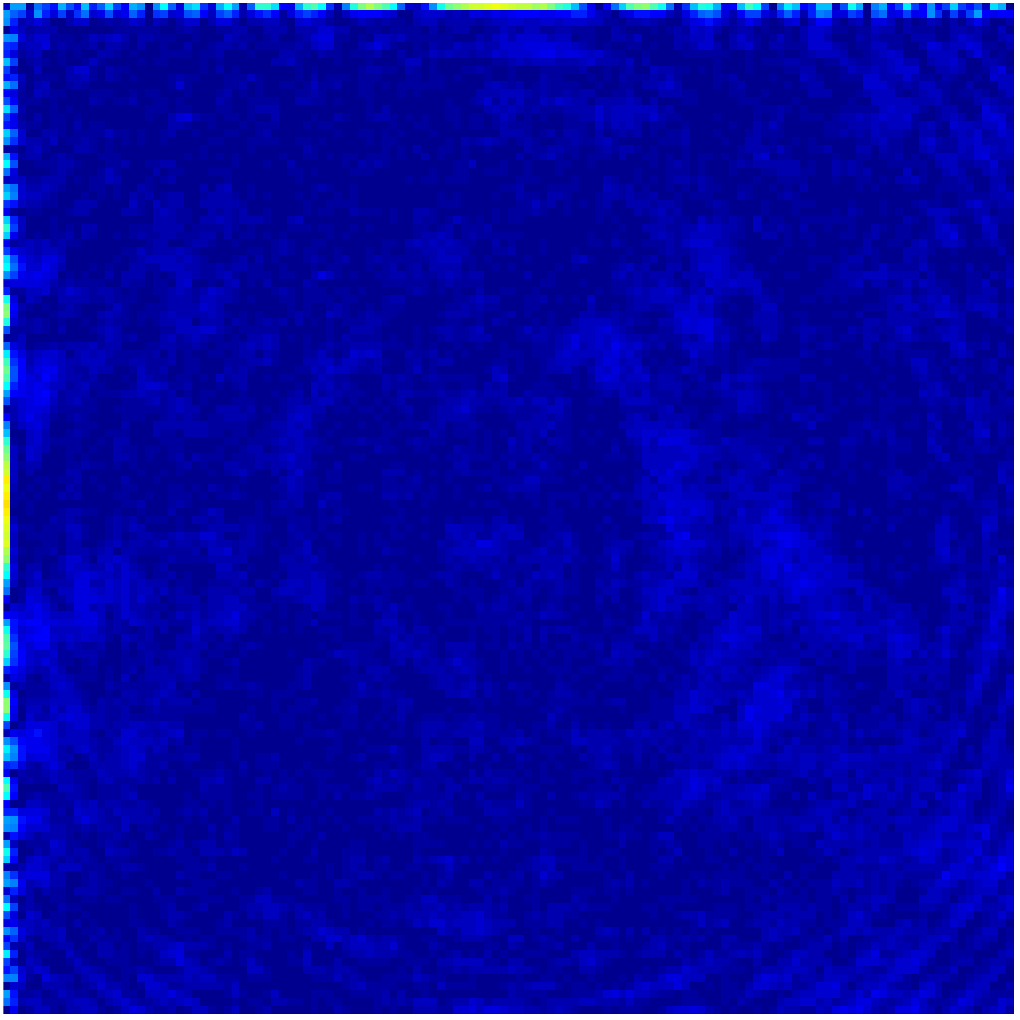}} &
{\includegraphics[height=0.18\textwidth]{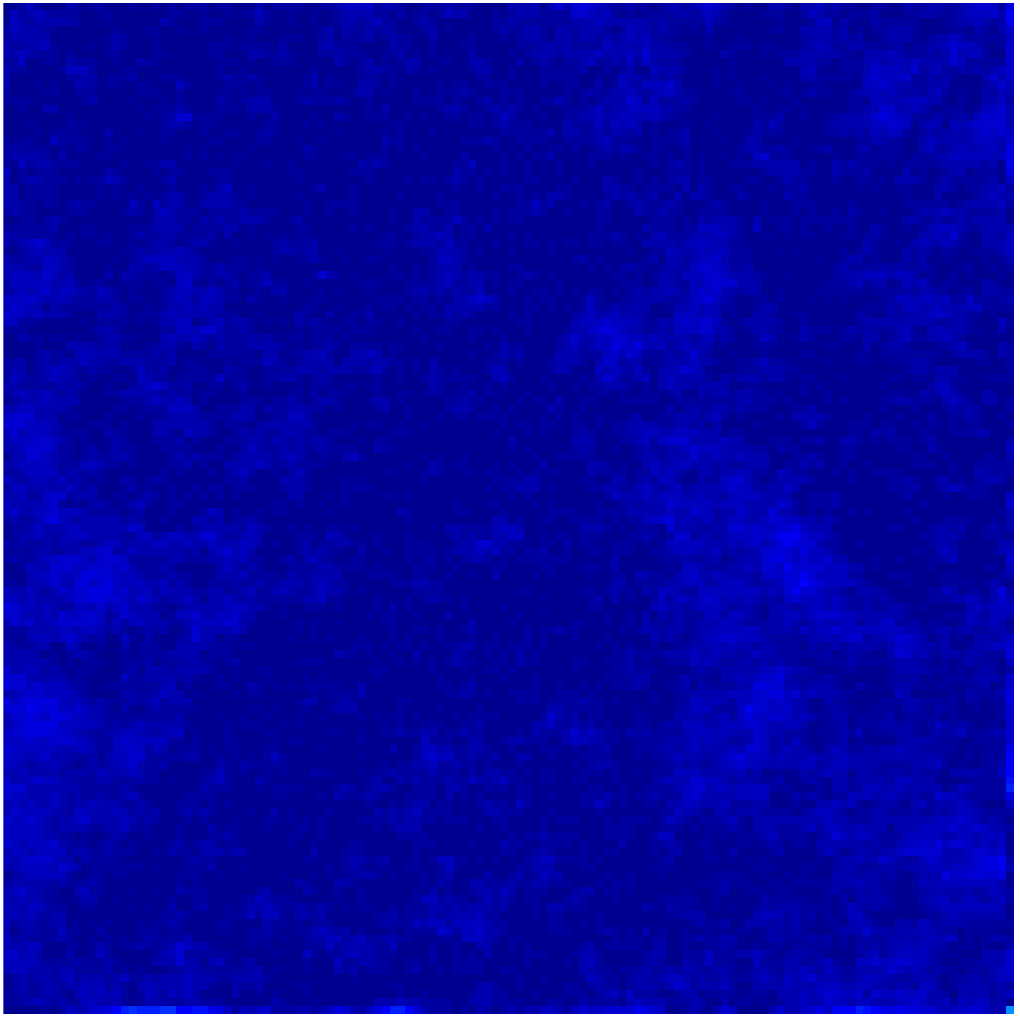}} \\ \cline{1-5}
\end{tabular}
\end{table}
With regards to the recovery performance at the boundary area, embedded by the lowest accuracy $l=1$, the centralized scheme does a better job compared to staggered scheme particularly on higher frequencies which matches with the reported results in \cite{harker2015regularized}. However, once the polynomial accuracy is increased to $l=5$, the staggered formulation overtakes the rank position across all the varying simulation parameters. This simply demonstrates the strength of case--I derivative matrix on recovering both interior and boundary images by means of high polynomial accuracies. A demo example is shown in \ref{fig_zone_plate_recovery_sigma_2} to display the residual errors obtained from both categories. We just select two tap-length polynomials $l\in{1,5}$ for visualization. The perturbed oscillations, known as the chequerboard effect \cite{gambaruto2015processing}, is evident on recovering the image associated by case-II (centralized scheme) design. Moreover, recovering the image by case--I using the lowest accuracy also contains perturbing residuals mostly around the boundaries. Whereas most of these residuals are removed using a high-accuracy staggered derivative matrix (a variant of case--I design).

\subsection{Real-world application in image stitching}\label{section_stitching}
Image stitching is the final application studied in this section to recover seamless portrait from multiple camera images captured in different illumination conditions. The term of \textit{stitching} can be found in different applications such as panoramic imaging \cite{szeliski2006image} or  image mosaicking \cite{levin2004seamless}. Provided by multiple image tiles, the task is to stitch these tiles in a seamless fashion without noticing any shade effect or transition artifacts from one border to another. The shade of grays in raw stitching can be caused by different sources such as multiple camera acquisition setup that have different parameter tuning, variation of illumination under different lighting conditions, etc. In particular here we offer a convenient stitching algorithm for preview camera imaging in digital pathology (DP) solutions known as the whole slide imaging (WSI) scanners \cite{LironPantonowitz}. DP scanners use high resolution imaging optics to acquire submicron resolution images of organ tissues on glass-slides. Prior to such scan, a low resolution preview image of the whole mount tissue slide ($\geq 50\text{mm} \times 75\text{mm}$) is required to guide the scanner to find regions-of-interest (ROI) including tissue, label, and barcode on each slide. The preview image is usually obtained by stitching together several tiles of multiple cameras within the imaging system. The reconstructed images contain artifacts caused by non-uniform illumination on the tile borders and aberrations due to the optical arrangement, which cause discontinuities in the image and degrade the image quality. This effect is shown in \ref{fig_raw_stitch_tissue} where three image tiles are stitched together to construct the whole field-of-view (FOV). The goal here is to correct the corresponding distortion/shading effects to result in a seamless high quality image.

\begin{figure}[htp]
\centerline{
\subfigure[Raw stitch]{\includegraphics[height=0.3\textwidth]{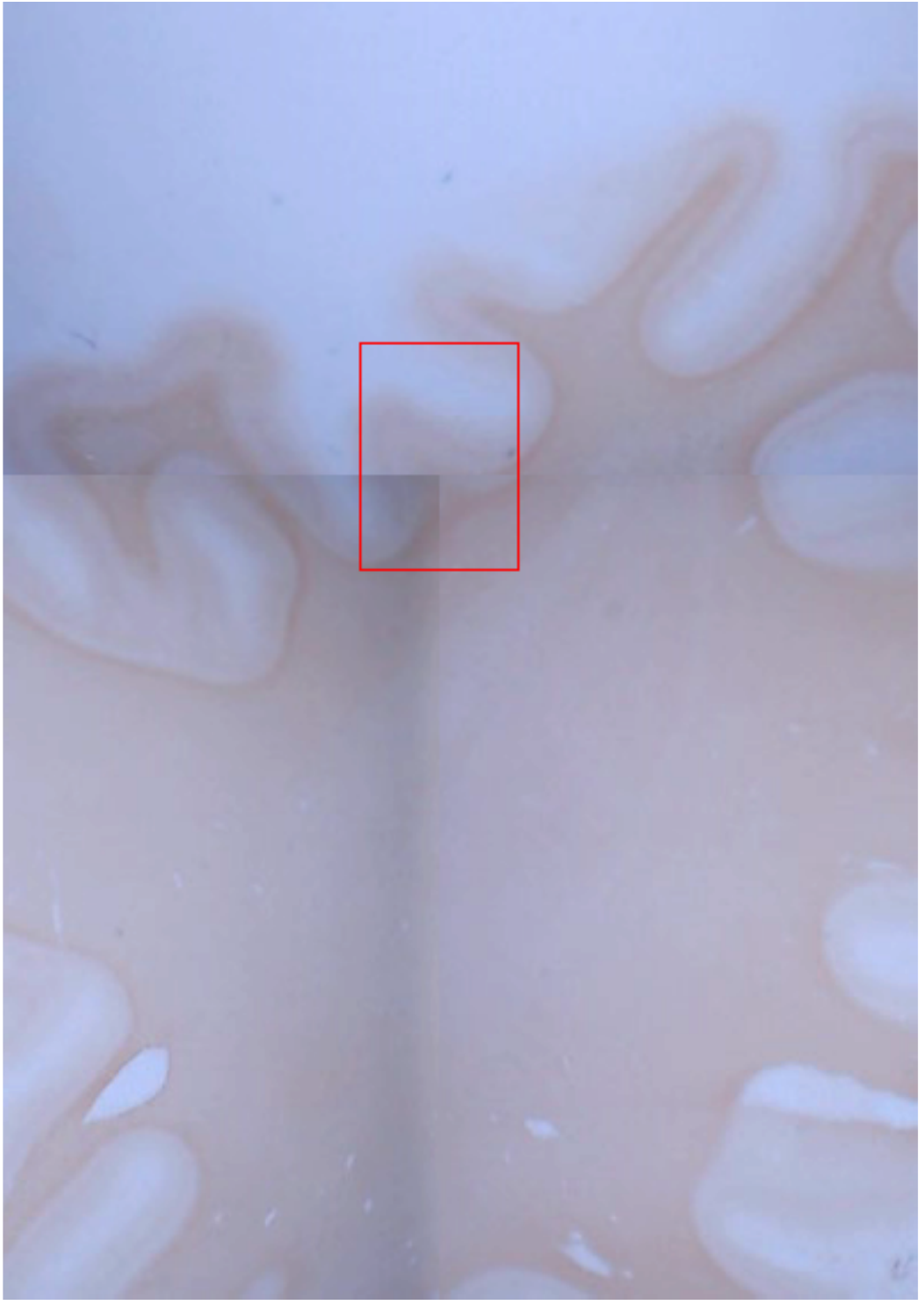}\llap{\includegraphics[width=0.1\textwidth]{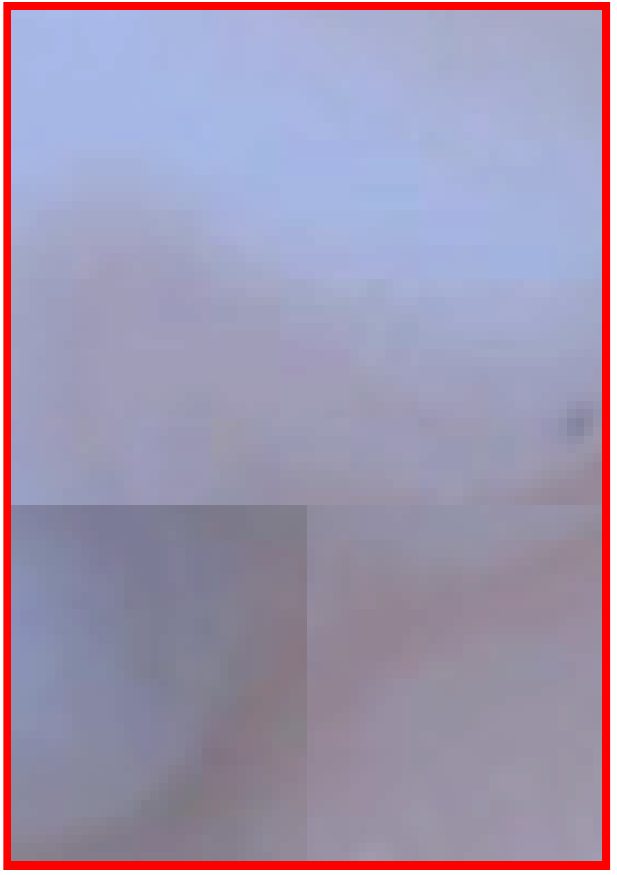}}\label{fig_raw_stitch_tissue}}
\subfigure[Gain]{\includegraphics[height=0.3\textwidth]{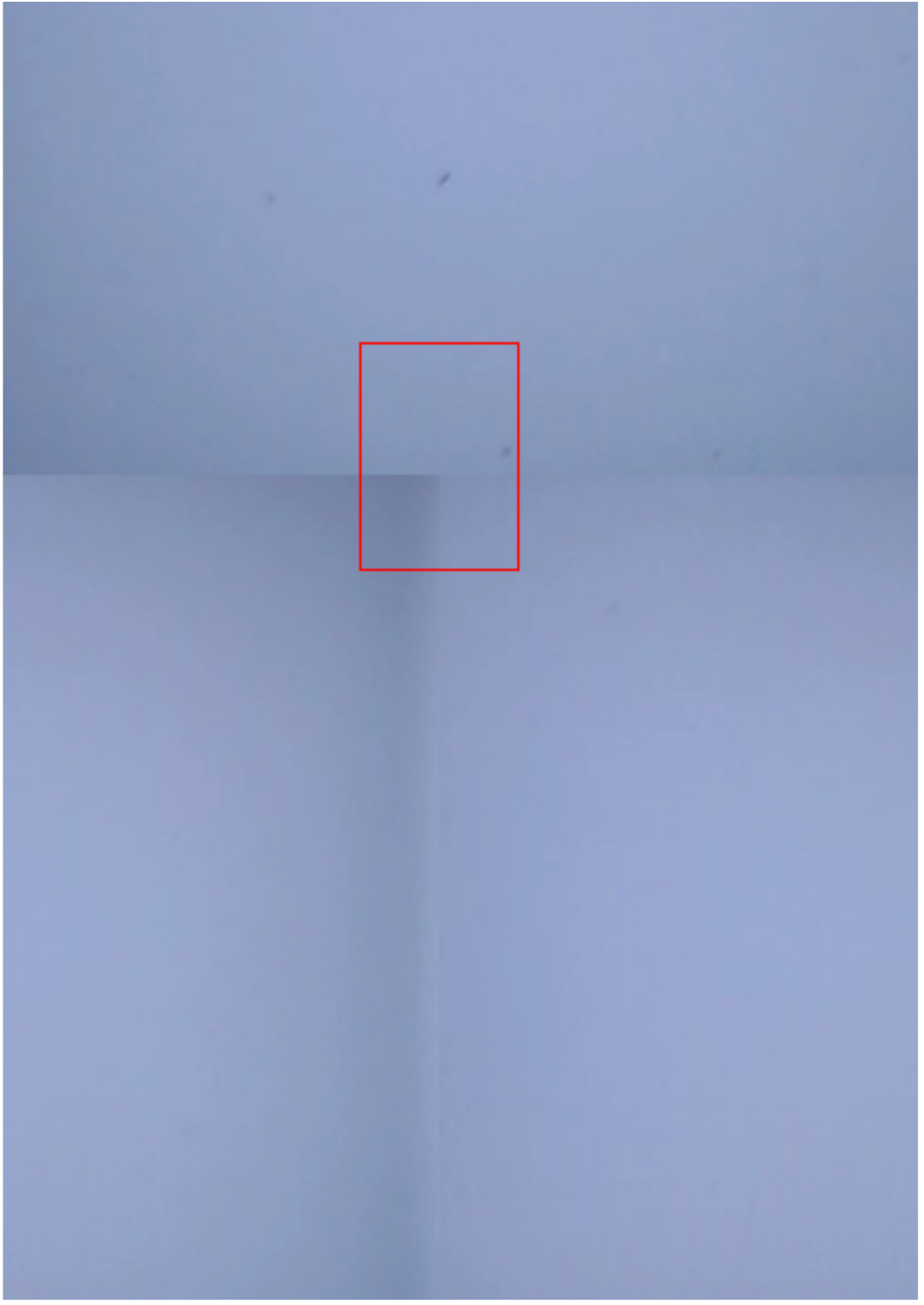}\llap{\includegraphics[width=0.1\textwidth]{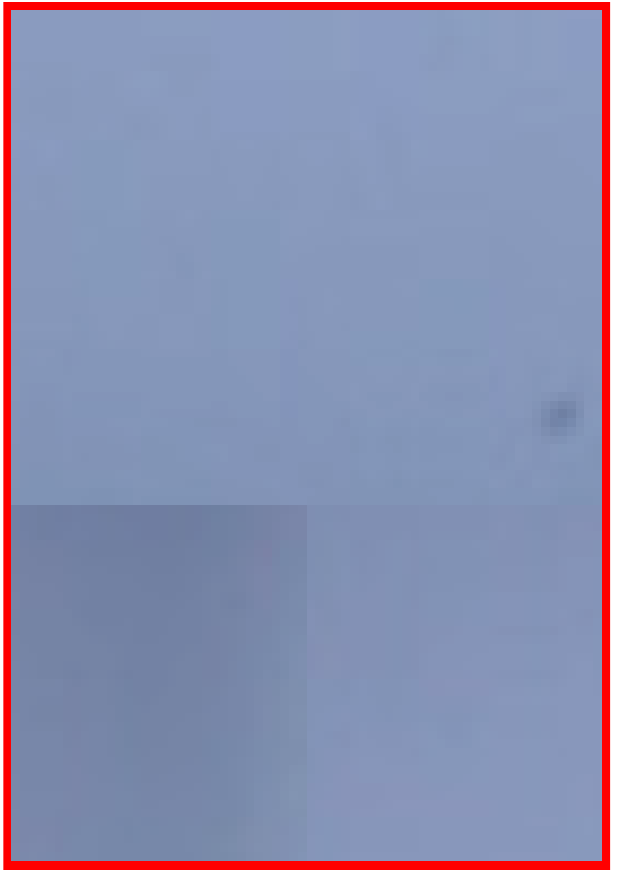}}\label{fig_ROI_gain}}
\subfigure[Gain enhance]{\includegraphics[height=0.3\textwidth]{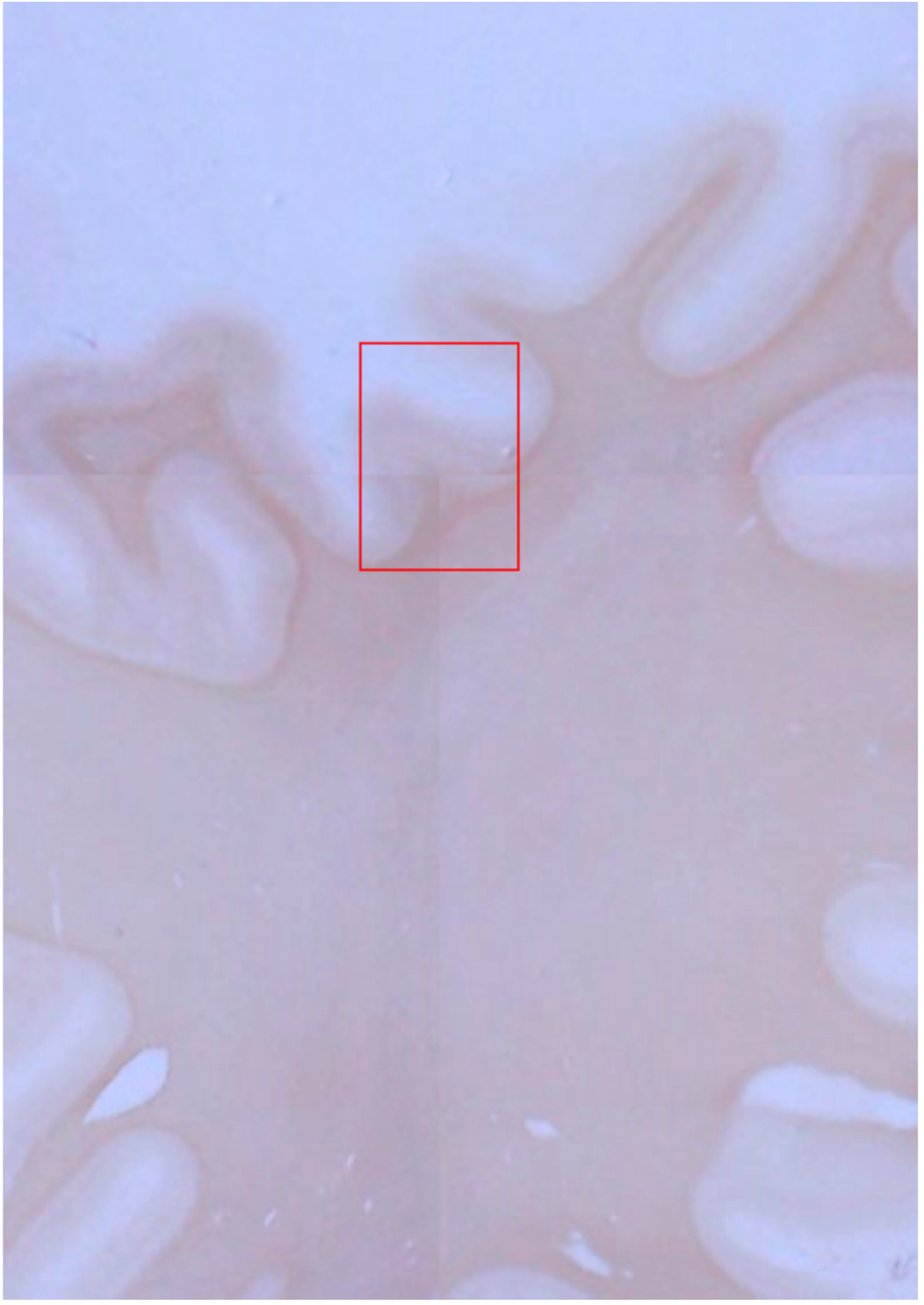}\llap{\includegraphics[width=0.1\textwidth]{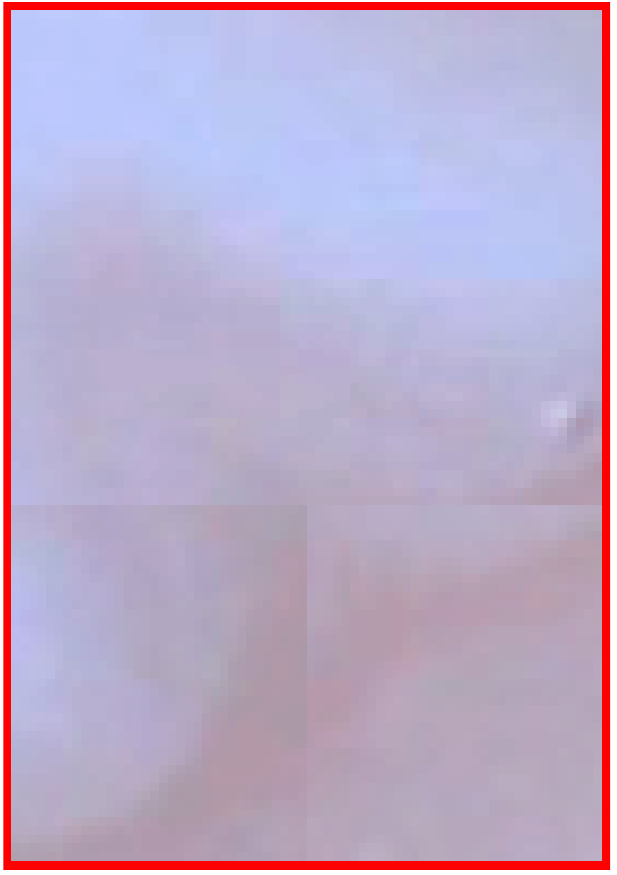}}\label{fig_raw_stitch_enhanced_tissue}}
}
\caption{Raw image stitching with shade/seam effect on the borders.}
\label{fig_stitch_tissue_problem}
\end{figure}

To formulate the problem, multiple image tiles $T_i$ are provided from different sections of each slide to cover the whole FOV. Here, the subscript $i$ indicates the section index from FOV. For instance, here three different tiles are considered to construct the whole image $I = \left[T_1;T_2, T_3\right]$. Note that the image tiles contain overlapping border area that are excluded for raw stitching. Two cameras are used to cover upper ($T_1$) and lower ($T_2$,$T_3$) tiles, respectively. Flat-field correction is performed on each tile to cancel the vignetting effect of both cameras by replacing every tile $T_i\leftarrow (T_i-C_D)\odot G$ where $C_D$ is the dark camera image and $G=1/(C_F-C_D)$ is the gain factor (\ref{fig_ROI_gain}) obtained from flat-filed camera image $C_F$, see \ref{fig_raw_stitch_enhanced_tissue} for correction. The flat-field image is collected from an empty slide with white background. The gain factor indicates the corresponding pixel magnifying ratio from dark to white background. Despite effective image quality improvement  by means of gain enhancement, still the tile borders contain seams which need to be canceled. In the next step the gradient of each tiles are stitched together to form the gradient of FOV $\nabla I = \left[\nabla T_{1};\nabla T_{2}, \nabla T_{3}\right]$. Note that the forward differentiation of the tiles can be associated with lowpass derivative kernels to decrease the noise effect introduced from the raw images. This is a useful step in stitching gradient tiles to mitigate the border miss-alignments and construct a coherent gradient field across the whole FOV. The overall rationality of approximating the gradients is the discrepant artifacts such as bias/smooth illumination will be mostly canceled in the gradient field while coherent transitions are maintained on the tile borders. The final step is applied to reconstruct the image $I$ from the stitched gradients $\nabla I$ by means of the Sylvester recovery introduced in \ref{sec_gradient_surface}. The overall pipeline of the reconstruction algorithm is demonstrated in \ref{fig_preview_stitching_diagram}. The seamless recovery is followed by stitching the input tiles in the gradient field and then reconstruct the whole image domain by means of Sylvester recovery.

\begin{figure}[htp]
\centerline{
\includegraphics[height=0.35\textwidth]{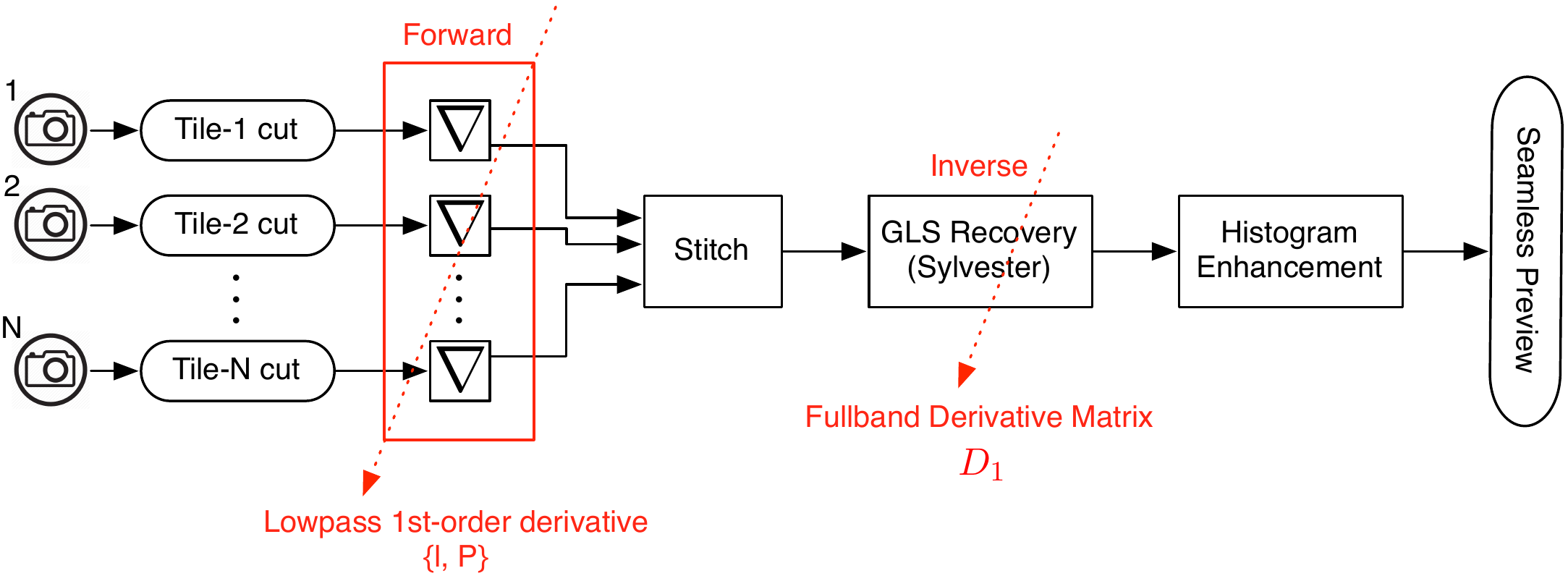}
}
\caption{Camera image stitching flow diagram to recover seamless preview. The gradient of every tile is calculated using a lowpass derivative kernel. The stitched gradient is then recovered using GLS method associated by a fullband derivative matrix.}
\label{fig_preview_stitching_diagram}
\end{figure}

The empirical analysis is done here using different configurations of derivative kernels for both forward and inverse gradient calculations. To obtain tile gradients associated with lowpass differentiation, we emphasis to use first order derivative filter with $l=5$ tap-length polynomial in staggered form. This filter offers a reasonable range of cutoff parameters $P=\{1,3,5,7,9\}$. Note that $P=1$ is the lowest cutoff level and $P=9$ is a fullband (no-cutoff) differentiator. Choosing a staggered scheme here for forward differentiation is to simply avoid deviation of filter response to mitigate the information loss at the frequency spectrum. The associated derivative matrix in Sylvester equation is also set by different parameters. Both staggered and centralized schemes are deployed by two tap-length polynomials $l=\{1,5\}$. 

\begin{table}[htp]
\renewcommand{\arraystretch}{1.3}
\caption{Preview image stitching by means of gradient surface recovery technique. Two derivative matrices of case--I and case--II are used to solve Sylvester equation.}
\label{table_preview_image_stitching}
\centering
\scriptsize
\begin{tabular}{|l|c|c|c|c|c|}
\cline{2-6}
\multicolumn{1}{c|}{} & $P=1$ & $P=3$ & $P=5$ & $P=7$ & $P=9$ \\ \cline{1-6}
{\hspace{-.05in}\begin{sideways} \hspace{.15in} centralized,~$l=1$ \end{sideways}\hspace{-.05in}}   & 
{\includegraphics[height=0.23\textwidth]{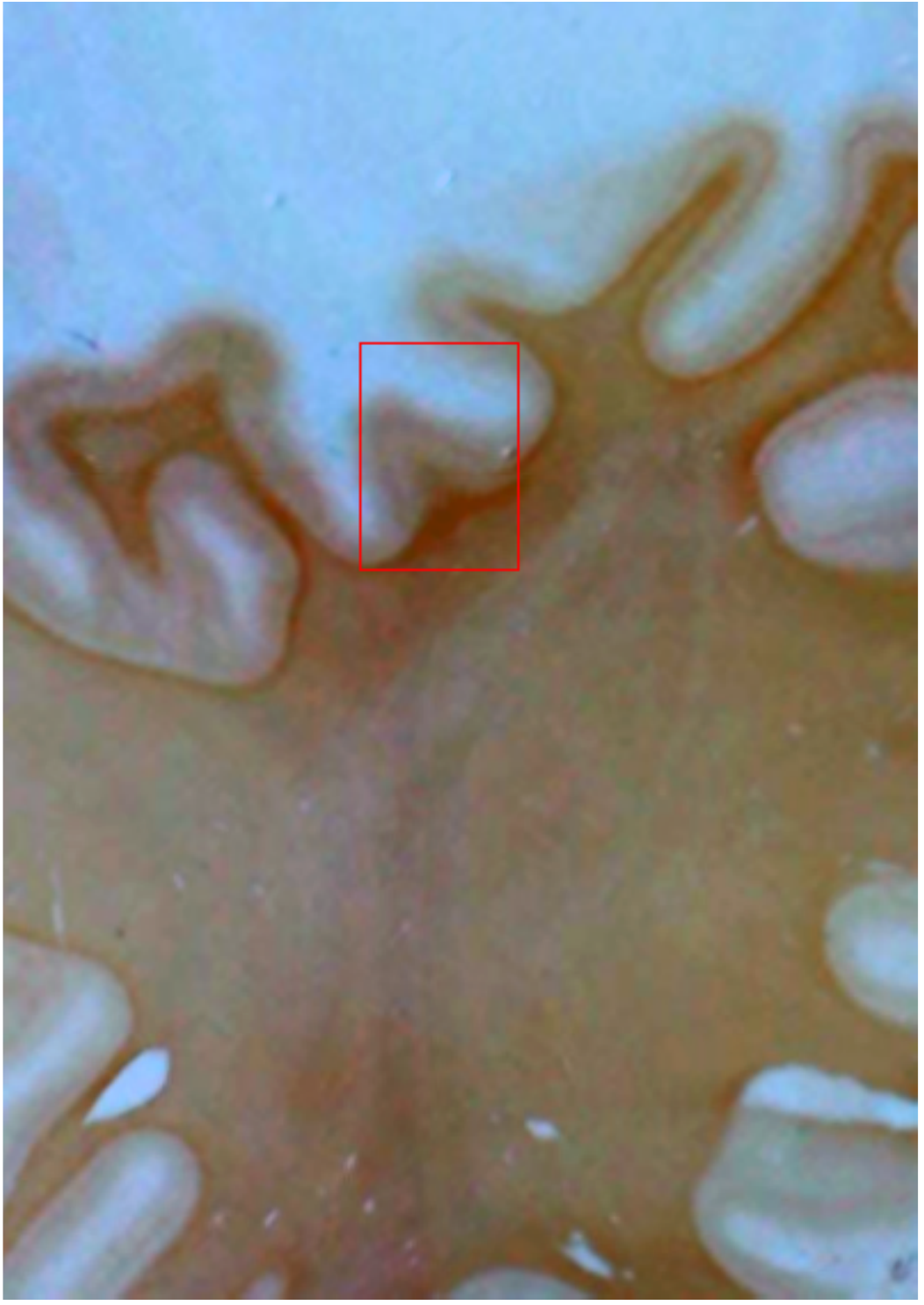}\llap{\includegraphics[width=0.08\textwidth]{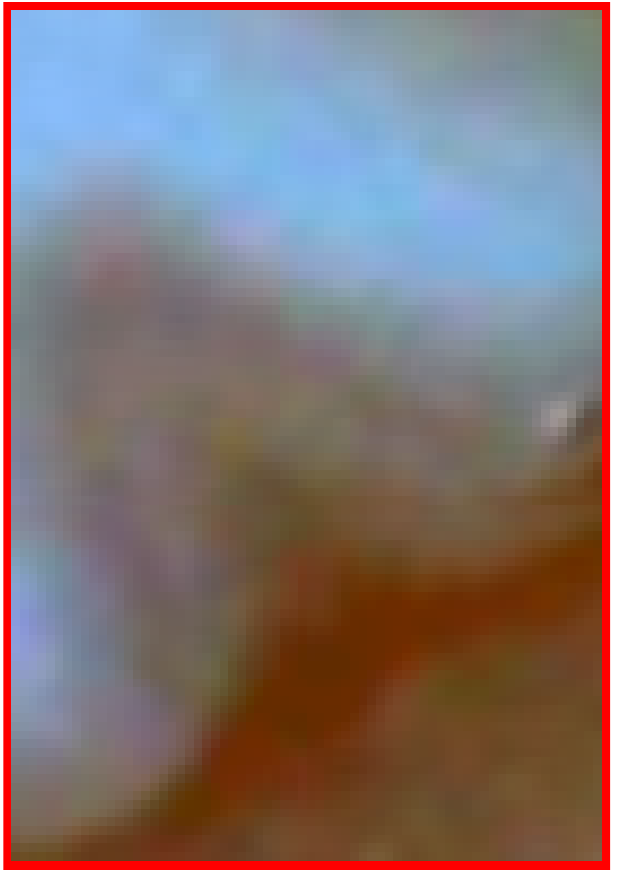}}} &
{\includegraphics[height=0.23\textwidth]{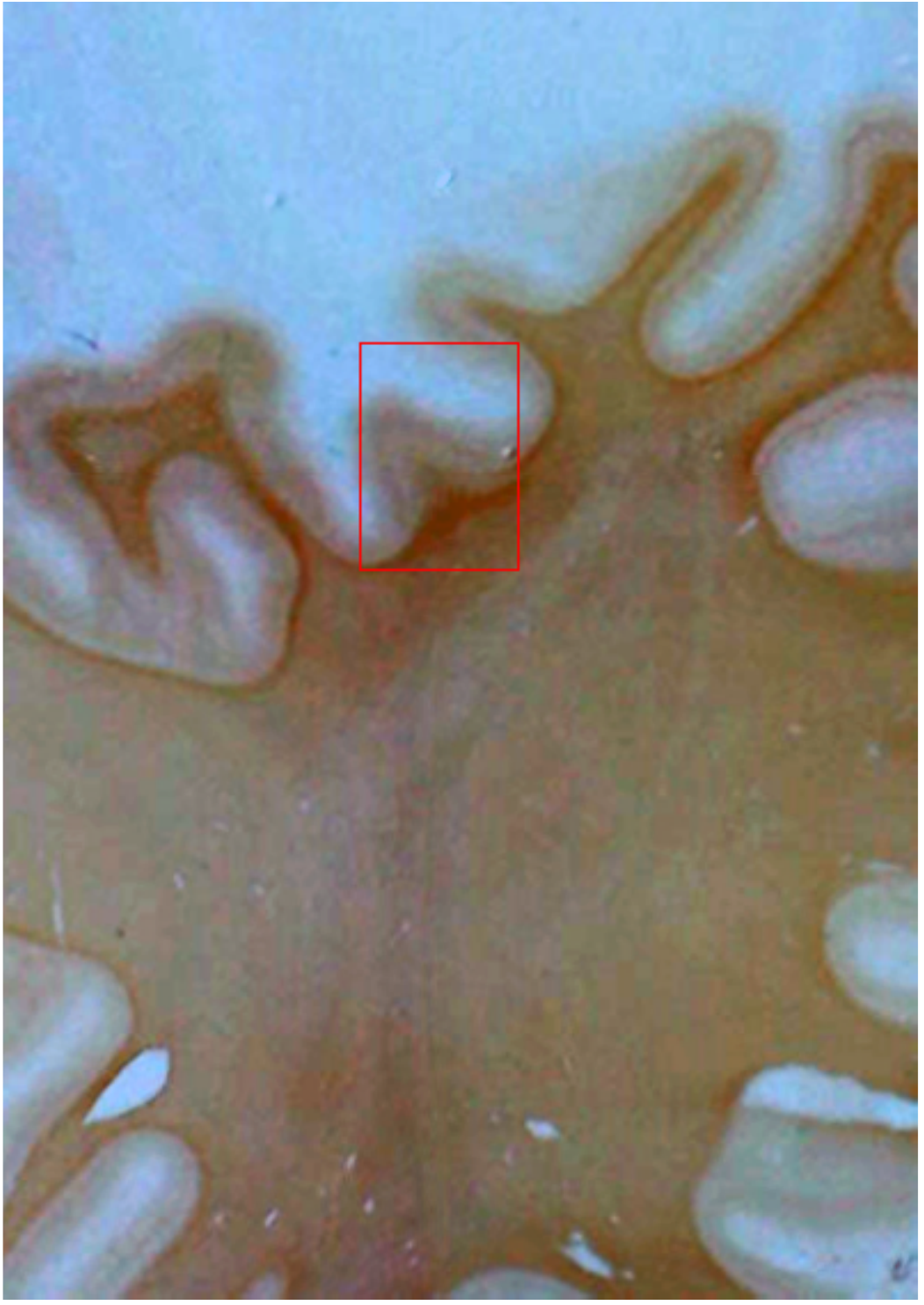}\llap{\includegraphics[width=0.08\textwidth]{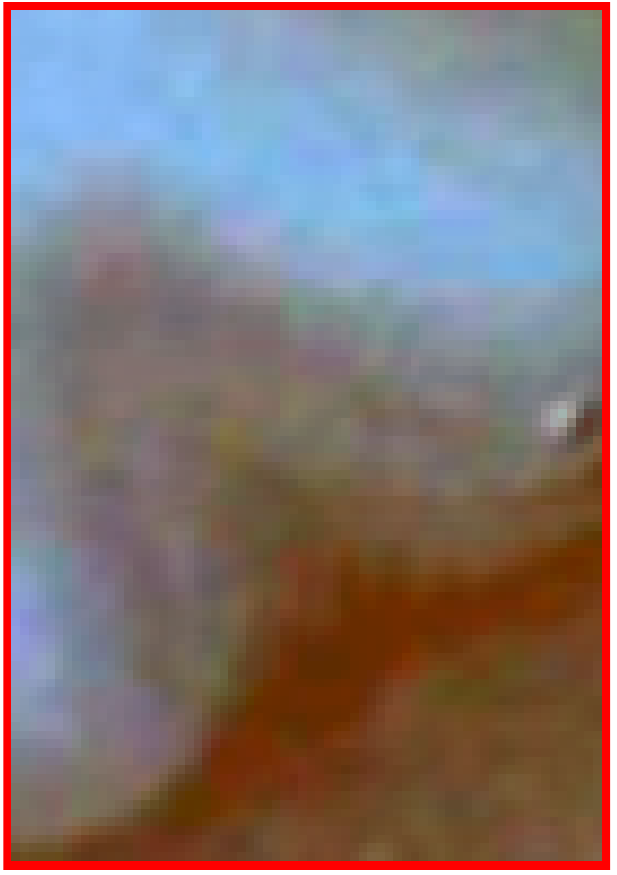}}} &
{\includegraphics[height=0.23\textwidth]{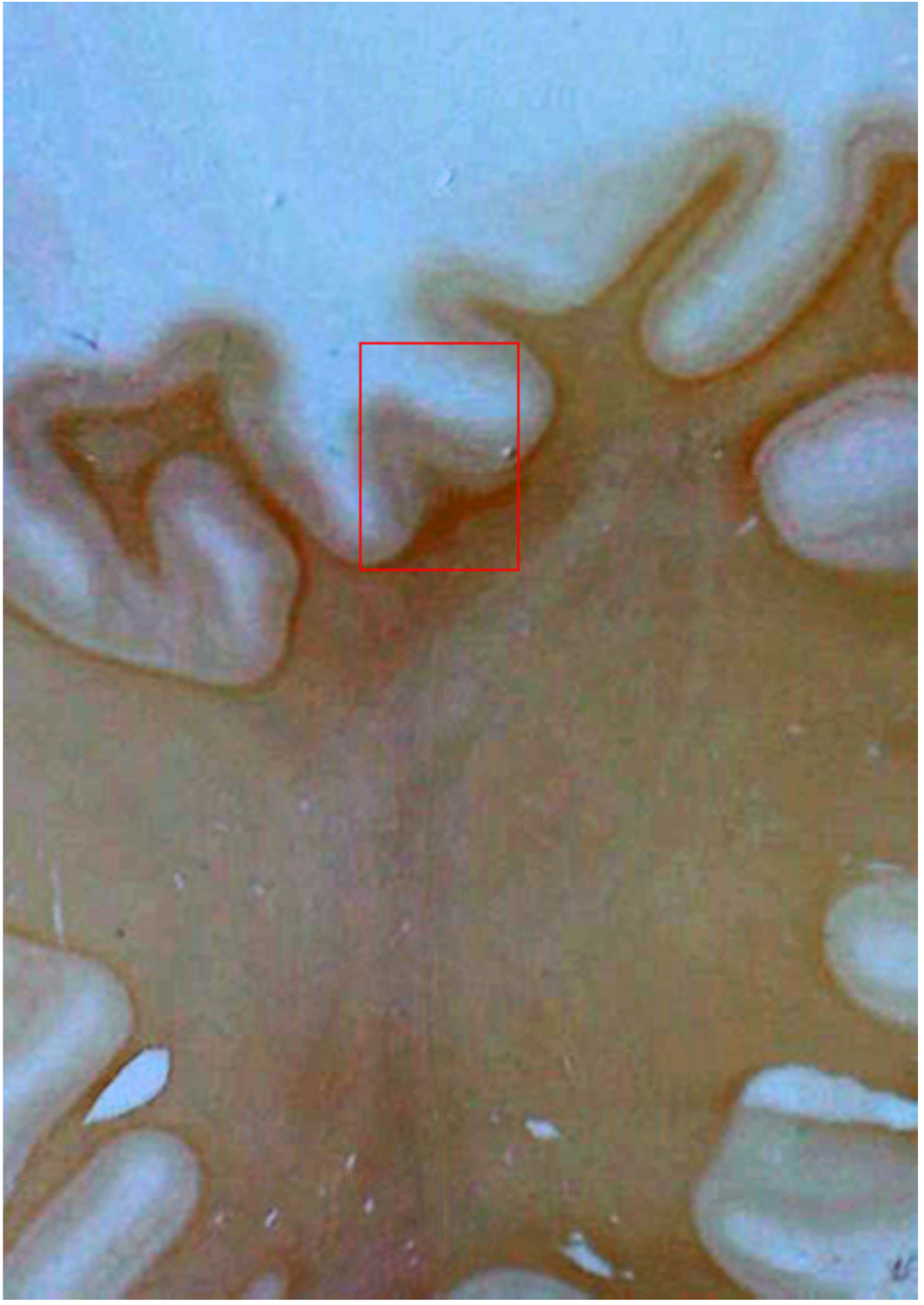}\llap{\includegraphics[width=0.08\textwidth]{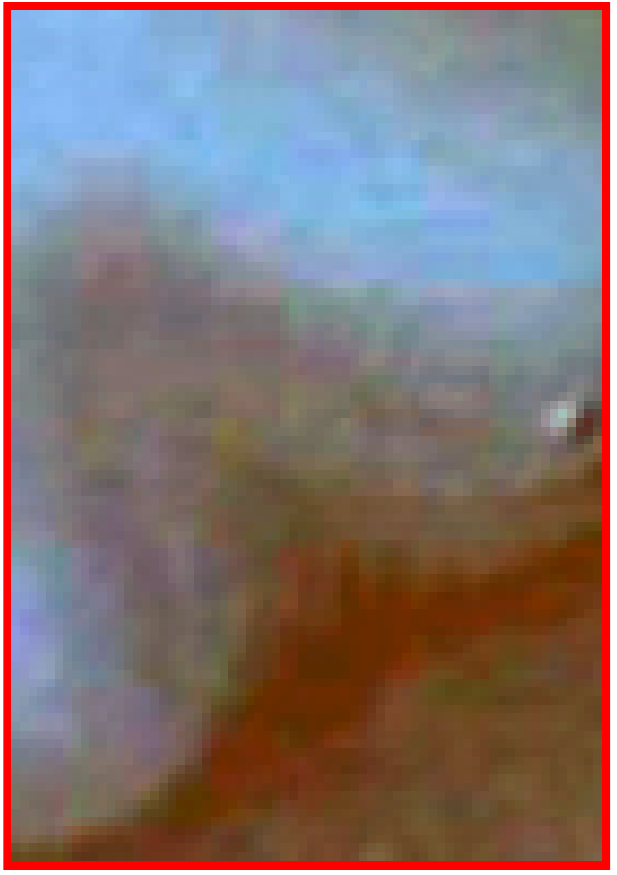}}} &
{\includegraphics[height=0.23\textwidth]{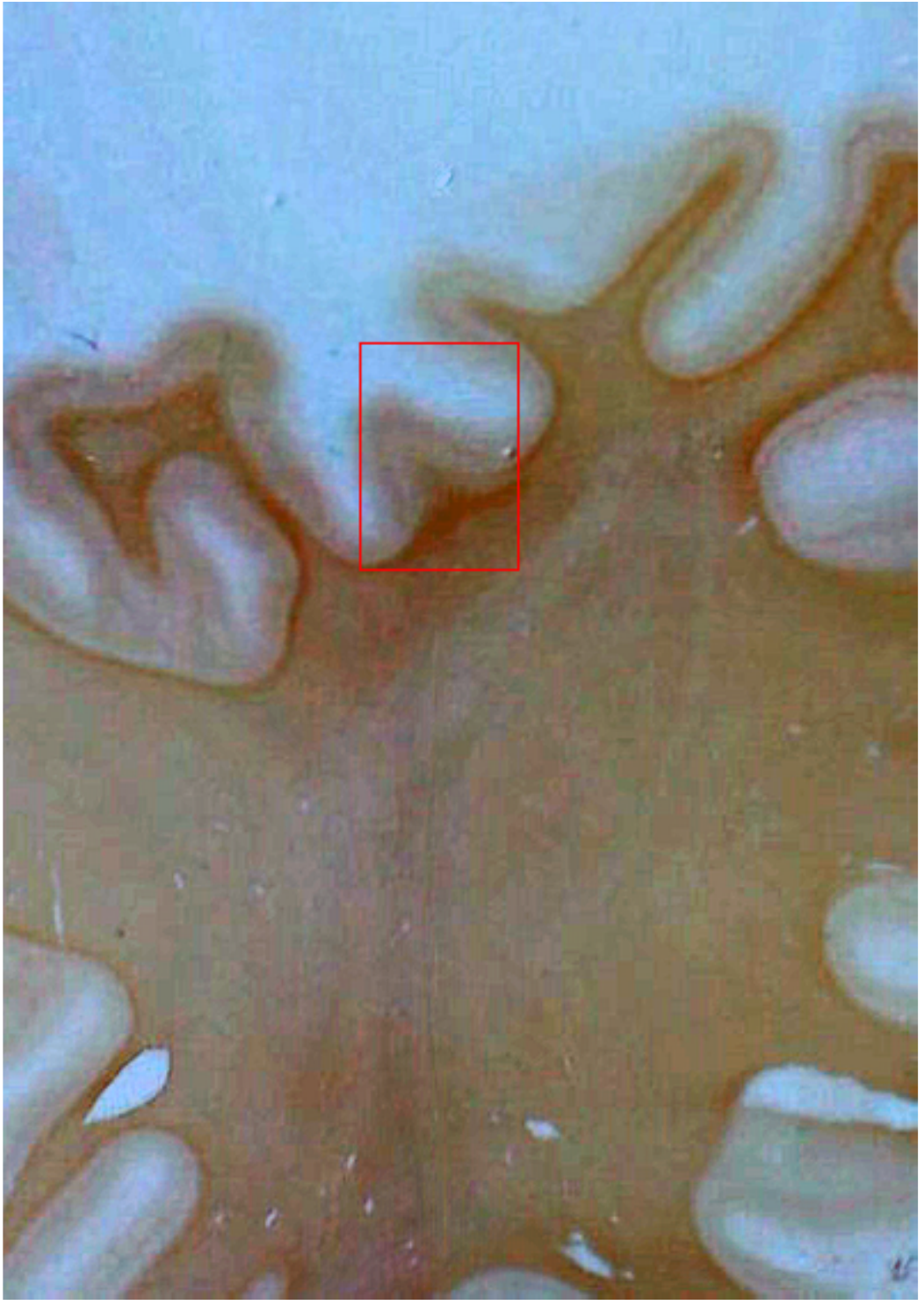}\llap{\includegraphics[width=0.08\textwidth]{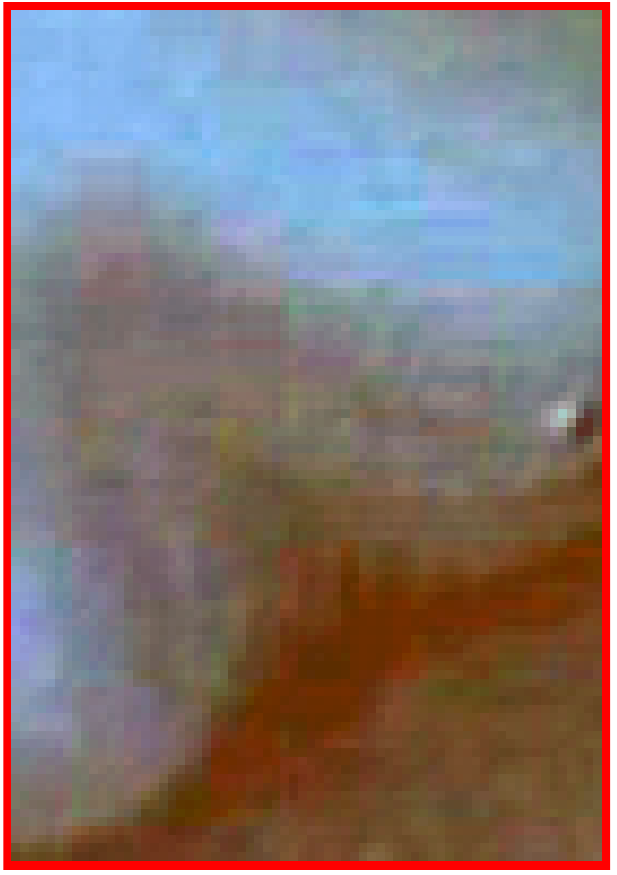}}} &
{\includegraphics[height=0.23\textwidth]{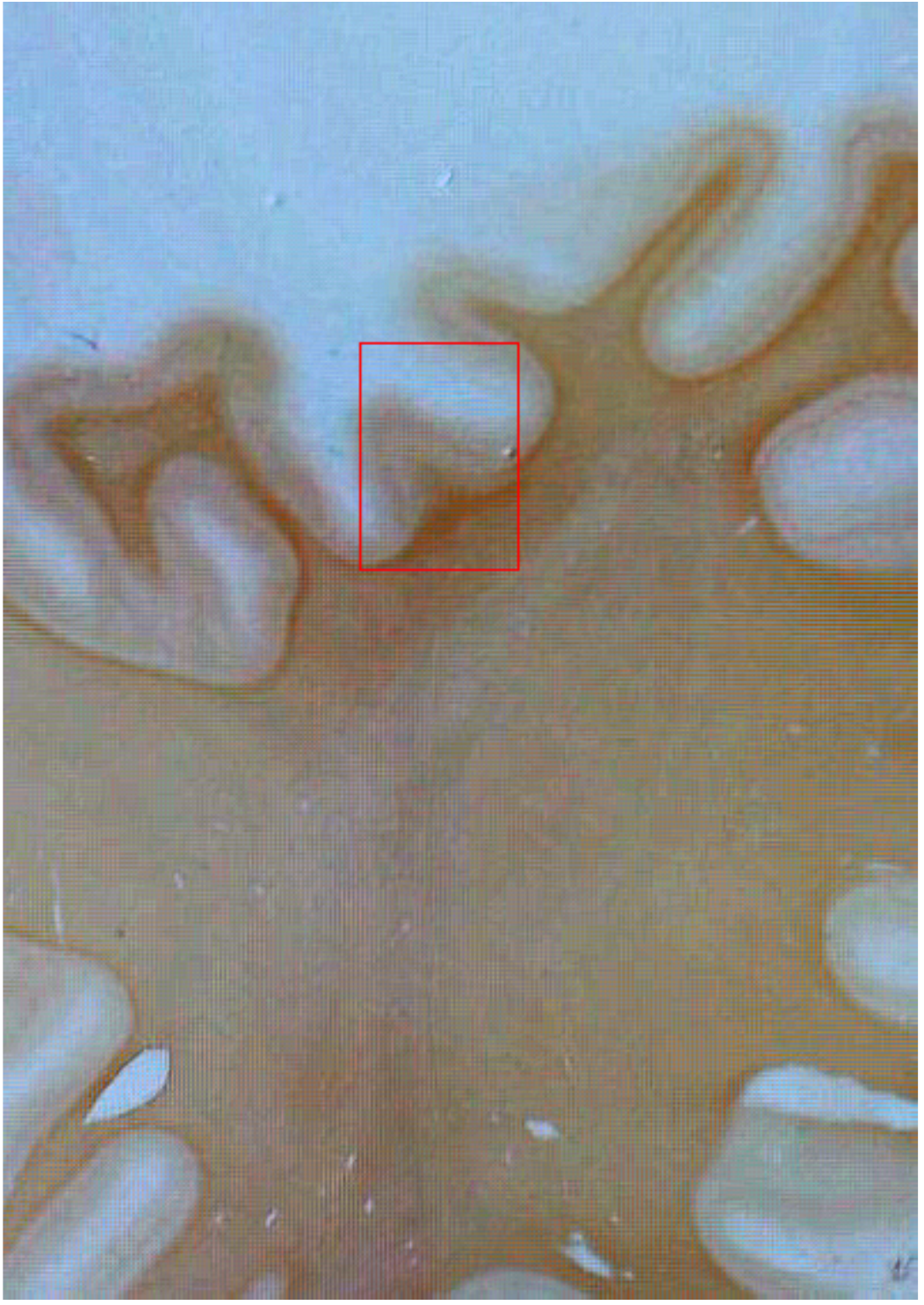}\llap{\includegraphics[width=0.08\textwidth]{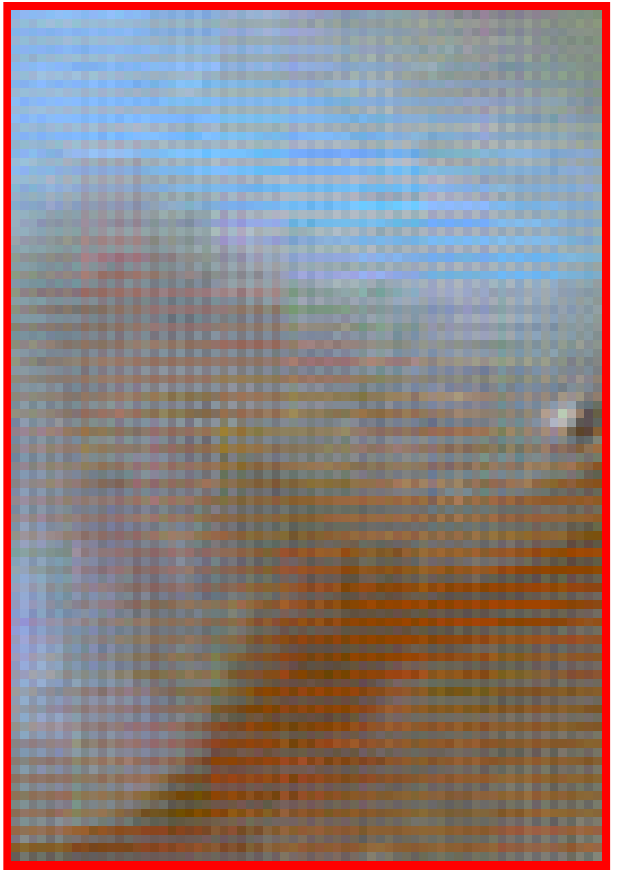}}} \\ \cline{1-6}
{\hspace{-.05in}\begin{sideways} \hspace{.15in} centralized,~$l=5$ \end{sideways}\hspace{-.05in}}   & 
{\includegraphics[height=0.23\textwidth]{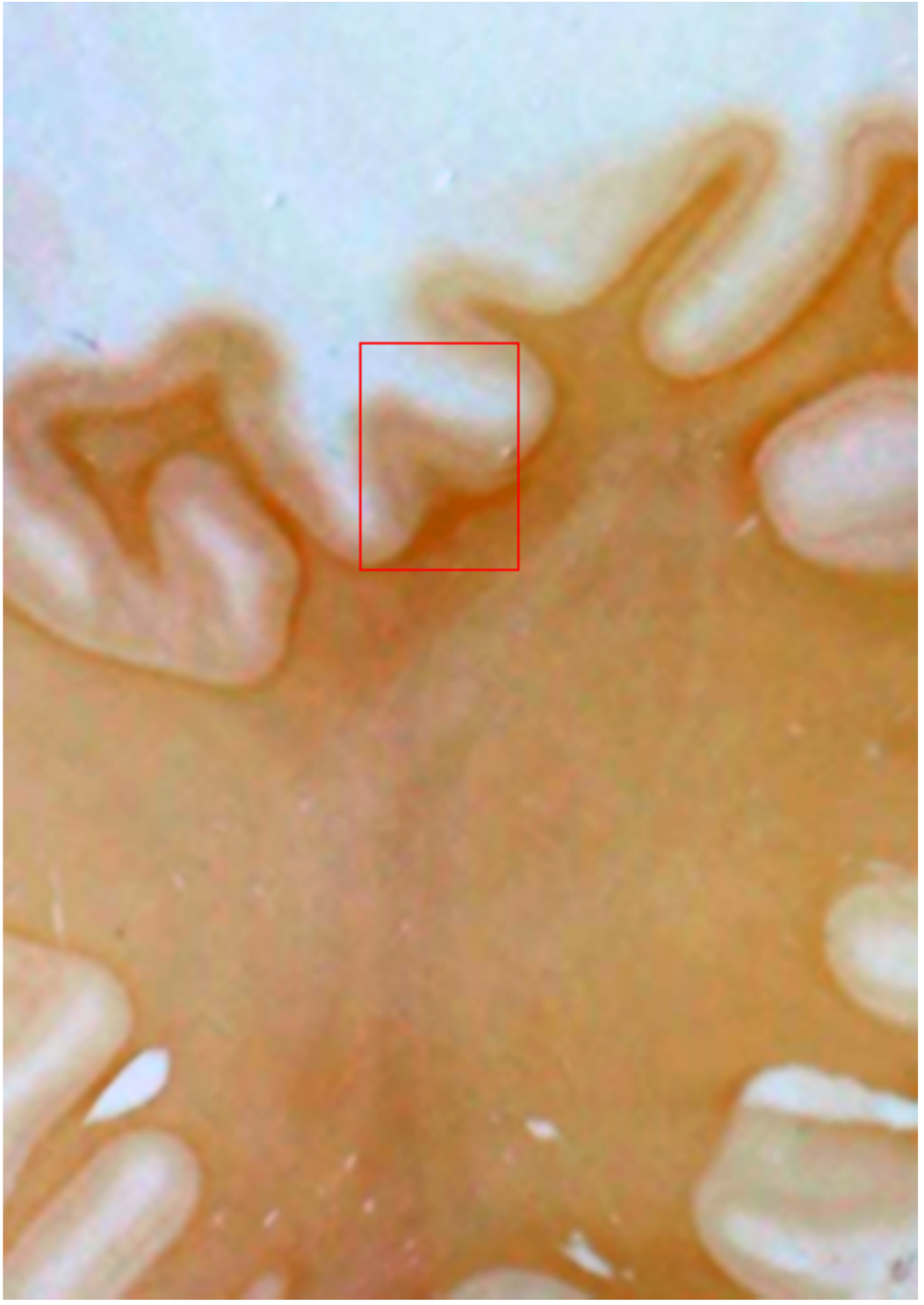}\llap{\includegraphics[width=0.08\textwidth]{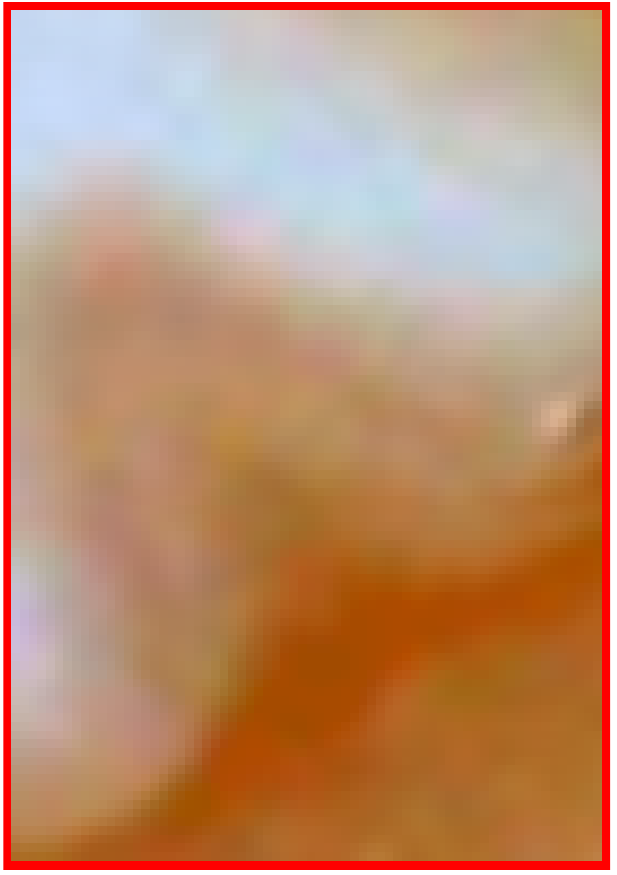}}} &
{\includegraphics[height=0.23\textwidth]{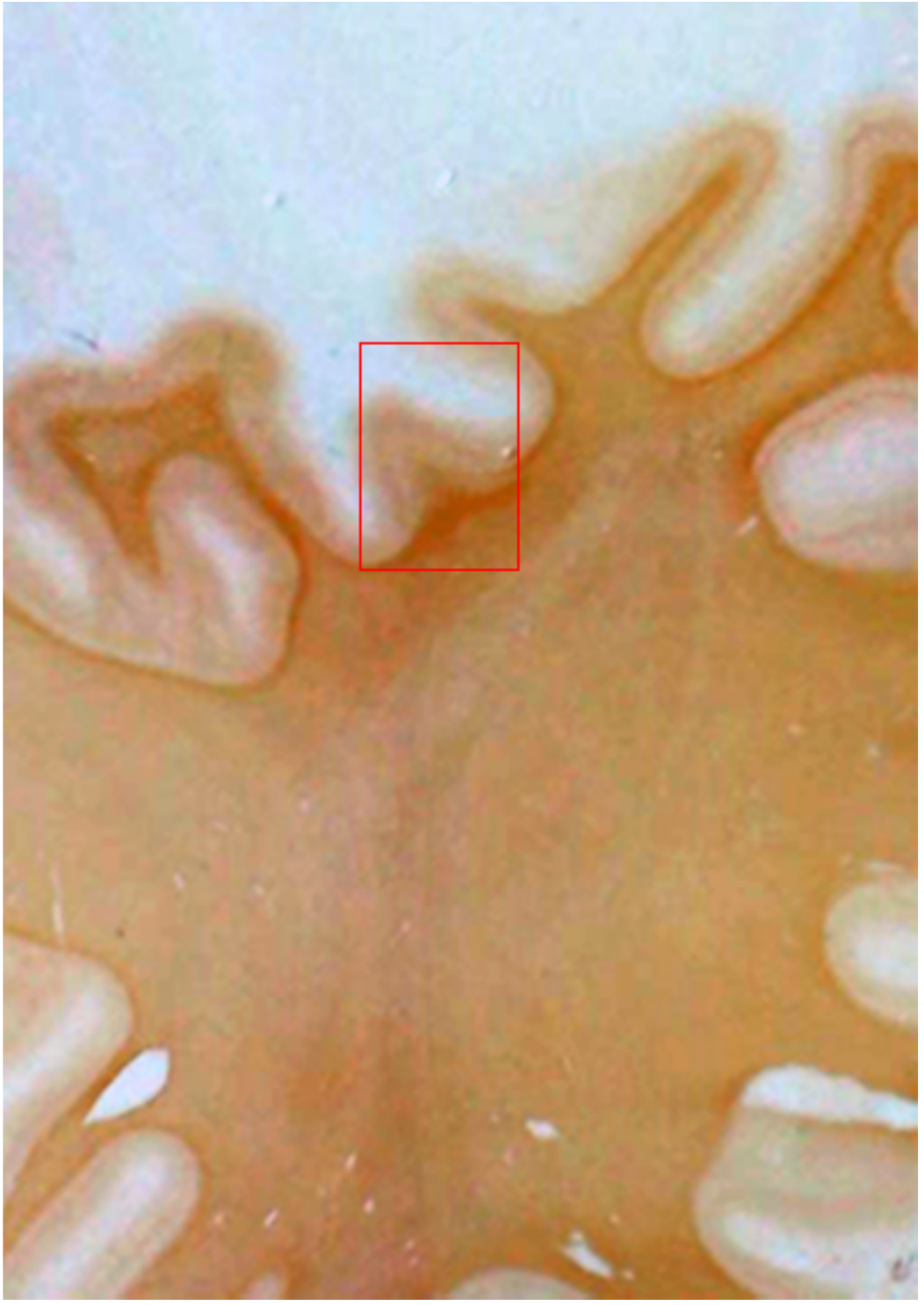}\llap{\includegraphics[width=0.08\textwidth]{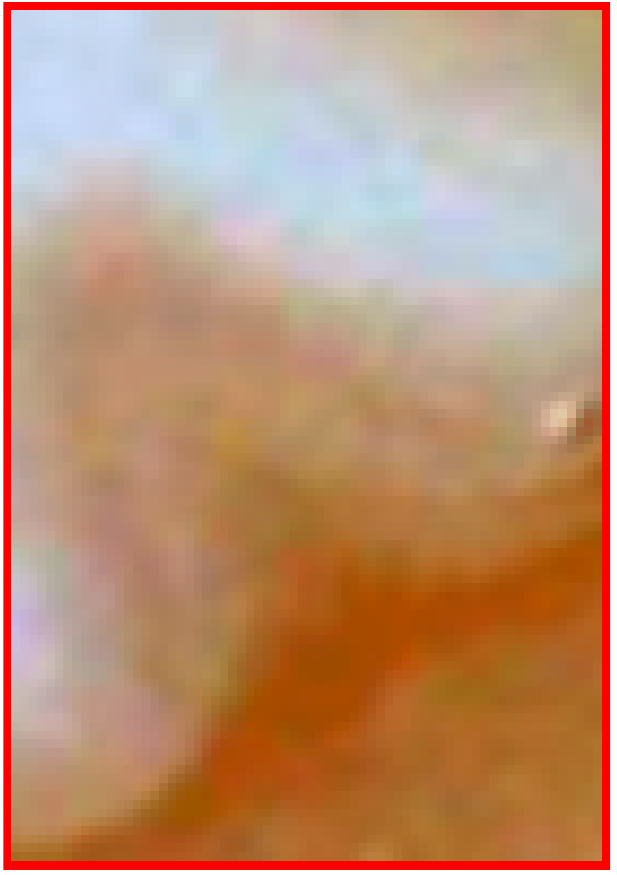}}} &
{\includegraphics[height=0.23\textwidth]{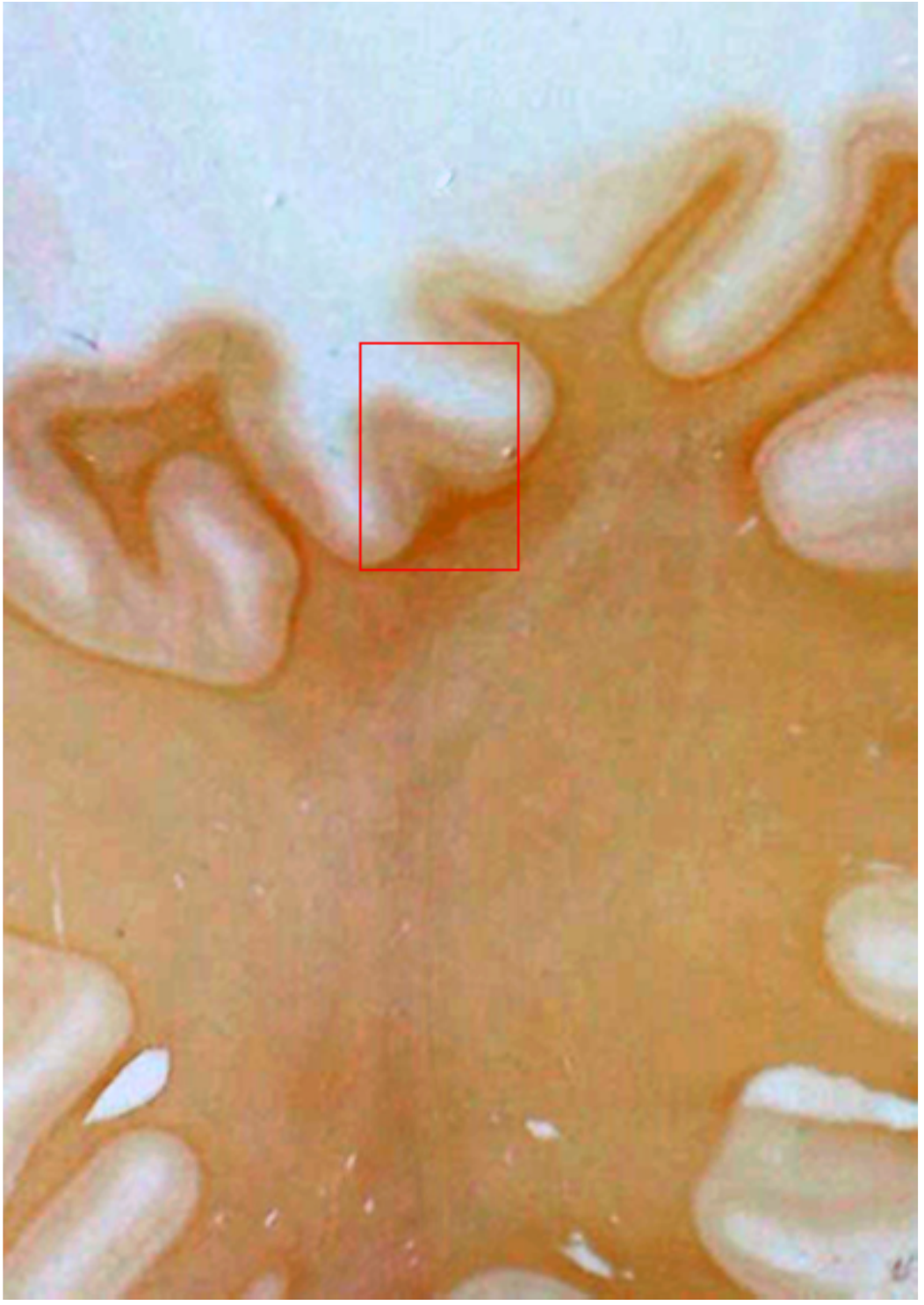}\llap{\includegraphics[width=0.08\textwidth]{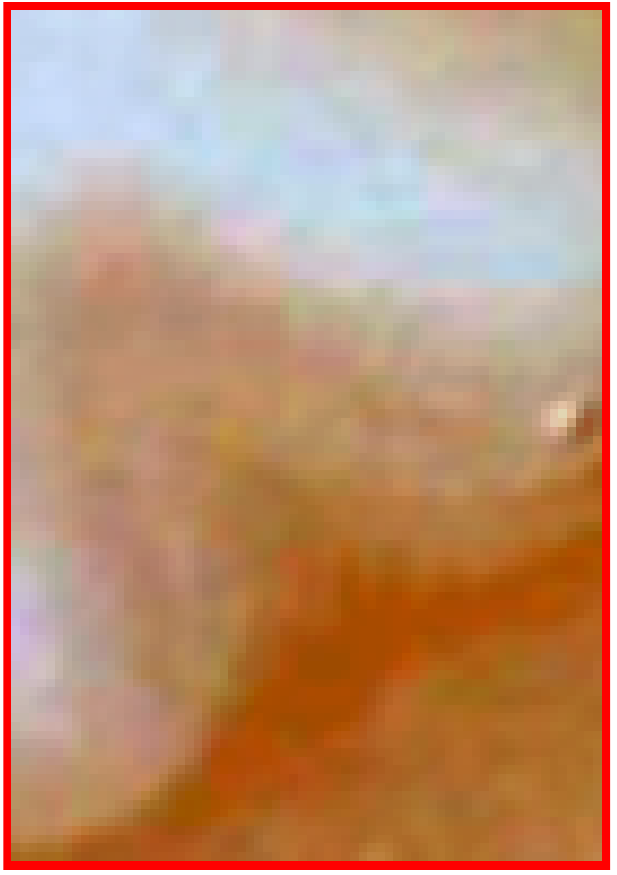}}} &
{\includegraphics[height=0.23\textwidth]{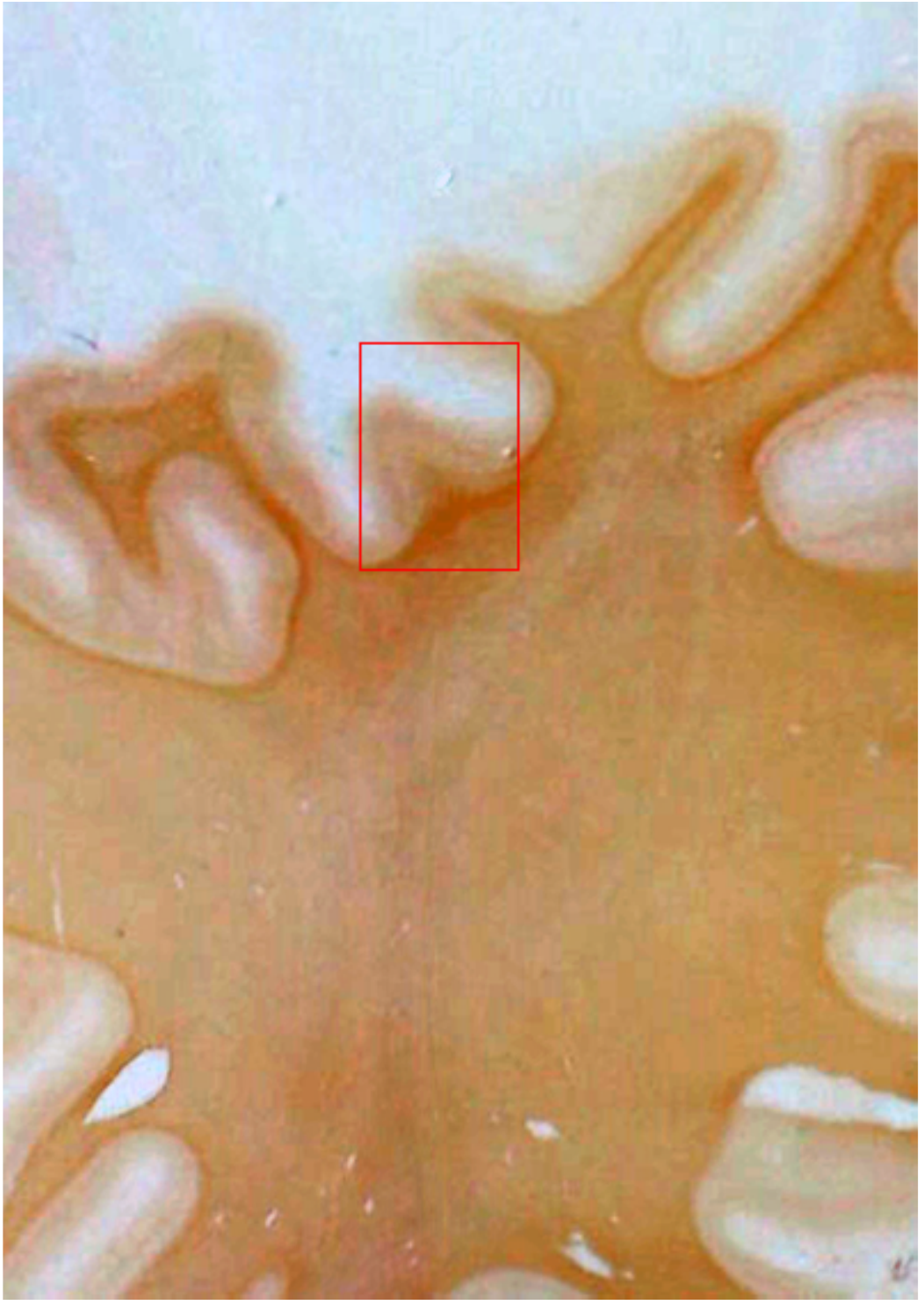}\llap{\includegraphics[width=0.08\textwidth]{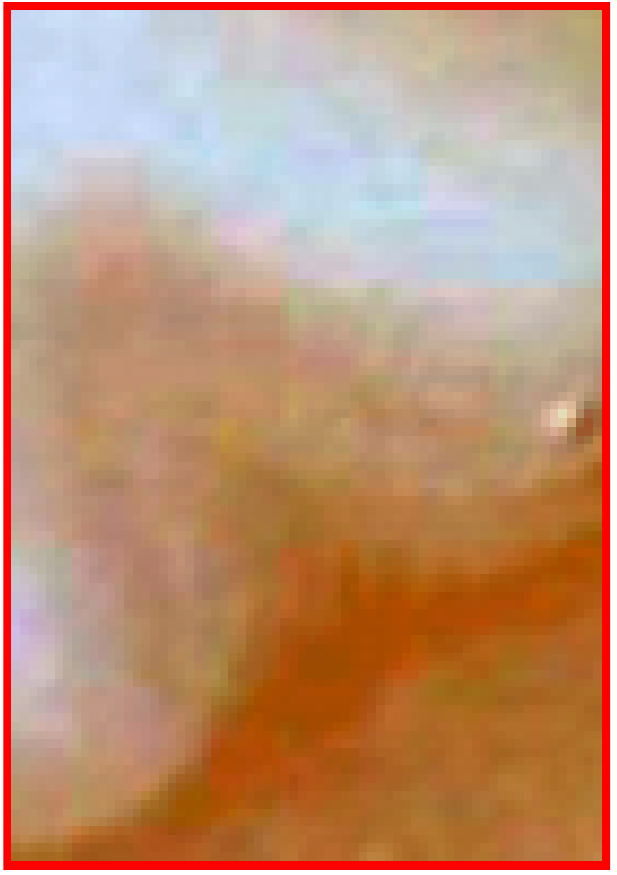}}} &
{\includegraphics[height=0.23\textwidth]{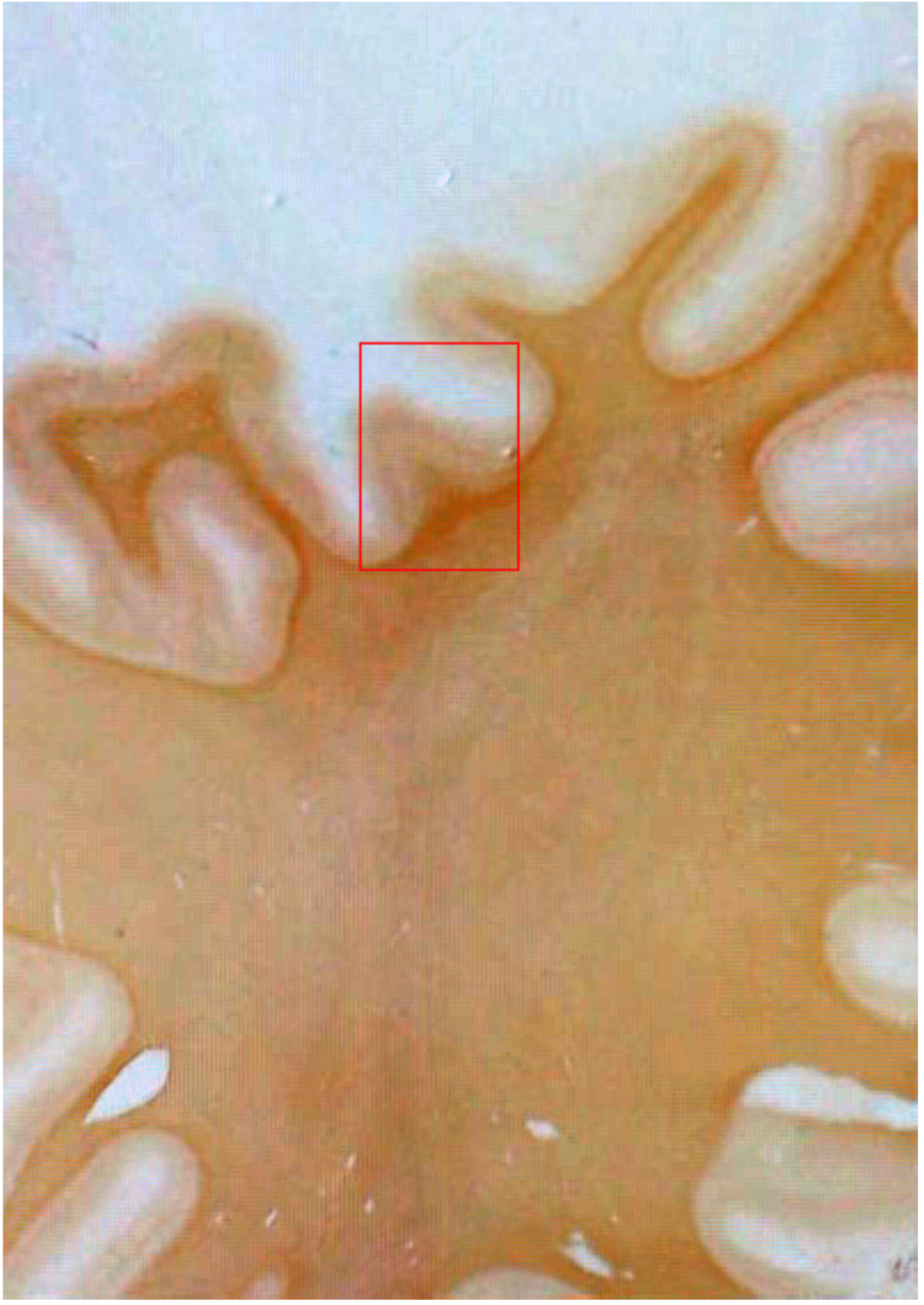}\llap{\includegraphics[width=0.08\textwidth]{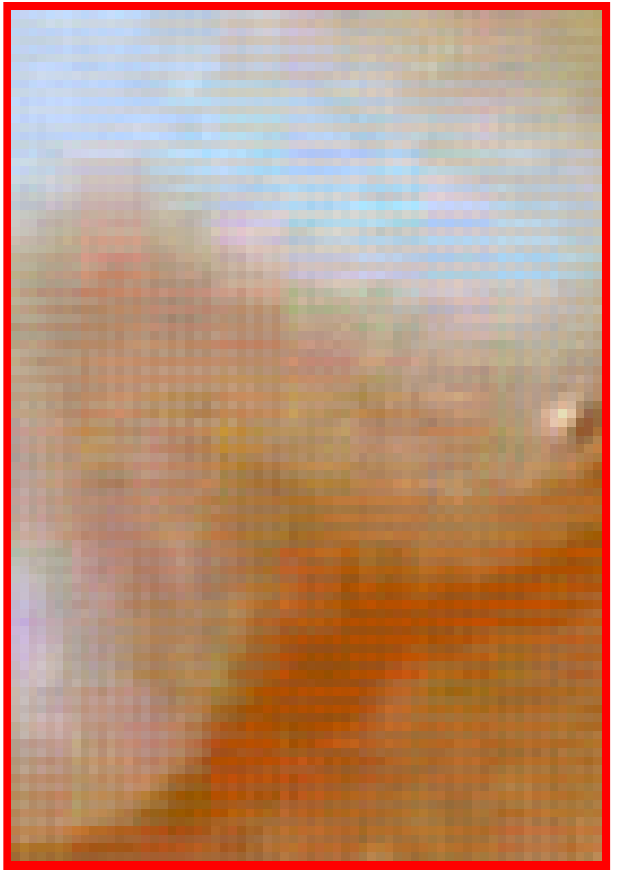}}} \\ \cline{1-6}
{\hspace{-.05in}\begin{sideways} \hspace{.15in} staggered,~$l=1$ \end{sideways}\hspace{-.05in}}   & 
{\includegraphics[height=0.23\textwidth]{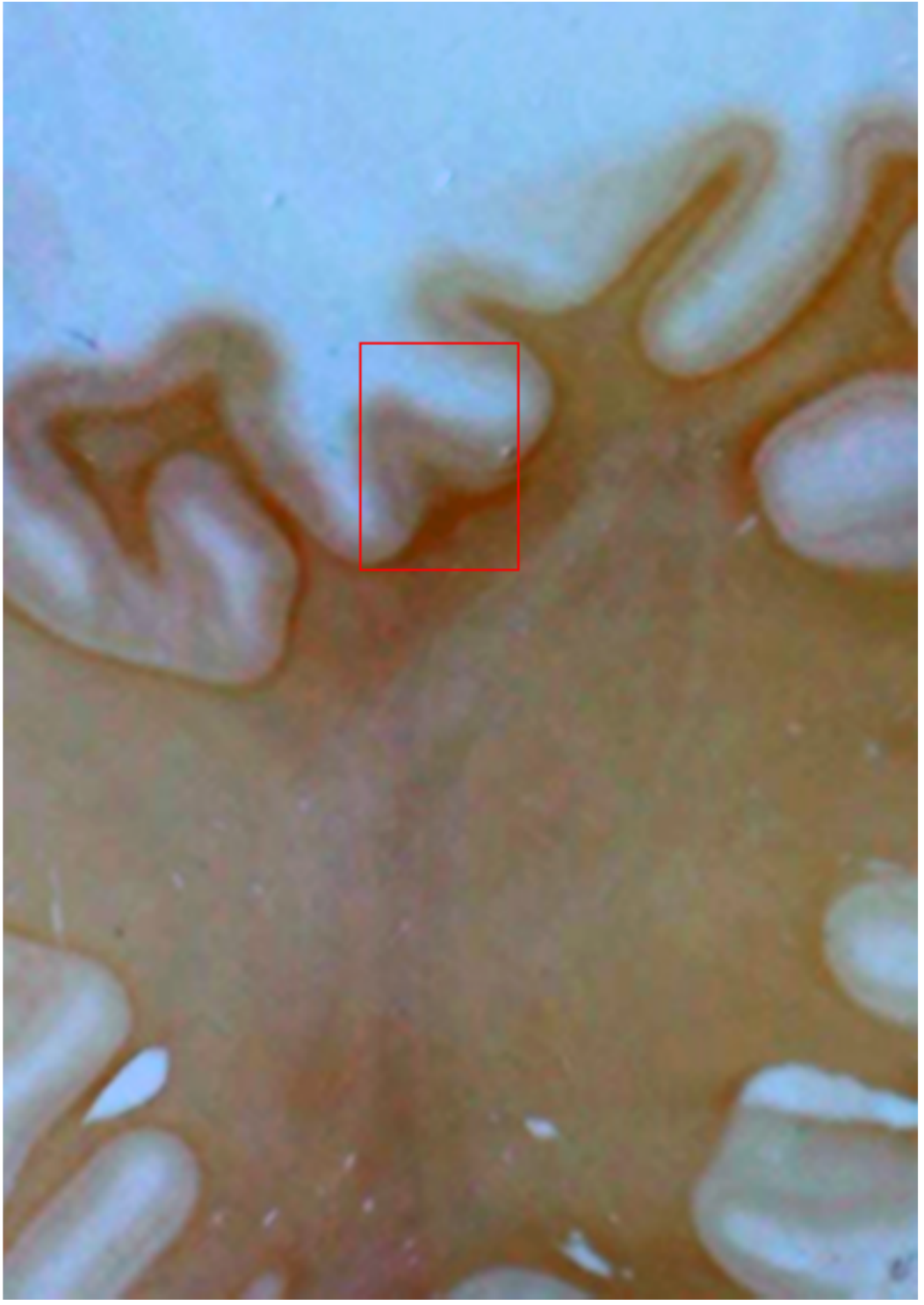}\llap{\includegraphics[width=0.08\textwidth]{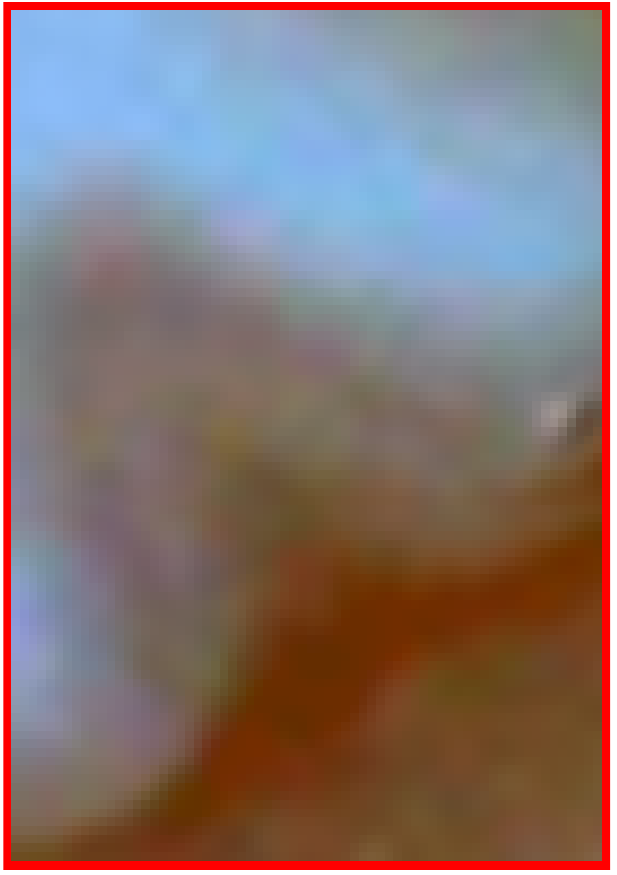}}} &
{\includegraphics[height=0.23\textwidth]{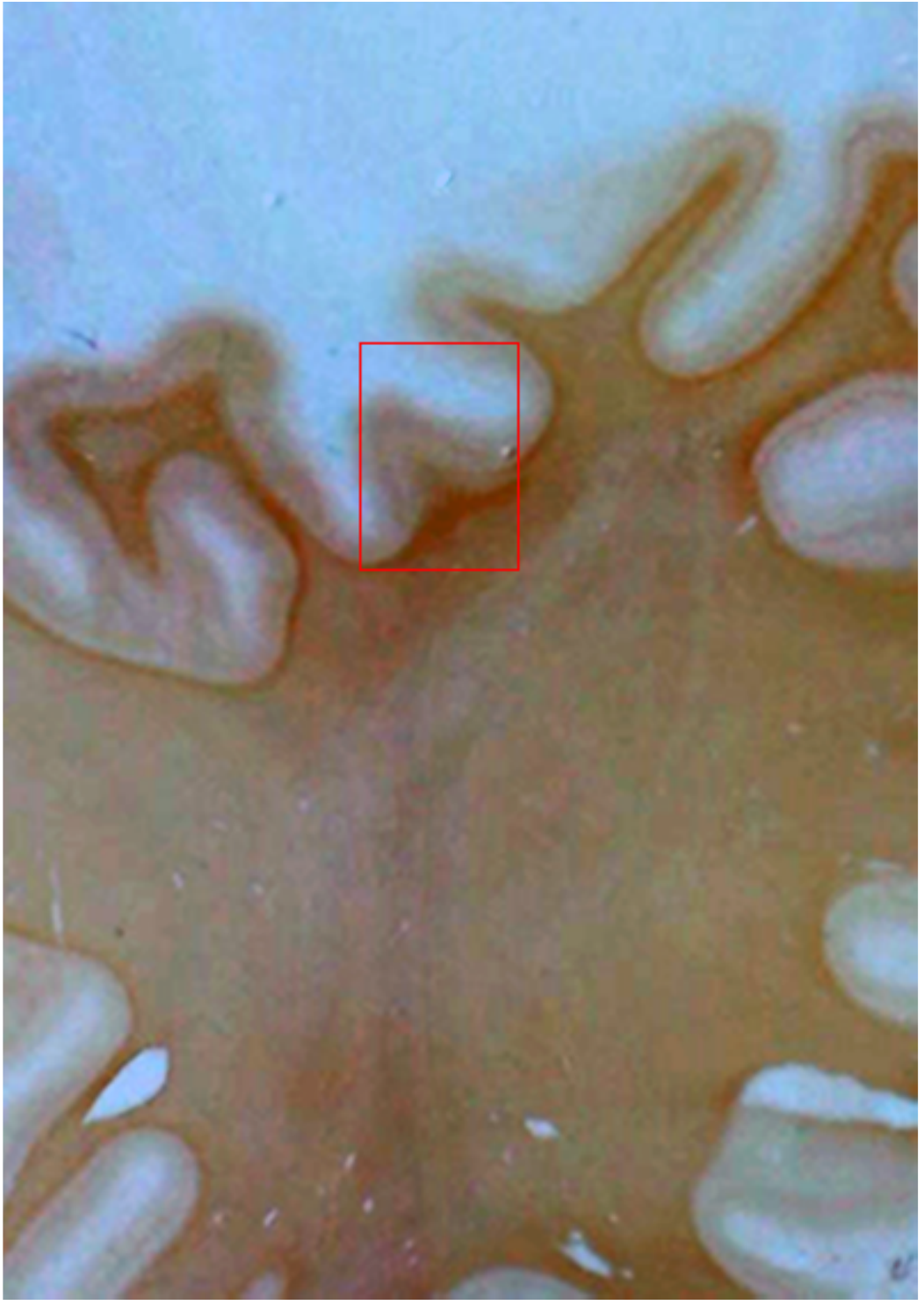}\llap{\includegraphics[width=0.08\textwidth]{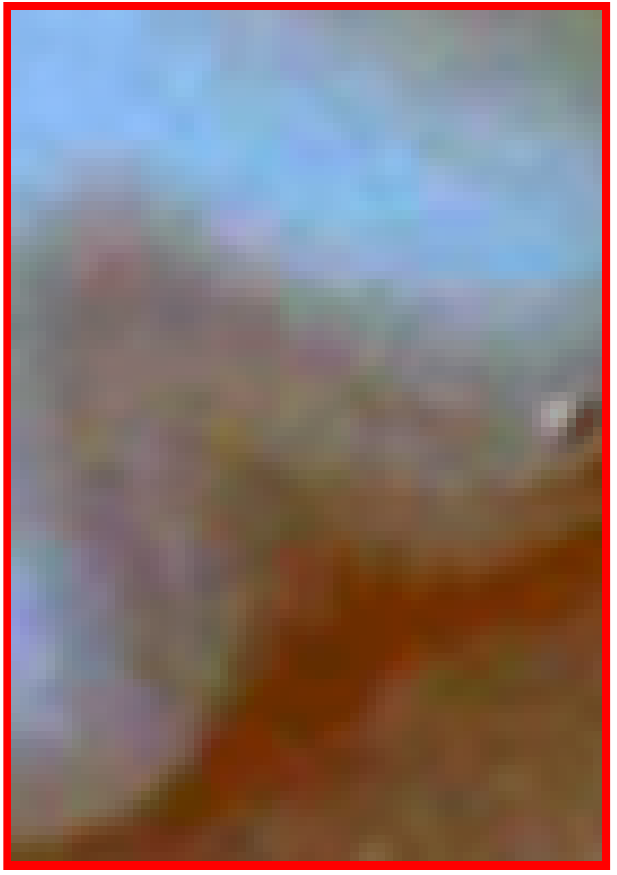}}} &
{\includegraphics[height=0.23\textwidth]{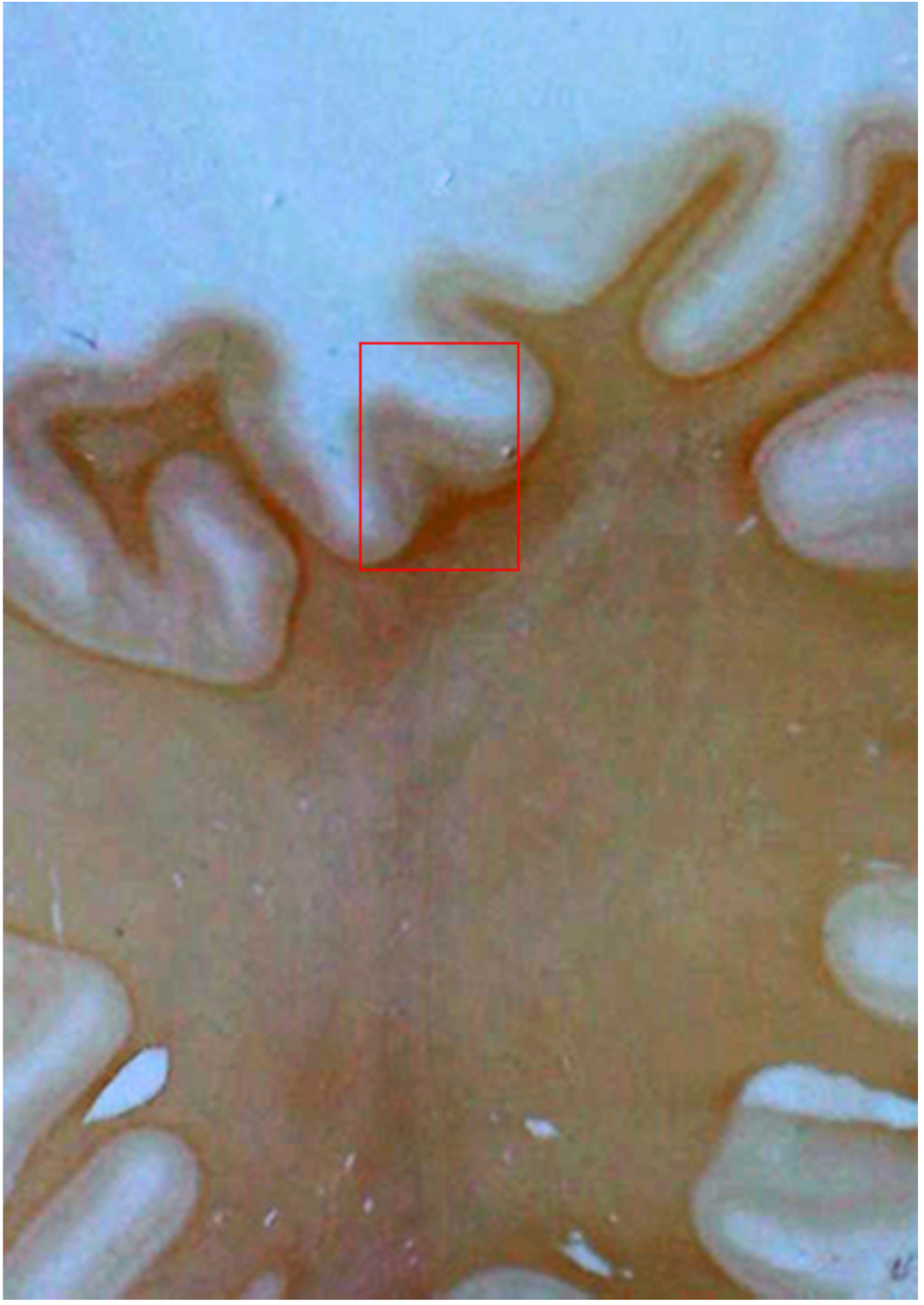}\llap{\includegraphics[width=0.08\textwidth]{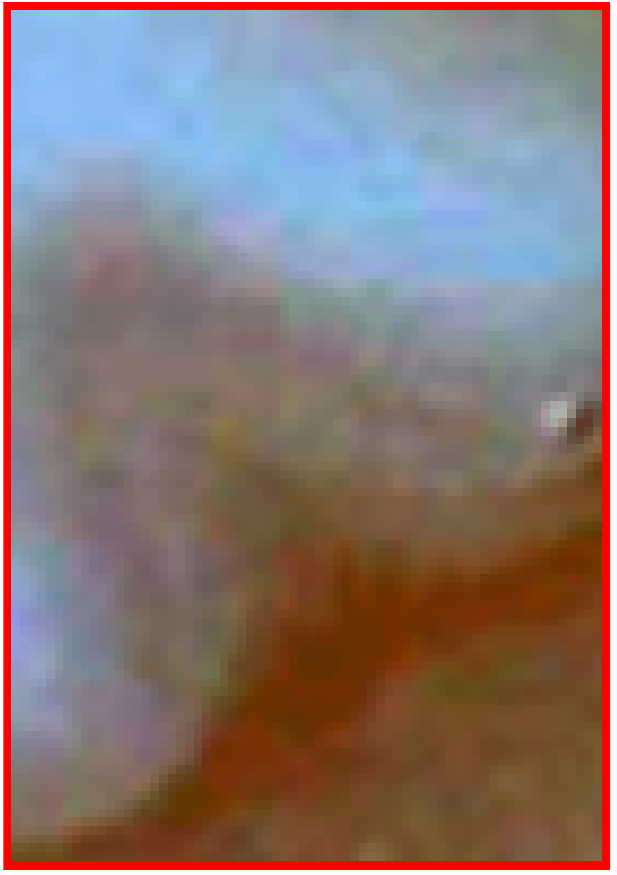}}} &
{\includegraphics[height=0.23\textwidth]{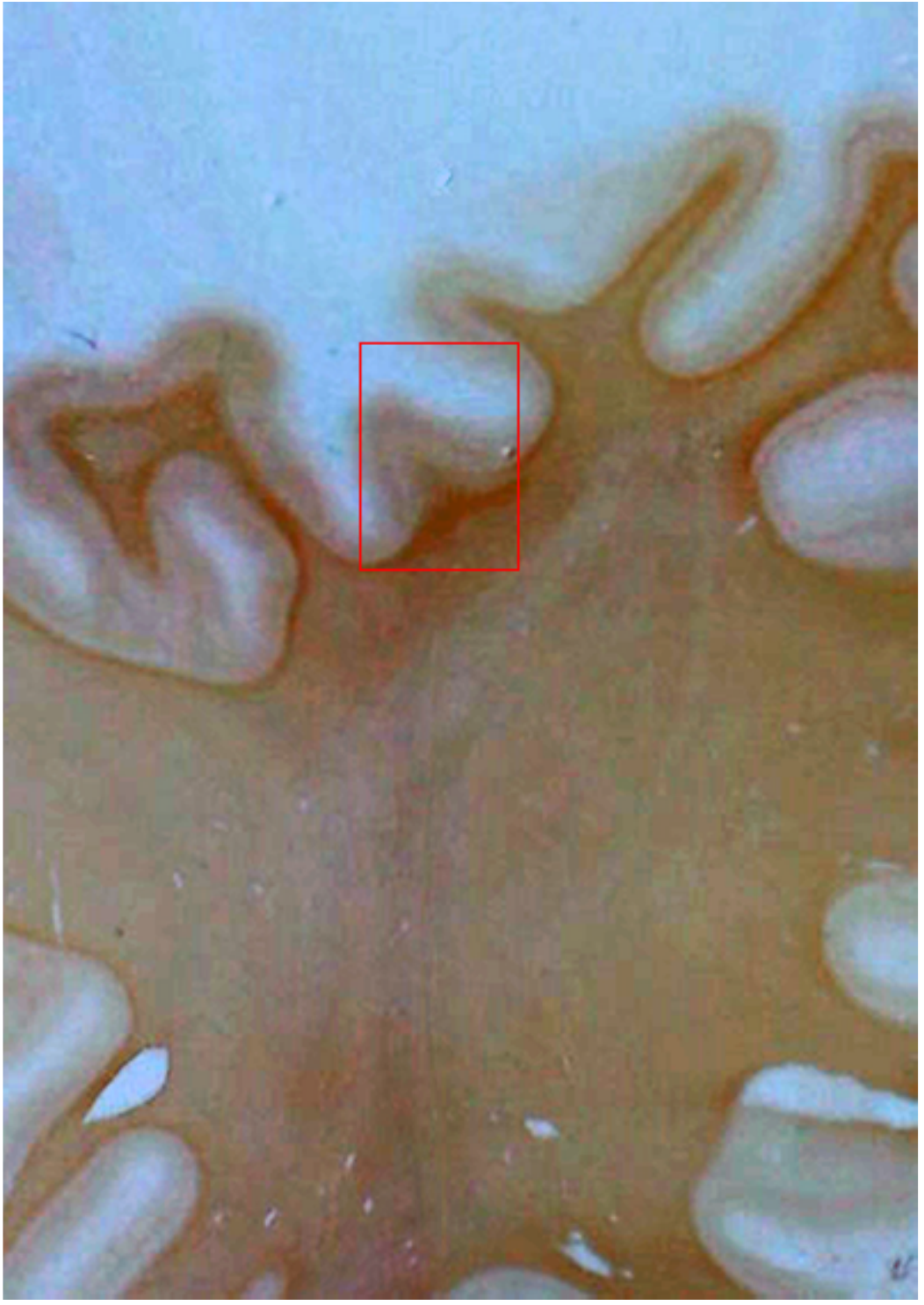}\llap{\includegraphics[width=0.08\textwidth]{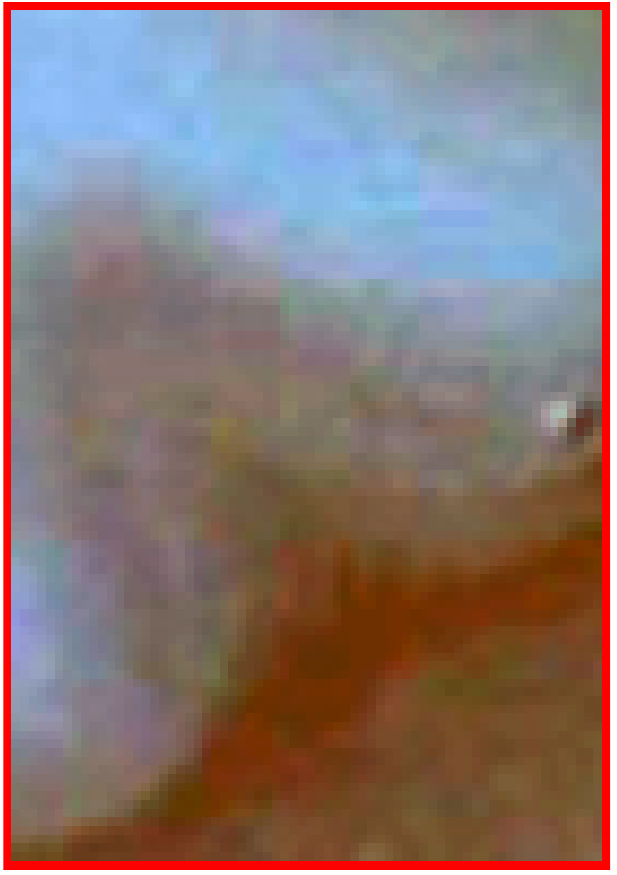}}} &
{\includegraphics[height=0.23\textwidth]{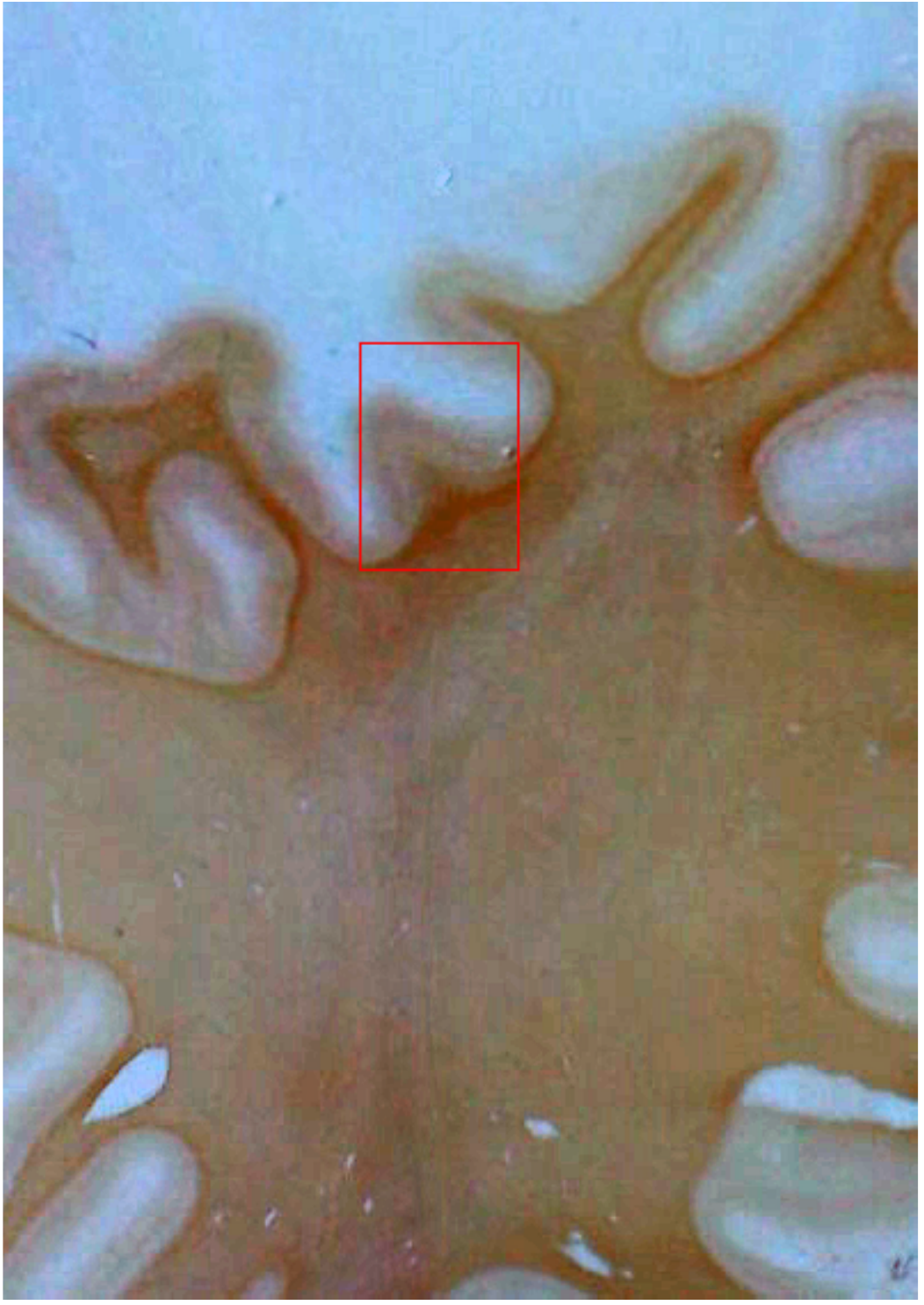}\llap{\includegraphics[width=0.08\textwidth]{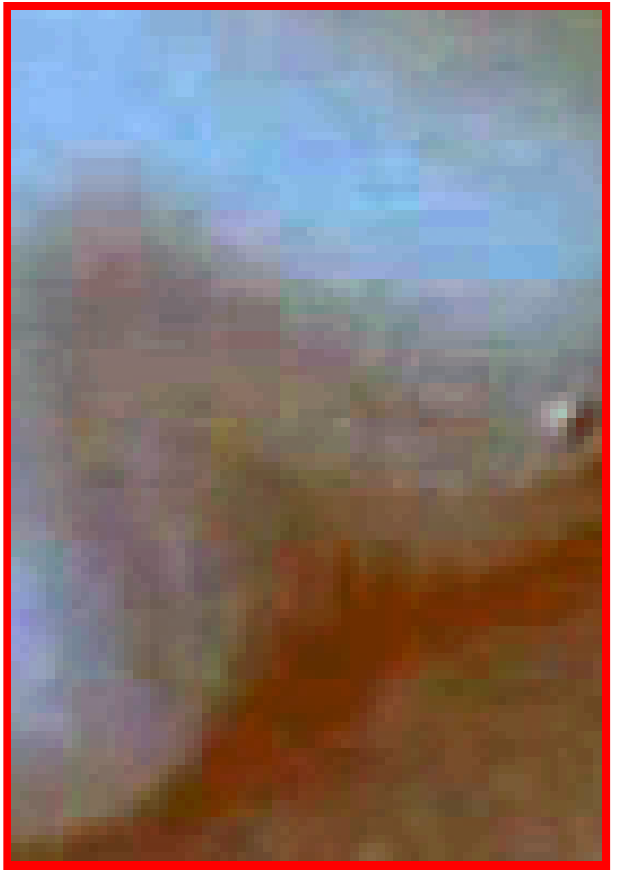}}} \\ \cline{1-6}
{\hspace{-.05in}\begin{sideways} \hspace{.15in} staggered,~$l=5$ \end{sideways}\hspace{-.05in}}   & 
{\includegraphics[height=0.23\textwidth]{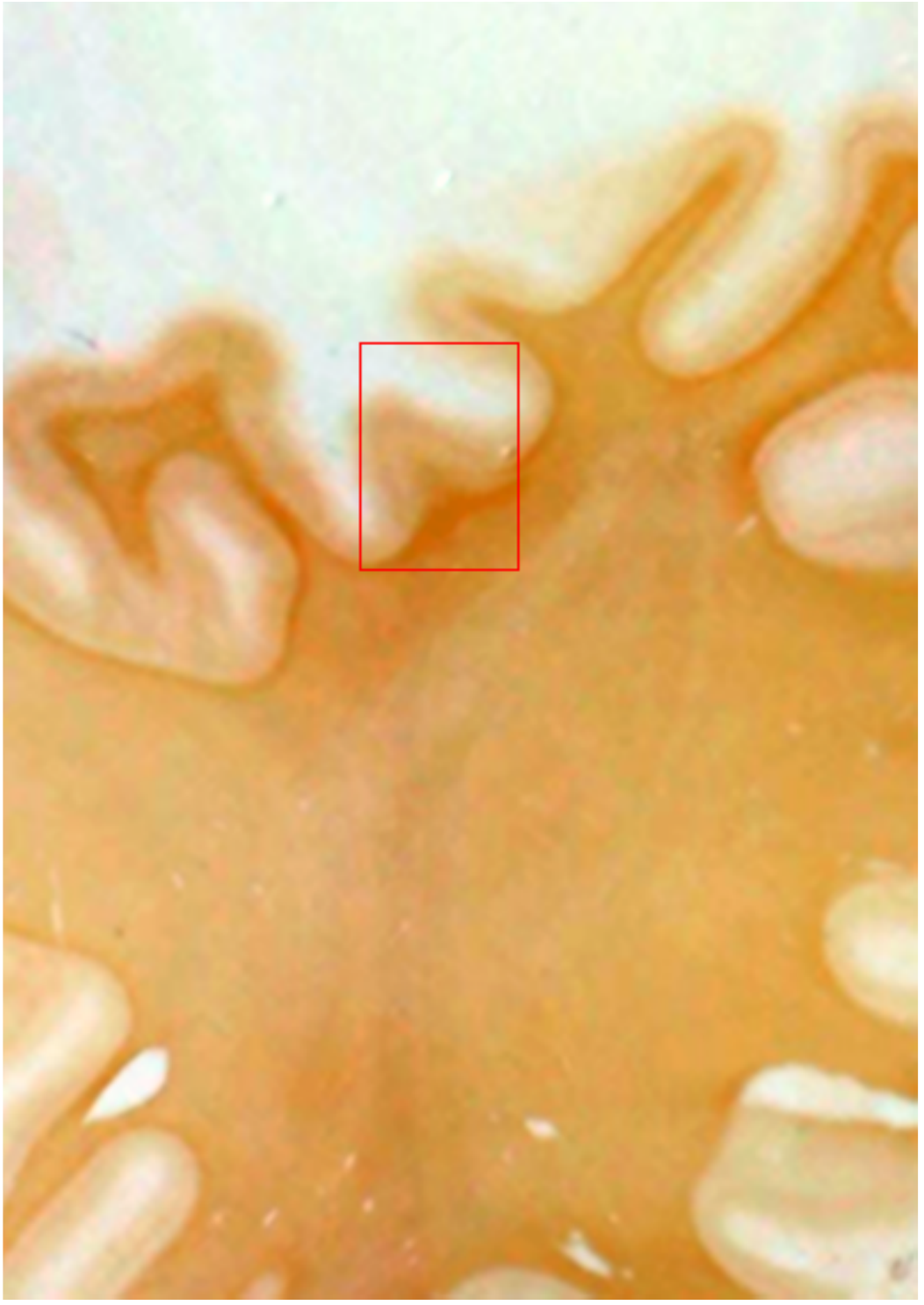}\llap{\includegraphics[width=0.08\textwidth]{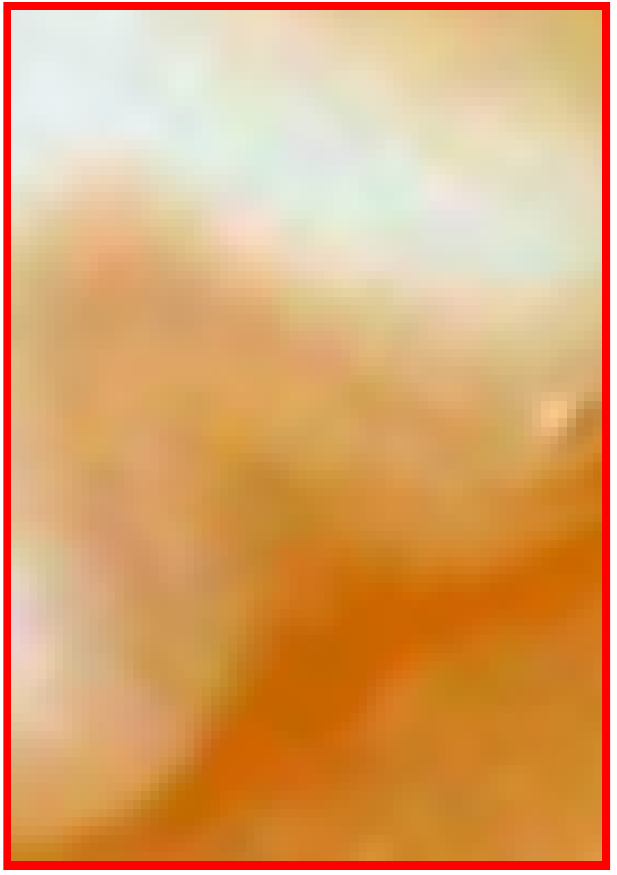}}} &
{\includegraphics[height=0.23\textwidth]{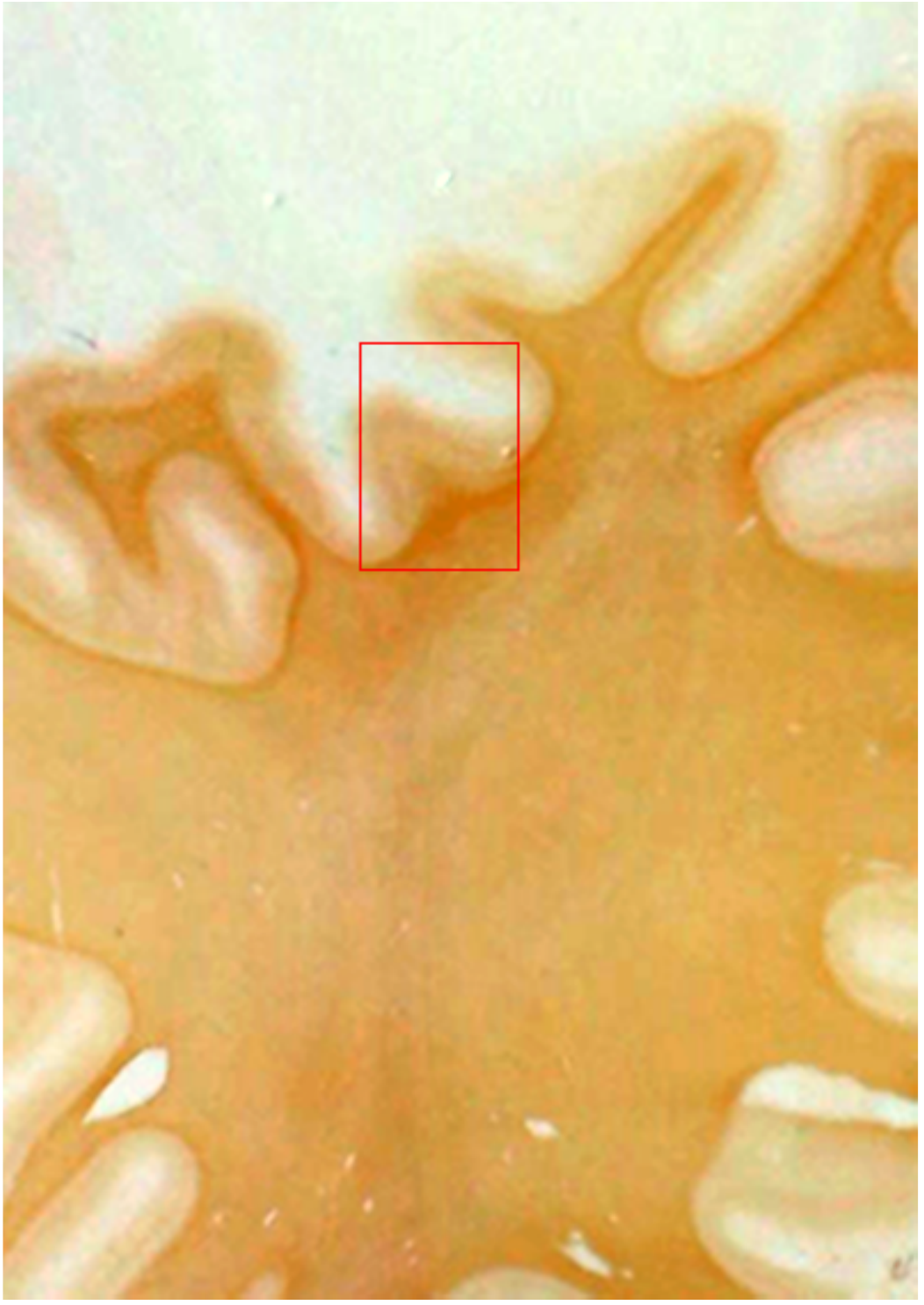}\llap{\includegraphics[width=0.08\textwidth]{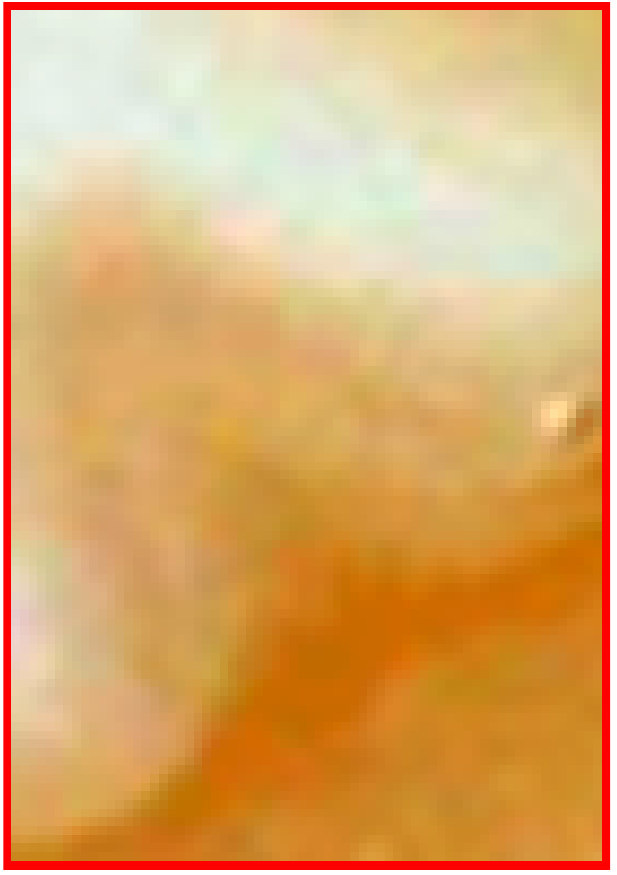}}} &
{\includegraphics[height=0.23\textwidth]{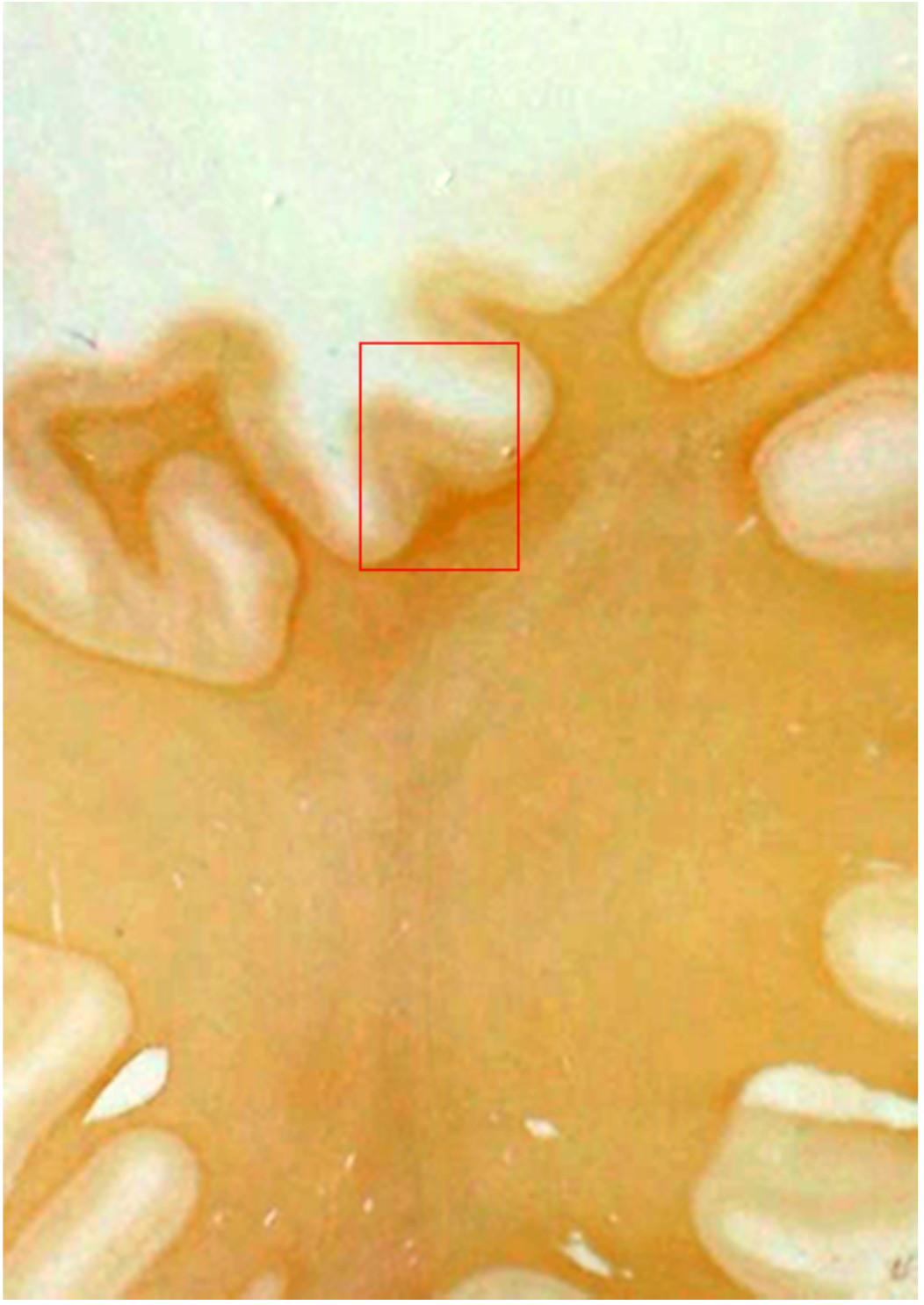}\llap{\includegraphics[width=0.08\textwidth]{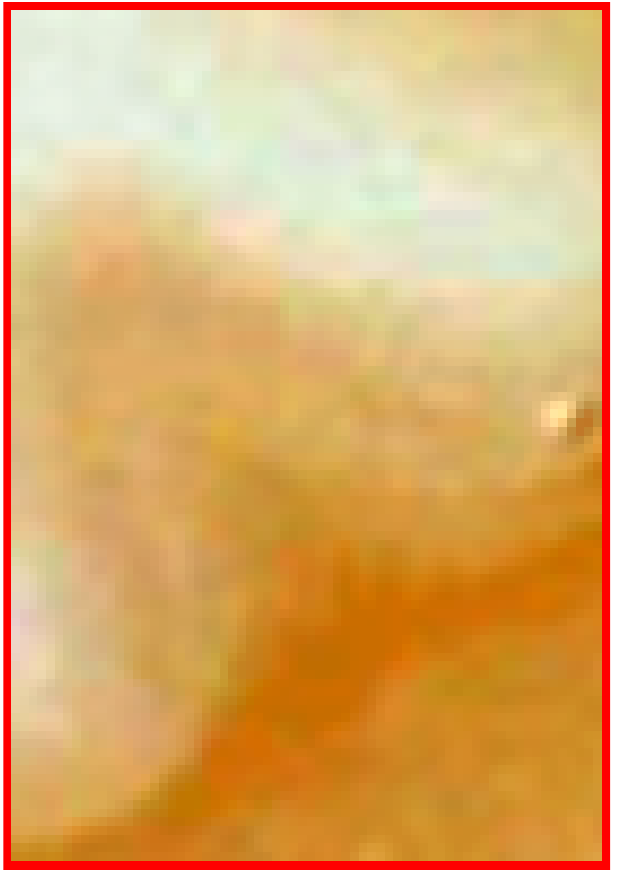}}} &
{\includegraphics[height=0.23\textwidth]{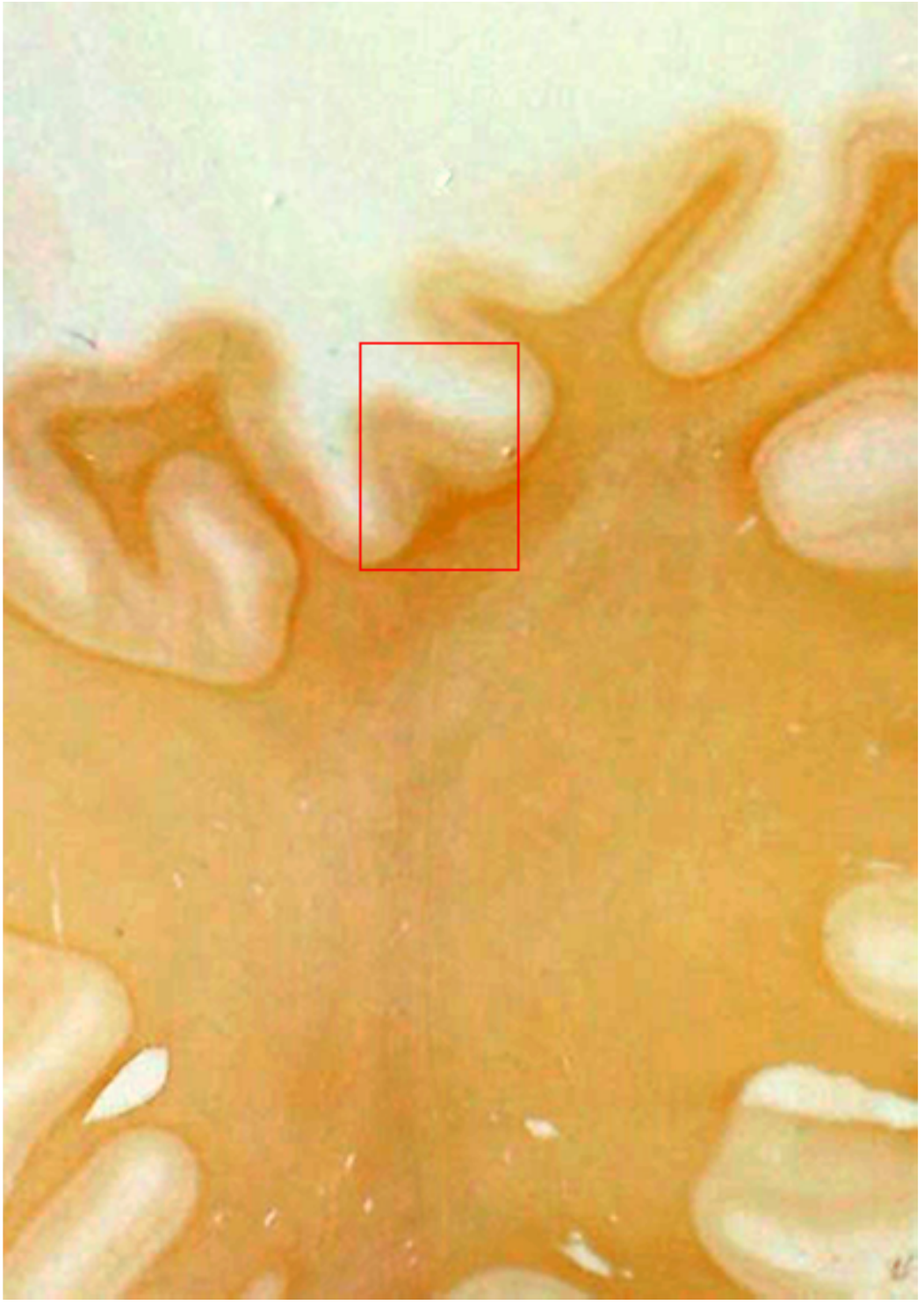}\llap{\includegraphics[width=0.08\textwidth]{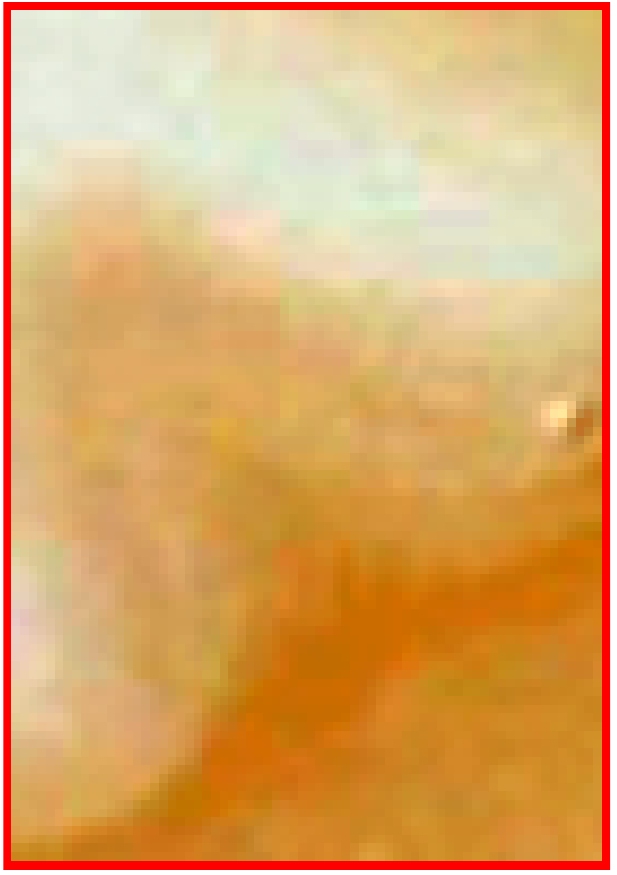}}} &
{\includegraphics[height=0.23\textwidth]{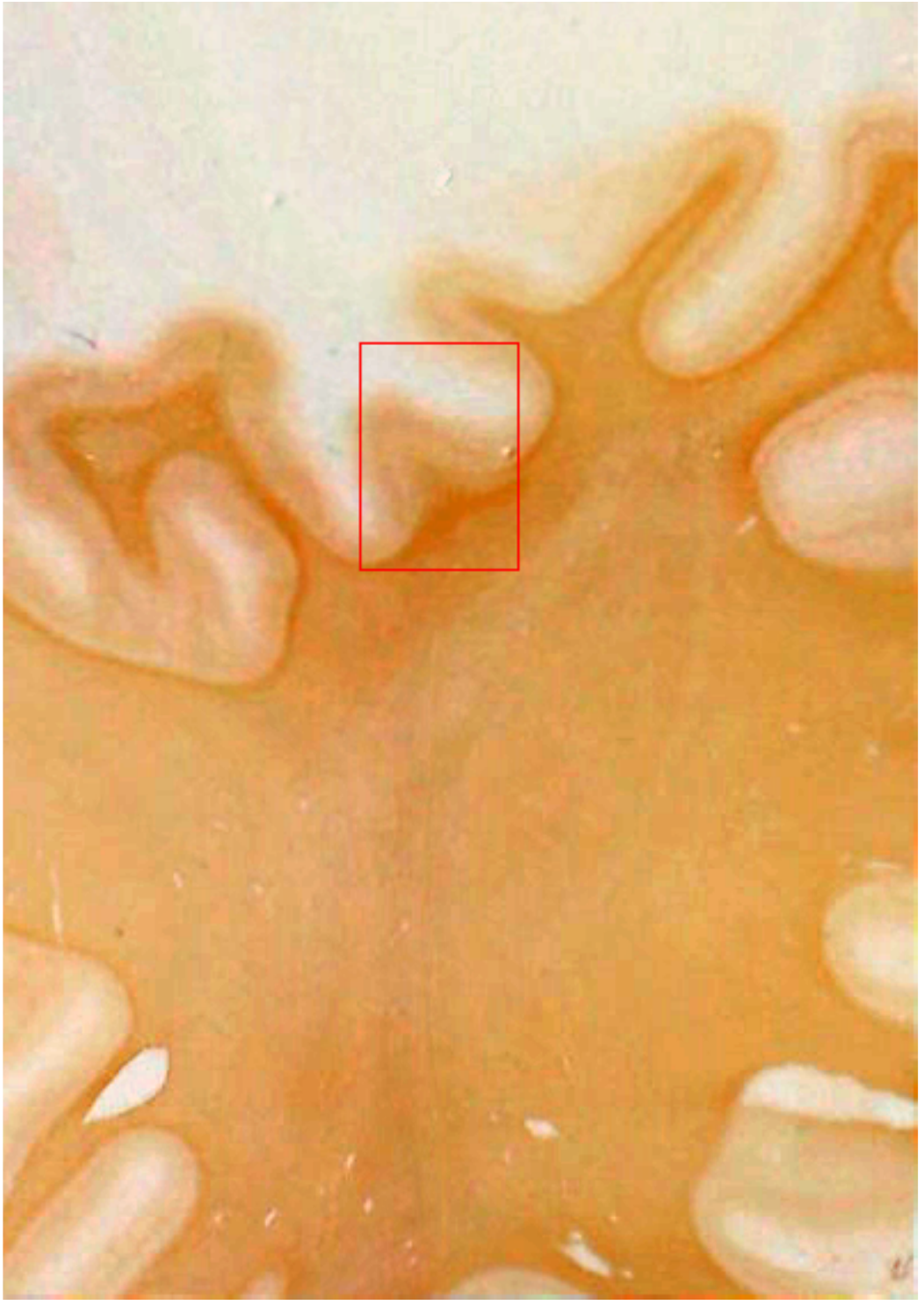}\llap{\includegraphics[width=0.08\textwidth]{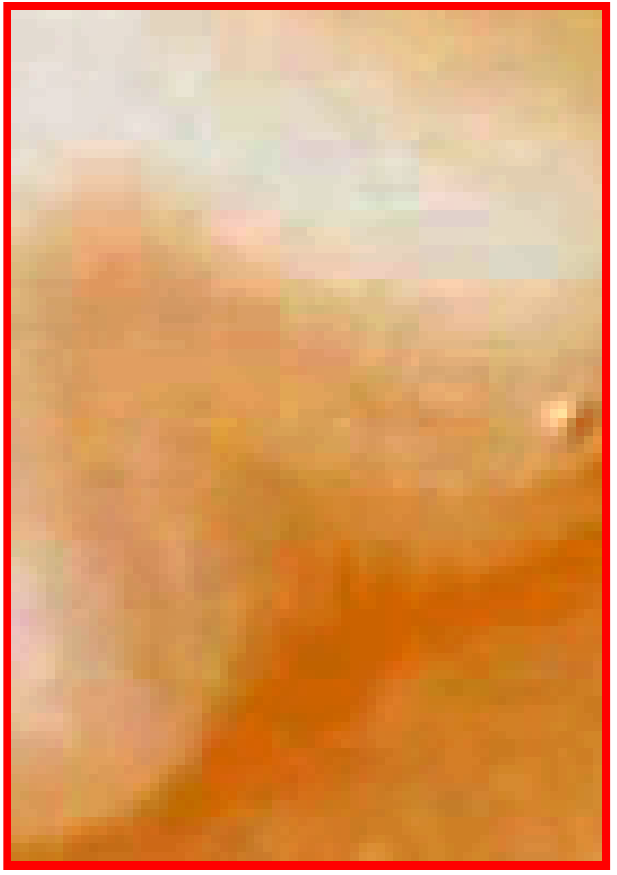}}} \\ \cline{1-6}
\end{tabular}
\end{table}

The result of reconstruction using all four possibilities are shown in \ref{table_preview_image_stitching}. In the case of centralized scheme with $l=1$ (first row of the table) the recovery contain severe artifacts such as unbalanced gray values across color channels and oscillating perturbations on the fullband recovery ($P=9$) known as the chequerboard effect \cite{gambaruto2015processing}. This comply with the discussions made in \ref{sec_stability_eigenvlaues} where case--II matrix is sensitive to such artifacts. One way to improve this result is to increase the polynomial order similar to \cite{o2008algebraic, o2012framework, harker2015regularized} where white balance issue is greatly enhanced. However, the oscillations still remain at the recovery stage. To prevent the oscillation, the staggered scheme should be used by deploying Case--I introduced in \ref{sec_stability_eigenvlaues}. The third and forth row of \ref{table_preview_image_stitching} elaborates on this case. As it shown, the oscillations are removed on fullband recovery. However, low polynomial order $l=1$ at third row contains white balance issue. Whereas both latter problems are fixed by increasing the polynomial order shown in fourth row of the \ref{table_preview_image_stitching}. The columns of this table correspond to recovering the gradient fields approximated by different cutoff levels. In fact by setting a reasonable cutoff range many perturbation artifacts are canceled in the gradient field. As a result, the oscillations in the recovery stage are mainly disappeared for centralized cases. The downside of this approach, however, is the reconstructed images are smooth and may not be pleasant in terms of sharp quality presentation.

\section{Concluding remarks}\label{sec_conclusion}
We have proposed a generalized numerical framework to derive comprehensive lowpass/fullband derivative kernels in a closed form solution based on the maximally flat design technique. We extended our design to matrix operation to decompose the derivatives of tensor data such as image and volumes. Four possible cases of derivative matrices are designed which encode high accuracy boundary formulation with arbitrary derivative order, polynomial accuracy, and lowpass/fullband design. The stability of these matrices were analyzed where significant improvement were achieved for odd derivative matrix with staggered formulation compared to the existing centralized solutions in the literature. The utility of this matrix were tested in gradient surface recovery and image stitching applications. The overall results in both applications suggest that the proposed matrix provides highly robust recoveries compared to the existing numerical solvers in the literature.

MaxPol package provides comprehensive solution for numerical differentiation that can be of interest to broad audiences in engineering and science. They may emerge as solutions to discrete operators for image diffusion and variational regularization problems for restoration purposes applied in image in-painting, enhancement/de-noising, and deconvolution problems. The package can be also utilized as a numerical solver to the PDE problems applied in fluid mechanics, acoustics, and wave equations. Another possible interest could be in dynamic control systems to estimate state variables in a direct fashion by avoiding numerical integration due to noise sensitivity issues.



\section*{Acknowledgments}
The authors would like to greatly thank Huron Digital Pathology Inc. for providing valuable discussion and database for the developing parts of the experiments conducted in this paper.
\bibliographystyle{siamplain}
\bibliography{references}

\end{document}